\documentclass[11pt, reqno]{amsart}

\usepackage{amsmath}
\usepackage{amssymb}
\usepackage{latexsym}
\usepackage{amsthm}
\usepackage{mathrsfs}
\usepackage{color}
\usepackage{amsfonts}
\usepackage{array}
\usepackage{setspace}
\usepackage{mathtools}
\usepackage{tikz}
\usepackage{tikz-cd}
\usepackage{quiver}
\usepackage{hyperref}
\usetikzlibrary{arrows}
\usetikzlibrary{knots,patterns}
\usepackage{graphicx}
\usepackage[all]{xy}
\usepackage{mathrsfs} 
\usepackage[marginratio=1:1,height=584pt,width=500pt,tmargin=117pt]{geometry}
\usepackage{enumerate}
\usepackage{enumitem}

\usetikzlibrary{decorations.markings}
\tikzset{double line with arrow/.style args={#1,#2}{decorate,decoration={markings,%
mark=at position 0 with {\coordinate (ta-base-1) at (0,1pt);
\coordinate (ta-base-2) at (0,-1pt);},
mark=at position 1 with {\draw[#1] (ta-base-1) -- (0,1pt);
\draw[#2] (ta-base-2) -- (0,-1pt);
}}}}
\tikzset{Equal/.style={-,double line with arrow={-,-}}}

\numberwithin{itemcounter}{subsection}

\makeatletter
\newcommand{\nolisttopbreak}{\vspace{\topsep}\nobreak\@afterheading}
\makeatother

\makeatletter
\newcommand{\rellab}[2]{%
   \protected@write \@auxout {}{\string \newlabel {#1}{{#2}{\thepage}{#2}{#1}{}} }%
   \hypertarget{#1}{#2}
}
\makeatother

\def\det{\mathrm{det}}

\def\Z{\mathbb{Z}}
\def\Q{\mathbb{Q}}

\def\BBH{\Lambda}
\def\H{\mathrm{H}}
\def\IH{\mathrm{IH}}
\def\K{\mathrm{K}}
\def\N{\mathbb{N}}

\def\C{\mathbb{C}}
\def\G{\mathbb{G}}

\def\g{\mathfrak{g}}
\def\U{\mathbf{U}}
\def\A{\mathbf{A}}
\def\wotimes{{{\widehat{\otimes}}}}

\def\pmin{p^{\text{min}}\hspace{-0.01in}}
\def\pmax{p^{\text{max}}\hspace{-0.01in}}

\def\NS{\mathrm{NS}}
\def\red{\mathrm{red}}
\def\rd{\mathrm{rd}}

\def\Pic{\mathrm{Pic}}

\def\BB{\mathrm{B}}
\def\CC{\mathrm{C}}

\def\HH{\mathrm{H}}

\def\BM{{\mathrm{BM}}}

\def\calA{\mathcal{A}}
\def\calB{\mathcal{B}}
\def\calC{\mathcal{C}}

\def\calF{\mathcal{F}}
\def\calH{\mathcal{H}}
\def\calG{\mathcal{G}}
\def\calJ{\mathcal{J}}
\def\calK{\mathcal{K}}
\def\calL{\mathcal{L}}
\def\calJ{\mathcal{J}}

\def\calN{\mathcal{N}}
\def\calP{\mathcal{P}}

\def\calX{\mathcal{X}}
\def\calY{\mathcal{Y}}
\def\calZ{\mathcal{Z}}

\def\bfA{\mathbf{A}}
\def\bfH{\mathbf{H}}

\DeclareMathOperator{\hdg}{\mathbf{h}}

\def\frakg{\mathfrak{g}}

\def\calM{{{\mathfrak{M}}}}

\def\frakU{\mathfrak{U}}

\def\HDT{{\operatorname{HDT}\nolimits}}
\def\BPS{{\operatorname{BPS}\nolimits}}

\def\HN{{\operatorname{HN}\nolimits}}
\def\an{{\operatorname{an}\nolimits}}
\def\Perv{{\operatorname{Perv}\nolimits}}
\def\CC{{\operatorname{CC}\nolimits}}

\def\ch{{\operatorname{ch}\nolimits}}
\def\id{{\operatorname{id}\nolimits}}
\def\Hom{{\operatorname{Hom}\nolimits}}
\def\supp{{\operatorname{supp}\nolimits}}

\def\Ext{{\operatorname{Ext}\nolimits}}
\def\Exp{{\operatorname{Exp}\nolimits}}
\def\ker{{\operatorname{ker}\nolimits}}

\def\deg{{\operatorname{deg}\nolimits}}
\def\rk{{\operatorname{rk}\nolimits}}
\def\Ind{{\operatorname{Ind}\nolimits}}
\def\Im{{\operatorname{Im}\nolimits}}

\def\Aff{{\operatorname{Aff}\nolimits}}
\def\Perf{{\operatorname{Perf}\nolimits}}
\def\Spc{{\operatorname{Spc}\nolimits}}
\def\CAlg{{\operatorname{CAlg}\nolimits}}

\def\gr{{\operatorname{gr}\nolimits}}

\def\cl{{\operatorname{cl}\nolimits}}
\def\c{{\operatorname{c}\nolimits}}
\def\res{{\operatorname{res}\nolimits}}
\def\m{{\operatorname{m}\nolimits}}
\def\snilp{{\operatorname{snilp}\nolimits}}
\def\prim{{\operatorname{prim}\nolimits}}
\def\pr{{\operatorname{pr}\nolimits}}

\def\ad{{{\operatorname {ad}}}}
\DeclareMathOperator{\Spec}{Spec}
\def\Sym{{{\operatorname {Sym}}}}
\def\prim{{{\operatorname {prim}}}}

\DeclareMathOperator{\vdim}{vdim}
\def\vd{{{{\operatorname{vd}}}}}
\def\nil{{{{\operatorname{nil}}}}}
\def\snil{{{{\operatorname{snilp}}}}}
\def\DT{{{{\operatorname{DT}}}}}
\def\Vect{{{{\mathsf{Vect}}}}}
\def\p{{{\operatorname{p}}}}
\def\muss{{\mu-{{{\operatorname{ss}}}}}}
\DeclareMathOperator{\Map}{Map}

\newcommand{\BGM}{\mathrm{B}\mathbb{G}_{\mathrm{m}}}
\newcommand{\GM}{\mathbb{G}_{\mathrm{m}}}
\newcommand{\Grad}{\Grad}
\newcommand{\Filt}{\Filt}
\newcommand{\wM}{\widetilde{M}}
\newcommand{\wev}{\widetilde{\ev}}
\newcommand{\wgr}{\widetilde{\gr}}

\newcommand{\Pro}{\mathsf{Pro}}
\newcommand{\bigboxtimes}{\mathop{\scalebox{1.5}{$\boxtimes$}}}
\newcommand{\Lim}[1]{{\displaystyle{\mathop{``\mathrm{lim}"}_{#1}}}}

\newcommand{\RGpro}{R \Gamma_\pro}
\newcommand{\RGproBM}{ \RGpro^{\!\!\!\!\!\!\!\!\!\mathrm{BM}}}
\newcommand{\cBPS}{\mathcal{BPS}}
\newcommand{\uum}{\underline{m}}

\newcommand{\uui}{{\underline{i}}}

\newcommand{\uunu}{{\underline{\nu}}}
\newcommand{\uugamma}{{\underline\gamma}}
\newcommand{\uualpha}{{\underline\alpha}}

\newcommand{\uuell}{{\underline\ell}}

\DeclareMathOperator{\Hilb}{Hilb}

\def\MHM{{\mathsf{MHM}}}
\def\MHMmon{{\mathsf{MHM^{mon}}}}
\def\Dhmon{{ \Dh^{\mathsf{mon}}}}
\def\Dhcmon{{ \Dhc^{\mathsf{mon}}}}
\def\Dhcmonb{{ \Dhc^{\mathsf{mon}, \mathsf{b}}}}
\newcommand{\Dh}{{\mathsf{D}}_{\mathsf{H}}} 
\newcommand{\Dhc}{{\mathsf{D}}_{\mathsf{H}, \mathsf{c}}} 

\newcommand{\Dhcb}{{\mathsf{D}}_{\mathsf{H}, \mathsf{c}}^{\mathsf{b}}}
\newcommand{\Dcb}{{\mathsf{D}}_{\mathsf{c}}^{\mathsf{b}}}
\def\RGpro{{{R\Gamma_{\mathrm{{pro}}}}}}

\def\calF{\mathcal{F}}
\def\calG{\mathcal{G}}
\def\calM{\mathcal{M}}
\def\calE{\mathcal{E}}
\def\calO{\mathcal{O}}
\def\calU{\mathcal{U}}
\def\calV{\mathcal{V}}

\def\D{\mathbb{D}}
\def\G{\mathbb{G}}
\def\R{\mathbb{R}}
\def\A{\mathbb{A}}
\def\L{\mathbb{L}}
\def\Q{\mathbb{Q}}
\def\N{\mathbb{N}}
\def\Z{\mathbb{Z}}
\def\G{\mathbb{G}}

\def\pro{{{\mathrm{pro}}}}
\def\Filt{{{\mathrm{Filt}}}}
\def\Grad{{{\mathrm{Grad}}}}
\def\ev{\mathrm{ev}}
\def\rk{\mathrm{rk}}
\def\Td{\mathrm{Td}}
\def\Bun{\mathrm{Bun}}
\def\ch{\mathrm{ch}}
\def\bfH{\mathbf{H}}
\def\Z{\mathbb{Z}}

\def\ex{{\mathrm{ex}}}

\def\Ext{\mathrm{Ext}}
\def\Hom{\mathrm{Hom}}
\def\taut{\mathrm{taut}}
\def\vir{\mathrm{vir}}
\def\ss{{\mathrm{ss}}}
\def\st{{\mathrm{st}}}
\def\sst{{{\text{-}\mathrm{ss}}}}
\def\Coh{{{\mathsf{Coh}}}}
\def\Kthree{{{\mathsf{K3}}}}
\def\Abelian{{{\mathsf{Ab}}}}

\theoremstyle{plain}
\newtheorem*{thm}{Theorem}

\newtheorem{theorem}{Theorem}[section]
\newtheorem{lemma}[theorem]{Lemma}
\newtheorem{lemma-definition}[theorem]{Lemma-Definition}
\newtheorem{definition-lemma}[theorem]{Definition-Lemma}

\newtheorem{proposition}[theorem]{Proposition}
\newtheorem{conjecture}[theorem]{Conjecture}
\newtheorem{corollary}[theorem]{Corollary}
\theoremstyle{definition}
\newtheorem{definition}[theorem]{Definition}
\theoremstyle{remark}
\newtheorem{remark}[theorem]{Remark}
\newtheorem{example}[theorem]{Example}
\numberwithin{equation}{section}

\def\BGM{{{\mathrm{B}\G_m}}}

\newcommand{\kinjo}[1]{\color{magenta}{{\bf kinjo:} \small{#1}}}

\newcommand{\orange}[1]{{\color{orange} #1}} 

\title[Hecke operators on symplectic surfaces and $\chi$-independence]{Hecke operators on symplectic surfaces and $\chi$-independence}
\author{Ben Davison, Lucien Hennecart, Tasuki Kinjo, Olivier Schiffmann and Eric Vasserot}
\date{}

\address{B. Davison: School of Mathematics and Maxwell Institute for Mathematical Sciences, University of Edinburgh, Edinburgh EH9 3FD, United Kingdom}
\email{ben.davison@ed.ac.uk}

\address{L. Hennecart: LAMFA, Universit\'e de Picardie Jules Verne, 80000 Amiens, France}
\email{lucien.hennecart@u-picardie.fr}

\address{T. Kinjo: Research Institute for Mathematical Sciences, Kyoto University, Kyoto 606-8502, Japan}
\email{tkinjo@kurims.kyoto-u.ac.jp}

\address{O. Schiffmann: Laboratoire de Mathématiques d'Orsay, Bât 307, Faculté des Sciences d'Orsay, Université Paris-Saclay, 91405 Orsay, France AND Simion Stoilow Institute of Mathematics, P.O. Box 1-764, RO-014700 Bucharest, Romania}
\email{olivier.schiffmann@universite-paris-saclay.fr}

\address{E. Vasserot: IMJ-PRG, UMR7586, Universit\'e Paris-Cit\'e, 75013 Paris, France}
\email{vasserot@imj-prg.fr}

\begin{document}

\maketitle

\begin{abstract}
We prove Toda’s $\chi$-independence conjecture for the BPS cohomology of moduli spaces of one-dimensional sheaves on quasi-projective symplectic surfaces, relative to the Chow variety. We also identify the BPS Lie algebra associated with one-dimensional Mukai vectors with the subspace of tautological classes, giving an extension of Markman’s tautological generation theorem from primitive to arbitrary Mukai vectors.

The main structural input is a bialgebra structure on the cohomological Hall algebra of coherent sheaves on a quasi-projective symplectic surface $S$. The coproduct is obtained, by dimensional reduction, from a factorization coproduct for 3d cohomological Hall algebras, and gives rise to a global BPS Lie algebra attached to the stack of coherent sheaves on $S$.

The link between this structure and the applications to $\chi$-independence and tautological generation is provided by Hecke operators on BPS cohomology, which modify one-dimensional sheaves by zero-dimensional quotients. To make this construction work, we prove that there is an identification between the affinized BPS cohomology of the semistable locus and the primitive part of the coproduct on the entire moduli stack.
\end{abstract}

\setcounter{tocdepth}{1}
\tableofcontents

\section{Introduction}

Let $(S,H)$ be a polarized complex projective K3 or Abelian surface. 
For each effective Mukai vector  $\nu \in \H^{\ev}(S,\Z)$, let $M_S(\nu)$ be the moduli space of polystable coherent sheaves of Mukai 
vector $\nu$.  The geometry of $M_S(\nu)$ has attracted a lot of study, starting with the seminal work of 
Mukai \cite{MukaiInvent84}, who showed that if $\nu$ is {primitive} and $H$-generic, 
then $M_S(\nu)$ is a smooth projective symplectic 
variety of dimension $\nu\cdot \nu + 2$. 
Subsequently, Yoshioka described $M_S(\nu)$ explicitly in \cite{Yoshioka}
when the class $\nu$ is primitive and $H$-generic. For $S$ Abelian, the variety $M_S(\nu)$ is either an Abelian 
surface if $\nu \cdot \nu =0$, or it is deformation equivalent to $\widehat{S}\times \text{Hilb}^{\nu \cdot \nu / 2}(S)$ 
when $\nu \cdot \nu >0$, where $\widehat{S}$ is the dual Abelian surface. For $S$ a K3 surface, $M_S(\nu)$ is deformation equivalent to $\text{Hilb}^{\nu \cdot \nu / 2+1}(S)$ when $\nu$ is either of 
positive rank or an ample curve class. In particular, the Betti numbers and the mixed Hodge polynomial of $M_S(\nu)$ are then 
fully determined, thanks to \cite{Gottsche-Soergel}.

If $\nu=n \nu_\pr$ is a multiple of a primitive vector $\nu_\pr$, with $n>1$, i.e., in the presence of strictly semistable coherent sheaves, the moduli 
space $M_S(\nu)$ becomes singular in general. 
If $\nu \cdot \nu\leqslant 0$ then $M_S(\nu)$ is either a point or isomorphic to the 
symmetric power
$\Sym^n(S')$ of the Abelian (resp. K3) surface $S'=M_S(\nu_\pr)$. 
If $\nu \cdot \nu >0$, in only a few notable instances 
does $M_S(\nu)$ admit, as discovered by O'Grady, a symplectic resolution, see \cite{OGrady,KaledinLehnSorger}. 
One is thus led to consider the intersection homology of $M_S(\nu)$, to better understand its singularities.  Since stabilizers vary, in the presence of strictly semistable coherent sheaves, we are also naturally led to the study of the cohomology of the stack of semistable coherent sheaves. This stack has a quasi-smooth derived enhancement which is a 0-shifted symplectic stack, 
see \cite{BD21}.

One of the main outcomes of the present work is a computation of both the intersection homology of $M_S(\nu)$ and the mixed Hodge 
polynomial of $\calM^\ss_S(\nu)$ if $\nu=(0,\alpha,n\delta)$ with $n\delta\in \H^4(S,\Z)$ and $\alpha\in \H^2(S,\Z)$ the class of an effective one-cycle on $S$.  See \S \ref{formulae_sec} for exact formulae.  The theory for zero-dimensional sheaves (i.e. $\alpha=0$) is 
well-known.  We expect that one can reduce the case for general $\nu$ to the case of one-dimensional sheaves up to deformation equivalence using Fourier--Mukai transforms \cite{BrThesis, Orlov} and the minimal model program for moduli of sheaves on K3 surfaces \cite{BM14}, extending our work to arbitrary 
$\nu$. The case in which $\nu=(0,\alpha,n\delta)$ is especially important due to its close relationship to Gromov--Witten theory and to the general problem of 
counting curves on Calabi--Yau (=CY) threefolds.

Our approach  is through the theory of Donaldson--Thomas (=DT) invariants, BPS cohomology, BPS Lie 
algebras and cohomological Hall algebras (=CoHAs). Thanks to the technique of dimensional reduction, see \cite{BenDimRed,K}, our work in fact 
settles one important case of Toda's $\chi$-independence conjecture for Gopakumar--Vafa invariants \cite{T23} and provides a 
natural extension of Markman's theorem on the generation of the cohomology of moduli spaces of stable sheaves on symplectic 
surfaces from primitive to arbitrary Mukai vectors. Our main new conceptual tools are the construction of Hecke operators acting on the direct sum of the BPS cohomologies for all Mukai vectors corresponding to a fixed curve class, 
and the related construction of cocommutative coproducts on cohomological Hall algebras of not necessarily semistable coherent sheaves.

Before stating our results in detail in \S \ref{sec_main_theorems}, we briefly review the relevant notions. To improve the readability of this introduction, we do not include all 
the grading shifts.

\medskip

\subsection{Donaldson--Thomas theory and BPS cohomology}  Originally designed to provide a way to count holomorphic curves
or ideal sheaves on a 
projective complex 3CY manifold $X$, see~\cite{Th00},
Donaldson--Thomas theory may also be formulated for a smooth complex surface $S$. The main 
geometric object of study is the moduli stack $\calM^\ss_S(\nu)$ of Gieseker $H$-semistable coherent sheaves of Mukai 
vector $\nu$. This stack admits a good moduli space $p_S\colon \calM^\ss_S(\nu) \to M_S(\nu)$. 
As shown in \cite{D24}, for a symplectic surface $S$ the complex $(p_S)_*\mathbb{DQ}_{\calM^\ss_S(\nu)}$ is semisimple and is pure as a complex of mixed Hodge 
modules. Moreover, when $S$ is symplectic, there exists a perverse sheaf $\mathcal{BPS}_{M_S(\nu)}$, the {BPS sheaf}, such that
\begin{equation}\label{intro1}
\begin{split}
\Sym_{\oplus}\Big(\bigoplus_{\nu \neq 0} \cBPS_{M_S(\nu)} \otimes \H^*(\BGM,\Q) \Big) 
&\cong \bigoplus_{\nu} ({p}_S)_*\mathbb{DQ}_{\calM^\ss_S(\nu)},\\
\Sym_{\oplus}\Big(\bigoplus_{\nu \neq 0} \cBPS_{M_S(\nu)} \Big) 
&\cong  \bigoplus_{\nu} {}^{\text{p}}\calH^0\big(({p}_S)_* \mathbb{DQ}_{\calM^\ss_S(\nu)}\big).
\end{split}
\end{equation}
Here a reduced polynomial $p$ is fixed, and the sum ranges 
over the set $\Gamma_p$ of effective Mukai vectors $\nu$ whose reduced Hilbert polynomial is equal to $p$. The notation $\Sym_{\oplus}$ is explained in \S\ref{good}. The decomposition \eqref{intro1} is known as the
 {cohomological integrality} property and is proved in \cite{DHSM2},
 see also \cite{Bu25b,HK_BPS} for a wider class of stacks for which a similar statement holds.
In words, the perverse sheaves $\mathcal{BPS}_{M_S(\nu)}$ provide the building blocks for the complexes 
$(p_S)_*\mathbb{DQ}_{\calM^\ss_S(\nu)}$, and hence, by passing to derived global sections, for the Borel--Moore homology groups 
$\H^\BM_*(\calM^\ss_S(\nu),\Q)$. 

A Mukai vector $\nu$ is called generic if any semistable coherent sheaf of Mukai vector $\nu$ is also stable, see e.g., \cite{Yoshioka2}.
If $\nu$ is primitive and $H$-generic, 
then the stack $\calM^\ss_S(\nu)$ is a $\GM$-gerbe over 
$M_S(\nu)$, which is smooth, and $\cBPS_{M_S(\nu)}$ is the perverse shift of $\mathbb{Q}_{M_S(\nu)}$.
For general $\nu$, the cohomology of the complex $\mathcal{BPS}_{M_S(\nu)}$, called the {BPS cohomology}, is a better 
cohomological invariant of $\calM^\ss_S(\nu)$ than $\H^\BM_*(\calM^\ss_S(\nu),\Q)$. For instance, in the context of preprojective quiver 
varieties or moduli spaces of Higgs bundles, the Poincaré polynomial of the BPS cohomology recovers the Kac polynomials counting 
absolutely indecomposable quiver representations or vector bundles over finite fields, see \cite{DavisonKac,SchiffmannKac}.

\medskip

\subsection{BPS associative algebra and BPS Lie algebra} Fixing $p$ again, the decomposition \eqref{intro1} reflects an additional 
structure on the complex $\bigoplus_{\nu\in \Gamma_p} ({p}_S)_* (\mathbb{DQ}_{\calM^\ss_S(\nu)})$:
it is a graded associative algebra object in 
the category of complexes of mixed Hodge modules on $\bigsqcup_{\nu \in \Gamma_p} M_S(\nu)$.
Further, the complex $\bigoplus_{\nu \in \Gamma_p} \cBPS_{M_S(\nu)}$ carries the structure 
of a Lie algebra object in the category of mixed Hodge modules on $\bigsqcup_{\nu\in \Gamma_p} M_S(\nu)$, 
and \eqref{intro1} identifies the sum
$\bigoplus_{\nu\in \Gamma_p} {}^{\text{p}}\calH^0\big(({p}_S)_*\mathbb{DQ}_{\calM^\ss_S(\nu)}\big)$ with its enveloping algebra, 
see \cite{DHSM1,DHSM2} for details.

For a reduced polynomial $p$ corresponding to a primitive Mukai vector $\nu_\pr$ which is $H$-generic  
and satisfies $\nu_\pr \cdot \nu_\pr >0$, 
the sum in \eqref{intro1} restricts to $\{n\nu_\pr\,;\, n \geqslant 1\}$ and \cite{DHSM1} proves the following isomorphism
\begin{equation}\label{intro2}
\bigoplus_{n \geqslant 1} \cBPS_{M_S(n\nu_\pr)} \cong \text{FreeLie}\Big( \bigoplus_{n \geqslant 1} \text{IC}_{M_S(n \nu_\pr)}\Big),
\end{equation}
where $\text{IC}_{M_S(n \nu_\pr)}$ is the intersection complex on $M_S(n \nu_\pr)$.  Again, this isomorphism is written in the category of mixed Hodge modules.
Combining \eqref{intro1} and \eqref{intro2} we see that the knowledge, for all $n \geqslant 1$, of any one of the three collections of mixed 
Hodge modules 
$$\{(p_S)_*\mathbb{DQ}_{\calM^\ss_S(n\nu_\pr)}\,,\, \cBPS_{M_S(n\nu_\pr)}\,,\, \text{IC}_{M_S(n\nu_\pr)}\}$$ fully determines the other two. 
The same is thus true after taking derived global sections.

By passing to derived global sections in \eqref{intro1}, the graded vector space 
$$\bfH^{\ss}_S(p)\coloneq\bigoplus_{\nu \in \Gamma_p} \H^\BM_*(\calM_S^\ss(\nu),\Q)$$ acquires the structure of an 
associative algebra. Likewise, the graded vector space 
\\$\bigoplus_{\nu\in \Gamma_p} \H^*(M_S(\nu),\cBPS_{M_S(\nu)})$ acquires the structure of a graded 
Lie algebra, the {BPS Lie algebra}, whose enveloping algebra, the {BPS associative algebra}, is the $0$th piece of the perverse filtration of  
$\bfH^{\ss}_S(p)$ with respect to the map $p_S$. 

\medskip

\subsection{Cohomological Hall algebras} The associative algebra 
$\bfH^{\ss}_S(p)$
is 
an example of a semistable 2D CoHA. It was introduced in the context of preprojective quiver stacks in \cite{SVIHES} and in more geometric 
contexts in \cite{KV19}. 
The CoHA multiplication is given by the usual pull-push formalism applied to 
the natural correspondence
\begin{align}\label{intro3}
	\calM^\ss_S(\nu_1) \times \calM^\ss_S(\nu_2) \xleftarrow[]{\gr} 
	\calM^{\ss,\ex}_S(\nu_1,\nu_2)\xrightarrow[]{\ev}  \calM^\ss_S(\nu)
\end{align}
where the middle term of \eqref{intro3} classifies short exact sequences of sheaves $0 \to \calF  \to \calG \to \calE \to 0$ with $\calF,\calG,\calE$ 
semistable of respective Mukai vectors $\nu_2,\nu_1+\nu_2, \nu_1$. Note that the correspondence \eqref{intro3} makes sense when we 
replace $\calM^\ss_S(\nu)$ by the stack $\calM_S(\nu)$ classifying {all} coherent sheaves on $S$ of Mukai vector $\nu$, and allow $\nu$ to 
take all values. As the stacks $\calM_S(\nu)$ are only locally of finite type, this yields a graded associative algebra $\bfH_S$ in the category of 
topological vector spaces.

CoHAs of preprojective quiver stacks naturally act on the cohomology of Nakajima quiver varieties and as such have been much studied,
see, e.g., \cite{SVIHES,SVgenerators,Davison_affine_gl_1,Yang-Zhao,DHSM2}, etc., in relation to affine Yangians of 
Kac--Moody algebras and Maulik--Okounkov Lie algebras. In particular, it is proven in \cite{Davison_affine_gl_1,Shivang} that the CoHAs 
associated to cyclic quivers are isomorphic to enveloping algebras of explicit Lie algebras
of matrix-valued differential operators, which may be thought of as affinized versions of BPS Lie algebras, and it is proven in 
\cite{DHSM2} that the BPS Lie algebra of an arbitrary preprojective algebra is a graded Borcherds algebra. In contrast, much less is known about the 
CoHA $\bfH_S$ associated to a surface $S$. However, the subalgebra 
$$\bfH_S^0=\bigoplus_{n\geqslant 0} \H^\BM_*(\calM_S(n\delta),\Q)$$ 
corresponding to the stack of {zero}-dimensional sheaves can be explicitly described. In particular, when $S$ is symplectic,
$\bfH_S^0$ is isomorphic to the enveloping algebra of a $\H^*(S,\Q)$-colored version of $\mathfrak{w}_{1+\infty}$, the Lie algebra of differential 
operators on the plane, see \cite{MMSV} and also \S\ref{zero_dim_2d_sec} for a more detailed statement. In general, it is expected that the CoHA of the stack of objects of any reasonable 2CY Abelian 
category is the enveloping algebra of some affinized version of the BPS Lie algebra; the cocommutative coproduct that we construct in this paper, along with the Milnor--Moore theorem, provides a general mechanism for proving such results.

\medskip 

\subsection{Toda's $\chi$-independence conjecture} 
In his generalization of the mathematical approach to defining Gopakumar--Vafa invariants for projective 3CY varieties presented by Maulik and Toda in \cite{MT18},
Toda formulated a $\chi$-independence conjecture for the pushforward of the 3D BPS 
perverse sheaf to the Chow variety, see \cite{T23}.  The conjecture is a fundamental test for the Maulik--Toda definition of Gopakumar--Vafa invariants; a basic requirement of the definition, if they are to capture the definitions provided by physicists, is that they depend only on the underlying curve class of the sheaves considered in order to define them.  For 
$X=S \times \A^1$, using dimensional reduction, this conjecture may be reformulated in terms of the stacks $\calM^\ss_S(\nu)$ 
and the 2D BPS perverse sheaf $\cBPS_{M_S(\nu)}$. Let $\alpha \in \H^2(S,\Z)$ be a curve class, $\delta \in \H^4(S,\Z)$ be the class of a
 point and set $\nu=\alpha+n\delta$. Let $B_S(\alpha)$ be the Chow variety parametrizing algebraic one-cycles on $S$ of class $\alpha$,
 see \cite{Rydh}, and let $\pi_S\colon M_S(\nu) \to B_S(\alpha)$ be the support map.

\begin{conjecture}[Toda, {\cite[Conj. 1.2]{T23}}]\label{conjintro:Toda}
For any $\gamma \in B_S(\alpha)$, the invariant
$$\Phi(\alpha+n\delta)=\sum_{i\in\Z}\chi\big({}^{\p}\calH^i((\pi_S)_*\cBPS_{M_S(\alpha+n\delta)})|_{\gamma}\big)\cdot y^i\in\Z[y^{\pm 1}]$$
is independent of $n \in \Z$.
\end{conjecture}
The Euler characteristic specialization of this conjecture (obtained by putting $y=-1$) was proved by Maulik and Thomas in \cite{MT19}.
The importance of this conjecture is that it not only confirms the prediction from physics that the Gopakumar--Vafa invariants for a curve class $\alpha$ can be defined independently of the choice of $n$, but shows moreover that they can even be defined unambiguously by considering singular moduli spaces of (one-dimensional) sheaves.

\medskip

\subsection{Markman's theorem} Let us again consider a smooth and projective polarized symplectic surface $(S,H)$. 
Fix a $H$-generic and primitive Mukai vector $\nu$. 
In that situation there are no strictly semistable sheaves of Mukai vector $\nu$ and the moduli space $M_S(\nu)$ is smooth and proper. The 
cohomology ring of $M_S(\nu)$ contains a subalgebra $\H^*_{\taut}(M_S(\nu),\Q)$ of tautological classes, generated by the Künneth 
components of the universal sheaf. In that context, Markman proves the following result.

\begin{thm}[\cite{Markman}] For $\nu$ a primitive and $H$-generic Mukai vector, 
the cohomology ring of $M_S(\nu)$ is tautologically generated, i.e., we have $\H^*(M_S(\nu),\Q)=\H^*_{\taut}(M_S(\nu),\Q)$.
\end{thm}

This theorem has numerous applications in enumerative geometry.
It plays a crucial role in all of the proofs of the $P=W$ conjecture in nonabelian Hodge theory, see \cite{MMHS,MSP=W}. Using Poincaré duality, Markman's theorem may be restated in the Borel--Moore homology $\text{H}^\BM_*(M_S(\nu),\Q)$.
The naive extension of the latter statement is false when $\nu$ is not primitive, even if one replaces the moduli space of semistable sheaves by the corresponding moduli stack.  One of our main theorems (Theorem \ref{thmintro:Markman}) is the correct generalization of Markman's theorem, written in terms of affinized BPS Lie algebras.

\medskip

\subsection{Main theorems}\label{sec_main_theorems} Let $(S,H)$ be a polarized projective symplectic surface. 
Our first main result is a proof of Toda's $\chi$-independence conjecture for the local surface $X=S\times \A^1$. We show the following strong form of Conjecture~\ref{conjintro:Toda}.

\begin{theorem}[=Theorem~\ref{thm:Toda}]\label{thmintro:toda} Let $\alpha \in \H^2(S,\Z)$ be the class of an effective one-cycle. 
The complex of mixed Hodge modules over the Chow variety $B_S(\alpha)$
\begin{equation}\label{eqintro: Todacomplex}
\bigoplus_{i\in\Z}{}^{\p}\calH^i((\pi_S)_*\cBPS_{M_S(\alpha+n\delta)})[-i]
\end{equation}
 is independent of $n$, up to an explicit isomorphism. 
Moreover, it is independent of the choice of polarization $H$.
\end{theorem}

The fact that the invariant $\Phi(\alpha+n\delta)$ is independent of the polarization was proven in the context of projective CY threefolds by 
Toda in \cite{T23}. Our proof in the local surface case is quite different, see \S\ref{sec:main ideas} below.  Our proof covers the 
more general case of a polarized quasi-projective surface $(S,H)$ and a curve class $\alpha \in \H^2(S,\Z)$ for which there exists a 
compactification with a canonical divisor which does not intersect members of $|\alpha|$, see \S\ref{sec:Toda generalization}.

Theorem \ref{thmintro:toda} has been proved simultaneously and independently by Groechenig, Wyss, and Ziegler \cite{GWZ}, using entirely different techniques.

\medskip

The isomorphisms provided by Theorem~\ref{thmintro:toda} between the complexes associated to $\alpha+n\delta$ and $\alpha+m\delta$ are given by the action of cohomological Hecke operators of pointwise modification. More precisely, our argument crucially hinges on the following result:

\begin{proposition}\label{propintro:heckeops} There is a canonical $\N \times \Z$-graded action of the CoHA of zero-dimensional sheaves $\bfH^0_S$ on the complex of mixed Hodge modules
\begin{equation}\label{eq:propHecke}
\bigoplus_{n\in\Z}\bigoplus_{i\in\Z}{}^{\p}\calH^i((\pi_S)_*\cBPS_{M_S(\alpha+n\delta)})[-i].
\end{equation}
\end{proposition}

Note that Hecke modifications do \textit{not} preserve semistable sheaves, so that the general formalism only provides an action of $\bfH^0_S$ on the (much bigger) complex
\begin{equation}\label{eq:propHecke2}
	\bigoplus_{n\in\Z}\bigoplus_{i\in\Z}{}^{\p}\calH^i((\rho_S)_*\mathbb{DQ}_{\calM_S(\alpha+n\delta)})[-i],
\end{equation} 
where $\rho_S\colon \bigsqcup_n\calM_S(\alpha+n\delta) \to B_S(\alpha)$ is the support map for the full stack. The content of Proposition~\ref{propintro:heckeops} is that \eqref{eq:propHecke} is naturally a subcomplex of \eqref{eq:propHecke2} which is moreover stable under the action of $\bfH^0_S$.

\medskip

Passing to cohomology in Theorem~\ref{thmintro:toda} yields a $\chi$-independence result for 
graded pieces of different semistable BPS Lie algebras, 
see Theorem~\ref{thmintro:BPS Lie} below. Part of the setup is to work 
directly with the stacks $\calM_S(\nu)$ without imposing any semistability condition. Our second main result concerns the existence of a 
topological coproduct for the corresponding CoHA $\bfH_S=\bigoplus_\nu \H^\BM_*(\calM_S(\nu),\Q)$.

\begin{theorem}[=Corollary~\ref{cor:exists coproduct}]\label{thmintro:exists coproduct} 
Let $p$ be any reduced polynomial.
\hfill
\begin{enumerate}[label=$\mathrm{(\alph*)}$,leftmargin=8mm,itemsep=2mm]
\item 
There is a cocommutative coproduct $\Delta\colon \bfH_S \to \bfH_S\, \wotimes \,\bfH_S$
turning $\bfH_S$ into a topological graded bialgebra.
\item 
There is a cocommutative coproduct 
$\Delta^\ss\colon \bfH^{\ss}_S(p) \to \bfH^{\ss}_S(p) \otimes \bfH_S^{\ss}(p)$
turning $\bfH^{\ss}_S(p)$ into a graded bialgebra.
\item The restriction 
$\bigoplus_{\nu \in \Gamma_p} \bfH_S(\nu) \to \bfH^{\ss}_S(p)$
induced by the open immersions $\calM^\ss_S(\nu) \to \calM_S(\nu)$ is a surjective morphism of graded bialgebras.
\end{enumerate}
\end{theorem}

Our definition of the coproduct is based on the technique of factorization algebras, see, e.g., \cite{CG} and \cite{KV19}. 
We note that a related construction, in the context of vertex algebras for quivers with potential, has independently appeared in \cite{Jindal-Latyntsev-Sarunas}. 
Let  $\frakg_S=\bigoplus_\nu \frakg_S(\nu)$ and $\frakg^{\ss}_S(p)=\bigoplus_{\nu \in \Gamma_p} \frakg^\ss_S(\nu)$ be the Lie 
algebras of primitive elements of $\bfH_S$ and of $\bfH^{\ss}_S(p)$. 
Set also $\frakg_S(p)=\bigoplus_{\nu \in \Gamma_p} \frakg_S(\nu)$.
We deduce the following Milnor--Moore statements: the 
canonical morphisms 
$$\U(\frakg_S) \to \bfH_S, \quad \U(\frakg^{\ss}_S(p)) \to \bfH^{\ss}_S(p)$$ 
are respectively a dense embedding and an isomorphism.

The $\chi$-independence property for Lie algebras may now be stated as follows. 
We have $\calM_S(\delta)^\cl=S \times \BGM$. Hence there is a canonical identification 
$$\H^*(S,\Q)\otimes \H^*(\BGM,\Q) \cong \H^\BM_*(\calM_S(\delta),\Q) = \frakg_S(\delta).$$ 
Let $D_{1,0}(\beta)$ be the element corresponding to $\beta\otimes 1\in \H^*(S,\Z)\otimes \H^*(\BGM,\Q)$.

\begin{theorem}[=Theorems~\ref{thm:restriction-isom} and \ref{thm:chi-indep-Lie-algebra}]\label{thmintro:BPS Lie}
\hfill
\begin{enumerate}[label=$\mathrm{(\alph*)}$,leftmargin=8mm,itemsep=2mm]
\item The restriction $\frakg_S(p)\to \frakg^{\ss}_S(p)$
is an isomorphism of graded Lie algebras.
\item For any $\nu$ there is an isomorphism of graded vector spaces 
$\H^*(M_S(\nu),\cBPS_{M_S(\nu)}) \otimes \H^*(\BGM,\Q) \cong \frakg^\ss_S(\nu)$.
\item Let $\alpha \in \H^2(S,\Z)$ be the class of an effective one-cycle, $\nu=\alpha+n\delta$ and let $\beta \in \H^2(S,\Z)$ be ample. The 
adjoint action by the element $D_{1,0}(\beta) \in \frakg_S(\delta)$ restricts to a graded isomorphism 
$\frakg_S(\nu) \cong \frakg_S(\nu+\delta)$ 
preserving the Hodge filtration and the perverse filtration given by the proper morphism to the Chow variety.
\end{enumerate}
\end{theorem}

The element $D_{1,0}(\beta)$ corresponds to the length one Nakajima operator associated to the class $\beta$.
Statement (c) can be deduced from Theorem~\ref{thmintro:toda} by taking derived global sections; we also provide a proof directly using Hecke operators on the (affinized) BPS Lie algebra.
Statements (a) and (b) allow us to regard $\frakg_S$ or $\frakg^\ss_S(p)$ as affinized versions of the semistable BPS 
Lie algebra of \cite{DHSM2}.

Theorem~\ref{thmintro:BPS Lie} provides a way to compute the graded dimensions of $\text{IH}^*(M_S(\nu),\Q)$ and 
$\H^\BM_*(\calM^\ss_S(\nu),\Q)$, along with the graded dimensions of their Hodge filtrations, for any Mukai vector $\nu$ of dimension one. 
Indeed, both are determined, thanks to \eqref{intro1} and \eqref{intro2}, by the collection of graded dimensions and Hodge filtrations of 
the spaces $\frakg^\ss_S(\nu) \cong \frakg_S(\nu)$ for all one-dimensional Mukai vectors $\nu$. By the $\chi$-independence property, 
we may replace any $\frakg_S(\alpha+n\delta)$ by $\frakg_S(\alpha+\delta)\cong \frakg^\ss_S(\alpha+\delta)$ and pick a polarization which is 
$(\alpha+\delta)$-generic. As $\alpha+\delta$ is primitive, we have
$$\frakg^\ss_S(\alpha+\delta) =\bfH_S^\ss(\alpha+\delta) =\H^\BM_*(M_S(\alpha+\delta),\Q) \otimes \H^*(\BGM,\Q)$$
and the latter is known explicitly thanks to \cite{Gottsche-Soergel,Yoshioka}.

\smallbreak
Our final main result provides a characterization of the subspace $\frakg^\ss_S(\nu)$ of $ \H^\BM_*(\calM_S^\ss(\nu),\Q)$ for one-dimensional Mukai 
vectors. Let  $[\calM_S(\nu)^\rd]$ be the {reduced} virtual fundamental class of $\calM_S(\nu)$ (see \S\ref{red_vc_sec} for the definition). We 
set 
$$\H_*^\taut(\calM_S(\nu),\Q)= \H^*_\taut(\calM_S(\nu),\Q) \cap [\calM_S(\nu)^\rd] \subseteq \H^\BM_*(\calM_S(\nu),\Q)\ $$
and define $\H_*^\taut(\calM^\ss_S(\nu),\Q)$ similarly.

\begin{theorem}[=Theorem~\ref{thm:Markman}]\label{thmintro:Markman} 
Let $\alpha \in \H^2(S,\Z)$ be an effective curve class.
Then for any $n \in\Z$ and $\nu=\alpha+n\delta$, we have
$\H_*^\taut(\calM_S(\nu),\Q) = \frakg_S(\nu)$ and $\H_*^\taut(\calM^\ss_S(\nu),\Q)=\frakg^\ss_S(\nu).$
\end{theorem}

All of the above theorems have analogs with similar proofs for the quasi-projective surface $S=T^*C$, where $C$ is a smooth projective curve. 
In that case, $\calM_S(\nu)$ is the moduli stack of Higgs sheaves on $C$ and Theorem~\ref{thmintro:toda} was proved by a different method 
in \cite{KinjoKoseki}. In particular, Theorem~\ref{thmintro:Markman} says that for $C$ of genus $g>1$ and for any rank $r$ and degree $d$, the 
restriction of $\frakg_S(r,d)$ to the stack of Higgs {bundles} is equal to $\H_*^\taut(\mathcal{H}iggs_{C}(r,d),\Q)$.

\medskip

\subsection{Main ideas}\label{sec:main ideas}
Let us say a few words about the main ideas used in the proof of Theorem~\ref{thmintro:toda}. 
The isomorphisms between 
the complexes \eqref{eqintro: Todacomplex} for various values of $n$ are given by cohomological Hecke operators. To 
apply these, we use a characterization of the pushforward of the BPS sheaf to the Chow variety in Proposition~\ref{prop:H0-action}. 
This is done by first passing from 2D to 3D
via the dimensional reduction isomorphism and then by using a microlocal characterization of the BPS sheaf (see Corollary \ref{cor:snilp-BPS} and Proposition \ref{prop:coker}). Once we have 
constructed the action of the algebra of Hecke operators $\bfH_S^0$ on the pushforward of the BPS sheaf, we 
show that the tensor product with a sufficiently ample line bundle is realized as a Hecke operator of high length. From this 
and the fact that $\bfH^0_S$ is generated by Hecke operators of length one, we deduce that the length one Hecke operator $D_{1,0}(\beta)$ is 
an isomorphism as well, see Lemma~\ref{lem:keylemma}.

Theorems~\ref{thmintro:exists coproduct},~\ref{thmintro:BPS Lie} and~\ref{thmintro:Markman} are proven using similar ideas.
The coproduct uses dimensional reduction and a factorization structure on  the  3D Donaldson--Thomas sheaf. 
This gives compatible coproducts on both $\bfH_S$ and $\bfH^\ss_S(p)$. 
This also yields a Hecke action of the Lie algebra version $\frakg_S(\N \delta)$ on $\frakg_S$ and we may apply the same argument as 
for the BPS sheaf. Finally, by \cite{MMSV}, Hecke operators preserve the subspace of tautological classes 
$\H_*^\taut(\calM_S(\alpha+\Z\delta),\Q)$. The fundamental class 
$[\calM_S(\alpha+n\delta)^\rd]$ is quasi-primitive; Theorem~\ref{thmintro:Markman} then essentially follows from the $\chi$-independence 
of the Lie algebra $\frakg_S$ in Theorem~\ref{thmintro:BPS Lie} together with Markman's theorem which implies that for $\nu$ primitive 
 and $H$-generic we have $\H_*^\taut(\calM^\ss_S(\nu),\Q)=\frakg^\ss_S(\nu)$.

\medskip

\subsection{Conventions}\label{conventions_sec}
For ease of reading, we adopt certain notational conventions.
The category of derived affine schemes $\Aff$ over $\C$
is opposite to the $\infty$-category $\CAlg^{\leqslant 0}$ of connective commutative 
dg-algebras over $\C$.
A derived $\infty$-prestack is a functor $\CAlg^{\leqslant 0}\to\Spc$ where $\Spc$ is the $\infty$-category of $\infty$-groupoids.
By a derived $\infty$-stack we mean the sheafification of a derived $\infty$-prestack for the \'etale topology.
Unless specified otherwise, all stacks are derived $\infty$-stacks, and we simply use the term stacks for them.
Further, an Artin stack is always assumed to be locally of finite type.
The classical truncation $\calM\mapsto\calM^\cl$ is the right adjoint functor of the fully faithful embedding of $\infty$-stacks into 
derived stacks. Let $\Map(\calM,\calN)$ be the derived mapping stack in the $\infty$-category of derived stacks.
If $\calM$ is a classical 1-stack and $\calN$ a derived 1-stack, then 
$\Map(\calM,\calN)^\cl\cong\Map(\calM,\calN^\cl)$
is the classical mapping stack of the classical stacks $\calM$ and $\calN^\cl$.
For a classical Artin stack $\calM$, the reduced substack $\calM_\red$ is the unique Artin stack with a morphism
$\calM_\red\to\calM$ which yields a natural isomorphism 
$\Hom(X,\calM_\red)=\Hom(X_\red,\calM)$ for any scheme $X$.
If $\calM$ is a derived Artin stack we abbreviate
\begin{align}\label{reduced-stack}\calM_\red=(\calM^\cl)_\red.\end{align}

Let $X$ be a separated scheme of finite type. Denote by $\Dcb(X)$ (resp.~$\Dhcb(X)$) the stable $\infty$-category of algebraically constructible complexes of sheaves with rational coefficients (resp.~mixed Hodge modules) on $X$.
We let $\mathsf{D}(X) \coloneqq \Ind\Dcb(X)$ and $\Dh(X)=\Ind\Dhcb(X)$.
The functors $\mathsf{D}(-)$ and $\Dh(-)$ have extensions to the $\infty$-category of locally of finite type Artin $\infty$-stacks with a 6-functor formalism, thanks to the work of Tubach \cite{Tubstack}.
For any such $\infty$-stack
$\calM$, let $\Dcb(\calM)\subset \mathsf{D}(\calM)$ (resp.~$\Dhcb(\calM)\subset\Dh(\calM)$) be the full subcategory of constructible complexes of
sheaves (resp.~mixed Hodge modules), i.e., the $\infty$-category of complexes of
sheaves (resp.~mixed Hodge modules) whose pull-back to any separated scheme $X$ of finite type belongs to $\Dcb(X)$ (resp.~$\Dhcb(X)$).
There is  a forgetful functor 
$\Dh(\calM)\to\mathsf{D}(\calM)$ which restricts to $\Dhcb(\calM)\to \Dcb(\calM)$.
The derived category $\Dcb(\calM)$ (resp.~$\Dhcb(\calM)$) 
is equipped with a perverse t-structure whose heart is denoted by $\Perv(\calM)$ (resp.~$\MHM(\calM)$).

Let $\Dhmon(\calM)$,  $\Dhcmonb(\calM)$, etc.,  
be the corresponding $\infty$-categories of monodromic mixed Hodge modules. 
There is a monodromy-forgetting functor $\Dhmon(\calM)\to\Dh(\calM).$
Conversely, equipping a mixed Hodge module with the trivial monodromy yields a fully faithful functor
$\Dh(\calM)\to\Dhmon(\calM)$
whose essential image consists of the monodromy-free complexes.
Let 
$\MHMmon(\calM)\subset\Dhmon(\calM)$ be the full subcategory  consisting of the objects whose underlying
mixed Hodge modules belong to $\MHM(\calM)$.

If $\calM$ is an Artin $\infty$-stack we have
$$\Dhcb(\calM)=\Dhcb(\calM^\cl),\quad
\Dhcmonb(\calM)=\Dhcmonb(\calM^\cl),\quad \text{etc.}$$
Set 
$\bullet=\Spec(\C)$
and let
$\L=\H^*_c(\A^1)$ in $\Dh(\bullet)$. 
Let ${}^\p\calH^i(-)$ and ${}^\p\gr(-)=\bigoplus_i{}^\p\calH^i(-)[-i]$ denote the perverse cohomology 
and the associated graded for the perverse filtration.
 
Given a finite type compactly generated $\infty$-category  $\calC$
let $\calM _{\calC}$ be the moduli stack of objects in $\calC$ in the sense of \cite[\S 3.1]{TV07}, i.e., 
on an affine scheme $U$ we set
$\calM _{\calC}$ to be the space of exact functors $\calC^c\to\Perf(U)$ from compact objects in $\calC$ to perfect complexes on $U$.
An $n$CY (=$n$-Calabi--Yau) category is a compactly generated finite type $\mathbb{C}$-linear $n$-Calabi--Yau stable $\infty$-category in the sense of \cite[def.~3.5]{BD19}.

\bigskip

\section{BPS cohomology and CoHAs of 2CY categories} \label{sec:2CY CoHAs}

This section is a reminder of some results on cohomological DT(=Donaldson--Thomas) theory, mainly from \cite{KKPS} and \cite{Bu25b}. 

\subsection{The DT mixed Hodge module}\label{sec: BPS and DT}

\subsubsection{Shifted symplectic structures}

Let $\calX$ be a derived Artin stack. 
For an integer $n$, an {$n$-shifted symplectic structure} on $\calX$ is an $n$-shifted closed $2$-form $\omega_{\calX} \in \calA^{2, \mathrm{cl}}(\calX, n)$ such that the 
induced pairing map $\mathbb{T}_{\calX} \to \mathbb{L}_{\calX}[n]$ is invertible.
See \cite[\S 1.2]{PTVV13} for the definition of closed forms for derived stacks and basic results in shifted symplectic geometry.
Such a pair is called an {$n$-shifted symplectic stack}.
In this paper, we are only interested in the cases $n = 0$ and $n = -1$.
The following two examples are the main objects to be considered in our paper.
 
\begin{example}
Let $\calC$ be a 2CY category.
Then it is shown in \cite[thm.~5.5]{BD21} that the moduli stack of objects $\calM _{\calC}$  in $\calC$ 
is equipped with a $0$-shifted symplectic structure.
\end{example}

\begin{example}\label{ex-shifted-cotangent}
Let $\calY$ be a quasi-smooth derived Artin stack.
The $(-1)$-shifted cotangent stack $\mathrm{T}^*[-1]\calY$ is the total space of the cotangent complex $\mathbb{L}_{\calY}[-1]$.
 We abbreviate $\widetilde\calY=\mathrm{T}^*[-1]\calY$.
 Let $\kappa \colon\widetilde\calY  \to \calY$ be the projection.
 Let $s_{\mathrm{sta}} \in \mathrm{R} \Gamma(\widetilde\calY, \kappa^* \mathbb{L}_{\calY}[-1])$ be the tautological section
and  $\lambda \in \mathrm{R} \Gamma(\widetilde\calY, \mathbb{L}_{\widetilde\calY}[-1])$ be its image under the map induced by the natural map $\kappa^*\mathbb{L}_{\calY}\rightarrow \mathbb{L}_{\widetilde\calY}$.
It is proved in \cite[thm.~2.4]{Cal19} that the de Rham differential $\omega_{\mathrm{sta}} = d_{\mathrm{dR}} \lambda$ is $(-1)$-shifted 
symplectic and we call it the {standard $(-1)$-shifted symplectic form}.
\end{example}

 \medskip

\subsubsection{The orientation}\label{sec-orientation}

For a $(-1)$-shifted symplectic stack $\calX$, an {orientation} is a pair consisting of a line bundle with locally constant $\mathbb{Z} / 2 \mathbb{Z}$-grading
$S$ on $\calX$ and an isomorphism $o \colon S^{\otimes 2} \cong \det(\L_{\calX})$.
If there is no risk of confusion, an orientation $(S, o)$ is denoted simply as $o$.
Throughout the paper, we use the Koszul sign rule for $\mathbb{Z} / 2 \mathbb{Z}$-graded line bundles: see \cite[\S 2.2]{KPS} for more details on our sign conventions on graded line bundles and the determinant.
For an orientation $(S, o)$, we let $\deg(o)$ denote the grading of $S$ defined as a locally constant function valued in $\mathbb{Z} / 2 \mathbb{Z}$.

For a $(-1)$-shifted cotangent stack $\calX = \widetilde\calY$ as in Example \ref{ex-shifted-cotangent}, there 
exists a standard orientation $(S_{\mathrm{sta}},o_{\mathrm{sta}})$ with $S_{\mathrm{sta}}= \kappa^* \det(\L_{\calY})$,
see e.g. \cite[ex.~3.20]{KPS}.
Here, $\det(\L_{\calY})$ is graded by $\vdim \calY \; (\mathrm{mod}\ 2\mathbb{Z})$.
 
  \medskip

\subsubsection{The DT mixed Hodge module}
For an oriented $(-1)$-shifted symplectic stack $(\calX, \omega_{\calX}, o)$, 
a monodromic mixed Hodge module
$\varphi_{\calX, \omega_{\calX}, o}$ in $\MHMmon(\calX)$ is constructed in \cite{BBDJS15}, 
which we refer to as the {DT mixed Hodge module}.
See \cite{Tubstack} and \cite[\S 5.2]{Bu25b} for the treatment of monodromic mixed Hodge modules on stacks.
Reasons to work with monodromic mixed Hodge modules are that in this category one has a square root of the Lefschetz
Hodge structure $\L$, and the Thom--Sebastiani theorem holds only after taking the monodromy action.
If the $(-1)$-shifted symplectic structure and the orientation are obvious from the context, 
we simply write $\varphi_{\calX} =  \varphi_{\calX, \omega_{\calX}, o}$.
The underlying perverse sheaf of $\varphi_{\calX}$ will be called the {DT perverse sheaf}.
If there is no risk of confusion, we use $\varphi_{\calX}$ also for the underlying perverse sheaf.

\medskip

\subsubsection{Compatibility with \'etale maps}

Let $(\calX_2, \omega_2, o_2)$ be an oriented $(-1)$-shifted symplectic stack and  $\eta \colon \calX_1 \to \calX_2$ be an \'etale morphism.
Then $(\calX_1, \eta^{\star}\omega_2, \eta^{\star}o_2)$ defines a structure of an oriented $(-1)$-shifted symplectic stack.
It follows from the construction of the DT mixed Hodge module that there is a natural isomorphism in $\Dhmon(\calX_1)$
\begin{equation}\label{eq-etale-compatibility}
	\varphi_{\calX_1} \cong \eta^* \varphi_{\calX_2}.
\end{equation}

 \medskip

\subsubsection{Thom--Sebastiani theorem}

Let $(\calX_1, \omega_1, o_1)$ and $(\calX_2, \omega_2, o_2)$ be oriented $(-1)$-shifted symplectic stacks.
Then $(\calX_1 \times \calX_2, \omega_1 \boxplus \omega_2, o_1 \boxtimes o_2)$ defines a structure of an oriented $(-1)$-shifted symplectic stack.
It is shown in \cite[prop.~4.3]{KPS} and \cite[cor.~5.9]{Des} that the formation of the DT mixed Hodge module is monoidal.
Namely, there is a natural isomorphism in $\Dhmon(\calX_1 \times \calX_2)$
\begin{equation}\label{eq-Thom-Sebastiani}
 \varphi_{\calX_1} \boxtimes \varphi_{\calX_2} \cong \varphi_{\calX_1 \times \calX_2}	
\end{equation}
 called the {Thom--Sebastiani isomorphism}, satisfying the associativity relation.
Here, we use the monodromic tensor product, see, e.g., \cite[\S 5.2.4]{Bu25a} for the details.
The monodromic tensor product is compatible with the standard tensor product for perverse sheaves, but not with the standard tensor product for mixed Hodge modules.
It is shown in \cite[prop.~5.23(3)]{KKPSpervpull} that the Thom--Sebastiani isomorphism is graded commutative.
Namely, if $\sigma \colon \calX_2 \times \calX_1 \to \calX_1 \times \calX_2$ is the swapping isomorphism, 
the following diagram commutes: 
\begin{equation}\label{eq-TS-commutative}
	\begin{tikzcd}
	{\sigma^*(\varphi_{\calX_1} \boxtimes \varphi_{\calX_2})} &[70pt] {\sigma^*\varphi_{\calX_1 \times \calX_2}} \\
	{\varphi_{\calX_2} \boxtimes \varphi_{\calX_1}} & {\varphi_{\calX_2 \times \calX_1}.}
	\arrow["{(-1)^{\deg(o_1) \cdot \deg(o_2)} \cdot \eqref{eq-Thom-Sebastiani}}", from=1-1, to=1-2]
	\arrow["\cong"', from=1-1, to=2-1]
	\arrow["\cong", from=1-2, to=2-2]
	\arrow["{\eqref{eq-Thom-Sebastiani}}"', from=2-1, to=2-2] 
\end{tikzcd}
\end{equation}

 \medskip

\subsubsection{Translation invariance}

If $\calX = \widetilde\calY$ as in Example \ref{ex-shifted-cotangent}, we abbreviate
$\varphi_{\widetilde\calY}\cong\varphi_{\calX, \omega_{\mathrm{sta}}, o_{\mathrm{sta}}}.$
If further $\calY$ is a $0$-shifted symplectic stack with an action of $\BGM$, 
we have an action of $\mathrm{T}[-1] \BGM \cong \A^1 / \G_m$ on $\mathrm{T}[-1]\calY$.
Under
the identification $\mathrm{T}[-1]\calY \cong \widetilde\calY$, it yields an $\A^1$-action on $\calX$.

\begin{proposition}\label{prop:equivariance}
Let $\calY$ be a $0$-shifted symplectic stack with an action of $\BGM$.
\hfill
\begin{enumerate}[label=$\mathrm{(\alph*)}$,leftmargin=8mm,itemsep=2mm]
\item
$\varphi_{\widetilde\calY}$ is equivariant with respect to the $\A^1$-action.
\item
$\varphi_{\widetilde\calY}$ is locally constant along the orbits of the scaling $\G_m$-action.
\end{enumerate}
\end{proposition}

\begin{proof}
Part (a)  is shown in the proof of \cite[thm.~7.2.15]{Bu25b}.
Part (b) follows from the construction of the DT mixed Hodge module in terms of vanishing cycles relative to a function which is homogeneous 
with respect to dilations.
\end{proof}

\medskip 

\subsection{Shifted Lagrangian classes}

We explain basic functorial properties for the DT perverse sheaves under Lagrangian correspondences conjectured by Joyce \cite[conj~1.1]{Joyce-Safronov} and proved in \cite{KKPS}.
This construction unifies the dimensional reduction isomorphism \cite{K}, the cohomological Hall algebra for 3CY categories 
\cite{KPS}, \cite{Des}, and the virtual fundamental classes for quasi-smooth derived Artin stacks \cite{K19}.
This construction will be used when comparing the two-dimensional CoHAs and three-dimensional CoHAs under the dimensional reduction isomorphism (Theorem~\ref{thm: comparison 2D 3D}), and also when proving the primitivity of the reduced virtual fundamental classes for the moduli stack of objects in 2CY categories in \S\ref{sec:proofmarkman}. 

\subsubsection{Lagrangian correspondence}

Let $(\calX_1, \omega_{\calX_1})$ and $(\calX_2, \omega_{\calX_2})$ be $(-1)$-shifted symplectic stacks and consider a correspondence
$\calX_1 \xleftarrow{f_1} \calL \xrightarrow[]{f_2} \calX_2$.
An isotropic structure for the correspondence is a choice of homotopy
\[
 \eta \colon f_1^{\star} \omega_{\calX_1}  \sim f_2^{\star} \omega_{\calX_2}
\]
in the space $\calA^{2, \mathrm{cl}}(\calL, -1)$ of $(-1)$-shifted closed $2$-forms on $\calL$.
The isotropic structure gives a map $$\mathbb{T}_{\calL / \calX_1} \to \mathbb{L}_{\calL / \calX_2}[-2].$$
We say that $\eta$ is a Lagrangian structure if this map is invertible. Such a correspondence is called a $(-1)$-shifted Lagrangian correspondence.
A Lagrangian structure induces an isomorphism 
\begin{equation}\label{eq-Lag-det-isom}
    \det(\mathbb{L}_{\calX_1}|_{\calL})  \otimes \det(\mathbb{L}_{\calL / \calX_1})^{\otimes 2} \cong \det(\mathbb{L}_{\calX_2}|_{\calL}).
\end{equation}
An orientation for the $(-1)$-shifted Lagrangian correspondence is a choice of orientations $(S_i, o_i)$ for $\calX_i$ and an isomorphism
$o \colon S_1 |_{\calL} \otimes \det(\mathbb{L}_{\calL / \calX_1}) \cong S_2|_{\calL}$ such that the following diagram commutes
\[\begin{tikzcd}
	{(S_1 |_{\calL} \otimes \det(\mathbb{L}_{\calL / \calX_1}))^{\otimes 2}} & {(S_2 |_{\calL})^{\otimes 2}} \\
	{\det(\mathbb{L}_{\calX_1} |_{\calL}) \otimes \det(\mathbb{L}_{\calL / \calX_1})^{\otimes 2}} & {\det(\mathbb{L}_{\calX_2} |_{\calL}). }
	\arrow["{o^{\otimes 2}}", from=1-1, to=1-2]
	\arrow["{o_1 \otimes \id}"', from=1-1, to=2-1]
	\arrow["{o_2}", from=1-2, to=2-2]
	\arrow["{\eqref{eq-Lag-det-isom}}"', from=2-1, to=2-2]
\end{tikzcd}\]
The product of oriented Lagrangian correspondences is equipped with the structure of an oriented Lagrangian correspondence.
Also, a composition of oriented Lagrangian correspondences is defined in the obvious way.

For an oriented Lagrangian correspondence $\calL$ from $\calX_1$ to $\calX_2$, one can define the swapped Lagrangian correspondence
$\mathrm{swap}(\calL)$ from $\calX_2$ to $\calX_1$ with the same underlying derived stacks. 
We warn the reader that $\mathrm{swap}$ is {not} an involution: The orientations of $\mathrm{swap}^2(\calL)$ and $\calL$ differ by a factor of $(-1)^{\mathrm{vdim} \calL}$.
This is essentially due to our Koszul sign convention for graded line bundles, and will be explained in \cite{KKPS}.

\subsubsection{Shifted Lagrangian classes}

The DT perverse sheaves behave functorially with respect to oriented $(-1)$-shifted Lagrangian correspondences.

\begin{theorem}[\cite{KKPS}]\label{thm-shifted-Lag-class}
    Let $\calX_1 \xleftarrow{f_1} \calL \xrightarrow[]{f_2} \calX_2$ be an oriented $(-1)$-shifted Lagrangian correspondence.
    Set  $d = \vdim \calL$.
    Then there is a map
    \begin{equation}\label{eq-shifted-lag-class}
     [\calL]_{\mathrm{Lag}} \colon f_1^* \varphi_{\calX_1} \to f_2^{!} \varphi_{\calX_2} [- d]   
	\end{equation}
    in $\mathsf{D^b_c}(\calL)$ which we call the {shifted Lagrangian class}.
The shifted Lagrangian class satisfies the following properties.
\begin{enumerate}[label=$\mathrm{(\alph*)}$,leftmargin=8mm,itemsep=2mm]
\item \label{item-shifted-Lag-class-unital} Unitality: If $f_1$ and $f_2$ are identity maps and the Lagrangian structure is the diagonal one, then the shifted Lagrangian class $[\calL]_{\mathrm{Lag}}$ is the identity map.
In this case $d=0$ since the virtual dimension of a $(-1)$-shifted symplectic stack is $0$. 
        \item \label{item-shifted-Lag-class-associative} Associativity: Assume that we are given another oriented $(-1)$-shifted Lagrangian correspondence $\calX_2 \xleftarrow{g_2} \calL' \xrightarrow[]{g_3} \calX_3$ and let
        $\calX_1 \xleftarrow{h_1} \calL'' \xrightarrow[]{h_3} \calX_3$ be the composite.
Set $d' = \vdim \calL'$ and $d'' = \vdim \calL''$.
        Then the following diagram commutes
        \[\begin{tikzcd}
            {h_1^* \varphi_{\calX_1}} &[20pt] &&[20pt] {h_3^! \varphi_{\calX_3}[-d'']} \\
            {(g_2')^* f_1^* \varphi_{\calX_1}} & {(g_2')^* f_2^! \varphi_{\calX_2}[-d]} & {(f_2')^! g_2^* \varphi_{\calX_2}}[-d] & {(f_2')^!g_3^! \varphi_{\calX_3}[-d''].}
            \arrow["{[\calL'']_{\mathrm{Lag}} }", from=1-1, to=1-4]
            \arrow["\cong"', from=1-1, to=2-1]
            \arrow["\cong", from=1-4, to=2-4]
            \arrow["{(g_2')^* [\calL]_{\mathrm{Lag}} }"', from=2-1, to=2-2]
            \arrow[from=2-2, to=2-3]
            \arrow["{(f_2')^![\calL']_{\mathrm{Lag}} }"', from=2-3, to=2-4]
        \end{tikzcd}\]
        Here $g_2' \colon \calL'' \to \calL$ and $f_2' \colon \calL'' \to \calL'$ are the base change of $g_2$ and $f_2$ respectively.

        \item \label{item-shifted-Lag-class-multiplicative} Multiplicativity: Assume that we are given another oriented $(-1)$-shifted Lagrangian correspondence $\calX_1' \xleftarrow{f_1'} \calL' \xrightarrow[]{f_2} \calX_2'$. Set $d' = \vdim \calL'$.
        Then the following diagram commutes
        \[\begin{tikzcd}
            {f_1^*\varphi_{\calX_1} \boxtimes f_1'^*\varphi_{\calX_1'}} &[40pt] {f_2^!\varphi_{\calX_2}[-d] \boxtimes f_2'^!\varphi_{\calX_2'}[-d']} \\
            {(f_1 \times f_1')^* \varphi_{\calX_1 \times \calX_1'}} & {(f_2 \times f_2')^! \varphi_{\calX_2 \times \calX_2'}[-d - d'].}
            \arrow["{[\calL]_{\mathrm{Lag}}  \boxtimes [\calL']_{\mathrm{Lag}} }", from=1-1, to=1-2]
            \arrow["\eqref{eq-Thom-Sebastiani}"', from=1-1, to=2-1]
            \arrow["\eqref{eq-Thom-Sebastiani}"', from=1-2, to=2-2]
            \arrow["{[\calL \times \calL']_{\mathrm{Lag}} }"', from=2-1, to=2-2]
        \end{tikzcd}\]

        \item \label{item-shifted-Lag-class-etale-compat} \'Etale compatibility: Assume that we are given the following commutative diagram
        \[\begin{tikzcd}
            {\calX_1'} & {\calL'} & {\calX_2'} \\
            {\calX_1} & \calL & {\calX_2}
            \arrow["{u_1}"', from=1-1, to=2-1]
            \arrow["{f_1'}"', from=1-2, to=1-1]
            \arrow["{f_2'}", from=1-2, to=1-3]
            \arrow["{u_{\calL}}"', from=1-2, to=2-2]
            \arrow["{u_2}"', from=1-3, to=2-3]
            \arrow["{f_1}", from=2-2, to=2-1]
            \arrow["{f_2}"', from=2-2, to=2-3]
        \end{tikzcd}\]
        such that the vertical maps are \'etale. We equip the upper correspondence with the $(-1)$-shifted Lagrangian structure induced from the bottom one.
        Then the following diagram commutes:
        \[\begin{tikzcd}
            {(f_1')^*\varphi_{\calX_1'}} &&& {(f_2')^!\varphi_{\calX_2'}[-d]} \\
            {(f_1')^*u_1^*\varphi_{\calX_1}} & {u_{\calL}^* f_1^*\varphi_{\calX_1}} & {u_{\calL}^! f_2^!\varphi_{\calX_2}[-d]} & {(f_2')^!u_2^!\varphi_{\calX_2}[-d].}
            \arrow["{[\calL']_{\mathrm{Lag}}}", from=1-1, to=1-4]
            \arrow["\eqref{eq-etale-compatibility}"', from=1-1, to=2-1]
            \arrow["\eqref{eq-etale-compatibility}", from=1-4, to=2-4]
            \arrow["\cong"', from=2-1, to=2-2]
            \arrow["{[\calL]_{\mathrm{Lag}}}"', from=2-2, to=2-3]
            \arrow["\cong"', from=2-3, to=2-4]
        \end{tikzcd}\]

		\item \label{item-shift-Lag-smooth} 
		Smooth case: Assume that $f_1$ is smooth and $f_2$ induces an isomorphism on the classical truncation. Then the map \eqref{eq-shifted-lag-class} is invertible.
\qed
    \end{enumerate}

\end{theorem}

\subsubsection{Example: Dimensional reduction}

Let $\calY$ be a quasi-smooth derived Artin stack. Set $\widetilde{\calY} = \mathrm{T}^*[-1]\calY$ as in Example~\ref{ex-shifted-cotangent}.
Let $0 \colon \calY \to \widetilde{\calY}$ be the zero section.
Then it is shown in \cite[Theorem 2.4]{Cal19} that the correspondences
\[
	\bullet \leftarrow  \calY \xrightarrow{0} \widetilde{\calY}, \quad \widetilde{\calY} \xleftarrow{0} \calY \to  \bullet 
\]
are Lagrangian correspondences. It follows from \cite[lem.~3.19]{KPS} that these correspondences are equipped with orientations.
Therefore Theorem~\ref{thm-shifted-Lag-class} yields maps
\[
	\mathbb{Q}_{\calY}[\vdim \calY] \to 0^! \varphi_{\widetilde{\calY}}, 
	\quad 0^* \varphi_{\widetilde\calY} \to \mathbb{D}\mathbb{Q}_{\calY}[-\vdim \calY].
\]
It is proved in \cite{KKPS} that these maps are invertible.
Now let $\kappa \colon \widetilde{\calY} \to \calY$ be the projection.
Then we have the following statement, which is called the {dimensional reduction theorem}.

\begin{theorem}\label{thm-dimensional-reduction}
	There are isomorphisms of complexes of monodromic mixed Hodge modules in $\Dhcmonb(\calY)$
	\begin{equation}\label{eq:dimensional-reduction}
		\kappa_! \varphi_{\widetilde\calY}  \cong    \mathbb{Q}_{\calY}\otimes\mathbb{L}^{- \frac12\vdim \calY}, \quad \kappa_* \varphi_{\widetilde\calY}  \cong   \mathbb{D} \mathbb{Q}_{\calY}\otimes\mathbb{L}^{\frac12\vdim \calY}.
	\end{equation}
\end{theorem}
\begin{proof}
	At the level of constructible sheaves, this follows from the above argument together with isomorphisms 
	$0^! \varphi_{\widetilde{\calY}} \cong \kappa_! \varphi_{\widetilde{\calY}}$ and $0^* \varphi_{\widetilde{\calY}} \cong \kappa_* \varphi_{\widetilde{\calY}}$ from \cite[prop.~3.7.2]{KS13}.
	A lift of these isomorphisms to monodromic mixed Hodge module complexes is explained in \cite[\S~6.2.1]{Bu25b}.
\end{proof}

\begin{remark}
We constructed the isomorphisms \eqref{eq:dimensional-reduction} using shifted Lagrangian classes.
Alternatively, one can prove the existence of the isomorphisms \eqref{eq:dimensional-reduction} using the local model of quasi-smooth 
derived Artin stacks, see \cite[thm.~4.14]{K}.
\end{remark}

\subsubsection{Example: Integral isomorphism}\label{sssec-ex-integral-isom}

Let $\calX$ be a $(-1)$-shifted symplectic Artin $1$-stack equipped with an orientation $o \colon S^{\otimes 2} \cong \det(\mathbb{L}_{\calX})$.
Assume that $\calX$ is $1$-Artin, quasi-separated and has affine stabilizers.
We consider the derived stacks
\[
    \Filt(\calX) = \Map(\A^1 / \mathbb{G}_{\mathrm{m}}, \calX)
,\quad
\Grad({\calX} ) = \Map(\mathrm{B} \mathbb{G}_{\mathrm{m}}, \calX).
\]
The classical truncations are the classical mapping stacks 
$$\Filt(\calX)^\cl =\Filt(\calX^\cl) = \Map(\A^1 / \mathbb{G}_{\mathrm{m}}, \calX^\cl)
,\quad
\Grad(\calX)^\cl=\Grad(\calX^\cl) = \Map(\mathrm{B} \mathbb{G}_{\mathrm{m}}, \calX^\cl).$$
The diagram $\mathrm{B} \mathbb{G}_{\mathrm{m}} \to \mathbb{A}^1/\mathbb{G}_{\mathrm{m}} \leftarrow \bullet$ defines the following correspondence
\begin{align}\label{attr1}\mathrm{attr}_{\calX} \colon \Grad(\calX) \xleftarrow{\gr} \mathrm{Filt}(\calX) \xrightarrow{\ev} \calX.\end{align}
It is shown in \cite[cor.~5.19]{KPS} that $\Grad({\calX})$ inherits the structure of $(-1)$-shifted symplectic 
stack and $(\gr, \ev)$ that of a Lagrangian correspondence. Further, \cite[lem.~3.19]{KPS} shows that there is a unique orientation 
$\mathrm{attr}_{\calX}^{\star} (o)$ on $\Grad(\calX)$ and the above Lagrangian correspondence up to a unique isomorphism.
In particular, Theorem~\ref{thm-shifted-Lag-class} implies that there is a map
\[
   \varphi_{\Grad(\calX)} \to \gr_* \ev^! \varphi_{\calX}[- \vdim \Filt(\calX)].
\]
It is proved in \cite{KKPS} that this map is an isomorphism. We call it the {integral isomorphism}.
There is a more direct construction of the integral isomorphism as in \cite[cor.~7.19]{KPS} and \cite[thm.~1.2]{Des} based on the compatibility 
between the vanishing cycle functor and hyperbolic localization.
With this approach, one can show that the above map refines to an isomorphism of monodromic mixed Hodge modules
\begin{equation}\label{eq-integral-isom}
\varphi_{\Grad(\calX)} \cong \mathbb{L}^{\frac{1}{2} \vdim \Filt(\calX)} \otimes \gr_* \ev^! \varphi_{\calX}.
\end{equation}

For later use, we will show that the orientation $\mathrm{attr}_{\calX}^{\star} (o)$ is preserved under the reversion of the filtrations.
Let $\mathbb{A}^1_{-}$ be $\mathbb{A}^1$ equipped with $\mathbb{G}_{\mathrm{m}}$-action of weight $-1$, 
set $\Filt_{-}(\calX) = \Map(\mathbb{A}^1_{-} / \mathbb{G}_{\mathrm{m}}, \calX)$ and define the attractor correspondence 
\begin{align}\label{attr2}
\mathrm{attr}_{\calX, -} \colon\Grad(\calX) \xleftarrow{\gr_{-}} \Filt_{-}(\calX) \xrightarrow{\ev_{-}} \calX
\end{align} 
in a similar manner.
This naturally defines an oriented Lagrangian correspondence and we let $\mathrm{attr}_{\calX, -}^{\star}(o)$ denote the induced orientation on $\Grad(\calX)$.
Consider the following composite of the positive attractor correspondence and the swap of the negative attractor correspondence:
\[\begin{tikzcd}
	&& {\Filt(\calX) \times_{\calX}  \Filt_{-}(\calX)} && \\
	& {\Filt(\calX)} && {\Filt_{-}(\calX) } \\
	{\Grad(\calX)} && \calX && {\Grad(\calX).}
	\arrow[from=1-3, to=2-2]
	\arrow[from=1-3, to=2-4]
	\arrow["\gr"', from=2-2, to=3-1]
	\arrow["\ev", from=2-2, to=3-3]
	\arrow["{\ev_{-}}"', from=2-4, to=3-3]
	\arrow["{\gr_{-}}", from=2-4, to=3-5]
\end{tikzcd}\]
As it is the composite of oriented Lagrangian correspondences, the large correspondence is naturally an oriented Lagrangian correspondence.
By \cite[lem.~A.2]{KPS}, the natural map $\Grad(\calX) \to \Filt(\calX) \times_{\calX}  \Filt_{-}(\calX)$ is \'etale.
By restricting the upper correspondence to $\Grad(\calX)$, we obtain a correspondence
$\Grad(\calX) \xleftarrow{\id} \Grad(\calX) \xrightarrow{\id} \Grad(\calX)$ where the source and target are equipped with the orientations
$\mathrm{attr}_{\calX}^{\star}(o)$ and $\mathrm{attr}_{\calX, -}^{\star}(o)$ respectively.
In particular, we obtain a natural isomorphism of orientations
\begin{equation}\label{eq-rho-X}
	\rho_{\calX} \colon \mathrm{attr}_{\calX}^{\star}( o) \cong \mathrm{attr}_{\calX, -}^{\star} (o).
\end{equation}

\subsubsection{Purity transformations and virtual fundamental classes}\label{sssec-purity-transform}

In this subsection, we recall the purity transformation from \cite{K19}, which is a sheaf-theoretic version of the virtual pullback.
Later, we will see that the shifted Lagrangian class  gives a microlocal refinement of the purity transformation.
Let $f \colon \calY_1 \to \calY_2$ be a quasi-smooth morphism of relative virtual dimension $d$. 
Then Khan \cite[const.~3.6]{K19}  constructed a natural morphism
\[
 [f] \colon \mathbb{Q}_{\calY_1} \to \mathbb{L}^{d} \otimes f^! \mathbb{Q}_{\calY_2}
\]
in $\Dhcb(\calY_1)$. Khan \cite{K19} works with motivic sheaves, but the same construction works in the setting of mixed Hodge modules, using the six-functor formalism of mixed Hodge modules for stacks recently constructed by Tubach \cite{Tubstack}.
For $M \in \Dhcb(\calY_2)$, we define the {purity transformation}
\begin{equation}\label{eq-purity-transform}
 \mathrm{pur}_f \colon f^* M \to \mathbb{L}^{d} \otimes f^!M
\end{equation}
to be the composite
\[
    f^* M \cong f^* M \otimes \mathbb{Q}_{\calY_1} \xrightarrow[]{\id \otimes [f]} f^* M \otimes (\mathbb{L}^{d} \otimes  f^!    \mathbb{Q}_{\calY_2}) \to \mathbb{L}^{d} \otimes f^! M.
\]
The purity transformation satisfies the following properties
\begin{itemize}
\item Associativity: assume that we are given quasi-smooth morphisms $\calY_1 \xrightarrow{f} \calY_2 \xrightarrow{g}  \calY_3$ and denote $d = \vdim f$ 
and $d' = \vdim g$.
Take $N \in \Dhcb(\calY_3)$. Then the following diagram commutes
		  \[\begin{tikzcd}
			{(g\circ f)^* N} && {\mathbb{L}^{d + d'} \otimes (g\circ f)^! N} \\
			{f^* g^*  N} & {\mathbb{L}^{d} \otimes f^* g^!  N} & {\mathbb{L}^{d + d'} \otimes f^! g^! N.}
			\arrow["{\mathrm{pur_{g \circ f}}}", from=1-1, to=1-3]
			\arrow["\cong"', from=1-1, to=2-1]
			\arrow["\cong", from=1-3, to=2-3]
			\arrow["{\mathrm{pur_{g}}}"', from=2-1, to=2-2]
			\arrow["{\mathrm{pur_{f}}}"', from=2-2, to=2-3]
		\end{tikzcd}\]
		This is a consequence of \cite[prop.~3.12]{K19}.
	\item Functoriality: assume that we are given the following Cartesian diagram of derived Artin stacks 
	\[\begin{tikzcd}
		{\calY_1'} & {\calY_2'} \\
		{\calY_1} & {\calY_2}
		\arrow["{f'}", from=1-1, to=1-2]
		\arrow["g"', from=1-1, to=2-1]
		\arrow["\lrcorner"{anchor=center, pos=0.125}, draw=none, from=1-1, to=2-2]
		\arrow["h", from=1-2, to=2-2]
		\arrow["f", from=2-1, to=2-2]
	\end{tikzcd}
	\]
	such that $f$ is quasi-smooth. Let $d = \vdim f$. Take $M \in \Dhcb(\calY'_2)$. Then the following diagram commutes
	\[\begin{tikzcd}
		{f^* h_* M} & {g_*f'^*M} \\
		{\mathbb{L}^d \otimes f^! h_*M} & {\mathbb{L}^d \otimes g_*f'^! M.}
		\arrow[from=1-1, to=1-2]
		\arrow["{\mathrm{pur}_f}"', from=1-1, to=2-1]
		\arrow["{g_*\mathrm{pur}_f}", from=1-2, to=2-2]
		\arrow["\cong"', from=2-1, to=2-2]
	\end{tikzcd}\]
	This is a consequence of \cite[thm.~3.13]{K19}.
\end{itemize}

The purity transformation for $M = \mathbb{D} \mathbb{Q}_{\calY_2}$ yields the following map on the global sections
\[
 f^! \colon \H^{\mathrm{BM}}_*(\calY_2,\Q) \to \H^{\mathrm{BM}}_{* + 2 \vdim f}(\calY_1,\Q), 
\]
which is called the {virtual pullback}.
In particular, for a quasi-smooth derived Artin stack $\calY$, there is a canonical class
\[
 [\calY] \in \H^{\mathrm{BM}}_{2 \vdim \calY}(\calY,\Q)
\]
which is called the virtual fundamental class.

\subsubsection{Purity transformation via shifted Lagrangian classes}\label{sssec-purity-via-Lag}

We now explain that the shifted Lagrangian class gives a microlocal refinement of the purity 
transformation \eqref{eq-purity-transform} in a special case.
Consider the correspondence $$\zeta \colon \calY_1 \xleftarrow[]{q} \calZ \xrightarrow[]{p} \calY_2.$$
Set $\widetilde{\calY}_1 = \mathrm{T}^*[-1]\calY_1$, $\widetilde{\calY}_2 = \mathrm{T}^*[-1]\calY_2$ and $\calL = \mathrm{T}^*[-2](\calZ / \calY_1 \times \calY_2)$.
Consider the following diagram
\begin{equation}\label{L-to-Z}
\begin{tikzcd}
	{\widetilde{\calY}_1} & \calL & {\widetilde{\calY}_2} \\
	{\calY_1} & \calZ & {\calY_2.}
	\arrow["{\kappa_1}"', from=1-1, to=2-1]
	\arrow["{\tilde{q}}"', from=1-2, to=1-1]
	\arrow["{\tilde{p}}", from=1-2, to=1-3]
	\arrow["{\kappa_{\calZ}}"', from=1-2, to=2-2]
	\arrow["{\kappa_2}"', from=1-3, to=2-3]
	\arrow["q"', from=2-2, to=2-1]
	\arrow["p", from=2-2, to=2-3]
\end{tikzcd}
\end{equation}
By \cite[thm.~2.13]{Cal19} the upper correspondence carries a natural Lagrangian structure together with an orientation, which we call the conormal Lagrangian correspondence and denote by $\mathrm{T}^*[-1](\zeta)$.
We will later use the following properties of conormal Lagrangian correspondences, which are proved in \cite{KKPS}.
\begin{enumerate}[label=(C\arabic*), ref=(C\arabic*)]
	\item Functoriality: The conormal correspondence of the identity correspondence is equivalent to the identity correspondence. Also, the composition of conormal Lagrangian correspondences is naturally equivalent to the conormal Lagrangian correspondence of the composition. \label{item-conormal-1}
	\item Compatibility with swaps: The swap of the conormal Lagrangian correspondence $\tilde{\calY}_1 \leftarrow \calL \to \tilde{\calY}_2$ is equivalent to the conormal Lagrangian correspondence for $\calY_2 \xleftarrow{p} \calZ \xrightarrow{q} \calY_1$ without orientations. The orientation differs by the sign $(-1)^{\vdim(p)}$. \label{item-conormal-2}
\end{enumerate}

Assume from now on that $q$ is quasi-smooth and $p$ is proper.
Since the fibre of the map $\mathbb{L}_{\mathcal{Z} / \mathcal{Y}_1 \times \mathcal{Y}_2}[-2] \to p^*\mathbb{L}_{\mathcal{Y}_2}[-1]$ is equivalent to $\mathbb{L}_{\mathcal{Z} / \mathcal{Y}_1}[-2]$,
the quasi-smoothness of $q$ implies that the map $\mathcal{L} \to \tilde{\calY}_2 \times_{\calY_2} \calZ$ is a closed immersion.
In particular, the map $\tilde{p}$ is proper.
Therefore Theorem~\ref{thm-shifted-Lag-class} yields a map
\begin{equation}\label{eq-micro-virtual-pullback}
  \tilde{p}_* \tilde{q}^* \varphi_{\tilde\calY_1} \to \varphi_{\tilde\calY_2}[- \vdim \calL].
\end{equation}
The following statement is proved in \cite{KKPS}. 
\begin{proposition}\label{prop-dim-red-vs-purity}
    The following diagram in $\mathsf{D}(\calY_2)$ commutes
    \[\begin{tikzcd}
        {p_* q^* \kappa_{1, *} \varphi_{\tilde\calY_1}} & {p_* \kappa_{\calZ, *} \tilde{q}^*  \varphi_{\tilde\calY_1}} & {\kappa_{2, *} \tilde{p}_* \tilde{q}^*  \varphi_{\tilde\calY_1}} & {\kappa_{2, *}   \varphi_{\tilde\calY_2}[- d_{\calL}]} \\
        {p_* q^* \mathbb{D}\mathbb{Q}_{\calY_1}[- d_1]} && {p_* \mathbb{D}\mathbb{Q}_{\calZ}[- d_1 - 2d_q]} & { \mathbb{D}\mathbb{Q}_{\calY_2}[- d_{\calL} -d_2].}
        \arrow[from=1-1, to=1-2]
        \arrow["{\eqref{eq:dimensional-reduction}}"', from=1-1, to=2-1]
        \arrow["\cong", from=1-2, to=1-3]
        \arrow["{\eqref{eq-micro-virtual-pullback}}", from=1-3, to=1-4]
        \arrow["{\eqref{eq:dimensional-reduction}}", from=1-4, to=2-4]
        \arrow["{\mathrm{pur}_q}"', from=2-1, to=2-3]
        \arrow["{\mathrm{counit}}"', from=2-3, to=2-4]
    \end{tikzcd}\]
Here we denote $d_1 = \vdim \calY_1$, $d_2 = \vdim \calY_2$, $d_{\calL} = \vdim \calL$ and $d_q = \vdim q$. 
The left upper map is the Beck--Chevalley map.
\end{proposition}

Applying the above proposition for the correspondence $\bullet \leftarrow \calY = \calY$ for quasi-smooth $\calY$, we obtain the following statement, generalizing \cite[thm.~5.2]{kinjo2021virtual}.

\begin{corollary}\label{cor-VFC-via-dim-red}
Let $d = \vdim \calY$ and $\tilde\calY=\mathrm{T}^*[-1]\calY$.
The following diagram in $\Dcb(\calY)$ commutes:
	\[\begin{tikzcd}
		{\kappa_* 0_!\mathbb{Q}_{\calY}} & {\kappa_* \varphi_{\tilde\calY}[-d]} \\
		{\mathbb{Q}_{\calY}} & {\mathbb{DQ}_{\calY}[-2d].}
		\arrow["{\eqref{eq-shifted-lag-class}}", from=1-1, to=1-2]
		\arrow["\cong"', from=1-1, to=2-1]
		\arrow["{\eqref{eq:dimensional-reduction}}", from=1-2, to=2-2]
		\arrow["{[\calY]}"', from=2-1, to=2-2]
	\end{tikzcd}\]
	
\end{corollary}

\medskip 

\subsection{Microlocal homology}

\subsubsection{Microlocal homology}

Let $\calY$ be a quasi-smooth derived Artin stack and $\Lambda \subset\widetilde\calY^\cl$ a closed conic subset.
Following \cite[def.~1.24]{Schefers22} and \cite[\S 6.2]{kinjo2021virtual}, we define the {microlocal homology} associated with $\Lambda$ as follows
\[
\H^{\Lambda}_{- * + \vdim \calY} (\calY,\Q) = \H^*(\Lambda, \varphi_{\widetilde\calY} |^!_{\Lambda}).
\]
Let $\Lambda_{0} \subset \widetilde\calY^\cl$ denote the image of the zero section of $\kappa \colon \widetilde{\calY} \to \calY$.
Proposition \ref{prop:equivariance} and Theorem \ref{thm-dimensional-reduction} yield the following isomorphisms.

\begin{corollary}\label{cor-microlocal-dim-red}
We have
$\H^{\widetilde\calY}_{*}(\calY,\Q) \cong \H^\BM_{*}(\calY,\Q)$
and
$\H^{\Lambda_{0}}_{- * + 2 \vdim \calY}(\calY,\Q) \cong \H^*(\calY,\Q)$.
\qed
\end{corollary} 

The microlocal homology is functorial with respect to inclusions of cones:
for a given inclusion $\Lambda \subset \Lambda'$ there is a morphism
$\H^{\Lambda}_* (\calY,\Q)  \to \H^{\Lambda'}_{*} (\calY,\Q).$
In particular, setting $\Lambda'=\widetilde\calY^\cl$ we get a morphism
\begin{equation}\label{eq:cone-BM1}
\H^{\Lambda}_* (\calY,\Q) \to \H^\BM_{*} (\calY,\Q).
\end{equation}

 \medskip

\subsubsection{The nilpotent microlocal homology}
\label{sec:nilp-micro}
Let $\calC$ be a 2CY category and $\calM \subset \calM_{\calC}$ be an open substack which is $1$-Artin.
Then $\calM$ is 0-shifted symplectic and quasi-smooth.
Under the identification $\widetilde\calM \cong \mathrm{T}[-1] \calM$, 
a point in $\widetilde\calM$ corresponds to a pair $(\calF, \phi)$ where
$\calF \in \calM$ and $\phi \colon\calF \to \calF$.
Let
\begin{align}\label{Lambda-Lambda}
\Lambda_{\nil},\,\Lambda_{\snil}\subset\widetilde\calM_\red\end{align}
be the set of the pairs $(\calF, \phi)$ where $\phi$ is nilpotent and $\phi - c$ is nilpotent for some $c \in \C$ respectively.

\begin{proposition}
\hfill
\begin{enumerate}[label=$\mathrm{(\alph*)}$,leftmargin=8mm,itemsep=2mm]
\item 
The subsets $\Lambda_{\nil}, $ $\Lambda_{\snil}$ of $\widetilde \calM_\red$ are closed.
\item
There are natural graded vector space morphisms
\begin{align}\label{eq:cone-BM2}
\HH_*^{\Lambda_{\nil}}(\calM,\Q ) \to\HH_*^{\Lambda_{\snil}}(\calM,\Q ) \to\HH_*^\BM(\calM,\Q ).
\end{align}
\end{enumerate}
\end{proposition}

\begin{proof}
Part (b) follows from (a). Let us concentrate on the proof of (a).
We first prove that the nilpotent locus $\Lambda_{\nil} \subset \widetilde\calM_\red$ is closed.
Take an affine scheme $T$ and a morphism $T \to \calM$. Consider the following Cartesian diagram
\[\begin{tikzcd}
	{\Lambda_{\nil, T}} & {\widetilde{T}} & T \\
	{\Lambda_{\nil}} & {\widetilde{\calM}_\red} & {\calM_\red.}
	\arrow[from=1-1, to=1-2]
	\arrow[from=1-1, to=2-1]
	\arrow["\lrcorner"{anchor=center, pos=0.125}, draw=none, from=1-1, to=2-2]
	\arrow[from=1-2, to=1-3]
	\arrow[from=1-2, to=2-2]
	\arrow["\lrcorner"{anchor=center, pos=0.125}, draw=none, from=1-2, to=2-3]
	\arrow[from=1-3, to=2-3]
	\arrow[from=2-1, to=2-2]
	\arrow[from=2-2, to=2-3]
\end{tikzcd}\]
It is enough to show that $\Lambda_{\nil, T}$ is closed in $\widetilde{T}$.
The dimension of the fibre $\widetilde{T} \to T$ is bounded by some integer $N$.
Let $\theta_N \colon \widetilde\calM_\red \to \widetilde\calM_\red$ be the map given by 
$(E, \phi) \mapsto (E, \phi^N)$, and $\theta_{N, T} \colon \widetilde{T} \to \widetilde{T}$ be its base change.
Then we  have 
    \[
        \Lambda_{\nil, T} = \theta_{N, T}^{- 1}(0_T(T))
    \]
 where $0_T \colon T \to \widetilde{T}$ denotes the zero section. In particular, we conclude that $\Lambda_{\nil, T}$ is closed.
The statement for $\Lambda_{\snil} \subset \widetilde\calM_\red$ follows from the following lemma applied to the action map 
$\A^1 \times  \Lambda_{\nil} \to \widetilde\calM_\red$.

\end{proof}

\begin{lemma}
Let $T$ be an affine scheme of finite type over $\C$. 
Let $C_1$ and $C_2$ be schemes with $\mathbb{G}_{\mathrm{m}}$-equivariant closed immersions into vector bundles over $T$.
Let $f \colon C_1 \to C_2$ be a $\mathbb{G}_{\mathrm{m}}$-equivariant injective morphism over $T$. Then $f$ is a finite morphism.
\end{lemma}

\begin{proof}
Set $\mathbb{P}(C_i \times \A^1) = (C_i \times \A^1 \smallsetminus T) / \mathbb{G}_{\mathrm{m}}$ for $i =1, 2$.
The assumption implies that $\mathbb{P}(C_i \times \A^1)$ is projective over $T$ and the map
$\bar f \colon \mathbb{P}(C_1 \times \A^1) \to \mathbb{P}(C_2 \times \A^1)$ is injective.
So this map is proper and quasi-finite, hence it is finite.
Since the map $\bar f$ restricts to $f$, we obtain the desired claim.
\end{proof}

The microlocal homology associated with the nilpotent and semi-nilpotent cones will play an important role.

\begin{proposition}
The morphism
$\HH_*^{\Lambda_{\nil}}(\calM,\Q ) \to\HH_*^{\Lambda_{\snil}}(\calM,\Q )$ in \eqref{eq:cone-BM2}
is zero.
We have an isomorphism
\begin{equation}\label{eq-nilp-vs-seminilp}
\HH_*^{\Lambda_{\snil}}(\calM,\Q ) \cong  \HH_*^{\Lambda_{\nil}}(\calM,\Q )\otimes\mathbb{L}^{-1}.
\end{equation}
\end{proposition}

\begin{proof}
By Proposition \ref{prop:equivariance} the DT perverse sheaf $\varphi_{\widetilde\calM}$ is $\A^1$-equivariant.
Hence, under the identification $\Lambda_{\snil}\cong\A^1\times\Lambda_{\nil}$, we have an isomorphism
$$\varphi_{\widetilde\calM}\big|_{\Lambda_\snil}^!\cong\D\Q_{\A^1}\boxtimes\varphi_{\widetilde\calM}\big|_{\Lambda_\nil}^!$$
such that the morphism
$\HH_*^{\Lambda_{\nil}}(\calM,\Q ) \to\HH_*^{\Lambda_{\snil}}(\calM,\Q )$ in \eqref{eq:cone-BM2}
is given by the obvious morphism $\Q_{\{0\}}\to\D\Q_{\A^1}$, which is zero in cohomology.
Finally, we have $\H^*(\A^1,\D\Q_{\A^1})\cong\Q[2]\cong \L^{-1}$.
\end{proof}

\medskip 

\subsubsection{Reduced virtual fundamental class via microlocal homology}\label{red_vc_sec}

We adopt the notation from the previous section.
Consider the natural $\mathrm{B}\G_{\mathrm{m}}$-action on $\calM$.
By \cite[prop.~3.1.9]{Bu25b} a $\BB\G_m$-action on a 0-shifted symplectic stack admits a moment map
\[
 \mu \colon \calM \to \mathbb{A}^1[-1].	
\]
We set $\calM^{\mathrm{rd}} = \mu^{-1}(0)$. Then $\calM^{\mathrm{rd}}$ is quasi-smooth and we have $(\calM^{\mathrm{rd}})^{\cl} = \calM^{\cl}$.
We define the reduced virtual fundamental class of $\calM$ by
\[
 [\calM]_{\mathrm{rd}} = [\calM^{\mathrm{rd}}] \in \H^{\mathrm{BM}}_{2 \vdim \calM^{\mathrm{rd}}}(\calM^{\mathrm{rd}}) \cong \H^{\mathrm{BM}}_{2 \vdim \calM + 2}(\calM).
\]
Note that $[\calM] = 0$ since the obstruction bundle contains $\calM \times \mathbb{A}^1$. 
Therefore the reduced virtual fundamental class is the
more reasonable object to study.
We now explain a construction of the reduced virtual fundamental class using microlocal homology.
Let 
\begin{equation}\label{eq-def-Lambda-scal}
	\Lambda_{\mathrm{scal}} = \calM^\cl \times \mathbb{A}^1 \subset \widetilde{\calM}^{\cl}
\end{equation}
be the closed subset consisting of pairs $(\calF, \phi)$ 
where $\phi \colon \calF \to \calF$ is a scalar multiplication.
We have the following statement.

\begin{proposition}\label{prop-scal-microlocal} \hfill
\begin{enumerate}[label=$\mathrm{(\alph*)}$,leftmargin=8mm,itemsep=2mm]
\item There is a natural isomorphism 
\begin{equation}\label{eq-scal-micro}
\H^{\Lambda_{\mathrm{scal}}}_{- * + 2 \vdim \calM + 2}(\calM,\Q) \cong \H^*(\calM,\Q).
\end{equation} \label{item-scal-compute}
\item Consider the composition of the following maps
\begin{equation}\label{eq-composite-Lambda-scal}
\H^0(\calM,\Q) \xrightarrow[\cong]{\eqref{eq-scal-micro}} 
\H^{\Lambda_{\mathrm{scal}}}_{2 \vdim \calM + 2}(\calM,\Q) \xrightarrow{\eqref{eq:cone-BM1}} \H^{\mathrm{BM}}_{2 \vdim \calM + 2}(\calM,\Q).
\end{equation}
The image of $1 \in \H^0(\calM,\Q)$ under the above map is $[\calM]_{\rd}$. \label{item-scal-VFC}
\end{enumerate}
\end{proposition}

\begin{proof}
Let $\Lambda^{\rd}_0 \subset (\widetilde{\calM^{\rd}})^{\cl}$ be the image of the zero section.
Then we have the following Cartesian diagram
\[\begin{tikzcd}
		{\Lambda_{\mathrm{scal}}} & {\Lambda^{\rd}_0} \\
		{(\widetilde{\calM})^{\cl}} & {(\widetilde{\calM^{\rd}})^{\cl}}
		\arrow["{a_{\Lambda}}", from=1-1, to=1-2]
		\arrow[from=1-1, to=2-1]
		\arrow["\lrcorner"{anchor=center, pos=0.125}, draw=none, from=1-1, to=2-2]
		\arrow[from=1-2, to=2-2]
		\arrow["a", from=2-1, to=2-2]
\end{tikzcd}\]
where $a$ is induced by the map $\mathbb{L}_{\calM} [-1] |_{\calM^{\rd}} \to \mathbb{L}_{\calM^{\rd}} [-1]$.
Since $a$ is smooth, using Theorem~\ref{thm-shifted-Lag-class}\ref{item-shift-Lag-smooth}, we have $\varphi_{\widetilde{\calM}} \cong a^* \varphi_{\widetilde{\calM^{\rd}}}[1]$.
	In particular, we have
	\[
	 \varphi_{\widetilde{\calM}} |^! _{\Lambda_{\mathrm{scal}}} \cong (a^*\varphi_{\widetilde{\calM^{\rd}}}) |^! _{\Lambda_{\mathrm{scal}}}[1] \cong a_{\Lambda}^* (\varphi_{\widetilde{\calM^{\rd}}} |^!_{\Lambda^{\rd}_0})[1] \xrightarrow[\cong]{\eqref{eq:dimensional-reduction}} \mathbb{Q}_{\Lambda_{\mathrm{scal}}}[\vdim \calM + 2].
	\]
This proves the first claim.
By construction, the composite \eqref{eq-composite-Lambda-scal} is identified with
	\[
	  \H^0(\calM,\Q) \cong \H^0(\calM^{\rd},\Q) \xrightarrow[\cong]{\textnormal{cor.~\ref{cor-microlocal-dim-red}}} \H^{\Lambda_{0}^{\rd}}_{2 \vdim \calM^{\rd}}(\calM^{\rd},\Q) \xrightarrow{\eqref{eq:cone-BM1}} \H^{\mathrm{BM}}_{2 \vdim \calM^{\rd}}(\calM^{\rd},\Q) \cong \H^{\mathrm{BM}}_{2 \vdim \calM + 2}(\calM,\Q).
	\]
Therefore the second claim follows from Corollary~\ref{cor-VFC-via-dim-red}.

\end{proof}

\subsection{The CoHA}

Now, we recall the construction of the CoHA (=cohomological Hall algebra) for 3CY categories, 
specializing to the case of the $3$CY completion of a 2CY category.
Let $\calC$ be a 2CY category.

\subsubsection{The assumptions on the stack \texorpdfstring{$\calM$}{\calM}}\label{para-assumptions-M}
Let $\calM \subset \calM_{\calC}$ be an open substack which parametrizes objects in 
a $\C$-linear Abelian category.
Recall that the classical stack $\calM^\cl$  is $\Theta$-reductive if the map 
$$\Filt(\calM^\cl) = \Map(\A^1 / \mathbb{G}_{\mathrm{m}}, \calM^\cl) \to \calM^\cl$$ is proper on each 
connected component of the source.
We refer to \cite[def.~5.1.1]{halpern2014structure} for more details on $\Theta$-reductiveness.
We say that $\calM$ has quasi-compact graded points if the morphism
$$\Grad(\calM^\cl) = \Map(\mathrm{B} \mathbb{G}_{\mathrm{m}}, \calM^\cl)\to \calM^\cl$$ induced by pullback along the obvious morphism 
$\bullet\to\mathrm{B} \mathbb{G}_{\mathrm{m}}$ is quasi-compact on each connected component of the source.
We will consider the following conditions
\hfill
\begin{enumerate}[leftmargin=8mm,itemsep=2mm]
    \item $\calM$ is a quasi-smooth  Artin $1$-stack locally of finite type, \label{item-M-1-Artin}
    \item $\calM$ is closed under direct sums and summands and it contains $\{ 0 \}$ as an open and closed substack, \label{item-M-closed}
    \item $\calM$ has affine diagonal, \label{item-M-affine-diagonal}
    \item $\calM$  is $\Theta$-reductive,\label{item-M-theta-reductive}
    \item $\calM$ has quasi-compact connected components and quasi-compact graded points.\label{item-ac-graded}
 \end{enumerate}

 \medskip

\subsubsection{The integral isomorphism}\label{para-integral-identity}\label{sec:int-identity}
Let $\calM$ be an open substack of $\calM_{\calC}$ satisfying 
the conditions \eqref{item-M-1-Artin}-\eqref{item-M-theta-reductive}
in \S\ref{para-assumptions-M}.
Since $\widetilde{\calM}$ is affine over $\calM$, the stack $\widetilde\calM$  is also $\Theta$-reductive.
The diagram
    \[
     \mathrm{B} \mathbb{G}_{\mathrm{m}} \to \A^1 / \mathbb{G}_{\mathrm{m}} \leftarrow \bullet    
    \]
yields the correspondences
\begin{equation}\label{attractor1}
\Grad(\calM ) \xleftarrow[]{\mathrm{gr}} \Filt(\calM ) \xrightarrow[]{\mathrm{ev}} \calM 
\end{equation}
and
\begin{equation}\label{attractor2}
\Grad(\widetilde\calM) \xleftarrow[]{\widetilde{\gr}} \Filt(\widetilde\calM) \xrightarrow[]{\widetilde{\ev}} \widetilde\calM.
\end{equation}
By \cite[lem.~1.1.5, prop.~1.1.13, lem.~1.3.8]{halpern2014structure}, the following hold
   \hfill
\begin{enumerate}[leftmargin=8mm,itemsep=2mm]
    \item the morphisms $\ev$, $\tilde\ev$ are representable and proper on each connected component of the source, 
    \item the morphisms $\gr$, $\tilde\gr$ are quasi-compact and $\gr$ is quasi-smooth.
    \item the morphisms $\gr$, $\tilde\gr$ induce isomorphisms on connected components.
 \end{enumerate}
 Note also that both stacks $\Filt(\calM)$ and $\Grad(\calM)$ are quasi-smooth.

As we have seen in \S\ref{sssec-ex-integral-isom}, the $(-1)$-shifted symplectic structure and the orientation on 
$\widetilde\calM$ in \S\ref{sec-orientation} yield a $(-1)$-shifted symplectic structure on $\Grad(\widetilde{\calM })$, 
and a Lagrangian structure and orientation on the  correspondence \eqref{attractor2}.
The following lemma shows that it is a special case of the conormal Lagrangian correspondence \eqref{L-to-Z}.

\begin{lemma}\label{lem-attractor-is-conormal}
The oriented Lagrangian correspondence \eqref{attractor2} is equivalent to the conormal correspondence
	\begin{equation}\label{eq-conormal-of-attractor}
		\mathrm{T}^*[-1](\mathrm{attr}_{\calM}) \colon \mathrm{T}^*[-1] \Grad(\calM) \leftarrow \mathrm{T}^*[-2](\Filt(\calM) / \Grad(\calM) \times \calM) \to \mathrm{T}^*[-1]\calM.
	\end{equation}
\end{lemma}

\begin{proof}
Since the orientations of these Lagrangian correspondences are uniquely determined by the orientation on $\widetilde{\calM}$ by \cite[lem.~3.19]{KPS}, it is sufficient to compare them as Lagrangian correspondences.
This is a relative version of \cite[prop.~3.9]{CS}, which we briefly explain below.
First, an equivalence of $(-1)$-shifted symplectic stacks $\mathrm{T}^*[-1]\Grad(\calM) \cong \Grad(\widetilde{\calM})$ is a direct consequence of \cite[prop.~3.9]{CS}.
Next, we construct an equivalence 
	\begin{equation}\label{eq-Filt-is-shifted-conormal}
	 \Filt(\widetilde{\calM}) \cong 	\mathrm{T}^*[-2](\Filt(\calM) / \Grad(\calM) \times \calM) 
	\end{equation}
as derived stacks.
For a derived affine scheme $T$, a  $T$-valued point $T \to \Filt(\widetilde{\calM})$ is 
a pair of a morphism 
$T \times \mathbb{A}^1 / \mathbb{G}_{\mathrm{m}} \to \calM$ and a section 
$$s \in \Gamma(T \times \mathbb{A}^1 / \mathbb{G}_{\mathrm{m}}, 
\mathbb{L}_{\calM}[-1] |_{T \times \mathbb{A}^1 / \mathbb{G}_{\mathrm{m}}}).$$
On the other hand, a $T$-valued point $T \to \mathrm{T}^*[-2](\Filt(\calM) / \Grad(\calM) \times \calM)$ is a pair of a morphism 
$T \times \mathbb{A}^1 / \mathbb{G}_{\mathrm{m}} \to \calM$, or equivalently $T \to \Filt(\calM)$
and a section 
$$t \in \Gamma(T, \mathbb{L}_{\Filt(\calM) / \Grad(\calM) \times \calM}[-2]|_{T}).$$
Consider the following diagram
	\[\begin{tikzcd}
		{\Filt(\calM) \times (\mathrm{B}\mathbb{G}_{\mathrm{m}} \coprod \bullet)} & {(\Grad(\calM) \times \calM) \times (\mathrm{B}\mathbb{G}_{\mathrm{m}} \coprod \bullet)} \\
		{\Filt(\calM) \times (\mathbb{A}^1  / \mathbb{G}_{\mathrm{m}} )} & \calM \\
		{\Filt(\calM).}
		\arrow[from=1-1, to=1-2]
		\arrow["{\id \times \iota}"', from=1-1, to=2-1]
		\arrow["{\ev_{\mathrm{B}\mathbb{G}_{\mathrm{m}} \coprod \bullet}}", from=1-2, to=2-2]
		\arrow["{\ev_{\mathbb{A}^1  / \mathbb{G}_{\mathrm{m}}}}"', from=2-1, to=2-2]
		\arrow["{\pr_1}"', from=2-1, to=3-1]
	\end{tikzcd}\]
	By \cite[prop.~B.3.7]{Roz}, the relative cotangent complex $\mathbb{L}_{\Filt(\calM) / \Grad(\calM) \times \calM}$ is described by
	\[
	 \mathop{\mathrm{cofib}}( (\pr_1 \circ (\id \times \iota))_{\sharp} (\ev_{\mathbb{A}^1 / \mathbb{G}_{\mathrm{m}}} \circ (\id \times \iota))^* \mathbb{L}_{\calM} \to \pr_{1, \sharp} \ev_{\mathbb{A}^1 / \mathbb{G}_{\mathrm{m}}}^* \mathbb{L}_{\calM}),
	\]
	where $(-)_{\sharp}$ denotes the left adjoint of the pullback functor between perfect complexes.
	Since $\iota \colon \mathrm{B} \mathbb{G}_{\mathrm{m}} \coprod \bullet \to \mathbb{A}^1 / \mathbb{G}_{\mathrm{m}}$ is equipped with a relative $0$-orientation by \cite[prop.~5.17]{KPS},
	the above complex is naturally equivalent to $\pr_{1, *} \ev_{\mathbb{A}^1 / \mathbb{G}_{\mathrm{m}}, *} \mathbb{L}_{\calM}[-1]$, which implies an equivalence
	\[
		\Gamma(T \times \mathbb{A}^1 / \mathbb{G}_{\mathrm{m}}, \mathbb{L}_{\calM}[-1] |_{T \times \mathbb{A}^1 / \mathbb{G}_{\mathrm{m}}}) \cong  \Gamma(T, \mathbb{L}_{\Filt(\calM) / \Grad(\calM) \times \calM}[-2]|_{T})
	\]
	which is natural in $T$. Therefore we obtain the equivalence of derived stacks \eqref{eq-Filt-is-shifted-conormal}.

	The comparison of the Lagrangian structure is parallel to the absolute case \cite[prop.~3.9]{CS}.

\end{proof}

Since the correspondence \eqref{attractor1} is 0-shifted Lagrangian, we have
$$\vdim\Filt(\calM)=\frac12\big(\vdim\Grad(\calM)+\vdim\calM\big).$$
Combining this with Lemma~\ref{lem-attractor-is-conormal}, we obtain
\begin{equation}\label{eq-vanish-vdim}\vdim\Grad(\widetilde\calM)=\vdim\Filt(\widetilde\calM)=\vdim\widetilde\calM=0.\end{equation}
Let $o_{\mathrm{sta}, \widetilde{\calM}}$ and $o_{\mathrm{sta}, \Grad(\widetilde{\calM})}$ denote the standard orientations defined in Example~\ref{ex-shifted-cotangent} for $\widetilde{\calM}$ and $\Grad(\widetilde{\calM})$, where 
we identify $\Grad(\widetilde{\calM})$ with $\mathrm{T}^*[-1]\Grad(\calM)$.
By Lemma~\ref{lem-attractor-is-conormal} and the uniqueness of the orientation on the attractor correspondence \cite[lem.~3.19]{KPS}, there is a natural isomorphism of orientations
\begin{equation}\label{eq-loc-of-standard-ori-is-standard}
 \mathrm{attr}_{\widetilde{\calM}}^{\star} (o_{\mathrm{sta}, \widetilde{\calM}}) \cong o_{\mathrm{sta}, \Grad(\widetilde{\calM})}.
\end{equation}
The following technical statement will be used later for the compatibility between the CoHA multiplication and the coproduct.

\begin{lemma}\label{lem-ori-differ-by-Euler-form}
The following diagram commutes
\[\begin{tikzcd}
{\mathrm{attr}_{\widetilde{\calM}}^{\star} (o_{\mathrm{sta}, \widetilde{\calM}})} &[55pt] {\mathrm{attr}_{\widetilde{\calM}, -}^{\star} (o_{\mathrm{sta}, \widetilde{\calM}}) } \\
		{o_{\mathrm{sta}, \Grad(\widetilde{\calM})}} & {o_{\mathrm{sta}, \Grad(\widetilde{\calM})}.}
		\arrow["{(-1)^{\vdim \Filt(\calM)} \cdot \rho_{\widetilde{\calM}}}", from=1-1, to=1-2]
		\arrow["{\eqref{eq-loc-of-standard-ori-is-standard}}"', from=1-1, to=2-1]
		\arrow["{\eqref{eq-loc-of-standard-ori-is-standard}}", from=1-2, to=2-2]
		\arrow[equals, from=2-1, to=2-2]
	\end{tikzcd}\]
	See \eqref{eq-rho-X} for the definition of $\rho_{\widetilde{\calM}}$.
\end{lemma}

\begin{proof}
	By the construction of $\rho_{\widetilde{\calM}}$, it is enough to show that the composite of the conormal correspondence \eqref{eq-conormal-of-attractor} and the swap of the conormal correspondence of the negative attractor correspondence
	\begin{equation}\label{eq-negative-conormal-of-attractor}
	 \mathrm{swap}(\mathrm{T}^*[-1](\mathrm{attr}_{\calM, -})) \colon \mathrm{T}^*[-1]\calM \leftarrow \mathrm{T}^*[-2](\Filt_{-}(\calM) / \Grad(\calM) \times \calM) \to \mathrm{T}^*[-1]\Grad(\calM)	
	\end{equation}
	restricts to the identity correspondence for $\mathrm{T}^*[-1]\Grad(\calM)$ with the orientation given by multiplication by $(-1)^{\vdim \Filt(\calM)}$ under the \'etale map 
	\[
		\mathrm{T}^*[-1]\Grad(\calM) \to \mathrm{T}^*[-2](\Filt(\calM) / \Grad(\calM) \times \calM) \times_{\mathrm{T}^*[-1]\calM } \mathrm{T}^*[-2](\Filt_{-}(\calM) / \Grad(\calM) \times \calM).
	\]
	To see this, using \ref{item-conormal-2} in \S\ref{sssec-purity-via-Lag}, we see that \eqref{eq-negative-conormal-of-attractor} is identified with the conormal correspondence after inserting the sign $(-1)^{\vdim \Filt(\calM)}$ on the orientation.
	Therefore by \ref{item-conormal-1} the composite is the conormal correspondence for 
	\[
		\Grad(\calM) \leftarrow \Filt(\calM) \times_{\calM} \Filt_{-}(\calM) \to \Grad(\calM)
	\]
	after inserting the sign $(-1)^{\vdim \Filt(\calM)}$ on the orientation.
	In particular, the restriction to $\mathrm{T}^*[-1]\Grad(\calM)$ is identified with the conormal correspondence for the identity correspondence $\Grad(\calM)$ after inserting the sign $(-1)^{\vdim \Filt(\calM)}$ on the orientation.
	By using \ref{item-conormal-1} again, we obtain the desired claim.

\end{proof}

We now describe $\Grad(\calM )$ and $\Filt(\calM )$ explicitly, following \cite[\S 7.1]{Bu25a}.
Let 
\begin{align}\label{Gamma}
\Gamma = \pi_0(\calM ) \cong \pi_0(\widetilde\calM)
\end{align} 
be the monoid of the connected components.
For each $\gamma \in \Gamma$ let 
\begin{align}\label{M-gamma}
\calM _{\gamma} \subset \calM
,\quad
\widetilde\calM _{\gamma} \subset \widetilde\calM
\end{align}
denote the corresponding connected components.
Note that 
$$\calM_0=\widetilde\calM_0=\{0\}.$$
The stack $\Grad(\calM)$ is the union of the connected components of $\calM^\Z$ involving finitely many components 
$\calM_\gamma$ with $\gamma\neq 0$. Following \cite[(7.1.2.2)]{Bu25a}, we write
    \begin{equation}\label{eq-grad-explicit}
     \Grad(\calM )    \cong \bigsqcup_{\tilde{\gamma} \colon \mathbb{Z} \to \Gamma} \ \ 
     \Big( \prod_{i \in \supp(\tilde{\gamma})} \calM _{\tilde{\gamma}(i)} \Big) 
    \end{equation}
where $\tilde{\gamma}$ runs over all maps with finite support, where $\supp(\tilde\gamma)=\Z\smallsetminus\tilde\gamma^{-1}(0)$.
Since we have an isomorphism
$\pi_0(\Grad(\calM )) \cong \pi_0(\Filt(\calM ))$ 
by \cite[lem.~1.3.8]{halpern2014structure}, the same decomposition exists for $\Filt(\calM )$. 
Let 
$$\Filt _{\tilde{\gamma}} \subset \Filt(\calM )$$ denote the connected component corresponding to $\tilde{\gamma}$.
It is shown in \cite[\S 7.1.6]{Bu25a} that the stack $\Filt_{\tilde{\gamma}}$ depends only on the ordered sequence 
$\{ \tilde{\gamma}(i)\,;\,i \in \supp(\tilde{\gamma})\}$ up to an isomorphism.
Given a tuple
\begin{align}\label{uugamma}\uugamma = (\gamma_1, \ldots, \gamma_\ell)
,\quad
\gamma = \sum_{i=1}^\ell \gamma_i\end{align}
let $\Filt_{\uugamma}$ be the stack $\Filt _{\tilde{\gamma}}$ where $\tilde{\gamma}$ is the function 
$i \mapsto \gamma_{l - i}$ for $i = 0, \ldots, \ell-1$ and zero otherwise.
Set
$$\calM_{\uugamma}= \prod_{i=1}^\ell \calM _{\gamma_i}
,\quad
\varphi_{\widetilde\calM_{\uugamma}}= \mathop{\bigboxtimes_{i=1}^\ell}\displaylimits \varphi_{\widetilde\calM_{\gamma_i}}.$$
The stack $\Filt_{\uugamma}$ parametrizes ascending filtrations with ordered stepwise sub-quotients of classes 
$\gamma_{\ell}, \ldots, \gamma_1$.
We define $\widetilde\Filt_{\uugamma}$ and $\widetilde\calM_{\uugamma}$ in a similar manner.
The correspondences \eqref{attractor1} and \eqref{attractor2} restrict to the following correspondences
 \begin{equation}\label{attractor-explicit}
 \calM _{\uugamma} \xleftarrow[]{\gr_{\uugamma}} \Filt_{\uugamma} \xrightarrow{\ev_{\uugamma}} \calM _{\gamma}
 ,\quad 
 \widetilde\calM_{\uugamma} \xleftarrow[]{\tilde\gr_{\uugamma}} 
 \widetilde\Filt_{\uugamma} \xrightarrow{\tilde\ev_{\uugamma}} \widetilde\calM_{\gamma}.
 \end{equation}
The next result follows from the integral isomorphism \eqref{eq-integral-isom} together with the Thom--Sebastiani isomorphism \eqref{eq-Thom-Sebastiani}.
Note that in our setting the virtual dimension of  $\widetilde\Filt_{\uugamma}$ vanishes by \eqref{eq-vanish-vdim}.

\begin{theorem}
There is an isomorphism of monodromic mixed Hodge modules in $\MHMmon(\widetilde\calM_{\uugamma})$
\begin{equation}\label{eq-integral-identity}
\varphi_{\widetilde\calM_{\uugamma}}
\cong (\tilde\gr_{\uugamma})_* (\tilde\ev_{\uugamma})^! \varphi_{\widetilde\calM_{\gamma}}
\end{equation} 
satisfying the obvious unitality and associativity constraints.
\qed
\end{theorem}

\begin{remark}\label{rmk-etale-integral-identity}
We end this section with the following observation which will be useful later.
Assume that we are given the following diagram
\[\begin{tikzcd}
{\calU_{\uugamma}} & {\calU_{\uugamma}^+} & {\calU_{\gamma}} \\
{\widetilde{\calM}}_{\uugamma}  & \widetilde{\Filt}_{\uugamma} & {\widetilde\calM_{\gamma}}
\arrow[equals, from=1-1, to=1-2]
\arrow[from=1-1, to=2-1]
\arrow["{\tilde\ev_{\calU}}", from=1-2, to=1-3]
\arrow[from=1-2, to=2-2]
\arrow[from=1-3, to=2-3]
\arrow["{\tilde\gr_{\uugamma}}"', from=2-2, to=2-1]
\arrow["{\tilde\ev_{\uugamma}}", from=2-2, to=2-3]
\end{tikzcd}\]
where the vertical maps are open immersions and $\tilde\ev_{\calU}$ is \'etale.
The map $\tilde\ev_{\calU}$ preserves the orientation and the $(-1)$-shifted symplectic structure.
The integral isomorphism \eqref{eq-integral-identity} restricts to the obvious isomorphism
$\varphi_{\calU_{\uugamma}} \cong (\tilde\ev_{\calU_\gamma})^! \varphi_{\calU_\gamma}$.
This follows from the compatibility between the shifted Lagrangian classes with the \'etale pullback (Theorem~\ref{thm-shifted-Lag-class}\ref{item-shifted-Lag-class-etale-compat})
and the unitality of the shifted Lagrangian class (Theorem~\ref{thm-shifted-Lag-class}\ref{item-shifted-Lag-class-unital}) applied to $\calU_{\uugamma}$.
\end{remark}

 \medskip

\subsubsection{The CoHA}\label{sec:COHA}
We adopt the notation from \eqref{para-integral-identity} with $\uugamma = (\gamma_1, \gamma_2)$.
Taking the left adjoint in the isomorphism \eqref{eq-integral-identity} yields a map
 \begin{align}\label{adjunction}
 (\tilde\ev_{\uugamma})_!( \tilde\gr_{\uugamma})^* \left(\varphi_{\widetilde\calM_{\gamma_1}} \boxtimes \varphi_{\widetilde\calM_{\gamma_2}}  \right) \to \varphi_{\widetilde\calM_{\gamma}} .
 \end{align}
 We do not assume the condition \eqref{item-ac-graded}.
 However,  the stack $\widetilde\calM_\gamma$ is locally quasi-compact,
 the morphism $\tilde\ev$ is representable,
 and $\tilde\gr$ is quasi-compact.
 Thus, using the quasi-compact filtration as in \S~\ref{sec-pro} and the K\"unneth 
theorem in Proposition \ref{prop-kunneth}, the morphism \eqref{adjunction} yields a map of pro-complexes
 \begin{align}\label{mult0}\RGpro(\widetilde\calM_{\gamma_1}, \varphi_{\widetilde\calM_{\gamma_1}}) \,\widehat{\otimes}\, 
 \RGpro(\widetilde\calM_{\gamma_2}, \varphi_{\widetilde\calM_{\gamma_2}})  \to
  \RGpro(\widetilde\calM_{\gamma}, 
 \varphi_{\widetilde\calM_{\gamma}}).\end{align}
 It is proved in \cite[cor.~8.8]{KPS} that this map defines an associative algebra object
in $\Pro(\Dhcb(\bullet))$. 
In loc.~cit., the authors only verify the $1$-categorical associativity, but the $\infty$-categorical associativity is 
automatic since the construction of the multiplication is based on an isomorphism of perverse sheaves. 
For our purpose, the $1$-categorical associativity is sufficient.
Using the dimensional reduction isomorphism \eqref{eq:dimensional-reduction}, we obtain an associative multiplication 
\begin{equation}\label{eq-COHA-multi}
*\colon \calA_{\calM,\pro } \,\widehat{\otimes}\, \calA_{\calM,\pro } \to \calA_{\calM,\pro }
 \end{equation}
on the pro-complex defined as in \eqref{RGproBM} by
\begin{equation}\label{AMpro}
	\calA_{\calM,\pro } = \bigoplus_{\gamma \in \Gamma} \RGproBM(\calM _{\gamma},\Q)[- \vdim \calM_{\gamma}]
\end{equation}
  and a topological algebra structure on
\begin{align}\label{AM}\calA_{\calM } = \bigoplus_{\gamma \in \Gamma} R \Gamma^\BM(\calM _{\gamma},\Q)[- \vdim \calM_{\gamma}]. \end{align}
Taking the cohomology, we get a graded algebra structure 
on the graded pro-vector space in $\Pro(\Vect^{\heartsuit, \Gamma \times \mathbb{Z}})$ given by 
$$\H\calA_{\calM,\pro} =\bigoplus_{\gamma \in \Gamma} 
\H_{-*+\vdim\calM_\gamma}^{\BM}(\calM_{\gamma},\Q)_\pro $$
and, via the realization functor, a topological graded algebra structure 
on the graded vector space  in $\Vect^{\heartsuit, \Gamma \times \mathbb{Z}}$ given by
$$ \H\calA_\calM =\bigoplus_{\gamma \in \Gamma} 
\H_{-*+\vdim\calM_\gamma}^{\BM}(\calM_{\gamma},\Q).$$
This algebra is called the (absolute) CoHA associated with the $2$CY category $\calC$ and $\calM$.

\medskip

\subsubsection{The good moduli space}\label{good}
We assume that for each $\gamma \in \Gamma$ the stack $\calM _{\gamma}$ has a quasi-compact good moduli space  
$$p_{\gamma} \colon \calM _{\gamma} \to M_{\gamma}.$$ 
We adopt the definition of a good moduli space for derived Artin stacks as given in \cite[def.~2.1]{AHPS23}.
By  \cite[thm.~2.12]{AHPS23}, the derived stack $\calM_\gamma$ admits a good moduli space if and only if its classical truncation
$\calM_\gamma^\cl$ does.
This, in turn,  means that $M_{\gamma}$ is an algebraic space, the map $p_\gamma$ is quasi-compact and quasi-separated,
the pushforward of the structure sheaf of $\calM _{\gamma}$ is the structure sheaf of $M_{\gamma}$, and the pushforward functor is 
exact on quasi-coherent sheaves (cohomological affineness). 
 Since the morphism
 $$\kappa_\gamma \colon \widetilde\calM_{\gamma} \to \calM _{\gamma}$$ is affine, by base change the stack
 $\widetilde\calM_{\gamma}$ also admits a good moduli space
 $$\tilde{p}_{\gamma} \colon \widetilde\calM_{\gamma} \to \wM_{\gamma}.$$
We abbreviate 
$$M = \bigsqcup_{\gamma} M_{\gamma}
,\quad
\wM = \bigsqcup_{\gamma} \wM_{\gamma}.$$
The direct sum yields maps
$$\oplus \colon M \times M \to M
,\quad
\oplus \colon \wM \times \wM \to \wM.$$
There is a symmetric monoidal structure on locally bounded below monodromic mixed Hodge modules given by
 \[\boxtimes_{\oplus} \colon \Dh^{\mathsf{mon}, +} (\wM) \times \Dh^{\mathsf{mon}, +} (\wM) \to \Dh^{\mathsf{mon}, +} (\wM), \quad (E, F) 
 \mapsto \oplus_{*}(E \boxtimes F).\]
We define a symmetric product operation by
\begin{equation}\label{eq-sym-prod}
\Sym_{\oplus} \colon \Dh^{\mathsf{mon}, +}(\wM \smallsetminus \wM_0) \to \Dh^{\mathsf{mon}, +}(\wM), \quad E \mapsto \bigoplus_n (\underbrace{E 
\boxtimes_{\oplus} \cdots \boxtimes_{\oplus} E}_{n \text{ copies}})^{{\mathrm{S}_n}}.
\end{equation}

 \medskip

\subsubsection{The relative CoHA}\label{sec:relative-COHA}

In this section, we will introduce a sheaf-theoretic refinement of the cohomological Hall algebra introduced in \S~\ref{sec:COHA}.
The set-up we consider is a $\Gamma$-graded monoid scheme $(B = \coprod_{\gamma \in \Gamma} B_{\uugamma}, \oplus)$ together with a map $\tilde{\rho} = \coprod_{\gamma}  \tilde{\rho}_{\gamma} \colon \widetilde\calM \to B$ such that the following diagram commutes for pairs $\uugamma =(\gamma_1 , \gamma_2)$:
\begin{equation}\label{induction-X}
\begin{tikzcd}
	{\widetilde\calM_{\gamma_1} \times \widetilde\calM_{\gamma_2} } &{\widetilde\Filt_{\uugamma}}& {\widetilde\calM_{\gamma_1 + \gamma_2}} \\
	{B_{\gamma_1} \times B_{\gamma_2}} && {B_{\gamma_1 + \gamma_2}.}
	\arrow["\tilde\gr_{\uugamma}"', from=1-2, to=1-1]
	\arrow["\tilde\ev_{\uugamma}", from=1-2, to=1-3]
	\arrow["{\tilde{\rho}_{\gamma_1} \times \tilde{\rho}_{\gamma_2}}"', from=1-1, to=2-1]
	\arrow["{\tilde{\rho}_{\gamma_1 + \gamma_2}}", from=1-3, to=2-3]
	\arrow["\oplus", from=2-1, to=2-3]
\end{tikzcd}
\end{equation}
A basic example we consider is $B_{\gamma} = \wM_{\gamma}$, the good moduli space of $\widetilde\calM_{\gamma}$.  In this case, the commutativity of the above diagram is proved in \cite[\S 8.1.3]{Bu25b}.
Later, we will also consider the case where $B$ is the Chow variety of a Calabi--Yau surface.
We will define a relative version of the CoHA as an associative algebra object in $\Pro(\Dhmon(B))$.

For this, we consider the {Hall monoidal product} on $\Dhmon(\widetilde{\calM})$
\[
  \boxtimes_{\mathrm{Hall}} \colon 	\Dhmon(\widetilde{\calM}) \times \Dhmon(\widetilde{\calM}) \to \Dhmon(\widetilde{\calM}), \quad (E, F) \mapsto \tilde{\ev}_* \tilde{\gr}^* (E \boxtimes F).
\]
and the {direct sum monoidal product} on $\Pro(\Dhmon(B))$
\[
	\widehat{\boxtimes}_{\oplus} \colon \Pro(\Dhmon(B)) \times \Pro(\Dhmon(B)) \to \Pro(\Dhmon(B)), \quad (E, F) \mapsto \oplus_* (E \boxtimes F).
\]
The functor $\tilde{\rho}^{\pro}_* \colon \Dhmon(\widetilde{\calM}) \to \Pro(\Dhmon(B))$ is lax monoidal by construction.
In particular, the functor $\tilde{\rho}^{\pro}_*$ induces a functor between associative algebra objects
\[
	\tilde{\rho}^{\pro}_* \colon \mathsf{Alg}_{\boxtimes_{\mathrm{Hall}}}(\Dhmon(\widetilde{\calM})) \to \mathsf{Alg}_{\widehat{\boxtimes}_{\oplus}}(\Pro(\Dhmon(B))).
\]	
By using the map \eqref{adjunction}, we can equip $\varphi_{\widetilde{\calM}}$ with the structure of an associative algebra object in $\Dhmon(\widetilde{\calM})$.
Therefore the complex $\tilde{\rho}^{\pro}_* \varphi_{\widetilde{\calM}}$ upgrades to an object in $\mathsf{Alg}_{\widehat{\boxtimes}_{\oplus}}(\Pro(\Dhmon(B)))$.
In other words, there is a map 
\begin{equation}\label{relative-COHA-X}
	(\tilde{\rho}_{\gamma_1})^{\pro}_* \varphi_{\widetilde\calM_{\gamma_1}} \widehat{\boxtimes}_{\oplus} (\tilde{\rho}_{\gamma_2})^{\pro}_*
	\varphi_{\widetilde\calM_{\gamma_2}} \to 
	(\tilde{\rho}_\gamma)^{\pro}_* \varphi_{\widetilde\calM_{\gamma}}
\end{equation}
for $\gamma_1, \gamma_2 \in \Gamma$ with $\gamma = \gamma_1 + \gamma_2$, satisfying the associativity property.
Assume now that $\tilde{\rho}_{\gamma}$ factors as $\widetilde{\calM}_{\gamma} \to \calM_{\gamma} \xrightarrow{\rho_{\gamma}} B_{\gamma}$.
Then the dimensional reduction isomorphism \eqref{eq:dimensional-reduction} identifies the map \eqref{relative-COHA-X} with the following map
\begin{equation}\label{eq-relative-COHA-S}
	({\rho}_{\gamma_1})^{\pro}_* \D\Q_{\calM_{\gamma_1}} \boxtimes_{\oplus} ({\rho}_{\gamma_2})^{\pro}_*
	\D\Q_{\calM_{\gamma_2}} \to ({\rho}_\gamma)^{\pro}_*\D\Q_{\calM_{\gamma}}\otimes\L^{\vdim\gr_\uugamma}.
\end{equation}

Following \cite{KV19}, there is an alternative definition of relative CoHA on $\rho^{\pro}_* \D\Q_{\calM}$
using the quasi-smoothness of the map $\gr_\uugamma$.
For $\uugamma = (\gamma_1, \gamma_2)$, we denote $\rho_{\uugamma} = \rho_{\gamma_1} \times \rho_{\gamma_2}$.
Consider the following chain of morphisms
of monodromic mixed Hodge modules 
\begin{align}\label{eq-relative-COHA-KV}
\begin{split}
\oplus_*(\rho_{\uugamma})^{\pro}_*\D\Q_{\calM_{\uugamma}}
&\to\oplus_*(\rho_{\uugamma})^{\pro}_*(\gr_\uugamma)_*(\gr_\uugamma)^*\D\Q_{\calM_{\uugamma}}\\
&\cong(\rho_{\gamma})^{\pro}_*(\ev_{\uugamma})_*(\gr_{\uugamma})^*\D\Q_{\calM_{\uugamma}}\\
&\to (\rho_{\gamma})^{\pro}_*(\ev_{\uugamma})_*\D\Q_{\Filt_{\uugamma}}\otimes\L^{\vdim\gr_{\uugamma}}\\
&\cong (\rho_{\gamma})^{\pro}_*(\ev_{\uugamma})_* (\ev_{\uugamma})^! \D\Q_{\calM_{\gamma}}\otimes\L^{\vdim\gr_{\uugamma}}\\
&\to(\rho_{\gamma})^{\pro}_*\D\Q_{\calM_{\gamma}}\otimes \L^{\vdim\gr_{\uugamma}}.
\end{split}
 \end{align}
The first morphism is the unit,  the third one is  the purity transformation \eqref{eq-purity-transform}
and the last one is the counit, using the properness of the map $\ev_\uugamma$.
See also the discussion in \cite[\S8.1.8]{Bu25b}.
We call \eqref{eq-relative-COHA-S} the 3D relative CoHA and \eqref{eq-relative-COHA-KV} the 2D relative CoHA.
Taking global sections we recover the CoHA on $\H\calA_\calM$ in \S\ref{sec:COHA}. It also admits two realizations:
the 2D CoHA as in \cite{KV19} and the 3D CoHA  as in \cite[cor.~8.8]{KPS}, modulo dimensional reduction 
\eqref{eq:dimensional-reduction}. The following comparison result between the above two multiplications is a consequence of Proposition~\ref{prop-dim-red-vs-purity} and Lemma~\ref{lem-attractor-is-conormal}.

\begin{theorem}\label{thm: comparison 2D 3D}
	The following diagram commutes
	\begin{equation*}
	\begin{tikzcd}
			(\tilde{\rho}_{\gamma_1})^{\pro}_*(\varphi_{\widetilde{\calM}_{\gamma_1}})\boxtimes_{\oplus} (\tilde{\rho}_{\gamma_2})^{\pro}_*(\varphi_{\widetilde{\calM}_{\gamma_2}})
			 &	(\tilde{\rho}_{\gamma})^{\pro}_*(\varphi_{\widetilde{\calM}_{\gamma}})\\
			\big((\rho_{\gamma_1})^{\pro}_*\D\Q_{\calM_{\gamma_1}} \otimes \mathbb{L}^{\frac{1}{2} \vdim \calM_{\gamma_1}} \big) \boxtimes_\oplus \big((\rho_{\gamma_2})^{\pro}_*\D\Q_{\calM_{\gamma_2}} \otimes \mathbb{L}^{\frac{1}{2} \vdim \calM_{\gamma_2}} \big)  & (\rho_{\gamma})^{\pro}_*\D\Q_{\calM_{\gamma}} \otimes \mathbb{L}^{\frac{1}{2} \vdim \calM_{\gamma}}
			\arrow["\kappa_*\eqref{relative-COHA-X}", from=1-1, to=1-2]
		\arrow["\eqref{eq-relative-COHA-KV}", from=2-1, to=2-2]
			\arrow["\eqref{eq:dimensional-reduction}"', from=1-1, to=2-1]
			\arrow["\eqref{eq:dimensional-reduction}", from=1-2, to=2-2]		
	\end{tikzcd}
	\end{equation*}
	for any $\uugamma=(\gamma_1,\gamma_2)$ in $\Gamma^2$ with $\gamma=\gamma_1+\gamma_2$.
	\qed
\end{theorem}


\medskip

\subsection{BPS cohomology}

In this section, we recall the cohomological integrality theorem \cite[thm.~10.2.11]{Bu25b}.
It yields a decomposition of the 
cohomology of the DT perverse sheaf
on the moduli stack of objects in a 3CY category.
We will use the setting of \S \ref{sec:nilp-micro}, and assume the existence of good moduli spaces as in \S \ref{good}.
Let $\calC$ be a 2CY category and $\calM  \subset \calM _{\calC}$ an open substack
satisfying the assumptions in \S\ref{para-assumptions-M}.
Let $\Gamma$ and $\calM_\gamma$ be as in \eqref{Gamma}, \eqref{M-gamma}.

 \medskip

\subsubsection{The support lemma}
Under the identification in \S\ref{sec:nilp-micro}, 
 the zero section $[E] \mapsto [(E, 0)]$ and the action map yield closed immersions
 \begin{align}\label{0-i-gamma}
0_\gamma\colon M_{\gamma} \to \wM_{\gamma}
,\quad
i_{\gamma} \colon \A^1 \times M_{\gamma} \hookrightarrow \wM_{\gamma}.
\end{align}
We abbreviate
$$\Lambda_{\gamma,\snil}=\widetilde\calM_\gamma\cap\Lambda_\snil
,\quad
\Lambda_{\gamma,\nil}=\widetilde\calM_\gamma\cap\Lambda_\nil.$$

\begin{proposition}\label{prop:nilp-equiv-def}
We have the following Cartesian diagram
\begin{align}\label{nilp-equiv-def}
\vcenter{\hbox{$\xymatrix{
\widetilde\calM_{\gamma,\red}\ar[d]_-{\tilde p_\gamma}&\ar[l]\Lambda_{\gamma,\snil}\ar[d]&\ar[l]\Lambda_{\gamma,\nil}\ar[d]\\
\widetilde M_{\gamma,\red}&\ar[l]_-{i_\gamma}\A^1\times M_{\gamma,\red}&\ar[l]_-{0\times\id}M_{\gamma,\red}.\ar@/^1.5pc/[ll]^{0_\gamma}
}$}}
\end{align}
\end{proposition}

\begin{proof}
It is enough to prove that for each closed point $(\calF, \phi)$ of $\widetilde\calM_{\gamma}$ we have
\begin{enumerate}[label=$\mathrm{(\alph*)}$,leftmargin=8mm,itemsep=2mm]
\item 
$(\calF, \phi) \in \Lambda_{\nil}$
  if and only if $\tilde{p}_{\gamma}(\calF, \phi)$ is contained in the image of the map $0_\gamma\colon M_{\gamma} \to \wM_{\gamma}$,
 \item
$(\calF, \phi) \in\Lambda_{\snil}$ if and only if $\tilde{p}_{\gamma}(\calF, \phi)$ is contained in the image 
 of the map $i_\gamma\colon\A^1 \times M_{\gamma} \to \wM_{\gamma}$.
 \end{enumerate}

Part (b) follows from (a). Let us concentrate on the part (a).
    Assume first that $(\calF, \phi)$ is contained in $\Lambda_{\nil}$. 
Then, the filtration of $\calF$ induced by the nilpotent operator defines a 
    filtered point $x \in \widetilde\Filt_{\uugamma}$ for some $\uugamma = (\gamma_1, \ldots, \gamma_n)$ such that $\tilde\ev_{\uugamma}(x) = (\calF, \phi)$ and $\tilde\gr_{\uugamma}(x) = (\calF', 0)$ for some 
    $\calF' \in \calM _{\uugamma}$.
 We consider the commutative diagram \eqref{induction-X} for $B_{\gamma} = M_{\gamma}$. The object
    $\tilde{p}_{\gamma}(\calF, \phi) = \oplus  \tilde{p}_{\uugamma}(\calF', 0)$ 
    is contained in the image of $0_\gamma\colon M_{\gamma} \to \wM_{\gamma}$.

 Conversely, assume that $(\calF, \phi)$ is not contained in $\Lambda_{\nil}$.
By the Jordan decomposition in the finite dimensional algebra $\mathrm{End}(\calF)$, we see that there is a decomposition 
$\calF = \bigoplus_{c \in \C} \calF_c$ where $\calF_c$ is the generalized eigenspace of $\phi$ with eigenvalue $c$.
Take $ 0 \neq c_0 \in \C$ with $\calF_{c_0} \neq 0$.
Write $\gamma'$ for the class of $\calF_{c_0}$ and $\gamma''$ for the class of the complement $\calF_{\neq c_0}$.
It is clear that $\tilde{p}_{\gamma'}(\calF_{c_0}, \phi |_{\calF_{c_0}})$ is not contained in the image of 
 the map $0_{\gamma'}\colon M_{\gamma'} \to\wM_{\gamma'}$.
Since the direct sum map $\oplus_{\gamma', \gamma''}$ is a finite morphism preserving the scaling $\mathbb{G}_{\mathrm{m}}$-actions, we 
conclude that 
\[\tilde{p}_{\gamma}(\calF, \phi) = 
\oplus_{\gamma', \gamma''} \left( \tilde{p}_{\gamma'}(\calF_{c_0}, \phi |_{\calF_{c_0}}),   
\tilde{p}_{\gamma''}(\calF_{\neq c_0}, \phi |_{\calF_{\neq c_0}}) \right).\]
is not contained in the image of the map $M_{\gamma} \hookrightarrow \wM_{\gamma}$.
\end{proof}

The  BPS perverse cohomology sheaf on $\wM_{\gamma}$ is the object in $\MHMmon(\wM_{\gamma})$ 
given by
\[\cBPS_{\wM_{\gamma}} = {}^{\text{p}}\calH^1((\tilde{p}_\gamma)_* \varphi_{\widetilde\calM_{\gamma}}). \]
The following support property is proved in \cite[thm.~7.2.14]{Bu25b}.

\begin{proposition}\label{prop:support-lemma}
There is a pure Hodge module $\mathcal{BPS}_{M_{\gamma}}$ in $\MHM(M_{\gamma})$ with an isomorphism
\begin{align}\label{eq:support-lemma}
\cBPS_{\wM_{\gamma}} \cong (i_\gamma)_! (\mathcal{IC}_{\A^1} \boxtimes \mathcal{BPS}_{M_{\gamma}}).
\end{align}
 \qed
\end{proposition}

\medskip

\subsubsection{Global equivariant parameters}\label{sec-global}
To state the cohomological integrality theorem, we recall the definition of global equivariant parameters following 
\cite[\S 9.1.2]{Bu25b}. Take $\gamma \in \Gamma $. 
By \eqref{eq-grad-explicit}, there is a connected component of $\Grad(\calM _{\gamma})$ corresponding to the map 
$\mathbb{Z} \to \Gamma$ such that
$0 \mapsto\gamma$ and
$n\mapsto 0$ for each 
$n\neq 0.$
As explained in \cite[\S 3.4]{IN23}, the corresponding closed and open immersion
$\calM _{\gamma}\to\Grad(\calM _{\gamma})$
corresponds to a $\BGM$-action, which we refer to as the standard $\BGM$-action.
The action map gives rise to a linear map 
$$\H^*(\calM_\gamma,\Q)\to \H^*(\BGM,\Q).$$
A global equivariant parameter on $\calM _{\gamma}$ is a choice of a line bundle $\calL_{\gamma}$ which has non-zero weights with 
respect to the standard  $\BGM$-action, i.e., such that the Chern class $c_1(\calL_\gamma)$ 
generates $\H^*(\BGM,\Q)$ under the map above.
See \cite[\S 9.1.5]{Bu25b} for an existence criterion for a global equivariant parameter.

 \medskip

\subsubsection{The cohomological integrality isomorphism}\label{sec:nilp-int}
Set
$$\H^*(\BGM,\Q)_{\vir} = \H^*(\BGM,\Q)\otimes\mathbb{L}^{\frac12}.$$
Assume that we are given a global equivariant parameter $\calL_{\gamma}$ on $\widetilde\calM_{\gamma}$,
see \S\ref{sec-global}.
It yields an action of $\H^*(\BGM,\Q)$ on $(\tilde{p}_\gamma)_* \varphi_{\widetilde\calM_{\gamma}}$.
It is proved in \cite[prop.~7.2.9]{Bu25b} that 
${}^{\text{p}}\calH^{i}((\tilde{p}_\gamma)_* \varphi_{\widetilde\calM_{\gamma}})$ 
vanishes for $i < 1$. Thus, there is a map
 \[ \cBPS_{\wM_{\gamma}} \otimes\mathbb{L}^{\frac12} \to (\tilde{p}_\gamma)_* \varphi_{\widetilde\calM_{\gamma}}\]
which yields a map 
\begin{equation}\label{eq-BPS-inclusion}
        \cBPS_{\wM_{\gamma}} \otimes \H^*(\BGM,\Q)_{\vir} \to (\tilde{p}_\gamma)_* \varphi_{\widetilde\calM_{\gamma}}.
\end{equation}
If we are given $\uugamma$ and $\gamma$  as in \eqref{uugamma},
the CoHA multiplication map \eqref{relative-COHA-X} and \eqref{eq-BPS-inclusion} yield a map
    \begin{equation}\label{eq-BPS-induction}
        \Big(\cBPS_{\wM_{\gamma_1}} \otimes \H^*(\BGM,\Q)_{\vir}\Big)  \boxtimes_{\oplus} \cdots \boxtimes_{\oplus} 
        \Big(\cBPS_{\wM_{\gamma_n}} \otimes \H^*(\BGM,\Q)_{\vir}\Big) \to (\tilde{p}_\gamma)_* \varphi_{\widetilde\calM_{\gamma}}.
    \end{equation}
The relative cohomological integrality theorem proved in \cite[thm.~10.2.11]{Bu25b} states the following.

\begin{theorem}\label{thm:coh-integrality}
 The map \eqref{eq-BPS-induction} yields an isomorphism in $\Dh^{\mathsf{mon}, +}(\widetilde M)$
 \begin{equation}\label{eq:rel-cohint-X}
 \Sym_{\oplus}\Big(\bigoplus_{\gamma \in \Gamma \smallsetminus 0} \cBPS_{\wM_{\gamma}} 
 \otimes \H^*(\BGM,\Q)_{\vir} \Big) 
 \cong \bigoplus_{\gamma\in\Gamma} (\tilde{p}_\gamma)_* \varphi_{\widetilde\calM_{\gamma}}.
 \end{equation}
 \qed
\end{theorem}

\begin{corollary}\label{cor:coh-integrality}
 The CoHA multiplication yields an isomorphism in $\Dh^{\mathsf{mon}, +}(M)$
 \begin{equation}\label{eq:rel-cohint-S}
 \Sym_{\oplus}\Big(\bigoplus_{\gamma \in \Gamma \smallsetminus 0} \cBPS_{M_{\gamma}} \otimes \H^*(\BGM,\Q)_{\vir}
 \otimes\L^{-\frac12} \Big) 
 \cong \bigoplus_{\gamma\in\Gamma} (p_\gamma)_* \D\Q_{\calM_{\gamma}}\otimes\L^{\frac12\vdim\calM_\gamma}.
 \end{equation}
\end{corollary}

\begin{proof}
Applying the pushforward by the maps $\kappa_\gamma$ to \eqref{eq:rel-cohint-X} and using the commutativity of the following diagram
\begin{align*}
\begin{split}
\xymatrix{M_\uugamma\ar[r]^-{\oplus}&M_\gamma\\
\widetilde M_\uugamma\ar[r]^-{\oplus}\ar[u]^-{\kappa_\uugamma}&\widetilde M_\gamma\ar[u]_-{\kappa_\gamma}}
\end{split}
\end{align*}
we get the isomorphism
$$ \Sym_{\oplus}\Big(\bigoplus_{\gamma \in \Gamma \smallsetminus 0} 
 (\kappa_\gamma)_*\cBPS_{\wM_{\gamma}} \otimes \H^*(\BGM,\Q)_{\vir} \Big) 
 \cong \bigoplus_{\gamma\in\Gamma} (\kappa_\gamma)_*(\tilde{p}_\gamma)_* \varphi_{\widetilde\calM_{\gamma}}.$$
 Applying the dimensional reduction \eqref{eq:dimensional-reduction}
 and the support property \eqref{eq:support-lemma} to this isomorphism yields
 the corollary.
\end{proof}

We define the BPS cohomology of $\wM_{\gamma}$ and $\calM_{\gamma}$ to be the graded vector spaces
\[\H^*_{\BPS}(\wM_\gamma,\Q) = \H^*(\wM_{\gamma}, \cBPS_{\wM_{\gamma}});\quad \H^*_{\BPS}(\calM_\gamma,\Q) = \H^*(\wM_{\gamma}, \cBPS_{\calM_{\gamma}}).\]
By Proposition \ref{prop:support-lemma} there is an isomorphism $\H^*_{\BPS}(\wM_\gamma,\Q)\cong \H^*_{\BPS}(\calM_\gamma,\Q)\otimes\L^{1/2}$.  
Using the additive structure on $\Gamma$, we have as in \eqref{eq-sym-prod} a functor
$$ \Sym_{+} \colon \Dh^{\mathsf{mon}, +}(\Gamma \smallsetminus  0 ) \to \Dh^{\mathsf{mon}, +}(\Gamma).$$ 
Taking global sections, the dimensional reduction \eqref{eq:dimensional-reduction} yields the following 
cohomological integrality theorem.

\begin{corollary}\label{cor:absolute-integrality}
The CoHA multiplication yields the following isomorphism 
\begin{align}\label{eq:absolute-integrality}
\Sym_{+}\Big(\bigoplus_{\gamma \in \Gamma \smallsetminus 0} \H^*_{\BPS}(\wM_\gamma,\Q) \otimes \H^*(\BGM,\Q)_{\vir} \Big)  
\cong\bigoplus_{\gamma \in \Gamma}  \H^\BM_{- *}(\calM _\gamma,\Q)\otimes\mathbb{L}^{\frac12\vdim \calM _{\gamma}}.
\end{align}
\qed
\end{corollary}
We will need a version of the cohomological integrality theorem for the nilpotent and semi-nilpotent microlocal 
homology introduced in \S\ref{sec:nilp-micro}.

\begin{proposition}\label{prop:nilp-cohomological-integrality}
There are isomorphisms
\begin{align*}
& \Sym_{+}\Big(\bigoplus_{\gamma \in \Gamma \smallsetminus 0} \H^*_{\BPS}(\wM_\gamma,\Q) \otimes \H^*(\BGM,\Q)_{\vir} \otimes 
\mathbb{L} \Big)  \cong   \bigoplus_{\gamma \in \Gamma} 
\H^{\Lambda_\nil}_{- *}(\calM _\gamma,\Q)\otimes\mathbb{L}^{\frac12\vdim \calM _{\gamma}}, \\
& \Sym_{+}\Big(\bigoplus_{\gamma \in \Gamma \smallsetminus 0} \H^*_{\BPS}(\wM_\gamma,\Q) \otimes \H^*(\BGM,\Q)_{\vir} \otimes 
\mathbb{L} \Big) \otimes \mathbb{L}^{-1}  \cong   
\bigoplus_{\gamma \in \Gamma}  \H^{\Lambda_{\snil}}_{- *}(\calM _\gamma,\Q)\otimes\mathbb{L}^{\frac12\vdim \calM _{\gamma}}.
\end{align*}
\end{proposition}

\begin{proof}
The second isomorphism follows from the first one together with the isomorphism \eqref{eq-nilp-vs-seminilp}.
To prove the first isomorphism we apply the functor $(0_\gamma)^!$ to the isomorphism \eqref{eq:rel-cohint-X}.
The following commutative square is Cartesian
\begin{align}\label{oplus-S-X} 
\begin{split}
\xymatrix{M_\uugamma\ar[r]^-{\oplus}\ar[d]_-{0_\uugamma}&M_\gamma\ar[d]^-{0_\gamma}\\
\widetilde M_\uugamma\ar[r]^-{\oplus}&\widetilde M_\gamma.}
\end{split}
\end{align}
Hence, by base change, applying $(0_\gamma)^!$ to the isomorphism \eqref{eq:rel-cohint-X} yields the following isomorphism 
in $\Dh^{\mathsf{mon}, +} (M)$
\begin{align}\label{eq:int-isom0}
 \Sym_{\oplus}\Big(\bigoplus_{\gamma \in \Gamma \smallsetminus 0} 
(0_\gamma)^! \cBPS_{\wM_{\gamma}} \otimes \H^*(\BGM,\Q)_{\vir} \Big) \cong \bigoplus_{\gamma} 
(0_\gamma)^! (\tilde{p}_\gamma)_* \varphi_{\widetilde\calM_{\gamma}}. 
\end{align}
 Taking the cohomology, we get the following isomorphism in $\Dh^{\mathsf{mon}, +}(\Gamma)$
 \[ \Sym_{\oplus}\Big(\bigoplus_{\gamma \in \Gamma \smallsetminus 0} 
\H^* (M_{\gamma}, (0_\gamma)^! \cBPS_{\wM_{\gamma}}) \otimes \H^*(\BGM,\Q)_{\vir} \Big) \cong 
\bigoplus_{\gamma}
\H^*(M_{\gamma}, (0_\gamma)^! (\tilde{p}_\gamma)_* \varphi_{\widetilde\calM_{\gamma}}).\]
By Proposition \ref{prop:nilp-equiv-def}, we have the following Cartesian diagram
\begin{align*}
\begin{split}
\xymatrix{\Lambda_{\gamma,\nil}\ar[r]\ar[d]&M_\gamma\ar[d]^-{0_\gamma}\\
\widetilde\calM_\gamma\ar[r]^-{\tilde p_\gamma}&\widetilde M_\gamma.}
\end{split}
\end{align*}
Hence, by base change we get
\begin{align*}
\H^*(M_{\gamma}, (0_\gamma)^! (\tilde{p}_\gamma)_* \varphi_{\widetilde\calM_{\gamma}})
&=\H^{\Lambda_{\nil}}_{- *}(\calM _\gamma,\Q)\otimes\mathbb{L}^{\frac12\vdim \calM _{\gamma}}.
\end{align*}
The isomorphism \eqref{eq:support-lemma} yields
\begin{align}\label{eq:BPS-S-X}
(0_\gamma)^! \cBPS_{\wM_{\gamma}}\cong \cBPS_{M_{\gamma}}\otimes\L^{\frac12}\cong
(\kappa_\gamma)_*\cBPS_{\widetilde M_{\gamma}}\otimes\L.
\end{align}
Hence, we have the following isomorphism which proves the claim
$$\H^* (M_{\gamma}, (0_\gamma)^! \cBPS_{\wM_{\gamma}})
\cong \H_\BPS^* (\widetilde M_{\gamma}, \Q)\otimes\L.$$
\end{proof}
In order to construct Hecke modifications of BPS cohomology in \S\ref{sec:CoHA_hecke} we will use in an essential way the following \textit{microlocal characterization} of BPS cohomology:
\begin{corollary}\label{cor:snilp-BPS}
Under the isomorphism \eqref{eq:absolute-integrality}, the image of the map 
 $\H^{\Lambda_{\snil}}_*(\calM _\gamma,\Q) \to \H^\BM_*(\calM _\gamma,\Q)$
 in \eqref{eq:cone-BM2} is identified with the graded vector space
\[ \H^*_{\BPS}(\wM_\gamma,\Q) \otimes \H^*(\BGM,\Q)_{\vir}\otimes\mathbb{L}^{-\frac12 \vdim \calM _{\gamma}}.\]
\end{corollary}

\begin{proof}
By Corollary \ref{cor:absolute-integrality} and Proposition \ref{prop:nilp-cohomological-integrality}, for each 
tuple $\uugamma = (\gamma_1, \ldots, \gamma_\ell)$ it is enough to show that the map
\begin{align}\label{map00}
\bigboxtimes_{i} \left( \H^*_{\mathrm{BPS}}(\wM_{\gamma_{i}},\Q) \otimes \H^*(\BGM,\Q)_{\vir} \otimes 
\mathbb{L} \right) \otimes \mathbb{L}^{-1}
\to \bigboxtimes_{i} \left( \H^*_{\mathrm{BPS}}(\wM_{\gamma_{i}},\Q) \otimes \H^*(\BGM,\Q)_{\vir} \right)
\end{align}
is an isomorphism for $\ell = 1$ and $0$ for $\ell > 1$. As the former statement is obvious, we focus on the latter. 
Let $\gamma=\gamma_1+\gamma_2+\cdots+\gamma_\ell$ and
$i_\uugamma=i_{\gamma_1}\times i_{\gamma_2}\times\cdots\times i_{\gamma_\ell}$.
To unburden the notation, we abbreviate
$$
\calB_{M}=\cBPS_{M} \otimes \H^*(\BGM,\Q)_{\vir}
,\quad
\calB_{\widetilde{M}}=\cBPS_{\widetilde{M}} \otimes \H^*(\BGM,\Q)_{\vir}$$
and
$$
\calB_{{M}_\uugamma}= \calB_{{M}_{\gamma_1}}\boxtimes\calB_{{M}_{\gamma_2}} \boxtimes \cdots\boxtimes \calB_{{M}_{\gamma_\ell}}
,\quad
\calB_{\widetilde{M}_\uugamma}= \calB_{\widetilde{M}_{\gamma_1}}\boxtimes\calB_{\widetilde{M}_{\gamma_2}} \boxtimes \cdots\boxtimes \calB_{\widetilde{M}_{\gamma_\ell}}.$$
Let $\Delta\colon\A^1\to\A^\ell$ be the diagonal.
A base change along the Cartesian diagram
\begin{align*}
\begin{split}
\xymatrix{\A^1\times M_{\uugamma} \ar[rr]^-{i_\uugamma \circ( \Delta\times\id)} \ar[d]_-{\id\times\oplus} && 
\widetilde{M}_{\uugamma}  \ar[d]^-{\oplus}\\ 
\A^1\times M_{\gamma} \ar[rr]^-{i_\gamma} && \widetilde{M}_{\gamma}}			
\end{split}
\end{align*} 
yields the isomorphisms
\begin{equation}\label{eq:adjunction lemma1}
\begin{split}
(i_\gamma)_!(i_\gamma)^!\oplus_*\calB_{\widetilde{M}_{\uugamma}}
&\cong(i_\gamma)_!(\id\times\oplus)_*(\Delta\times\id)^!(i_\uugamma)^!
\calB_{\widetilde{M}_{\uugamma}}\\
&\cong\oplus_*(i_\uugamma)_!(\Delta\times\id)_!(\Delta\times\id)^!(i_\uugamma)^!
\calB_{\widetilde{M}_{\uugamma}}.
\end{split}	
\end{equation}
Using the isomorphism \eqref{eq:support-lemma} we also get
\begin{equation}\label{eq:adjunction lemma2}
\begin{split}
\oplus_*\calB_{\widetilde{M}_{\uugamma}}
&\cong\oplus_*(i_\uugamma)_!(\mathcal{IC}_{\A^\ell}  
\boxtimes\calB_{{M_{\uugamma}}})\\
&\cong\oplus_*(i_\uugamma)_!(i_\uugamma)^!(i_\uugamma)_!(\mathcal{IC}_{\A^\ell} 
\boxtimes\calB_{{M_{\uugamma}}})\\
&\cong\oplus_*(i_\uugamma)_!(i_\uugamma)^!\calB_{\widetilde{M}_{\uugamma}}.
\end{split}	
\end{equation}
Under the isomorphisms \eqref{eq:adjunction lemma1} and \eqref{eq:adjunction lemma2}, the morphisms
\begin{align}\label{map11}
(i_\gamma)_!(i_\gamma)^!\oplus_*\calB_{\widetilde{M}_{\uugamma}}
\to
\oplus_*\calB_{\widetilde{M}_{\uugamma}}
\end{align}
given by the counits $(i_\gamma)_!(i_\gamma)^!\to 1$
and
$(i_\uugamma)_!(\Delta\times\id)_!(\Delta\times\id)^! (i_\uugamma)^! \to 1$
are identified. 
If $\ell>1$ this last morphism vanishes in cohomology  by the lemma below. 
Further, taking the cohomology, the map \eqref{map11} gives the map \eqref{map00}.
The corollary follows.

\end{proof}
\begin{lemma}\label{claim}
Let $X$ be a scheme, $\mathcal{F}$ a constructible sheaf on $X$ and $\ell > 1$.
The counit $(\Delta\times\id)_!(\Delta\times\id)^! \to 1$ yields a morphism
$\psi\colon(\Delta \times \id)_!  (\Delta\times \id)^! (\mathcal{IC}_{\A^\ell} \boxtimes \mathcal{F}) \to 
\mathcal{IC}_{\A^\ell} \boxtimes \mathcal{F}$.
Applying the pushforward by the projection $\pi_X\colon \A^\ell\times X\to X$ to this map yields the 0 map.
\end{lemma}
\begin{proof}
The morphism $(\pi_X)_*\psi$ coincides with the action of $e(N)\in \H^{2\ell-2}(\A^{\ell}\times X,\mathbb{Q})\cong \H^{2\ell-2}(X,\mathbb{Q})\cong \H^{2\ell-2}(\Delta\times X,\mathbb{Q})$ on $\pi_*(\mathcal{IC}_{\A^\ell} \boxtimes \mathcal{F})$, where $e(N)$ is the Euler class of the normal bundle of $\Delta\times X\subset \A^{\ell} \times X$, and the isomorphisms are provided by the obvious homotopy equivalences.  Since the normal bundle is trivial, $e(N)=0$.
\end{proof}

\begin{remark}\label{rem:snilp-BPS}
The proof above implies that the morphism
$$\kappa_*(i_\gamma)_!(i_\gamma)^!(\tilde{p}_\gamma)_*\varphi_{\widetilde{\calM}_{\gamma}} 
\to \kappa_*(\tilde{p}_\gamma)_*\varphi_{\widetilde{\calM}_{\gamma}}$$ 
given by the counit $(i_\gamma)_!(i_\gamma)^!\to 1$
is identified with the composition of the chain of maps
$$\kappa_*(i_\gamma)_!(i_\gamma)^!(\tilde{p}_\gamma)_*\varphi_{\widetilde{\calM}_{\gamma}} 
\to\kappa_*(i_\gamma)_!(i_\gamma)^!(\cBPS_{\widetilde{M}_\gamma} \otimes \H^*(\BGM,\Q)_{\vir})
\cong\kappa_*\cBPS_{\widetilde{M}_\gamma} \otimes \H^*(\BGM,\Q)_{\vir}
\to\kappa_*(\tilde{p}_\gamma)_* \varphi_{\widetilde\calM_{\gamma}}.$$
The first map is the composition of the isomorphism
\eqref{eq:rel-cohint-X} and the obvious projection, the second map is the counit, which is an isomorphism
by \eqref{eq:adjunction lemma2}, 
and the last map is the inclusion of a direct summand, by \eqref{eq:rel-cohint-X}.
\end{remark}

\bigskip

\section{The BPS sheaf of the stack of coherent sheaves on a CY surface}\label{sec:BPSsurface}
\label{BPS-surface}
Fix a smooth connected projective CY surface $S$, let
$\calC = \mathsf{Coh}(S)$ be the stable $\infty$-category of coherent sheaves on $S$ and
let $\calM_S \subset \calM_{\calC}$ be the open derived $\infty$-substack parametrizing 
coherent sheaves in the heart of the standard t-structure.
Let $X$ be the CY 3-fold given by $X=\A^1\times S$.
Let $\mathsf{Coh}(X)$ be the stable $\infty$-category of coherent sheaves on $X$ and
$\calM_X$ the derived $\infty$-stack of compactly supported coherent sheaves in the heart of the standard t-structure.
Let $\Coh^{\leqslant d}(S)$ and $\Coh^{\leqslant d}(X)$ be the categories of sheaves of dimension $\leqslant d$.
Let $\calM_S^{\leqslant d}$ and $\calM_X^{\leqslant d}$ be the corresponding derived $\infty$-stacks.
We will apply the constructions from \S \ref{sec:2CY CoHAs} with 
$\calM=\calM_S$ and
$\widetilde\calM=\calM_X.$

\medskip

\subsection{The numerical invariants of $S$} 
We first recall the basic cohomological invariants of the $\infty$-category $\calC$. 
We refer to \cite{K3Global} and \cite{HL} for details.
Let $\cdot$ denote the intersection pairing on $\H^*(S,\Q)$.
Let 
$$\NS(S)\subset\H^2(S,\Z)$$ be the N\' eron--Severi group and $\NS(S)^\mathrm{eff}$ the cone of effective classes. 
Let $\K_0(S)$ be the Grothendieck group of coherent sheaves on $S$ and let
$\K_0(S)^+$ be the submonoid of the classes of coherent sheaves. 
Via the Chern character, we identify 
$$\K_0(S) \otimes \Q\cong
\H^0(S,\Q) \oplus (\NS(S)\otimes\Q) \oplus \H^4(S,\Q)\subset\H^{\mathrm{ev}}(S,\Q).$$ 
Let $\rk \colon \K_0(S) \otimes \Q \to \Q$ denote the generic rank.
There is a canonical morphism  
$$\Pic(S) \to \NS(S).$$
By the Hodge index theorem, the restriction of the 
intersection form to 
$\NS(S) \otimes\R$
has signature $(1,\rho(S)-1)$ where $\rho(S)$ is the Picard 
number of $S$. A class $H\in \NS(S)\otimes \R$ is ample if and only if 
$$H^2>0, \quad H \cdot C>0,\quad \forall\, C \in \NS(S)^{\mathrm{eff}}\smallsetminus\{0\}.$$ 
In that case, the restriction of the intersection pairing to $\Q H^\perp$ is negative definite.
The class of the structure sheaf of a point of $S$ is denoted $\delta\in \H^4(S,\Q)$. 
Let $t,-t$ be the Chern roots of $S$. The Chern classes of $S$ are
$c_1=0$, $c_2=-t^2$. The Todd class is 
\begin{equation}\label{E:ToddclassS}
\Td_S=1-\frac{t^2}{12}=1+\chi(\calO_S)\delta.
\end{equation}
The last equality comes from the Riemann--Roch theorem. We have $\Td_S=1+a\delta$ where
$a=2$ for a K3 surface and $a=0$ for an Abelian surface. 
The Euler form $$\langle \calF,\calF'\rangle=\dim\Hom(\calF,\calF')-\dim\Ext^1(\calF,\calF')+\dim\Ext^2(\calF,\calF')$$
descends to the Grothendieck group $\K_0(S)$.
If $\alpha=\sum_i \alpha^{(i)}$ with $\alpha^{(i)} \in \H^{2i}(S,\Q)$ then we set $\alpha^\vee=\sum_i (-1)^i \alpha^{(i)}$. 
We define the Mukai pairing
$$\langle \alpha,\beta \rangle=\int_S \alpha^\vee \cup \beta.$$
This pairing is symmetric and even.
The Mukai vector of a coherent sheaf $\calF$ is the class $\nu(\calF)\in \H^{\mathrm{ev}}(S,\Z)$ defined as 
\begin{align}\nu(\calF)=\ch(\calF) \cup \Td(S)^{\frac12}.
\end{align}
By Riemann--Roch 
\[
\chi(\calF,\calF')=\int_S \ch(\calF)^{\vee}\cup \ch(\calF')\cup\mathrm{Td}_S=\langle \nu(\calF),\nu(\calF')\rangle.
\]
Note that for any $\alpha$ we have
\begin{equation}\label{E:eulerdeltaO}
\langle \alpha,\delta \rangle=\langle \delta, \alpha \rangle=\rk(\alpha).
\end{equation}
Given a Mukai vector $\nu$ we denote by $\nu^{(i)}$ the projection of $\nu$ to $\H^{2i}(S,\Z)$.  If a Mukai vector $\nu$ factors through the extension by zero $\H^{\geqslant 4-2n}(S,\Z)\hookrightarrow \H^{\mathrm{ev}}(S,\Z)$ we say it is $n$-dimensional.
The group $\H^{\mathrm{ev}}(S,\Z)$ inherits a bilinear pairing from the symmetric bilinear form $(-,-)$ on $\H^{\mathrm{ev}}(S,\Q)$ given by
$$ (\nu,\nu')=  \nu^{(0)} \cdot \nu'^{(2)} -\nu^{(1)} \cdot \nu'^{(1)} + \nu^{(2)} \cdot \nu'^{(0)}.$$
Equipped with this pairing, this group is called the Mukai lattice.
An element $\nu$ of the Mukai lattice is called primitive if it is indivisible. 
We may identify the Mukai vector $\nu(\calF)$ with the triple
$$\big( \ch_2(\calF)+\frac{a}{2}\rk(\calF)\delta\,,\,c_1(\calF)\,,\,\rk(\calF)\big)=\big( \chi(\calF)\delta-\frac{a}{2}\rk(\calF)\delta\,,\,c_1(\calF)\,,\,\rk(\calF)\big)$$
where $\chi(\calF)$ is the Euler characteristic of $\calF$.
Note that if $\calF$ has dimension $\leqslant 1$ then 
$\rk(\calF)=0$, the Chern class $c_1(\calF)=\alpha$ is the class of the curve supporting $\calF$, and $\ch_2(\calF)=\chi(\calF)\,\delta$. 
In particular we have
$$\nu(\calF)=\alpha+\chi(\calF)\,\delta= \big(\chi(\calF)\,\delta\,,\,\alpha\,,\,0\big).$$
Let $\calM_S$ and $\calM_X$ be
the derived moduli stack of coherent sheaves on $S$ and coherent sheaves on $X$ with proper support, respectively.
Let $\calM_S(\nu)\subset\calM_S$ be
the substack of coherent sheaves with Mukai vector $\nu$.
Let  
\begin{align}\label{kappa}\kappa\colon\calM_X\to\calM_S\end{align}
be the pushforward by the projection $\varkappa\colon X\to S$.
We set $\calM_X(\nu)=\kappa^{-1}(\calM_S(\nu))$.
Since $\Td(\A^1)=1$, for any properly supported coherent sheaf $\calF$ on $X$
we have $\varkappa_*\ch(\calF)=\ch(\kappa_*\calF)$.
Hence $\calM_X(\nu)$ is the derived moduli stack of coherent sheaves on $X$
with proper support satisfying $\varkappa_*\ch(\calF)\mathrm{Td}_S^{1/2}=\nu$.

\medskip

\subsection{The stack of coherent sheaves on $S$}

The virtual dimension of $\calM_S$ is a locally constant function, with (even) value on $\calM_S(\nu)$ given by
\begin{align}\label{dimM}
\vdim \calM_S(\nu)=-(\nu,\nu)=(\nu^{(1)})^2-2\nu^{(0)}\cdot\nu^{(2)}.\end{align}

\subsubsection{The dimension filtration}
Any coherent sheaf $\calF$ has a canonical support filtration
\begin{equation}\label{eq:support filtration}
	\calF^0 \subseteq \calF^{\leqslant 1} \subseteq \calF^{\leqslant 2} =\calF
\end{equation}
such that $\calF^0$ is zero-dimensional and $\calF^d=\calF^{\leqslant d}/\calF^{\leqslant d-1}$ is pure of dimension $d$ for $d>0$.
Set $\calF^{\geqslant d}= \calF/\calF^{\leqslant d-1}$.

\begin{lemma}\label{lem:support} Let $f\colon \calF \to \calG$ be a morphism in $\Coh(S)$ and $d=0,1,2.$ 
	\hfill
	\begin{enumerate}[label=$\mathrm{(\alph*)}$,leftmargin=8mm,itemsep=2mm]
		\item
		The map $f$ restricts to a map $f^{\leqslant d}\colon\calF^{\leqslant d} \to \calG^{\leqslant d}$ and
		induces maps
		$f^{\geqslant d}\colon \calF^{\geqslant d} \to \calG^{\geqslant d}$
		and
		$f^d\colon \calF^d \to \calG^d$.
		\item If $f$ is injective then so is $f^{\leqslant d}$.
		If $f$ is  surjective then so is  $f^{\geqslant d}$.
		\item If $f$ is injective then so is $f^d$. If $f$ is surjective then so is $f^2$.
	\end{enumerate}
\end{lemma}

\begin{proof} 
Parts (a) and (b) are well-known and obvious. We deal with (c). The first statement is clear if $d=0$, so we assume $d\geq 1$.  Assume that $f$ is 
injective and let $\calK=\mathrm{ker}(f^d)$.  It is a pure $d$-dimensional sheaf, since it is a subsheaf of a pure $d$-dimensional sheaf.  
Let $s\in \calK(U)$ be a non-zero section of an affine open $U$.  Shrinking $U$ if necessary, we can lift $s$ to a section of $\calF^{\leqslant d}(U)$, 
inducing a section $s'\in\calG^{\leqslant d}(U)$, which by supposition lifts to a section of $\calG^{\leqslant d-1}(U)$.  
Thus $s'$ (and so $s$, by injectivity of $f^{\leqslant d}$) is zero on an open dense subspace of its support, and is zero, a contradiction.
The second claim in (c) is a direct application of (b) since $f^{\geqslant 2}=f^2$.
\end{proof}

Fix Mukai vectors 
$\nu(\calF^2)=\nu_2,$ $\nu(\calF^1)=\nu_1$ and $\nu(\calF^0)=\nu_0$ of dimensions $2,1,0$ respectively.  We define $\calM_S(\nu_2,\nu_1,\nu_0)$ to be the open substack
of $\Filt_{\nu_2,\nu_1,\nu_0}$ given by the following Cartesian square
$$\xymatrix{
\calM_S(\nu_2,\nu_1,\nu_0)\ar[r]\ar[d]&\ar[d]\Filt_{\nu_2,\nu_1,\nu_0}\\
 \calM_S(\nu_2,0,0) \times \calM_S(0,\nu_1,0) \times \calM_S(\nu^0)\ar[r]& \calM_S(\nu_2) \times \calM_S(\nu_1) \times \calM_S(\nu_0)
}
$$
where the substacks $$\calM_S(0,\nu_1,0)\subset\calM_S(\nu_1)
,\quad
\calM_S(\nu_2,0,0)\subset\calM_S(\nu_2)$$ 
are the open substacks parametrizing pure sheaves.
In particular $\calM_S(\nu_2,\nu_1,\nu_0)$ has a canonical derived structure which is a
quasi-smooth locally closed substack of $\calM_S(\nu)$ where $\nu=\nu_2+\nu_1+\nu_0$.
Thus these stacks admit canonical derived enhancements which are quasi-smooth. 
We have an obvious map
$$\pi\colon \calM_S(\nu_2,\nu_1,\nu_0) \to \calM_S(\nu_2,0,0) \times \calM_S(0,\nu_1,0) \times \calM_S(\nu_0) \quad
\calF\mapsto (\calF^2,\calF^1,\calF^0).$$
It is a quasi-smooth morphism. 

\begin{lemma}\label{lem:affine fibration pure types}
\hfill
\begin{enumerate}[label=$\mathrm{(\alph*)}$,leftmargin=8mm,itemsep=2mm]
\item
The map $\pi$ is an iterated vector bundle stack.
\item 
Up to some cohomological shift, the pullback by the map $\pi$ gives an isomorphism
$$\H_*^\BM(\calM_S(\nu_2,0,0),\Q) \otimes \H_*^\BM(\calM_S(0,\nu_1,0),\Q) \otimes \H_*^\BM(\calM_S(\nu_0),\Q) 
\cong \H_*^\BM(\calM_S(\nu_2,\nu_1,\nu_0),\Q).$$
\end{enumerate}
\end{lemma}

\begin{proof}
The map $\pi$ factors as the composition
\begin{equation}
\calM_S(\nu_2,\nu_1,\nu_0) \xrightarrow{\pi_1} \calM_S(\nu_2,\nu_1,0) \times \calM_S(\nu_0) 
\xrightarrow{\pi_2 \times \id}  \calM_S(\nu_2,0,0) \times \calM_S(0,\nu_1,0) \times \calM_S(\nu_0)
\end{equation}
where $\pi_1(\calF)=(\calF/\calF^0,\calF^0)$ and $\pi_2(\calF)=(\calF/\calF^1,\calF^1)$. 
The map $\pi_1$ is identified with the projection  
$$R\Hom_{S}(\calE^{\geqslant 1},\calE^0)[1]\to\calM_S(\nu_2,\nu_1,0) \times \calM_S(\nu_0)$$
to its base, where $\calE^{\geqslant 1}$ and $\calE^0$ are the tautological sheaves of $\calM_S(\nu_2,\nu_1,0)$ and $\calM_S(\nu_0)$. 
Since the $\infty$-category $\Coh(S)$ is 2CY, we have an isomorphism
$$\Ext^2_S(\calE^{\geqslant 1},\calE^0) \cong \Hom_S(\calE^0,\calE^{\geqslant 1})^\vee=\{0\}.$$ 
Hence the complex $R\Hom_S(\calE^{\geqslant 1},\calE^0)[1]$ is  concentrated in degrees $[-1,0]$ and $\pi_1$ is a vector bundle stack. 
The same argument shows that the same holds for the map $\pi_2$. Part (b) is a direct consequence of (a).
\end{proof}

In particular, for each class $\alpha\in \H_2(S,\Q)$ and each integer $n$, 
we have the functor $\calM_S(\alpha+n\delta)$ which parametrizes coherent sheaves with one-dimensional support and Mukai vector
$\nu= \alpha+n\,\delta$.
Then, we have
$$\calM_S(\nu)=\bigsqcup_{\nu_1+\nu_0=\alpha+n\delta}\calM_S(0,\nu_1,\nu_0).$$
We abbreviate
$$\calM_S(\alpha+\Z\delta)=\bigsqcup_{n\in\Z}\calM_S(\alpha+n\delta).$$
We will also consider the open substack $\calM_S^{\geqslant 1}$ of $\calM_S$ 
of sheaves with no zero-dimensional subsheaf. It is given by
\begin{equation}\label{eq:defMgeq1}
\calM_S^{\geqslant 1}=\bigcup_{\nu_1,\nu_2}\calM_S(\nu_2,\nu_1,0).
\end{equation}
We abbreviate
$\calM_S^{\geqslant 1}(\nu)=\calM_S^{\geqslant 1} \cap \calM_S(\nu).$

 \medskip

\subsubsection{The stacks of $\mu$-semistable and Gieseker semistable sheaves}

A numerical polynomial $P$ is a polynomial $P\in\Q[x]$ such that $P(\Z)\subset\Z$.
Write 
$$P=\sum_{k=0}^d\alpha_k(P) x^k/k!
,\quad
\alpha_k(P)\in\Q.$$
Fix a polarization of $S$, i.e., an ample rational divisor $H$.
Let $\omega \in  \H^2(S,\Q)$ be the class of $H$.
For now, assume that $H$ is integral.  The Hilbert polynomial of a compactly supported pure $d$-dimensional coherent sheaf $\calF$ 
is the numerical polynomial $P_H(\calF)$ given for each integer $m$ by
$$P_H(\calF)(m)=\chi(F\otimes\calO(mH))=\sum_{k=0}^d\alpha_k(\calF)\frac{m^k}{k!}.$$
Now assume, more generally, that $H=\frac{1}{n}H'$ with $H'$ integral.  Then we define $P_H(\calF)(m)=P_{H'}(\calF)(\frac{m}{n})$.  
When the rational divisor $H$ is clear we omit it and just write $P(\calF)(m)$ instead of $P_H(\calF)(m)$.
The leading term $\alpha_d(\calF)$ is positive. 
Since $S$ is a smooth surface with trivial canonical class, we get
$$\alpha_2(\calF)=\rk(\calF)\,\omega\cdot\omega,\quad\alpha_1(\calF)=c_1(\calF)\cdot\omega,\quad \alpha_0(\calF)=\chi(\calF).$$
The slope is defined by $\mu_H^d(\calF)=\alpha_{d-1}(\calF)/\alpha_d(\calF)$ if $d>0$, and
$\mu_H^0(\calF)=l(\calF)$ (=the length) otherwise.
We get
\begin{align*}
\mu_H^2(\calF)=\frac{c_1(\calF)\cdot \omega}{\text{rk}(\calF) {\omega \cdot \omega }}
,\quad
\mu_H^1(\calF)=\frac{\chi(\calF)}{c_1(\calF)\cdot \omega}.
\end{align*}
The coefficients $\alpha_i(\calF)$ only depend on the Mukai vector $\nu(\calF)$. 
Thus we write $\alpha_i(\nu)=\alpha_i(\calF)$ and $\mu(\nu)=\mu_H^d(\calF)$ if $\nu=\nu(\calF)$.
If either $d$ or $H$ is clear from the context then we omit them from the notation and write $\mu=\mu^d_H$.
For $d$-dimensional $\nu$, the subset of pure $d$-dimensional $\mu$-semistable  sheaves forms a quasi-compact quasi-smooth open substack 
$$\calM_S^\muss(\nu)\subset\calM_S(\nu).$$

A Mukai vector $\nu$ is called effective if $\calM_S(\nu)\neq\emptyset$.
The Bogomolov inequality states that for an effective Mukai vector $\nu$ of positive rank such that $\calM^\muss(\nu) \neq \emptyset$, we have
\begin{equation}\label{eq:Bogomolov}
	c_1(\nu)^2 \geqslant 2 \rk(\nu) \,\ch_2(\nu)
\end{equation}
where $\ch_2(\nu)=\nu^{(2)}-\frac{a}{2}\nu^{(0)}\,\delta$, $c_1(\nu)=\nu^{(1)}$ and $\rk(\nu)=\nu^{(0)}$.
Equivalently, we have
$$(\nu^{(1)})^2-a\nu^{(0)}\cdot\nu^{(2)}\geqslant -2(\nu^{(0)})^2.$$
We will mostly use the finer notion of Gieseker semistability. We equip the set of polynomials $\Q[x]$ with the lexicographic order of the coefficients:
we have $P\leqslant Q$ if and only if $P(m)\leqslant Q(m)$ for all $m$ large enough. 
For any pure $d$-dimensional non-zero sheaf $\calF$, the reduced Hilbert polynomial is
$$p(\calF)=\frac{P(\calF)}{\alpha_d(\calF)}.$$
The sheaf $\calF$ is Gieseker-semistable if for any nonzero $\calG \subset \calF$ 
we have $p(\calG) \leqslant p(\calF)$. 
Given a polynomial $P$, let
$\calM_S^\ss(P)\subset\calM_S$
be the quasi-compact quasi-smooth open substack of pure dimensional Gieseker semistable sheaves with Hilbert polynomial $P$,
see, e.g., \cite[thm.~3.3.7]{HL}. We also set $$\calM^\ss_S(\nu)=\calM_S^\ss(P(\nu)) \cap \calM_S(\nu),$$ where $P(\nu)$ is the Hilbert polynomial of pure sheaves of Mukai vector $\nu$.
We write
$\calM_S^\ss=\bigsqcup_\nu\calM_S^\ss(\nu).$ 
Given an effective curve class $\alpha\in \H_2(S,\Q)$, we will abbreviate
$$
\calM_S^\ss(\alpha+\Z\delta)=\bigsqcup_{n\in\Z}\calM_S^\ss(\alpha+n\delta).$$

\medskip

\subsubsection{The HN-stratification of pure sheaves}\label{sec:partial order1}

In this section we partition the stacks of pure sheaves into HN 
(=Harder--Narasimhan) strata. 
Recall that we have fixed a polarization $H$ of $S$.
Any pure $d$-dimensional coherent sheaf $\calF$ with $d>0$  has a unique HN filtration
\begin{align}\label{HNfiltration}
\{0\}=\calF_\ell \subsetneq \calF_{\ell-1} \subsetneq \cdots \subsetneq \calF_0 =\calF
\end{align}
where $\calF_{i-1}/\calF_i$ is a pure $d$-dimensional
Gieseker-semistable coherent sheaf with Mukai vector $\nu_{i}=\nu(\calF_{i-1}/\calF_{i})$ and 
\begin{align}\label{mu-nu}p(\nu_1)< \cdots < p(\nu_\ell).\end{align}
The tuple 
$\uunu(\calF)=(\nu_1,\dots,\nu_\ell)$ 
is called the HN-type of $\calF$. We set
\begin{align*}
\pmin(\calF)=p(\nu_1), \quad
\pmax(\calF)=p(\nu_{\ell}).
\end{align*}
A tuple $\uunu=(\nu_1, \ldots, \nu_\ell)$ of Mukai vectors of pure $d$-dimensional sheaves such that 
\eqref{mu-nu} holds is called a pure $d$-dimensional HN-type.
Note that we could have defined HN-types with respect to slope stability, and for $d=0,1$ the two definitions of pure $d$-dimensional HN-types coincide.
The stack parametrizing pure $d$ 
dimensional coherent sheaves $\calF$ of HN-type $\uunu$ is called an HN-stratum and is denoted $ \calM_S^+(\uunu)$. 
We may identify $ \calM_S^+(\uunu)$ with a moduli stack of filtrations, hence there is an open embedding
$ \calM_S^+(\uunu) \subset \Filt(\calM_S(\nu))$.
This defines a derived structure on the stack $\calM_S^+(\uunu)$.
This derived stack is quasi-smooth, quasi-compact, and
is a locally closed substack of $\calM_S(\nu)$, with $\nu=\sum_i\nu_i$ the weight of $\uunu$. 
If $\uunu=(\nu)$ consists of a single element, then $ \calM_S^+(\uunu)=\calM_S^\ss({\nu})$. 
If $d=0$, we define the HN-type of a 0 dimensional sheaf $\calF$ to be its length, i.e., we set  $\uunu(\calF)=\nu(\calF)=l(\calF)$.
We abbreviate
\begin{equation}\label{def:Mssuunu}
\calM_S^\ss(\uunu)= \prod_{i=1}^\ell \calM_S^\ss(\nu_i).
\end{equation}
In this setting, the correspondence \eqref{attractor-explicit} is the following 0-shifted Lagrangian correspondence
\begin{equation}\label{eq-strata-corresp}
\calM_S^\ss(\uunu) \xleftarrow{\gr_{\uunu}}  \calM_S^+(\uunu) \xhookrightarrow[]{\ev_{\uunu}} \calM_S(\nu).
\end{equation}
We have
\begin{equation}\label{eq:dnu}
\vdim\gr_\uunu=\frac{1}{2}\Big(\vdim \calM_{S}(\nu)-\sum_{i\leqslant l} \vdim \calM_S^\ss(\nu_i) \Big)=-\sum_{1\leqslant i<j\leqslant l} (\nu_i,\nu_j).
\end{equation}

\begin{proposition}\label{prop: purity-ss}
Let $\uunu$ be a pure HN-type.\hfill
\begin{enumerate}[label=$\mathrm{(\alph*)}$,leftmargin=8mm,itemsep=2mm]
\item
If $\calF,$ $\calG$ are Gieseker-semistable sheaves with $p(\calF) > p(\calG)$, then 
$\Hom(\calF,\calG)= \Ext^2(\calG,\calF)=\{0\}.$
\item
The map  $\gr_{\uunu}$  is an iterated vector bundle stack of rank $ \vdim\gr_\uunu$. 
\end{enumerate}
\end{proposition}

\begin{proof}
The first statement of (a) is well-known, the second one follows from Serre duality: 
as $\Omega_S=\calO_S$ the Serre functor coincides with the shift functor $\calF \mapsto \calF[2]$. 
For (b), note that the map $\gr_{\uunu}$ may be factored into a sequence of morphisms
$$ \calM_S^+(\uunu) \xrightarrow{\pi_1} \calM_S^\ss(\nu_1) \times \calM_S^+(\nu_2,\ldots,\nu_\ell) \xrightarrow{\pi_2} \cdots \xrightarrow{\pi_{\ell}} \calM_S^\ss(\uunu)$$
where each map $\pi_i$ is the projection to its base from
the total space of a complex $R\Hom(\calF_i/\calF_{i+1},\calF_{i+1})[1].$ 
By (a), we have 
$\Ext^2(\calF_i/\calF_{i+1},\calF_{i+1})=0,$ which means that each $\pi_i$ is in fact a stack vector bundle, as wanted. 
\end{proof}

We now define a partial order on the set of pure HN-types of a fixed weight. 
Let $d=1,2$. For any $d$-dimensional HN-type 
$\uunu=(\nu_1, \ldots, \nu_\ell)$ set $\nu_{\leqslant i}=\sum_{j\leqslant i} \nu_j$
and let $\nu=\sum_{i\leqslant l} \nu_i$ be the weight of $\uunu$.
We consider the convex path 
$\rho^d_\uunu$ in $\Q^2$ joining the vertices
$$(0,0), \quad (\alpha_d(\nu_1), \alpha_{d-1}(\nu_1)), 
\quad 
(\alpha_d(\nu_{\leqslant 2}), \alpha_{d-1}(\nu_{\leqslant 2})),
\quad \ldots \quad 
(\alpha_d(\nu_{\leqslant \ell}),\alpha_{d-1}(\nu_{\leqslant \ell})).$$
The path $\rho^d_\uunu$ starts at $0$ and ends at the point
$(\alpha_d(\nu),\alpha_{d-1}(\nu))$. If $\calF$ is a pure $d$-dimensional sheaf in 
$\calM^+_S(\uunu)$ then we set $\rho^d(\calF)=\rho^d_\uunu$. Let $\uunu,$ $\underline{\sigma}$ be pure $d$-dimensional 
HN-types of the same weight. 
We write $\uunu\preccurlyeq \underline{\sigma}$ if $\rho^d_\uunu$ lies above $\rho^d_{\underline{\sigma}}$. 
Hence $\uunu\preccurlyeq \underline{\sigma}$ if and only if $ \rho^d_\uunu \subset \rho^d_{\underline{\sigma}} + (\{0\} \times \Q^+).$
By Proposition~\ref{prop:good HN order} below, this partial order has nice compatibilities with the closure relation of HN-strata. 

\medskip
 
\subsubsection{Relation between Gieseker and slope HN stratification} For pure zero-dimensional or one-dimensional sheaves, the notions of Gieseker and $\mu$-stability coincide: 
for a one-dimensional Mukai vector $\nu=\alpha+\chi\delta$, the Hilbert polynomial is $P(\nu)(x)=(\alpha\cdot\omega)x+\chi$ and we have $\calM_S^\muss(\nu)=\calM_S^\ss(\nu)$.
For two-dimensional sheaves, we in general only have
$$(\text{Gieseker semistable}) \Rightarrow (\text{slope semistable})$$
and hence an open immersion $\calM^\ss_S(\nu) \subset \calM^\muss_S(\nu)$. Nevertheless, the following holds.

\begin{proposition} Let $\nu \in \H^\ev(S,\Z)$ be a two-dimensional Mukai vector. Then
$$\calM^\muss_S(\nu)=\bigsqcup_{\underline{\sigma}} \calM^+_S(\underline{\sigma})$$
where the union ranges over the finite set of HN-types $\underline{\sigma}=(\sigma_1,\ldots,\sigma_s)\in \HN(\nu)$ such that 
\begin{equation}\label{eq:Gieseker vs slope}
	\mu^2(\sigma_1)=\cdots =\mu^2(\sigma_s)=\mu^2(\nu).
\end{equation}
\end{proposition}
\begin{proof}
The only thing to prove is that the set of HN-types satisfying \eqref{eq:Gieseker vs slope} is finite. This follows from Bogomolov's inequality. We recall the argument for the reader's convenience. Let $\underline{\sigma}=(\sigma_1,\ldots,\sigma_s)\in \HN(\nu)$ satisfy \eqref{eq:Gieseker vs slope}. Then we have
\begin{equation}\label{eq:Gieseker vs slope2}
c_1(\sigma_i)\cdot \omega=\text{rk}(\sigma_i)
	\frac{c_1(\nu)\cdot \omega}{\text{rk}(\nu)}
\end{equation}
for all $i$, since all $\sigma_i$ have the same slope. Note also that
$$\sum_i \ch_2(\sigma_i)=\ch_2(\nu), \qquad \sum_i \text{rk}(\sigma_i)=\text{rk}(\nu).$$
The second equation, along with \eqref{eq:Gieseker vs slope}, implies that the tuples $(\text{rk}(\sigma_i))_i$ and $(c_1(\sigma_i)\cdot \omega)_i$ can take only finitely many values.
Now we use the Bogomolov inequality \eqref{eq:Bogomolov} for each of the $\sigma_i$. By the Hodge index theorem $c_1(\sigma_i)^2 \omega^2\leqslant (c_1(\sigma_i)\cdot \omega)^2= \text{rk}(\sigma_i)^2(\mu^2(\nu))^2(\omega^2)^2$. Combining this with the Bogomolov inequality this yields the inequality $\ch_2(\sigma_i)\omega^2\leqslant \frac{1}{2}\mu^2(\nu)^2\textrm{rk}(\sigma_i)(\omega^2)^2$. The relation $\sum_i \ch_2(\sigma_i)=\ch_2(\nu)$ now implies that the tuple $(\ch_2(\sigma_i))_i$ may only take finitely many values.  The Bogomolov inequality gives lower bounds for each $c_1(\sigma_i)^2$, while the Hodge index theorem gives upper bounds.  Since $\omega^{\perp}$ is negative definite, and $c_1(\sigma_i)^2$ is bounded above and below, there are only finitely many values of $c_1(\sigma_i)$ satisfying \eqref{eq:Gieseker vs slope2}.
\end{proof}

\medskip

\subsubsection{The HN-stratification of arbitrary sheaves}\label{sec:partial order2}
Now, we consider arbitrary coherent sheaves, not necessarily 
pure. The HN-type of $\calF$ is the triple of pure $\HN$-types
\begin{align}\label{HN}\uunu(\calF)=(\uunu(\calF^2)\,,\,\uunu(\calF^1)\,,\,\uunu(\calF^0)).\end{align}
Note that $\uunu(\calF^0)$ is the length of $\calF^0$. A coherent sheaf $\calF$ has a canonical Harder-Narasimhan filtration
$$\{0\}\subseteq \calF_{\ell+m} \subsetneq \calF_{\ell+m-1} \subsetneq \cdots \subsetneq \calF_{\ell}\subsetneq\calF_{\ell-1}\subsetneq 
\cdots \subsetneq\calF_0=\calF$$
such that $\calF^{\leqslant 1}=\calF_{\ell}$ and $\calF^0=\calF_{\ell+m}$,
 whose subquotients are semistable of strictly increasing reduced Hilbert polynomial.
Let $\calM_S^+(\uunu)^\cl$ be the classical stack parametrizing coherent sheaves on $S$ of HN-type $\uunu$.
 Let $\HN(\nu)$ be the set of all effective HN-types of weight $\nu$, i.e., for which $ \calM_S^+(\uunu)^\cl\neq\emptyset$.

\begin{proposition}\label{prop: purity-ss2} Let $\nu$ be a Mukai vector and $\uunu\in\HN(\nu)$. \hfill
\begin{enumerate}[label=$\mathrm{(\alph*)}$,leftmargin=8mm,itemsep=2mm]
\item $\calM_S^+(\uunu)^\cl$ has a derived enhancement $\calM_S^+(\uunu)$.
It is quasi-smooth, quasi-compact and locally closed in $\calM_S(\nu)$.
\item There is a correspondence 
\begin{equation}\label{eq-strata-corresp-not-pure}
\calM_S^\ss(\uunu) \xleftarrow{\gr_{\uunu}}  \calM_S^+(\uunu) \xhookrightarrow[]{\ev_{\uunu}} \calM_S(\nu).
\end{equation}
The map  $\gr_{\uunu}$  is an iterated vector bundle stack
of rank $ \vdim\gr_\uunu$ given by \eqref{eq:dnu}. 
\item 
There is a locally finite partition 
$$\calM_S(\nu)=\bigsqcup_{\uunu\in \HN(\nu)}  \calM_S^+(\uunu).$$
Any quasi-compact open substack $\calU \subset \calM_S(\nu)$ 
intersects finitely many strata nontrivially.
\end{enumerate}
\end{proposition}

\begin{proof}
The proof is the same as for Proposition~\ref{prop: purity-ss}. 
The last statement is obtained by combining Lemma~\ref{lem:affine fibration pure types}, 
the diagram \eqref{eq-strata-corresp}, Proposition~\ref{prop: purity-ss} and quasi-compactness of each of the stacks $\calM_S^\ss(\nu)$. 
\end{proof}

The HN partition depends on the polarization and is constrained by the 
Bogomolov inequality \eqref{eq:Bogomolov}. 
We extend the partial order from pure HN-types to arbitrary HN-types. 
To do this, for each tuple $\uunu=(\uunu_2,\uunu_1,\uunu_0)$ of pure HN-types of dimensions $2,1,0$ respectively we set
$$\rho^2_{\uunu}=\rho^2_{\uunu_2}
, \quad 
\rho^1_{\uunu}=(\alpha_1(\nu_2)\,,\,\alpha_0(\nu_2))+\rho_{\uunu_1}^1,$$
where $\nu_2$ is the weight of $\uunu_2$. Notice that the starting and ending points of the paths $\rho^d_{\underline{\nu}}$ depend on the individual weights of $\underline{\nu}_2, \underline{\nu}_1$ and $\underline{\nu}_0$.
Given two HN-types $\uunu,$ 
$ \uunu'$ of the same weight such that $(\rho^2_{\uunu}, \rho^1_{\uunu})\neq (\rho^2_{\uunu'},\rho^1_{\uunu'})$ we set 
$$\uunu \preccurlyeq \uunu'\quad \Leftrightarrow \quad  
\rho^2_{\uunu} \preccurlyeq \rho^2_{\uunu'} \quad \text{and}\quad \rho^1_{\uunu} \preccurlyeq \rho^1_{\uunu'}.$$
Note that if $\uunu \neq \uunu'$ and $\rho^d_{\uunu}=\rho^d_{\uunu'}$ for $d=1,2$,
 then $\uunu$ and $\uunu'$ are uncomparable.
 We define
 $$\calM^+_S(\succcurlyeq\!\uunu)=\bigsqcup_{\uunu'\succcurlyeq\uunu}\calM^+_S(\uunu')
 ,\quad
 \calM^+_S(\preccurlyeq\!\uunu)=\bigsqcup_{\uunu'\preccurlyeq\uunu}\calM^+_S(\uunu').$$
The following propositions are proved in \S \ref{app:C}.

\begin{proposition}\label{prop:finite interval partial} 
	Let $\nu$ be any Mukai vector. 
	\hfill
	\begin{enumerate}[label=$\mathrm{(\alph*)}$,leftmargin=8mm,itemsep=2mm]
		\item
		The poset $\HN(\nu)$ is locally finite, i.e., for any $\HN$ types $\uunu'$ and $\uunu$ the interval
		$[\uunu',\uunu]$ in $\HN(\nu)$ is finite.
		\item
		There is a total order $\leqslant$ refining the partial order $\preccurlyeq$ on $\HN(\nu)$ with an order preserving bijection 
		$\N \cong \HN(\nu)$.
		\qed
	\end{enumerate}
\end{proposition}

\begin{proposition}\label{prop:good HN order}
Let $\nu$ be a Mukai vector and $\uunu\in\HN(\nu)$. 
\hfill
\begin{enumerate}[label=$\mathrm{(\alph*)}$,leftmargin=8mm,itemsep=2mm]
\item
The substack  $\calM^+_S(\succcurlyeq\!\uunu)\subset \calM_S$ is closed.
\item
The substack $\calM^+_S(\preccurlyeq \!\uunu)\subset \calM_S$ is open and quasi-compact.
\item 
We have $\calM_S(\nu)=\bigcup_{\uunu\in\HN(\nu)}  \calM^+_S(\preccurlyeq\uunu)$.
\qed
\end{enumerate}
\end{proposition}  

We will always equip $\HN(\nu)$ with the above refined total order.
We define
 $$\calM^+_S(\geqslant\!\uunu)=\bigsqcup_{\uunu'\geqslant\uunu}\calM^+_S(\uunu')
 ,\quad
 \calM^+_S(\leqslant\!\uunu)=\bigsqcup_{\uunu'\leqslant\uunu}\calM^+_S(\uunu')$$
which are closed, resp. quasi-compact and open substacks of $\calM_S(\nu)$ respectively.

\medskip

\subsubsection{Abelian categories of semistable sheaves} \label{Abelian_subcats_sec}
In order to apply the DT formalism to stacks of semistable
 sheaves, we need to identify a suitable Abelian category.
 Fix a reduced polynomial $p=x^d/d!+\ldots$ and put
$$\Gamma_p=\{\nu\,\text{effective}\,;\,p_H(\nu)=p\}\cup\{0\}
,\quad
\Gamma_p^\times=\Gamma_p\smallsetminus\{0\}.$$  
Since the leading term of $P_H(\calF)$ is always positive, $\Gamma_p$ is a strictly convex cone.  We define
$$\calM^\ss_S(p)=\bigsqcup_{\nu\in\Gamma_p} \calM^\ss_S(\nu).$$

\begin{lemma}\label{lem:2CY semistable}
The full subcategory $\Coh^{\ss}(p)$ of $\calC$ whose objects are points of $\calM^\ss_S(p)$ is a $2$CY Abelian category.	
\end{lemma}

\begin{proof}
Let $\calF,\calG \in \calM^\ss_S(p)$ be non-zero and let $f \in \Hom_\calC(\calF,\calG)$. 
Let $\calH$ be the image of $f$ in the Abelian category $\calC$. 
Again, we may assume $\calH$ is non-zero, for otherwise $\text{Ker}(f)$ and $\text{Coker}(f)$ are obviously Gieseker-semistable with reduced Hilbert polynomial $p$.  
By semistability, we have 
$$p=p_H(\calF) \leqslant \pmin(\calH) \leqslant \pmax(\calH) \leqslant p_H(\calG)=p.$$
Hence $\calH$ is semistable and $p_H(\calH)=p$. Hence $p_H(\text{Ker}(f))=p$ if $\text{Ker}(f)\neq 0$ and $p_H(\text{Coker}(f))=p$ if $\text{Coker}(f)\neq 0$. Since
$$ \pmax(\text{Ker}(f)) \leqslant p_H(\calF)=p, \quad \pmin(\text{Coker}(f)) \geqslant p_H(\calF)=p$$
we deduce that $\text{Ker}(f)$ and $\text{Coker}(f)$ are semistable, and both are either zero or have reduced Hilbert polynomial $p$.
\end{proof}


\begin{remark}
Let $\nu_\pr$ be a primitive Mukai vector. Assume that $\nu_\pr$ is $H$-generic, 
see \cite[\S 4.C]{HL}. Then, the Mukai vector of any direct summand 
of a polystable coherent sheaf of Mukai vector in $\N \nu_\pr$ belongs also to $\N\nu_\pr$. 
Thus, the full subcategory of 
$\calC$ consisting of semistable sheaves of Mukai vectors in $\N\nu_\pr$ is an Abelian 2CY subcategory.
This may be false if $H$ is not generic.
\end{remark}

\medskip

\subsubsection{Geometric properties of the stacks $\calM_S$ and $\calM_S^{\ss}$}
We may now confirm that the assumptions in \S\ref{para-assumptions-M} hold.

\begin{lemma}
\hfill
\begin{enumerate}[label=$\mathrm{(\alph*)}$,leftmargin=8mm,itemsep=2mm]
\item
The stack $\calM_S$ satisfies the conditions \eqref{item-M-1-Artin}-\eqref{item-M-theta-reductive}.
\item 
For any reduced polynomial $p$, the Gieseker-semistable locus $\calM^\ss_S(p)$ in $\calM_S$ satisfies the conditions 
\eqref{item-M-1-Artin}-\eqref{item-ac-graded}.\end{enumerate}
\end{lemma}

\begin{proof}
The conditions \eqref{item-M-1-Artin} and \eqref{item-M-closed} are obvious.
Since the condition \eqref{item-M-affine-diagonal} follows from the condition \eqref{item-M-1-Artin}, 
it is enough to check the condition \eqref{item-M-theta-reductive}.
The condition \eqref{item-M-theta-reductive} for $\calM_S$ follows from the properness of the flag space,
see \cite[ex.~2.4]{halpern2016theta}.
Then the condition for the semistable locus follows from \cite[prop.~6.14]{alper2023existence} and Lemma~\ref{lem:2CY semistable}.
\end{proof}

\medskip

\subsubsection{The HN-stratification of $\calM_X(\nu)$}\label{sec:HN-X}
Recall the map $\kappa\colon\calM_X\to\calM_S$.
We define 
$$\calM_X^\ss(\nu)=\kappa^{-1}(\calM_S^\ss(\nu)).$$
Pulling back the HN-stratification we obtain a stratification
\begin{equation}\label{eq-HN-stratification-wMS}
\calM_X(\nu) =  \bigsqcup_{\uunu \in \HN(\nu)}  \calM_X^+(\uunu).
\end{equation}
Note that the open embedding
$\calM_X^+(\uunu) \subset \Filt(\calM_X(\nu))$
defines a derived structure on $\calM_X^+(\uunu)$.
In this setting, the correspondence \eqref{attractor-explicit} is
\begin{equation}\label{w-eq-strata-corresp}
\calM_X^\ss(\uunu) \xleftarrow{\wgr_{\uunu}} \calM_X^+(\uunu) 
\xhookrightarrow[]{\tilde\ev_{\uunu}} \calM_X(\nu).
\end{equation}
 an open part of the correspondence \eqref{attractor2}.
We abbreviate
\begin{align}\label{eq:phi-ss}\varphi_{\calM_X^\ss(\uunu)}=\varphi_{\calM_X^\ss(\nu_1)}\boxtimes\varphi_{\calM_X^\ss(\nu_2)}
\boxtimes\cdots\boxtimes\varphi_{\calM_X^\ss(\nu_\ell)}.
\end{align}
The integral isomorphism \eqref{eq-integral-identity} yields an isomorphism
\begin{equation}\label{eq-integrality-HN-strata}
\varphi_{\calM_X^\ss(\uunu)} \cong(\wgr_{\uunu})_* (\tilde\ev_{\uunu})^! \varphi_{\calM_X(\nu)}.
\end{equation}

\begin{remark}
Since $X=\mathbb A^1\times S$, a coherent sheaf $\mathcal E$ on $X$
with proper support is equivalent to a Higgs pair $(\mathcal F,\phi)$,
where $\mathcal F=\kappa_*\mathcal E$ is a coherent sheaf on $S$ and
$\phi\in \operatorname{End}(\mathcal F)$; see
\S\ref{sec:nilp-micro}. By \cite[Lemma 2.9]{TanakaThomas}, the Gieseker
semistability of $\mathcal E$ with respect to $\kappa^*H$ is equivalent
to the Gieseker semistability of the Higgs pair $(\mathcal F,\phi)$.
This is equivalent to the Gieseker semistability of $\mathcal F$ with
respect to $H$, since the maximal destabilizing subsheaf of $\mathcal F$
is preserved by $\phi$. This justifies our definition of
$\mathcal M_X^\ss(\nu)$.


\end{remark}

\subsubsection{The good moduli space and the Chow variety}\label{good-Chow}

The stack $\calM_S^\ss$ admits a good moduli space
given by the obvious morphism to its GIT moduli space.
Given the cohomology class $\alpha\in \H_2(S,\Q)$ of an effective 1-cycle on $S$ and any integer $n$, it yields the morphism
$$p_S\colon\calM_S^{\ss}(\alpha+n\delta)\to M_S(\alpha+n\delta).$$
For any scheme $M$ of finite type Rydh defines in \cite{Rydh}  a  
functor of proper equidimensional relative cycles on $M$. It is represented by a scheme
whose reduced subscheme coincides with the Suslin-Voevodsky quasi-projective variety if $M$ is a complex quasi-projective variety.
We call this reduced subscheme the Chow variety.
More precisely, let $B_S(\alpha)$ be the Chow variety parametrizing effective 1-cycles on the surface $S$ with homology class $\alpha$.
By \cite[thm.~7.14]{Rydh}, there is a morphism of reduced classical stacks
called the Hilbert-Chow map
$$\rho_S\colon\calM_S(\alpha+\Z\delta)_\red\to B_S(\alpha).$$
We now give a list of properties of good moduli spaces and Chow varieties for later use.

\smallskip

\noindent{(a)}
Since a good moduli space is a categorical quotient in the category of algebraic spaces,
the morphism $\rho_S$ factorizes through $\pi_S$.
To simplify the notation we omit the Mukai vector whenever it is clear from the context.
This yields a commutative square
\begin{align}\label{p-pi-rho-S}
\begin{split}
\xymatrix{
(\calM_S^\ss)_\red\ar@{^{(}->}[r]\ar[d]_-{p_S}&(\calM_S)_\red\ar[d]^-{\rho_S}\\
(M_S)_\red\ar[r]^-{\pi_S}&B_S.}
\end{split}
\end{align}
Replacing the surface $S$ by $X=\A^1\times S$, we have a good moduli space
$p_X\colon\calM_X^\ss\to M_X$ 
which yields the commutative square   
 \begin{align}\label{p-pi-rho-X}
\begin{split}
\xymatrix{
(\calM_X^{\ss})_\red\ar@{^{(}->}[r]\ar[d]_-{p_X}&(\calM_X)_\red\ar[d]^-{\rho_X}\\
(M_X)_\red\ar[r]^-{\pi_X}&B_X.}
\end{split}
\end{align}   
The morphisms $\pi_S$ and $\pi_X$ are projective.  For $\pi_S$ this is clear, since the domain is projective.  For $\pi_X$ it follows from the analogous statement for $\pi_{\mathbb{P}^1\times S}$ by base-changing along the morphism $B_{\mathbb{A}^1\times S}\hookrightarrow B_{X}$ induced by any inclusion $\mathbb{A}^1\hookrightarrow \mathbb{P}^1$.


\smallskip

\noindent{(b)} 
The direct sums of sheaves and cycles yield the morphisms
$$\oplus\colon\calM_S\times\calM_S\to\calM_S
,\quad
\oplus\colon M_S\times M_S\to M_S
,\quad
\oplus\colon B_S\times B_S\to B_S.$$
Indeed, by \cite{Rydh}, any finite morphism $f\colon Y\to Z$ yields a morphism of functors $B_Y\to B_Z$.
Applying this natural transformation to the obvious morphisms $S\coprod S\to S$  yields the
direct sum $\oplus\colon B_S\times B_S\to B_S$.
These morphisms fit into the following commutative diagram
 \begin{align}\label{oplus_M-B-S}
\begin{split}
\xymatrix{
(\calM_S)_\red\times (\calM_S)_\red\ar[d]_-{\rho_S\times\rho_S}\ar[r]^-\oplus&(\calM_S)_\red\ar[d]^-{\rho_S}\\
B_S\times B_S\ar[r]^-{\oplus}&B_S.}
\end{split}
\end{align}   
We also have direct sum morphisms $\oplus$ on $\calM_X$, $M_X$ and $B_X$ giving the diagrams
 \begin{align}\label{oplus_M-B-X}
\begin{split}
\xymatrix{
(\calM_X)_\red\times (\calM_X)_\red\ar[d]_-{\rho_X\times\rho_X}\ar[r]^-\oplus&(\calM_X)_\red\ar[d]^-{\rho_X}\\
B_X\times B_X\ar[r]^-{\oplus}&B_X.}
\end{split}
\end{align}   
Quasi-finiteness of the direct sum maps $\oplus$ on $M_S$, $B_S$, $M_X$ and $B_X$ is clear, since polystable sheaves admit only finitely many direct sum decompositions, and effective cycles admit only finitely many sum decompositions. Finiteness of these morphisms follows from projectivity.  For $S$ projectivity is clear, since the domains of both morphisms are projective, and for $X$ the statements can be deduced by base change along the morphisms induced by $X\hookrightarrow \mathbb{P}^1\times S$.

\smallskip

\noindent{(c)}
Given pairs $\uunu=(\nu_1,\nu_2)$ and $\uualpha=(\alpha_1,\alpha_2)$ with $\nu_i\in\alpha_i+\Z\delta$, 
we set $B_S(\uualpha)=\prod_iB_S(\alpha_i)$.
We have the following commutative diagram
\begin{align}\label{induction-B-S}
\begin{split}
\xymatrix{ 
\calM_S^\ss(\uunu)_\red\ar[d]_-{\rho_S} &\ar[l]_-{\gr_{\uunu}}  \calM_S^+(\uunu)_\red \ar[r]^-{\ev_{\uunu}}& 
\calM_S(\alpha)_\red\ar[d]^-{\rho_S}\\
B_S(\uualpha)  \ar[rr]^-\oplus&&B_S(\alpha)}
\end{split}
\end{align}
and
\begin{align}\label{induction-B-X}
\begin{split}
\xymatrix{ 
\calM_X^\ss(\uunu)_\red\ar[d]_-{\rho_X} &\ar[l]_-{\gr_{\uunu}}  \calM_X^+(\uunu)_\red \ar[r]^-{\ev_{\uunu}}& 
\calM_X(\alpha)_\red\ar[d]^-{\rho_X}\\
B_X(\uualpha)  \ar[rr]^-\oplus&&B_X(\alpha).}
\end{split}
\end{align}

\smallskip

\noindent{(d)}
The nilpotent and semi-nilpotent substacks introduced in \S\ref{sec:nilp-micro} are
$\Lambda_\nil\,,\,\Lambda_\snil\subset(\calM_X)_\red$.
We define
\begin{align*}
\Lambda_\nil^\ss=(\calM_X)_\red^\ss\cap\Lambda_\nil,\quad
\Lambda_\snil^\ss=(\calM_X)_\red^\ss\cap\Lambda_\snil.
\end{align*}
Following \eqref{0-i-gamma}, 
the zero section and the action map yield closed immersions
\begin{align}\label{0-MMB}
0\colon \Lambda_\nil\to(\calM_X)_\red,\quad
0\colon (M_S)_\red\to (M_X)_\red,\quad
0\colon B_S\to B_X
\end{align}
and
\begin{align}
i\colon \Lambda_\snil\to(\calM_X)_\red,\quad
i\colon \A^1\times (M_S)_\red\to (M_X)_\red,\quad
i\colon \A^1\times B_S\to B_X.
\end{align}
They yield the following commutative diagram 
\begin{align}\label{i-MMMB}
\begin{split}
\xymatrix@C=3em@R=3em{
 \A^1\times B_S  \ar[d]^-{i} &
\ar[l]_-{\id\times\pi_S} \A^1\times (M_S)_\red  \ar[d]^-{i} & \ar[l]_-{p_X}\Lambda_\snil^\ss \ar@{^{(}->}[r] \ar[d]^-i&
 \Lambda_\snil \ar[d]^-i\\
 B_X & \ar[l]_-{\pi_X} (M_X)_\red &\ar[l]_-{p_X}  (\calM_X)_\red^\ss \ar@{^{(}->}[r] & (\calM_X)_\red.
}
\end{split}
\end{align}
The right square is obviously Cartesian, the middle one by \eqref{nilp-equiv-def}.
We claim that the left square is also Cartesian.
Similarly, we have the following commutative diagram with Cartesian squares
\begin{align}\label{0-MMMB}
\begin{split}
\xymatrix@C=3em@R=3em{
B_S \ar[d]^-{0\times\id}
& (M_S)_{\red} \ar[l]_-{\pi_S} \ar[d]^-{0\times\id}
& \Lambda_\nil^\ss \ar[l]_-{p_X} \ar@{^{(}->}[r] \ar[d]
& \Lambda_\nil\ar[d] \\
\A^1\times B_S
& \A^1\times (M_S)_{\red} \ar[l]_-{\id\times\pi_S}
& \Lambda_\snil^\ss \ar[l]_-{p_X} \ar@{^{(}->}[r]
& \Lambda_\snil.
}
\end{split}
\end{align}

  
\smallskip

\noindent{(e)} 
We have the affine morphisms
$$\kappa\colon \calM_X\to\calM_S
,\quad
\kappa\colon M_X\to M_S
,\quad
\kappa\colon B_X\to B_S$$
yielding the following commutative square    
\begin{align}\label{oplus-S-X} 
\begin{split}
\xymatrix{B_X\times B_X\ar[r]^-{\oplus}\ar[d]_-{\kappa\times\kappa}&B_X\ar[d]^-\kappa\\
B_S\times B_S\ar[r]^-{\oplus}&B_S.}
\end{split}
\end{align}
We have also the Cartesian square
\begin{align}\label{oplus-X-S} 
\begin{split}
\xymatrix{B_S\times B_S\ar[r]^-{\oplus}\ar[d]&B_S\ar[d]\\
B_X\times B_X\ar[r]^-{\oplus}&B_X.}
\end{split}
\end{align}
The latter square is Cartesian since a union of families of cycles $C_1,C_2$ factors through the inclusion $S\hookrightarrow X$ if and only if both $C_1$ and $C_2$ do.

\medskip

\subsection{Kirwan surjectivity}\label{sec:deffiltration}

\subsubsection{Kirwan surjectivity}
 We now prove the Kirwan surjectivity property for the BM(=Borel--Moore) homology and the nilpotent 
 microlocal homology of the stack $\calM_S$.
We begin with the following purity statement.

\begin{theorem}\label{thm-purity}
Let $\uunu\in\HN(\nu)$. 
\hfill
\begin{enumerate}[label=$\mathrm{(\alph*)}$,leftmargin=8mm,itemsep=2mm]
\item
The mixed Hodge structures $\HH_*^{\BM}( \calM_S^+(\uunu),\Q)$ 
and  $\HH_*^{\Lambda_{\nil}}( \calM_S^+(\uunu),\Q)$ are pure.
\item
The mixed Hodge structures $\HH_*^{\BM}( \calM^+_S(\leqslant \!\uunu),\Q)$ 
and  $\HH_*^{\Lambda_{\nil}}(  \calM^+_S(\leqslant \!\uunu),\Q)$ are pure.
\end{enumerate}
\end{theorem}

\begin{proof}
The purity of the mixed Hodge module $\HH_*^{\BM}( \calM_S^+(\uunu),\Q)$ follows from \cite[thm.~7.10]{D24}.
Using the cohomological integrality in Corollary \ref{cor:absolute-integrality}, we deduce that 
the BPS cohomology $\HH_{\mathrm{BPS}}^*(M_X(\nu),\Q)$ is also a pure mixed Hodge module.
The nilpotent cohomological integrality in Proposition \ref{prop:nilp-cohomological-integrality} implies the purity of 
$\HH_*^{\Lambda_{\nil}}( \calM_S^+(\uunu),\Q)$, proving (a).
Part (b) follows from (a) and Propositions \ref{prop:good HN order} and \ref{prop:finite interval partial}.
\end{proof}

\begin{theorem}\label{thm-transition-surjective}
The restriction yields split surjections of complexes of mixed Hodge structures
\begin{align*}
R\Gamma^{\BM}( \calM^+_S(\leqslant \!\uunu),\Q) &\to R\Gamma^{{\BM}}( \calM^+_S(<\!\uunu),\Q),\\
R\Gamma^{\Lambda_{\nil}}( \calM^+_S(\leqslant \!\uunu),\Q) &\to R\Gamma^{\Lambda_{\nil}}( \calM^+_S(<\!\uunu),\Q).
\end{align*}
\end{theorem}

\begin{proof}
We first prove the claim for the nilpotent microlocal homology.
Proposition \ref{prop:good HN order} yields a closed immersion
$$\calM_X^+(\uunu)\to \calM^+_X(\leqslant \!\uunu).$$ 
We abbreviate
$$\Lambda_\nil^+(\uunu)=\Lambda_{\nil}\cap\calM_X^+(\uunu).$$
Consider the localization fibre sequence
\[R\Gamma(\Lambda_\nil^+(\uunu)\,,\,\varphi_{\calM_X(\nu)} |^!_{\Lambda_\nil^+(\uunu)})  \to 
R\Gamma^{\Lambda_{\nil}}( \calM^+_S(\leqslant \!\uunu),\Q) \to 
R\Gamma^{\Lambda_{\nil}}( \calM^+_S(<\!\uunu),\Q).\]
 We claim that this fibre sequence splits.
Indeed, it is enough to show that the first two complexes are pure, because the purity imposes the boundary map to be zero.
By an induction with respect to the total order on $\HN(\nu)$,  
 it is enough to show that the first complex is pure.
 The integral isomorphism \eqref{eq-integrality-HN-strata} implies that
 \[(\wgr_{\uunu})_* \left( \varphi_{\calM_X(\nu)} |^!_{\Lambda_\nil^+(\uunu)}  \right)  \cong 
  \bigboxtimes_{i=1}^\ell \left( \varphi_{\calM_X^+(\nu_{i})} |^!_{\Lambda^+(\nu_{i})_{\nil}}  \right)\]
where $\uunu=(\nu_1,\nu_2,\dots,\nu_\ell)$. Therefore the purity of the complex of mixed Hodge modules
$$R\Gamma\Big(\Lambda_\nil^+(\uunu)\,,\,
 \varphi_{\calM_X(\nu)} \big|^!_{\Lambda_\nil^+(\uunu)}\Big)$$
 follows from the purity of the nilpotent microlocal homology
 $\H^{\Lambda_{\nil}}_{*}(\calM_X^\ss(\nu_i),\Q)$ proved in Theorem \ref{thm-purity}.
 Repeating the argument  with $\calM_X(\nu)$ instead of $\Lambda_{\nil}$,
 we obtain the statement for the BM homology.
\end{proof}

\begin{corollary}\label{cor-BM-limit}
\hfill
\begin{enumerate}[label=$\mathrm{(\alph*)}$,leftmargin=8mm,itemsep=2mm]
\item
There are isomorphisms of pro-vector spaces
\begin{align*}
 \H^\BM_{*} (\calM_{S}(\nu),\Q)_\pro
& \cong \Lim{\uunu \in \HN(\nu)} \H^{\BM}_{*}(\calM^+_S(\leqslant\!\uunu),\Q), \\
 \H^{\Lambda_{\nil}}_{*} (\calM_{S}(\nu),\Q)_\pro
 &\cong \Lim{\uunu \in \HN(\nu)} \H^{\Lambda_{\nil}}_{*}(\calM^+_S(\leqslant\!\uunu),\Q).  
 \end{align*}
 \item
 There are isomorphisms of topological vector spaces
\begin{align*}
 \H^\BM_{*} (\calM_{S}(\nu),\Q)
& \cong \lim_{\uunu \in \HN(\nu)} \H^{\BM}_{*}(\calM^+_S(\leqslant\!\uunu),\Q), \\
 \H^{\Lambda_{\nil}}_{*} (\calM_{S}(\nu),\Q)
 &\cong \lim_{\uunu \in \HN(\nu)} \H^{\Lambda_{\nil}}_{*}(\calM^+_S(\leqslant\!\uunu),\Q).  
 \end{align*}

 \end{enumerate}
\end{corollary}

\begin{proof}
Part (a) follows from the fact that the open subsets $\calM^+_S(\leqslant\!\uunu)$ are quasi-compact and recover 
$\calM_S(\nu)$.
Part (b) follows from (a), because the homology groups 
$\H^{\BM}_i(\calM^+_S(\leqslant\!\uunu),\Q)$ and
$\H^{\Lambda_{\nil}}_i(\calM^+_S(\leqslant\!\uunu),\Q)$ are finite dimensional
for each $i$, 
hence the Mittag--Leffler condition is satisfied. 
\end{proof}

Here $\Lim{}$ is the limit as pro-vector space as in \S\ref{sec-pro}.
Theorem~\ref{thm-transition-surjective} and Corollary~\ref{cor-BM-limit} implies the following surjectivity result to which we will refer to as Kirwan surjectivity.

\begin{corollary}\label{cor:Kirwan-surj}\hfill
\hfill
\begin{enumerate}[label=$\mathrm{(\alph*)}$,leftmargin=8mm,itemsep=2mm]
\item
The restriction yields split surjections of mixed Hodge structures
\begin{align*}
\H_*^{\BM}(\calM_{S}(\nu),\Q) \to \H_*^{\BM}(\calM_{S}^{\ss}(\nu),\Q), \quad
\H_*^{\Lambda_{\nil}}(\calM_{S}(\nu),\Q) \to \H_*^{\Lambda_{\nil}}(\calM_{S}^{\ss}(\nu),\Q). 
\end{align*}
\item
We have the following isomorphisms of mixed Hodge structures
 \begin{align*}
\H_*^{\BM}(\calM_{S}(\nu),\Q) \cong \prod_{\uunu}\mathbb{L}^{ -\vdim\gr_\uunu} \otimes \H_*^{\BM}(\calM_S^\ss(\uunu),\Q)
,\quad
\H_*^{\Lambda_{\nil}}(\calM_{S}(\nu),\Q) \cong \prod_{\uunu}\mathbb{L}^{ -\vdim\gr_\uunu} \otimes 
\H_*^{\Lambda_{\nil}}(\calM_S^\ss(\uunu),\Q).
\end{align*}
\end{enumerate}
\qed
\end{corollary}

\begin{remark}
\hfill
\begin{enumerate}[label=$\mathrm{(\alph*)}$,leftmargin=8mm,itemsep=2mm]
\item
 $\H_i^\BM( \calM_S^\ss(\nu),\Q)$ is finite dimensional for all $i$ and $\nu$.
 However
 $\H_i^\BM(\calM_S(\nu),\Q)$ is infinite dimensional unless
$\nu \in \N \delta$. 
\item 
As a topological vector space $\H^\BM_i(\calM_{S}(\nu),\Q)$ is the limit of $\H^\BM_i(\mathcal{U},\Q)$
over all quasi-compact open subsets $\calU$, with the pro-topology.
It is also the image by the realization functor 
$\Pro(\Vect)\to\Vect$
of the pro-vector space $\H^\BM_{i}(\calM_{S}(\nu),\Q)_\pro$.
The action of $\H^*(\calM_S(\nu),\Q)_\pro$ on $\H^\BM_{*}(\calM_{S}(\nu),\Q)_\pro$ is the collection of actions of 
$\H^*(\mathcal{U},\Q)$ on $\H_*^\BM(\mathcal{U},\Q)$ for all quasi-compact open subset $\calU$.
\end{enumerate}
\end{remark}

\subsubsection{Relative Kirwan surjectivity}\label{sec:rel Kirwan}
Let $\nu$ be a one-dimensional Mukai vector.
In this section we decompose the complexes 
$(\rho_{X})_*\varphi_{\calM_X(\nu)}$ and $(\rho_S)_*\mathbb{D}\mathbb{Q}_{\calM_S(\nu)}$ 
into pieces associated with Harder-Narasimhan strata.
Recall that for each tuple $\uunu=(\nu_1, \ldots, \nu_\ell)$ in $\HN(\nu)$ we defined in \eqref{eq:phi-ss} the complex of mixed Hodge modules
$\varphi_{\calM_X^\ss(\uunu)}$ on $\calM_X^\ss(\uunu)$. 
Taking the pushforward by the map $\rho_X$ and the sum $\oplus$ we get
the complex of mixed Hodge modules over the Chow variety $B_X$ given by
\begin{equation}\label{eq:rho-phi-ss}
(\rho_X)_*\varphi_{\calM_X^\ss(\uunu)} =
(\rho_X)_*\varphi_{\calM_X^\ss(\nu_1)}\boxtimes_\oplus(\rho_X)_*\varphi_{\calM_X^\ss(\nu_2)}
\boxtimes_\oplus\cdots \boxtimes_{\oplus}(\rho_X)_*\varphi_{\calM_X^\ss(\nu_\ell)}\ .
\end{equation} 


\begin{proposition}\label{prop:rho-phi}

There exist isomorphisms of complexes of mixed Hodge modules in $\Pro(\Dh(B_X))$ and $\Pro(\Dh(B_S))$ 
\begin{align}\label{eq:rho-phi}
	\begin{aligned}
(\rho_X)^{\mathrm{pro}}_*\varphi_{\calM_X(\nu)} &\cong\Lim{\uunu \in \HN(\nu)}\bigoplus_{\uunu'\leqslant\uunu}(\rho_X)_*\varphi_{\calM_X^\ss(\uunu')} 
,\\
(\rho_S)^{\mathrm{pro}}_*\D\Q_{\calM_S(\nu)} &\cong\Lim{\uunu \in \HN(\nu)}\bigoplus_{\uunu'\leqslant\uunu}(\rho_S)_*\D\Q_{\calM_S^\ss(\uunu')} 
\otimes\L^{-\vdim\gr_\uunu}.
	\end{aligned}
\end{align} 
\end{proposition}

\begin{proof}
The proof is similar to the proof of Theorem \ref{thm-transition-surjective}.
The complex of mixed Hodge modules $(p_S)_*\D\Q_{\calM_S^\ss(\nu)}$ is semisimple and pure by \cite[thm.~B]{D24}.
The morphism $\pi_S$ is proper by \S\ref{good-Chow}.
Hence the complexes of mixed Hodge modules $(\rho_S)_*\D\Q_{\calM_S^\ss(\nu)}$ 
and $(\rho_S)_*\D\Q_{\calM_S^\ss(\uunu)}$ are also semisimple and pure.
Using Proposition~\ref{prop:support-lemma} and Theorem~\ref{thm:coh-integrality} we deduce that the complexes of mixed Hodge modules $(\rho_X)_*\varphi_{\calM_X^\ss(\nu)}$ 
and $(\rho_X)_*\varphi_{\calM_X^\ss(\uunu)}$ are also semisimple and pure.

Next, for each $\uunu\in\HN(\nu)$ we have the locally closed immersion 
$$\tilde\ev_\uunu\colon \calM_X^+(\uunu)\to\calM_X(\nu).$$
By Proposition \ref{prop:good HN order}, we also have the open immersion
$$\tilde\ev_{\leqslant\uunu}\colon \calM^+_X(\leqslant\!\uunu)\to\calM_X(\nu).$$
We claim that, in $\Dh(B_X(\nu))$, there is an isomorphism
\begin{equation}\label{decomp1}
(\rho_X)_*(\tilde\ev_{\leqslant\uunu})_*(\tilde\ev_{\leqslant\uunu})^!\varphi_{\calM_X(\nu)} \cong
\bigoplus_{\uunu'\leqslant\uunu}(\rho_X)_*\varphi_{\calM_X^\ss(\uunu')}.
\end{equation} 
We prove this claim by induction on the tuple $\uunu$.
Indeed, consider the triangle 
$$\xymatrix{(\tilde\ev_\uunu)_*(\tilde\ev_\uunu)^!\varphi_{\calM_X(\nu)}\ar[r]&
(\tilde\ev_{\leqslant\uunu})_*(\tilde\ev_{\leqslant\uunu})^!\varphi_{\calM_X(\nu)}\ar[r]
&(\tilde\ev_{<\uunu})_*(\tilde\ev_{<\uunu})^!\varphi_{\calM_X(\nu)}\ar[r]^-{+1}&}$$
Applying the pushforward by the map $\rho_X$ yields the triangle
$$\xymatrix{(\rho_X)_*(\tilde\ev_\uunu)_*(\tilde\ev_\uunu)^!\varphi_{\calM_X(\nu)}\ar[r]&
(\rho_X)_*(\tilde\ev_{\leqslant\uunu})_*(\tilde\ev_{\leqslant\uunu})^!\varphi_{\calM_X(\nu)}\ar[r]
&(\rho_X)_*(\tilde\ev_{<\uunu})_*(\tilde\ev_{<\uunu})^!\varphi_{\calM_X(\nu)}\ar[r]^-{+1}&}$$
By induction, the complex $(\rho_X)_*(\tilde\ev_{<\uunu})_*(\tilde\ev_{<\uunu})^!\varphi_{\calM_X(\nu)}$ is semisimple, pure 
and it decomposes as in \eqref{decomp1}.
The isomorphism \eqref{eq-integrality-HN-strata}
and the commutative diagram \eqref{induction-B-X} yield an isomorphism
\begin{align*}
(\rho_X)_*(\tilde\ev_\uunu)_*(\tilde\ev_\uunu)^!\varphi_{\calM_X(\nu)}
&\cong(\rho_X)_*\varphi_{\calM_X^\ss(\uunu)}.
\end{align*}
Since the complex of mixed Hodge modules 
$(\rho_X)_*\varphi_{\calM_X^\ss(\uunu)}$ is pure,
the triangle above splits, proving the proposition.
More precisely, the complex $(\rho_X)_*\varphi_{\calM_X(\nu)}$ is non canonically isomorphic to the image by the realization functor
of the system of semisimple objects  $$\bigoplus_{\uunu'\leqslant\uunu}{}^\p\gr\Big((\rho_X)_*\varphi_{\calM_X^\ss(\uunu')}\Big)$$
as $\uunu$ runs over $\HN(\nu)$.
The second isomorphism in \eqref{eq:rho-phi} follows from the first one, the dimensional reduction isomorphism 
\eqref{eq:dimensional-reduction} and the formula \eqref{eq:dnu}.
\end{proof}

Recall the closed immersion $i\colon \A^1\times M_S\to M_X$.
The counit $i_!i^!\to 1$ yields a morphism of pro-complexes of mixed Hodge modules 
\begin{align}\label{morphism1}
f_\nu\colon \kappa_*(\rho_X)^{\pro}_*i_!i^!\varphi_{\calM_X(\nu)}\to\kappa_*(\rho_X)^{\pro}_*\varphi_{\calM_X(\nu)}.
\end{align}
We have the following analog of Corollary \ref{cor:snilp-BPS}.

\begin{proposition}\label{prop:coker} 
The image of the associated graded
${}^\p\widehat\gr(f_\nu)$ of the morphism 
$f_\nu$ in \eqref{morphism1} is isomorphic to the complex of mixed Hodge modules
$$(\pi_S)_*\cBPS_{M_S(\nu)}\otimes \H^*(\BGM,\Q)_\vir\otimes\L^{-\frac12}.$$
\end{proposition}

\begin{proof}
For future use, recall the following commutative diagram with Cartesian squares
\begin{align}\label{triple-diag}
\begin{split}
\xymatrix@C=3em@R=3em{
\ar[d]\ar@/_2.2pc/[dd]_-{0^\ell}(B_S)^\ell \ar[d]^-{0\times\id} \ar[r]^-{\oplus} & B_S  \ar[d]^-{0\times\id} &
\ar[l]_-{\pi_S} M_{S,\red}  \ar[d]^-{0\times\id} & \ar[l]_-{p_X}\Lambda_\nil^\ss \ar@{^{(}->}[r] \ar[d]&
\ar@/_1.8pc/[lll]_-{\rho_X} \Lambda_\nil \ar[d]\ar@/^1.8pc/[dd]^-0 \\
\A^1\times (B_S)^\ell \ar[r]^-{\id\times\oplus} \ar[d]^-{i^\ell\circ(\Delta\times\id)} & \A^1\times B_S  \ar[d]^-{i} &
\ar[l]_-{\id\times\pi_S} \A^1\times M_{S,\red}  \ar[d]^-{i} & \ar[l]_-{p_X}\Lambda_\snil^\ss \ar@{^{(}->}[r] \ar[d]^-i&
\Lambda_\snil \ar[d]^-i\\
(B_X)^\ell \ar[r]^-{\oplus} & B_X & \ar[l]_-{\pi_X} M_{X,\red} &\ar[l]_-{p_X}  \calM_{X,\red}^\ss \ar@{^{(}->}[r] & \calM_{X,\red}
\ar@/^1.8pc/[lll]^-{\rho_X} 
}
\end{split}
\end{align}

We apply the continuous functor $i_!i^!$ to the first isomorphism \eqref{eq:rho-phi}.
A base change as in \eqref{eq:adjunction lemma1} yields the following isomorphism of
complexes of mixed Hodge modules
\begin{align}\label{eq:rho-0-phi4bis}
\begin{split}
(\rho_X)^{\pro}_*i_!i^!\varphi_{\calM_X(\nu)} 
&\cong\Lim{\uunu \in \HN(\nu)}\bigoplus_{\uunu'\leqslant\uunu}
(\rho_X)_*\oplus_*\uui_!(\Delta\times\id)_!(\Delta\times\id)^!\uui^!\boxtimes_{i=1}^\ell\varphi_{\calM_X^\ss(\nu'_i)}
\end{split}
\end{align} 
where we abbreviate $\uui=i^\ell$.
We apply the continuous functor $\kappa_*$. We get the isomorphism
\begin{align}\label{eq:rho-0-phi4}
\begin{split}
\kappa_*(\rho_X)^{\pro}_*i_!i^!\varphi_{\calM_X(\nu)} 
&\cong\Lim{\uunu \in \HN(\nu)}\bigoplus_{\uunu'\leqslant\uunu}
\kappa_*(\rho_X)_*\oplus_*\uui_!(\Delta\times\id)_!(\Delta\times\id)^!\uui^!\boxtimes_{i=1}^\ell\varphi_{\calM_X^\ss(\nu'_i)}.
\end{split}
\end{align} 
We claim that
all summands in \eqref{eq:rho-0-phi4} associated with unstable HN-strata, i.e., strata with $\ell>1$,  are taken to 0
by the morphism of complexes
\begin{align*}
f_\nu\colon \kappa_*(\rho_X)^{\pro}_*i_!i^!\varphi_{\calM_X(\nu)}\to\kappa_*(\rho_X)^{\pro}_*\varphi_{\calM_X(\nu)}
\end{align*}
given by  the counit $i_!i^!\to 1$.
Note that, since the union of all unstable strata is a closed subset, this sum is canonically defined as a subsheaf of \eqref{eq:rho-0-phi4}.
Hence $f_\nu$ factorizes in the following way
\begin{align}\label{eq:main-map1}
\kappa_*(\rho_X)^{\pro}_*i_!i^!\varphi_{\calM_X(\nu)}\to
\kappa_*(\rho_X)_*i_!i^!\varphi_{\calM_X^\ss(\nu)}\to
\kappa_*(\rho_X)^{\pro}_*\varphi_{\calM_X(\nu)}.
\end{align}
The first map is the restriction.
Further, a base change yields the isomorphism
\begin{align*}
\kappa_*(\rho_X)_*i_!i^!\varphi_{\calM_X^\ss(\nu)}
\cong\kappa_*(\pi_X)_*i_!i^!(p_X)_!\varphi_{\calM_X^\ss(\nu)}.
\end{align*}
Hence, by Remark \ref{rem:snilp-BPS}, the image of the complex 
${}^\p\widehat\gr\big(\kappa_*(\rho_X)_*i_!i^!\varphi_{\calM_X^\ss(\nu)}\big)$ 
by the associated graded of the second map in
\eqref{eq:main-map1} is isomorphic to the complex of
mixed Hodge modules
$$\kappa_*(\pi_X)_*\cBPS_{M_X(\nu)}\otimes \H^*(\BGM,\Q)_\vir.$$
The proposition follows because \eqref{eq:BPS-S-X} yields
$$\kappa_*(\pi_X)_*\cBPS_{M_X(\nu)}\cong(\pi_S)_*\cBPS_{M_S(\nu)}\otimes\L^{-\frac12}.$$

Now, we prove the claim.
The proof is done in two steps.
\hfill
\begin{enumerate}[leftmargin=8mm,itemsep=2mm]
\item
For each tuple $\uunu$ of length $\ell$ we consider the map
\begin{align*}
f_\uunu\colon \kappa_*(\rho_X)_*\oplus_*\uui_!(\Delta\times\id)_!(\Delta\times\id)^!\uui^!\boxtimes_{i=1}^\ell\varphi_{\calM_X^\ss(\nu_i)}
\to\oplus_*\boxtimes_{i=1}^\ell\kappa_*(\rho_X)_*\varphi_{\calM_X^\ss(\nu_i)}
=\kappa_*(\rho_X)_*\varphi_{\calM_X^\ss(\uunu)} 
\end{align*}
given by the counits $\uui_!\uui^!\to 1$ and $(\Delta\times\id)_!(\Delta\times\id)^!\to 1$. 
We claim that this map vanishes if $\ell>1$.
\item
We claim that we can choose the isomorphisms 
\begin{align*}
\kappa_*(\rho_X)^{\pro}_*\varphi_{\calM_X(\nu)} &\cong\Lim{\uunu \in \HN(\nu)}\bigoplus_{\uunu'\leqslant\uunu}\kappa_*(\rho_X)_*\varphi_{\calM_X^\ss(\uunu')} ,\\
\kappa_*(\rho_X)^{\pro}_*i_!i^!\varphi_{\calM_X(\nu)} 
&\cong\Lim{\uunu \in \HN(\nu)}\bigoplus_{\uunu'\leqslant\uunu}
\kappa_*(\rho_X)_*\oplus_*\uui_!(\Delta\times\id)_!(\Delta\times\id)^!\uui^!\boxtimes_{i=1}^\ell\varphi_{\calM_X^\ss(\nu'_i)}
\end{align*} 
in \eqref{eq:rho-phi} and \eqref{eq:rho-0-phi4} such that they identify the map
${}^\p\widehat\gr(f_\nu)$ on the left hand side with the sum of the maps ${}^\p\gr(f_{\uunu'})$ 
of each summand on the right hand side.
\end{enumerate}

The proof of (1) is similar to the proof of Corollary \ref{cor:snilp-BPS}, using the isomorphism 
\eqref{eq:support-lemma} and Lemma \ref{claim}.
The proof of (2) is a consequence of the construction.

\end{proof}


\bigskip

\section{The CoHA of coherent sheaves on a CY surface and Hecke operators}\label{sec:CoHA_hecke}

Given a 2CY category $\calC$ and an open substack $\calM$ of $\calM_{\calC}$ parametrizing objects in a $\C$-linear Abelian category and satisfying the assumptions 
\eqref{item-M-1-Artin}-\eqref{item-M-theta-reductive} in \S\ref{para-assumptions-M},
we defined in \eqref{eq-COHA-multi} a CoHA on the graded vector space $\H\calA_\calM$.
In this section we consider the particular case of the stack $\calM_S$ associated with the $\infty$-category $\calC=\Coh(S)$ 
of coherent sheaves on a smooth projective CY algebraic surface $S$.

\medskip

\subsection{The CoHA of the surface $S$}\label{sec:COHA-S}
We abbreviate
$$\bfH_{S,\pro}= \H\calA_{\calM_S,\pro}
,\quad
\bfH_S= \H\calA_{\calM_S}.$$
We have the $\K_0(S)^+\times\Z$-grading on $\bfH_S$ given by
\begin{equation}\label{eq:degreeshiftin defCOHA}
\bfH_S=\bigoplus_{\nu} \bfH_S(\nu)
,\quad
\bfH_S(\nu)= \bigoplus_n \bfH_{S,n}(\nu)
,\quad
\bfH_{S,n}(\nu)= \H_{-n+\vd_\nu}^\BM(\calM_S(\nu),\Q)
\end{equation}
where the summand $\bfH_{S,n}(\nu)$ has the degree $(\nu,n)$.
The $K$-theoretic degree $\nu$ is called the horizontal degree,
the cohomological degree $n$ is called the vertical degree. 
We also have
\begin{equation}\label{eq:degreeshiftin defCOHAX}
\bfH_X=\bigoplus_{\nu} \bfH_X(\nu)
,\quad
\bfH_X(\nu)= \bigoplus_n \bfH_{X,n}(\nu)
,\quad
\bfH_{X,n}(\nu)= \H^{n}(\calM_X(\nu),\varphi_{\calM_X(\nu)}).
\end{equation}
Dimensional reduction yields a graded vector space isomorphism
$\bfH_X\cong\bfH_S.$

Now, we consider the multiplication.
For any pair of Mukai vectors $\nu_1$ and $\nu_2$ with $\nu=\nu_1+\nu_2$, by \eqref{attractor2}
we have the following correspondence
\begin{align}\label{induction-diagram-X}
\calM_X(\nu_1) \times \calM_X(\nu_2) \xleftarrow[]{\tilde\gr} 
\calM_X^\ex(\nu_1,\nu_2)\xrightarrow[]{\tilde\ev}  \calM_X(\nu).
\end{align}
The map $\tilde\ev$ is representable and proper.
The integral isomorphism \eqref{eq-integral-identity} yields the topological product on $\bfH_X$ 
\begin{align}\label{mult}
\tilde\ev_*\tilde\gr^! \colon  \bfH_X \,\widehat\otimes\, \bfH_X \to \bfH_X.
\end{align}
This yields a topological algebra structure on $\bfH_X$ called the 3D CoHA, see \S\ref{sec:COHA}.
It is the Kinjo--Park--Safronov algebra introduced in \cite{KPS}.
We can also consider the correspondence
\begin{align}\label{induction-diagram-S}
\calM_S(\nu_1) \times \calM_S(\nu_2) \xleftarrow[]{\gr} 
\calM_S^\ex(\nu_1,\nu_2)\xrightarrow[]{\ev}  \calM_S(\nu).
\end{align}
The map $\gr$ is quasi-smooth. The map $\ev$ is representable and proper.
The purity transformation for the map $\gr$ equips $\bfH_S$ with the structure of a topological graded algebra called the 2D CoHA.
It is the Kapranov--Vasserot algebra introduced in \cite{KV19}.

As explained in \S\ref{sec:relative-COHA}, 
the 2D multiplication on $\bfH_S$ is identified with the 3D multiplication on $\bfH_X$ under the dimensional reduction isomorphism.
We use the 2D product in order to use the results of \cite{MMSV}.
We also need the 3D product in the proof of the Toda conjecture.

\subsection{The semistable CoHA of the surface $S$} 
Fix a reduced polynomial $p=x^d/d!+\ldots$ and define
 $\Gamma_p$ and $\Coh^{\ss}(p)$ as in \S \ref{Abelian_subcats_sec}.
Since the category $\Coh^{\ss}(p)$ is Abelian, 
stable under extensions, and 2CY, the correspondence \eqref{induction-diagram-S} restricts to a correspondence
on the open substack of semistable sheaves.
This yields an  algebra structure on the sums $\bfH_S^{\ss}(p)$ and $\bfH_S(p)$ given by
$$\bfH_S^{\ss}(p)= \bigoplus_{\nu\in\Gamma_p} \bfH_S^\ss(\nu)
,\quad
\bfH_S(p) = \bigoplus_{\nu\in\Gamma_p} \bfH_S(\nu)$$
where we write
$\bfH_S^\ss(\nu)= \H_{*+\vd_\nu}^\BM(\calM_S^\ss(\nu),\Q).$
The restriction 
$\bfH_S(p) \rightarrow  \bfH_S^\ss(p)$
is a surjective topological graded algebra morphism:
it is a morphism of algebras because of the open base change property, and
it is surjective by Corollary \ref{cor:Kirwan-surj} which gives a section of the restriction morphism and 
realizes $\bfH_S^{\ss}(p)$, as a direct summand of $\bfH_S(p)$ as a graded vector space.

\medskip

\subsection{The CoHA of zero-dimensional sheaves}\label{zero_dim_2d_sec}
We review some results from \cite{MMSV} on punctual Hecke modifications.
Let $\bfH_S^0$ and $\bfH_X^0$ be the graded subalgebras of $\bfH_S$ and $\bfH_X$ of sheaves with zero-dimensional support.
Let $\bfH_S^{0,+}$ be the augmentation ideal of $\bfH_S^0$. We have
$$\bfH_S^0= \bigoplus_{m\in\N} \bfH_S(m\delta)
,\quad
\bfH_S^{0,+}= \bigoplus_{m>0} \bfH_S(m\delta).$$
The algebra $\bfH_S^0$ is $\Z^2$-graded with $\bfH_{S,n}(m\delta)$ in degree $(m,n)$.
We define the horizontal subalgebra to be 
$$\bfA_S= \bfH_{S,0}^0= \bigoplus_{m\in\N} \bfH_{S,0}(m\delta).$$
Let $z=c_1(\calO(1))$ be the degree $2$ generator of $\H^*(\BB\G_m)$.
Since
$\calM_S(\delta)^\cl \cong S \times \BB\G_m$,
we have 
$$\bfH_S(\delta)=\H_{2-*}^\BM(S,\Q)[z].$$
The algebra $\bfH_S^0$ is computed in \cite{MMSV} for any smooth projective surface $S$. 
If $S$ is symplectic, the algebra  $\bfH_S^0$ is a deformed,  $\H^*(S,\Q)$-colored version of the $W_{1+\infty}$-algebra 
of differential operators on the plane.

\begin{definition} Let $\mathfrak{w}^+_S$ be the Lie algebra linearly spanned by elements 
$Z^mD^n(\beta)$
of degree $(m,2n-2+\deg\beta)$
with $\beta\in \H^*(S,\Q)$, $m\in\N\smallsetminus\{0\}$ and $n\in\N$, 
modulo the relations
\begin{equation*}
Z^mD^n(t\beta)=tZ^mD^n(\beta), \quad \quad \quad Z^mD^n(\beta+\beta')=Z^mD^n(\beta) + Z^mD^n(\beta')\\
\end{equation*}
for $t \in \Q$ and $\beta,\beta' \in \H^*(S,\Q)$.
Let us pick a class $q\in \H^2(S,\Z)$ such that $q^2=s_2(S)$; if $S$ is a K3 surface this amounts to picking a $q$ in the K3 lattice $U^{\oplus 3}\oplus E_8(-1)^{\oplus 2}$ with $q\cdot q=-24$, while if $S$ is Abelian $s_2(S)=0$ and we set $q=0$.
We  equip $\mathfrak{w}_S^{+}$ with the Lie bracket satisfying
$$[Z^mD^n(\beta),Z^{m'}D^{n'}(\beta')]=Z^{m+m'}\frac{(D+m'q)^nD^{n'}-D^n(D+mq)^{n'}}{q}(\beta \beta').$$ 
\end{definition}

\begin{theorem}[\cite{MMSV}] \label{thm:Hecke is gen in degree one} \hfill
\begin{enumerate}[label=$\mathrm{(\alph*)}$,leftmargin=8mm,itemsep=2mm]
\item
There is an algebra isomorphism
$\U(\mathfrak{w}^+_S) \to \bfH_S^0$
such that 
$$
ZD^n(\beta)\mapsto \beta z^n\cap [\calM_S(\delta)^\cl]
,\quad
n \in\N
\,,\,
\beta \in \H^*(S,\Q).$$
\item 
The Lie algebra $\mathfrak{w}_S^+$ is generated  by the subspace $\mathfrak{w}_S^+(\delta)$ spanned by the elements $ZD^n(\beta)$ for $n\in\N$ and $\beta\in \H^*(S,\Q)$.
\end{enumerate}
\qed
\end{theorem}

From now on, we will denote by $D_{m,n}(\beta)$ the element $Z^mD^n(\beta)$. Note that
$$[D_{m,n}(\beta),D_{m',n'}(\beta')]=(m'n-mn')D_{m+m',n+n'-1}(\beta\beta')$$
for any $m,m' \geqslant 1$ and $n,n'\leqslant 1$.
Recall that  $\bfA_S$ is the subalgebra of $\bfH_S^0$ of vertical (i.e. cohomological) degree zero.
Under the isomorphism $\U(\mathfrak{w}_S^+) \cong \bfH^0_{S}$ in 
Theorem \ref{thm:Hecke is gen in degree one} we have
$$D_{m,1}(1)\,,\,D_{m,0}(\beta)\in\bfA_S,\quad\forall m\geqslant 1,\quad\forall\beta \in \H^2(S,\Q).$$


\medskip

\subsection{The Hecke operators supported on a curve}\label{sec:Hecke-curve}
Let $C \subset S$ be a smooth, connected, projective curve of class $\beta\in\H_2(S,\Z)$. 
For any $m \geqslant 1$ let $\calM_{S,C}(m\delta)$ be the stack parametrizing length $m$ sheaves set-theoretically supported on $C$.
It is pure dimensional of dimension $0$. Its irreducible components are parametrized by the set 
$\calP(m)$ of partitions of $m$. Let $\calM_\lambda$ be the closure of the (irreducible) substack $\calM_\lambda^\circ$ 
parametrizing sheaves whose support cycle is of the form $\sum_i \lambda_i x_i$ for some distinct points $x_i \in C$.

Let $\bfA_\beta \subset \bfA_S$ 
be the commutative subalgebra generated by $\{D_{m,0}(\beta)\,;\, m \geqslant 1\}$. 
Set $\bfA_\beta(m\delta)=\bfA_\beta\cap\bfH_S(m\delta)$.
For each invertible sheaf $\calL$ on $S$, the tensor product with $\calL$
is an automorphism $\otimes \calL$ of the stack $\calM_S$, which induces
 an automorphism $\omega_\calL$ of the algebra $\bfH_S$
which preserves the subalgebras $\bfH_S^0$ and $\bfA_S$.

\begin{proposition}\label{prop:formula}Let $\calL \in \Pic(S)$ with Chern class $c_1(\calL)=\beta$. 
\begin{enumerate}[label=$\mathrm{(\alph*)}$,leftmargin=8mm,itemsep=2mm]
\item
The vector space $\bfA_\beta(m\delta)$ 
is linearly spanned by the fundamental classes $[\calM_\lambda]$ with $\lambda\in\calP(m)$. 
\item
The automorphism $\omega_\calL$ of $\bfA_S$ restricts to the identity on $\bfA_\beta$.
\item
We have $\omega_\calL(D_{1,1}(1))=D_{1,1}(1) + D_{1,0}(\beta)$.
\end{enumerate}
\end{proposition}

\begin{proof}
Let $i_C\colon  \calM_{S,C}(m\delta) \to \calM_S(m\delta)$ be  the obvious closed immersion.
For (a) we must prove that
$\bfA_\beta(m\delta)=\Im(i_C)_*$.
This is \cite[prop.~IV.9.12]{DPSSV}, see also \cite[\S7.4, prop.~7.9]{MMSV}.
For (b), note that $\otimes \calL$ preserves the support, hence it stabilizes the closed substacks $\calM_\lambda$ and acts trivially on the 
fundamental classes $[\calM_\lambda]$.
For (c), recall that $z=c_1(\calO(1))$ where $\calO(1)$ is the universal sheaf on $\BB\G_m$.
Let $f_\calL\colon  S \to \BB\G_m$ be the morphism associated to $\calL$.
The automorphism $\otimes \calL$ of $\calM_S(\delta)^\cl=S \times \BB\G_m$ is 
the composition of morphisms
$$S \times \BB\G_m \xrightarrow{(\id_S \times f_\calL) \times \id_{\BB\G_m}} S \times \BB\G_m \times \BB\G_m \xrightarrow{\id_S \times \otimes} S \times \BB\G_m$$
where $\otimes\colon  \BB\G_m \times \BB\G_m \to \BB\G_m$ is the tensor product. Thus, we have 
\begin{align*}
\omega_\calL(D_{1,1}(1))=\omega_\calL(z\cap[S \times \BB\G_m])
=c_1(\calO(1) \otimes \calL) \cap [S \times \BB\G_m]
=(z+\beta) \cap [S \times \BB\G_m]
=D_{1,1}(1) + D_{1,0}(\beta).
\end{align*}
\end{proof}

We will need the following algebraic lemma.

\begin{proposition}\label{prop:algebraic trick}
Let $V$ be a $\Z$-graded $\bfA_S$-representation 
with finite dimensional homogeneous components. 
Let $\beta \in \H^2(S,\Q)$ be a class such that 
\begin{enumerate}[label=$\mathrm{(\alph*)}$,leftmargin=8mm,itemsep=2mm]
\item there is a positive degree element $\omega \in \bfA_\beta$ 
which acts as an automorphism $\sigma$ of $V$,
\item $\sigma D_{1,1}(1)\sigma^{-1} \in D_{1,1}(1) + \Q D_{1,0}(\beta)$ as operators on $V$.
\end{enumerate}
Then $D_{1,0}(\beta)$ acts as a degree one automorphism of $V$.	
\end{proposition}

\begin{proof}
The operator $D_{1,0}(\beta)$ acts as an automorphism of $V$ if and only if it acts as an automorphism of $V\otimes_\Q \C$.  So in the proof we replace $V$ with its complexification.
Set $V=\bigoplus_{n \in \Z} V_n$. 
Let $N$ be the degree of $\omega$.
Let $D \in \bfA_\beta$ be homogeneous of degree $d>0$.
The degree zero operator $D^N\sigma^{-d}$ gives the generalized 
eigenspace decomposition 
$$V=\bigoplus_{\ell\in\C} V(D\,,\,\ell), \quad 
V(D\,,\,\ell)\coloneqq \text{Ker}\left( (D^N\sigma^{-d}-\ell)^\infty\right).$$
Since the algebra $\bfA_\beta$ is commutative, 
there is a simultaneous decomposition $V= \bigoplus_{\uuell} V(\uuell)$
parametrized by sequences $\uuell=(\ell_m)$ in $\C^\N$, where we define
$$V(\uuell)= \bigcap_{m\geqslant 1} V(D_{m,0}(\beta)\,,\, \ell_m). $$
The vector space $V(\uuell)$ is $\Z$-graded, with
$$V(\uuell)=\bigoplus_n V(\uuell)_n
,\quad
V(\uuell)_n=V(\uuell) \cap V_n.$$ 
We claim that $V(\uuell)$ is stable under the action of $D_{1,1}(1)$.
Indeed, since
$[D_{1,1}(1)\,,\, D_{m,0}(\beta)] \in \bfA_\beta$, we have $$[D_{1,1}(1)\,,\,\bfA_\beta] \subset \bfA_\beta,\quad 
[D_{1,1}(1)\,,\, \sigma^{\pm 1}] \in \sigma^{-1}\bfA_\beta .$$
The second inclusion comes from assumption (b) and the fact that $\sigma \in \bfA_\beta$. The claim now follows from the inclusion, for each integer $M>0$ and scalar $l$
$$[D_{1,1}(1)\,,\, (D_{d,0}(\beta)^N\sigma^{-d}-\ell)^M] \subset \bfA_\beta[\sigma^{-1}](D_{d,0}(\beta)^N\sigma^{-d}-\ell)^{M-1}.$$ 

Next, the operator $D_{1,0}(\beta)$ restricts to an isomorphism on $V(\uuell)$ if $\ell_1\neq 0$.
Hence it is enough to prove that $V(\uuell)=\{0\}$ if $\ell_1= 0$. 
Fix a sequence $\uuell$ such that $\ell_1=0$. 
The following degree zero operators commute with each other
\begin{equation}\label{eq:proof-alg-lemma1}
D_{\uum,0}(\beta)\sigma^{-1}=D_{m_1,0}(\beta) D_{m_2,0}(\beta)\cdots D_{m_s,0}(\beta)\sigma^{-1}
,\quad
\sum_im_i=N.
\end{equation}
The restriction of $D_{\uum,0}(\beta)\sigma^{-1}$ to $V(\uuell)$ is either invertible 
if $\ell_{m_i}\neq 0$ for all $i$ or nilpotent if $\ell_{m_i}= 0$ for some $i$. 
Since $\omega$ is a linear combination of monomials $D_{\uum,0}(\beta)$ and $\sigma$ is invertible, at least one of
the operators \eqref{eq:proof-alg-lemma1}
is invertible when restricted to $V(\uuell)$. 
We claim that if this happens then $m_i=1$ for all $i$. 
To prove the claim, it is enough to show that the restriction to $V(\uuell)$ of any monomial  $D_{\uum,0}(\beta)$
with $\sum_im_i=N$ and $m_1>1$ is nilpotent. Using the relation 
$$[D_{1,1}(1),D_{m_1-1,0}(\beta)]=(m_1-1)D_{m_1,0}(\beta),$$
we compute, for any integer $k \geqslant 1$ and any degree $n$ the trace
\begin{equation}\label{eq:proof-alg-lemma2}
\begin{split}
(m_1-1)\text{Tr}_{V(\uuell)_n}\left( D_{\uum,0}(\beta)^k  \sigma^{-k}\right)
&=\text{Tr}_{V(\uuell)_n}\left( [D_{1,1}(1)\,,\,D_{m_1-1,0}(\beta)]D_{m_1,0}(\beta)^{k-1} \cdots D_{m_s,0}(\beta)^k \sigma^{-k}\right)\\
&=\text{Tr}_{V(\uuell)_n}	\left( D_{1,1}(1)D_{m_1-1,0}(\beta)D_{m_1,0}(\beta)^{k-1} \cdots D_{m_s,0}(\beta)^k \sigma^{-k}\right)\\
&\quad 
-\text{Tr}_{V(\uuell)_{n-m_1+1}}\left( D_{1,1}(1)D_{m_1-1,0}(\beta)D_{m_1,0}(\beta)^{k-1} \cdots D_{m_s,0}(\beta)^k \sigma^{-k}\right).
\end{split}
\end{equation}
Summing up \eqref{eq:proof-alg-lemma2} for $n, n-m_1+1, \ldots, n-(N-1)(m_1-1)$, we obtain
\begin{equation}\label{eq:proof-alg-lemma3}
\begin{split}
(m_1-1)\sum_{r=0}^{N-1} \text{Tr}_{V(\uuell)_{n-r(m_1-1)}}\left(D_{\uum,0}(\beta)^k  \sigma^{-k}\right)
&=\text{Tr}_{V(\uuell)_n}	\left( D_{1,1}(1)D_{m_1-1,0}(\beta)D_{m_1,0}(\beta)^{k-1} \cdots D_{m_s,0}(\beta)^k \sigma^{-k}\right)\\
 -\text{Tr}_{V(\uuell)_{n-N(m_1-1)}}&	\left( D_{1,1}(1)D_{m_1-1,0}(\beta)D_{m_1,0}(\beta)^{k-1} \cdots D_{m_s,0}(\beta)^k \sigma^{-k}\right).\\ 
\end{split}
\end{equation}
Let $\sigma D_{1,1}(1) \sigma^{-1}=D_{1,1}(1)+a\,D_{1,0}(\beta)$ with $a\in\Q$. We have
\begin{equation}
\begin{split}
&\text{Tr}_{V(\uuell)_{n-N(m_1-1)}}\left( D_{1,1}(1)D_{m_1-1,0}(\beta)D_{m_1,0}(\beta)^{k-1} \cdots D_{m_s,0}(\beta)^k \sigma^{-k}\right)\\
&\quad
=\text{Tr}_{V(\uuell)_n}\left(\sigma^{m_1-1} D_{1,1}(1)D_{m_1-1,0}(\beta)D_{m_1,0}(\beta)^{k-1} \cdots D_{m_s,0}(\beta)^k \sigma^{-k-m_1+1}\right)\\
&\quad=\text{Tr}_{V(\uuell)_n}\left(\sigma^{m_1-1} D_{1,1}(1)\sigma^{1-m_1}D_{m_1-1,0}(\beta)D_{m_1,0}(\beta)^{k-1} \cdots D_{m_s,0}(\beta)^k \sigma^{-k}\right)\\
&\quad=\text{Tr}_{V(\uuell)_n}\left( D_{1,1}(1)D_{m_1-1,0}(\beta)D_{m_1,0}(\beta)^{k-1} \cdots D_{m_s,0}(\beta)^k \sigma^{-k}\right)\\
&\quad\quad +(m_1-1)\,a\,\text{Tr}_{V(\uuell)_n}\left( D_{1,0}(\beta)D_{m_1-1,0}(\beta)D_{m_1,0}(\beta)^{k-1} \cdots D_{m_s,0}(\beta)^k \sigma^{-k}\right)\\
&\quad=\text{Tr}_{V(\uuell)_n}\left( D_{1,1}(1)D_{m_1-1,0}(\beta)D_{m_1,0}(\beta)^{k-1} \cdots D_{m_s,0}(\beta)^k \sigma^{-k}\right)\\
\end{split}
\end{equation}
because $D_{1,0}(\beta)$ is nilpotent on $V(\uuell)$ since $\ell_1=0$ and the algebra $\bfA_\beta$ is commutative. 
We deduce that for all $k\geqslant 1$ and $n\in\Z$ we have
$$\sum_{r=0}^{N-1} \text{Tr}_{V(\uuell)_{n-r(m_1-1)}}\left(D_{\uum,0}(\beta)^k  \sigma^{-k}\right)=0.$$
This implies that the operator $D_{\uum,0}(\beta)\sigma^{-1}$ is nilpotent on $V(\uuell)$ as wanted. 
Putting this all together, we deduce that $D_{1,0}(\beta)^N$ is invertible on $V$,
hence $D_{1,0}(\beta)$ is also invertible on $V$.
\end{proof}

\medskip


\medskip

\subsection{Hecke operators on pure one-dimensional sheaves} 
\label{sec:Hecke-action}
Recall the open substack 
$\calM_S^{\geqslant 1}$ of $\calM_S$ parametrizing sheaves with no zero-dimensional subsheaf. 
Let $\nu$ be a one-dimensional Mukai vector.  We set
$$\calM_S^{\geqslant 1}(\nu)=\calM_S^{\geqslant 1} \cap \calM_S(\nu).$$
Thus, the stack $\calM_S^{\geqslant 1}(\nu)$ parametrizes pure  one-dimensional sheaves with Mukai vector $\nu$.
Taking the pullback by the morphism $\kappa\colon \calM_X\to\calM_S$ we get the open substacks
$\calM_X^{\geqslant 1}(\nu)$ and $\calM_X^{\geqslant 1}$ in $\calM_X$.
We define 
$$\bfH^{\geqslant 1}_S=\bigoplus_{\nu}\bfH^{\geqslant 1}_S(\nu)
,\quad
\bfH^{\geqslant 1}_S(\nu) =\H^\BM_{-*+\vd_\nu}(\calM_S^{\geqslant 1}(\nu),\Q).$$
By Corollary \ref{cor:Kirwan-surj}, the open restriction morphism
$\res^{\geqslant 1}\colon \bfH_S \to \bfH^{\geqslant 1}_S$ is surjective
because $\calM_S^{\geqslant 1}$ is a union of HN-strata.
For each class $\alpha\in\H^2(S,\Q)$, we define
$$\bfH_S^{\geqslant 1}(\alpha+\Z\delta)=\bigoplus_{n\in \Z} \bfH_S^{\geqslant 1}(\alpha+n\delta).$$
For any element $x\in \bfH_S^0$
we may define the Hecke operator $T_x$ on $\bfH_S^{\geqslant 1}(\alpha+\Z\delta)$
to be the left translation by $x$, thanks to
the following proposition.

\begin{proposition}\label{prop:Hecke-COHA}
The multiplication in $\bfH_S$ descends to a left $\bfH_S^0$-action 
$T$ on $\bfH_S^{\geqslant 1}(\alpha+\Z\delta)$.
\end{proposition}

The proposition can be proved directly. Instead, we first prove a relative version for a later use, and then take the global sections.
Consider the following commutative diagram
\begin{align}\label{induction-X-B}
\begin{split}
\xymatrix{
&B_X(\alpha)&&\\
\calM_X(\Z\delta)\times\calM_X(\alpha+\Z\delta)\ar[ur]^-{a_X\times \rho_X} &
\ar[l]_-{\tilde\gr} \calM_X^\ex(\Z\delta,\alpha+\Z\delta) \ar@{=}[r]&
\calM_X^\ex(\Z\delta,\alpha+\Z\delta) \ar[r]^-{\tilde\ev}& 
\calM_X(\alpha+\Z\delta)\ar[ull]_-{\rho_X}\\
\calM_X(\Z\delta)\times\calM_X^{\geqslant 1}(\alpha+\Z\delta)\ar[u]^-{\id\times j}&
\ar[l]_-{\tilde\gr^\circ} \calM_X^{\ex,\circ}(\Z\delta,\alpha+\Z\delta)\ar[u]^-{j_2}&
\calM_X^{\ex,\geqslant 1}(\Z\delta,\alpha+\Z\delta)\ar[r]^-{\tilde\ev^1}\ar[l]_-{j_1}\ar[u]_-j \ar@/^1.8pc/[ll]_-{\tilde\gr^1}&
\calM_X^{\geqslant 1}(\alpha+\Z\delta)\ar[u]_-{j}
}
\end{split}
\end{align}
where $a_X$ is the map $\calM_X(\Z\delta)\to\Spec(\C)$.
The maps $\tilde\ev$ and $\tilde\ev^1$ 
are representable and proper.
The map $j$ is the obvious open immersion.
The left and right squares are Cartesian.
Replacing the threefold $X$ by the surface $S$ everywhere yields a similar diagram
\begin{align}\label{induction-S-B}
\begin{split}
\xymatrix{
&B_S(\alpha)&&\\
\calM_S(\Z\delta)\times\calM_S(\alpha+\Z\delta)\ar[ur]^-{a_S\times \rho_S} &
\ar[l]_-{\gr} \calM_S^\ex(\Z\delta,\alpha+\Z\delta) \ar@{=}[r]&
\calM_S^\ex(\Z\delta,\alpha+\Z\delta) \ar[r]^-{\ev}& 
\calM_S(\alpha+\Z\delta)\ar[ull]_-{\rho_S}\\
\calM_S(\Z\delta)\times\calM_S^{\geqslant 1}(\alpha+\Z\delta)\ar[u]^-{\id\times j}&
\ar[l]_-{\gr^\circ} \calM_S^{\ex,\circ}(\Z\delta,\alpha+\Z\delta)\ar[u]^-{j_2}&
\calM_S^{\ex,\geqslant 1}(\Z\delta,\alpha+\Z\delta)\ar[r]^-{\ev^1}\ar[l]_-{j_1}\ar[u]_-j \ar@/^1.8pc/[ll]_-{\gr^1}&
\calM_S^{\geqslant 1}(\alpha+\Z\delta)\ar[u]_-{j}
}
\end{split}
\end{align}
The maps $\ev$ and $\ev^1$ are representable and proper.
The maps $\gr$, $\gr^\circ$ and $\gr^1$ are quasi-smooth.
The left and right squares are Cartesian.
We abbreviate
$$\varphi_{\calM_X^{\geqslant 1}(\alpha+\Z\delta)}=j^!\varphi_{\calM_X(\alpha+\Z\delta)}
,\quad
\rho_X^{\geqslant 1}=\rho_X\circ j
,\quad
\rho_S^{\geqslant 1}=\rho_S\circ j.$$

\begin{proposition}\label{prop:H0-action}\hfill
\begin{enumerate}[label=$\mathrm{(\alph*)}$,leftmargin=8mm,itemsep=2mm]
\item 
There is a morphism of complexes of mixed Hodge modules over $B_X(\alpha)$
\begin{align}\label{Hecke-X}\bfH_X^0\otimes(\rho_X^{\geqslant 1})^{\pro}_*\varphi_{\calM_X^{\geqslant 1}(\alpha+\Z\delta)}\to
(\rho_X^{\geqslant 1})^{\pro}_*\varphi_{\calM_X^{\geqslant 1}(\alpha+\Z\delta)}\end{align}
which defines an $\bfH_X^0$-action $T$ by Hecke operators on the pro-complex
$(\rho_X^{\geqslant 1})^{\pro}_*\varphi_{\calM_X^{\geqslant 1}(\alpha+\Z\delta)}$.
\item
There is a morphism of complexes of mixed Hodge modules over $B_S(\alpha)$
\begin{align}\label{Hecke-S}\bfH_S^0\otimes(\rho_S^{\geqslant 1})^{\pro}_*\D\Q_{\calM_S^{\geqslant 1}(\alpha+\Z\delta)}\to
(\rho_S^{\geqslant 1})^{\pro}_*\D\Q_{\calM_S^{\geqslant 1}(\alpha+\Z\delta)}\end{align} 
which defines an $\bfH_S^0$-action $T$ by Hecke operators on the pro-complex
$(\rho_S^{\geqslant 1})^{\pro}_*\D\Q_{\calM_S^{\geqslant 1}(\alpha+\Z\delta)}$.
\item \label{item-H0-action-compati}
Under the dimensional reduction and forgetting the mixed Hodge structures, the Hecke $\bfH_X^0$-action on the complex
$\kappa_*(\rho_X^{\geqslant 1})^{\pro}_*\varphi_{\calM_X^{\geqslant 1}(\alpha+\Z\delta)}$ is identified
with the Hecke $\bfH_S^0$-action on the complex
$(\rho_S^{\geqslant 1})^{\pro}_*\D\Q_{\calM_S^{\geqslant 1}(\alpha+\Z\delta)}$.
\end{enumerate}
\end{proposition}

\begin{proof}
The unit $1\to j_*j^*$ yields the following chain of maps
\begin{align*}
\varphi_{\calM_X(\Z\delta)}\boxtimes\varphi_{\calM_X(\alpha+\Z\delta)}
&\cong \tilde\gr_*\tilde\ev^!\varphi_{\calM_X(\alpha+\Z\delta)}\\
&\to \tilde\gr_*j_*j^*\tilde\ev^!\varphi_{\calM_X(\alpha+\Z\delta)}\\
&\cong (\id\times j)_*(\tilde\gr^1)_*(\tilde\ev^1)^!\varphi_{\calM_X^{\geqslant 1}(\alpha+\Z\delta)}.
\end{align*}
Taking the adjoint, we get a morphism of complexes of mixed Hodge modules
\begin{align}\label{hyp-res-2}\varphi_{\calM_X(\Z\delta)}\boxtimes\varphi_{\calM_X^{\geqslant 1}(\alpha+\Z\delta)}\to
(\tilde\gr^1)_*(\tilde\ev^1)^!\varphi_{\calM_X^{\geqslant 1}(\alpha+\Z\delta)}.\end{align}
Taking the pushforward by the map $a_X\times\rho_X$, this morphism yields as in \S\ref{sec:relative-COHA} a morphism
\begin{equation}\label{eq:Hecke action-X}
\bfH_X^0\otimes(\rho_X^{\geqslant 1})^{\pro}_*\varphi_{\calM_X^{\geqslant 1}(\alpha+\Z\delta)}\to
(\rho_X^{\geqslant 1})^{\pro}_*\varphi_{\calM_X^{\geqslant 1}(\alpha+\Z\delta)},
\end{equation}
proving part (a). The 2D version in part (b) is similar. Indeed, we have the following chain of maps
\begin{align}\label{rel-Hecke-S}
\begin{split}
(a_S\times\rho_S^{\geqslant 1})_*^{\pro}(\id\times j)^*(\D\Q_{\calM_S(\Z\delta)}\boxtimes\D\Q_{\calM_S(\alpha+\Z\delta)})
&\to(a_S\times\rho_S^{\geqslant 1})^{\pro}_*(\id\times j)^*\gr_*\gr^*(\D\Q_{\calM_S(\Z\delta)}\boxtimes\D\Q_{\calM_S(\alpha+\Z\delta)})\\
&\cong(a_S\times\rho_S^{\geqslant 1})^{\pro}_*(\gr^\circ)_*(\gr^\circ)^*(\D\Q_{\calM_S(\Z\delta)}\boxtimes\D\Q_{\calM_S^{\geqslant 1}(\alpha+\Z\delta)})\\
&\to(a_S\times\rho_S^{\geqslant 1})^{\pro}_*(\gr^\circ)_*\D\Q_{\calM_S^{\ex,\circ}(\Z\delta,\alpha+\Z\delta)}\\
&\to(a_S\times\rho_S^{\geqslant 1})^{\pro}_*(\gr^1)_*\D\Q_{\calM_S^{\ex,\geqslant 1}(\Z\delta,\alpha+\Z\delta)}\\
&\cong(\rho_S^{\geqslant 1})^{\pro}_*(\ev^1)_*(\ev^1)^!\D\Q_{\calM_S^{\geqslant 1}(\alpha+\Z\delta)}\\
&\to(\rho_S^{\geqslant 1})^{\pro}_*\D\Q_{\calM_S^{\geqslant 1}(\alpha+\Z\delta)}
\end{split}
\end{align}
where the first and fourth morphisms are the unit $1\to\gr_*\gr^*$ and $1\to (j_1)_*(j_1)^*$, 
the third is the purity transformation \eqref{eq-purity-transform} for the quasi-smooth map $\gr^\circ$,
the last is the properness of the map $\ev^1$ and the counit $(\ev^1)_!(\ev^1)^!\to 1$, 
and the other morphisms are given by base change and the commutativity of the diagram.
It yields a morphism of complexes of mixed Hodge modules
\begin{equation}\label{eq:Hecke action-S}
(a_S\times\rho_S^{\geqslant 1})^{\pro}_*\D\Q_{\calM_S(\Z\delta)}\boxtimes\D\Q_{\calM_S^{\geqslant 1}(\alpha+\Z\delta)}
\to(\rho_S^{\geqslant 1})^{\pro}_*\D\Q_{\calM_S^{\geqslant 1}(\alpha+\Z\delta)}
\end{equation}
which defines an $\bfH_S^0$-action
$$\bfH_S^0\otimes(\rho_S^{\geqslant 1})^{\pro}_*\D\Q_{\calM_S^{\geqslant 1}(\alpha+\Z\delta)}\to
(\rho_S^{\geqslant 1})^{\pro}_*\D\Q_{\calM_S^{\geqslant 1}(\alpha+\Z\delta)}.$$ 
Part (c) is proved as Theorem~\ref{thm: comparison 2D 3D}. Indeed, the arguments there show that the diagram
\begin{equation*}
	\begin{tikzcd}
		(a_X\times\rho_X^{\geqslant 1})^{\pro}_*(\varphi_{\calM_X(\Z\delta)}\boxtimes\varphi_{\calM_X^{\geqslant 1}(\alpha+\Z\delta)}) & (\rho_X^{\geqslant 1})^{\pro}_*(\varphi_{\calM_X^{\geqslant 1}(\alpha+\Z\delta)})\\
		(a_S\times\rho_S^{\geqslant 1})^{\pro}_*\D\Q_{\calM_S(\Z\delta)}\boxtimes\D\Q_{\calM_S^{\geqslant 1}(\alpha+\Z\delta)}
		&	(\rho_S^{\geqslant 1})^{\pro}_*\D\Q_{\calM_S^{\geqslant 1}(\alpha+\Z\delta)}
		\arrow["\eqref{eq:Hecke action-X}", from=1-1, to=1-2]
		\arrow["\eqref{eq:Hecke action-S}", from=2-1, to=2-2]
		\arrow["\eqref{eq:dimensional-reduction}"', from=1-1, to=2-1]
		\arrow["\eqref{eq:dimensional-reduction}", from=1-2, to=2-2]		
	\end{tikzcd}
\end{equation*}
commutes. 
\end{proof}

\begin{remark}
	In Proposition~\ref{prop:H0-action}(\ref{item-H0-action-compati}), we forgot the mixed Hodge structure since Theorem~\ref{thm: comparison 2D 3D} is stated only at the level of constructible complexes.
	We expect Theorem~\ref{thm: comparison 2D 3D} to upgrade to mixed Hodge modules with the same proof, and hence we expect that Proposition~\ref{prop:H0-action}(\ref{item-H0-action-compati}) holds without forgetting the mixed Hodge module structures.
\end{remark}

We can now prove Proposition \ref{prop:Hecke-COHA}.

\begin{proof}[Proof of Proposition $\ref{prop:Hecke-COHA}$]
Taking the global sections, the map \eqref{Hecke-X} yields a map 
$$\bfH_X^0\otimes\bfH_X^{\geqslant 1}(\alpha+\Z\delta)\to\bfH_X^{\geqslant 1}(\alpha+\Z\delta).$$
This map defines an $\bfH_X^0$-action on $\bfH_X^{\geqslant 1}(\alpha+\Z\delta).$
Under the dimensional reduction isomorphisms
$$
\bfH_X^0\cong\bfH_S^0
,\quad
\bfH_X^{\geqslant 1}(\alpha+\Z\delta)\cong\bfH_S^{\geqslant 1}(\alpha+\Z\delta)$$
we obtain the desired action of $\bfH_S^0$ on $\bfH_S^{\geqslant 1}(\alpha+\Z\delta)$. 
\end{proof}

\medskip

\subsection{Hecke operators and twisting by invertible sheaves} 

Let $\calL$ be an invertible sheaf on $S$.
The tensor product with $\calL$
is an automorphism $\otimes \calL$ of the stack $\calM_S$, and the
pushforward by $\otimes\calL$ an automorphism $\omega_\calL$ of the algebra $\bfH_S$.
Similarly, the tensor product with $\kappa^*\calL$
is an automorphism $\otimes \calL$ of the stack $\calM_X$, and the
pushforward by $\otimes\calL$ an automorphism $\omega_\calL$ of the algebra $\bfH_X$.
These automorphisms are intertwined by dimension reduction, so we abuse notation and denote them by the same symbol $\omega_\calL$.
We denote by the same symbol $\omega_\calL$ the automorphism of $\bfH^{\geqslant 1}_S(\alpha + \Z \delta)$ defined the same way.

The goal of this section is to prove that the
 automorphism $\omega_\calL$ of $\bfH^{\geqslant 1}_S(\alpha + \Z \delta)$ can be realised as a Hecke operator of horizontal degree $m\delta$, where $m=\alpha \cdot c_1(\calL)$. 
To do this, given a smooth projective curve $C\subset S$ and an integer $m >0$, we consider the closed substack 
$\calM_C(m\delta)$ of $\calM_S(m\delta)^\cl$
parametrizing length $m$ coherent sheaves on $S$ which are scheme theoretically supported on $C$. 
The fundamental class $[\calM_C(m\delta)]$ is an element in $\bfH_S^0$. Hence, the Hecke operator 
$T_{[\calM_C(m\delta)]}$ on $\bfH_S^{\geqslant 1}(\alpha+\Z\delta)$ is well defined. 

\begin{proposition}\label{prop:PicardisHecke}
Let $C\subset S$ be a smooth connected curve and $\Sigma$ an effective one-cycle on $S$. 
Assume that $C$ is not a component of any element of the linear system $|\Sigma|$. 
Set $\alpha=[\Sigma]$, $\beta=[C]$ and $m= \alpha\cdot\beta$. Assume that $m>0$. 
\begin{enumerate}[label=$\mathrm{(\alph*)}$,leftmargin=8mm,itemsep=2mm]
\item 
The automorphism $\omega_{\calO(C)}$ of $\bfH^{\geqslant 1}_S(\alpha + \Z \delta)$ is of degree $m\delta$.
\item
The Hecke operator 
$T_{[\calM_C(m\delta)]}$ on $\bfH_S^{\geqslant 1}(\alpha+\Z\delta)$ is invertible of degree $m\delta$.
\item
We have $T_{[\calM_C(m\delta)]}=\omega_{\calO(C)}$.
\end{enumerate}
\end{proposition}

The proof hinges on the following lemmas.

\begin{lemma}\label{lem:keylemma}
Let $C\subset S$ be a smooth curve and let $\Sigma$ be an effective one-cycle on $S$. 
Assume that $C$ is not a component of any element of the linear system 
$|\Sigma|$. Set $\alpha=[\Sigma]$, $\beta=[C]$ and $m= \alpha\cdot\beta$. Assume that $m>0$. 
If $\calF \in \calM_S^{\geqslant 1}(\alpha+n\delta)$, 
then the only subsheaf $\calG$ such that $\calF/\calG \in \calM_C(m\delta)$ is $\calF(-C)$.
\end{lemma}

\begin{proof}
Let  $i_C\colon  C \to S$ be the obvious closed immersion.
Let us first consider the exact sequence
$$0 \longrightarrow \mathrm{Tor}^1(\calO_C,\calF)   \longrightarrow \calF \otimes
 \calO(-C)  \longrightarrow \calF  \longrightarrow \calF \otimes \calO_C  \longrightarrow 0.$$
We have $ \mathrm{Tor}^1(\calO_C,\calF) =0$ since $\calF$ is pure one-dimensional and $C$ is not a component of the support of $\calF$. 
Thus
$\calF \overset{L}{\otimes} \calO_C=\calF \otimes \calO_C.$ 
We deduce that 
$$\ch(\calF \otimes \calO_C)=\ch(\calF) -\ch(\calF \otimes \calO(-C))=\ch(\calF)(1-e^{-[C]})=\nu \cdot [C] =m\delta.$$
Hence $\calF\otimes \calO_C$ is a zero-dimensional sheaf scheme-theoretically supported on $C$ of length $m$. 
It thus only remains to show that $\calF$ has a unique quotient of length $m$ which is scheme-theoretically supported on $C$. 
For any $\calG \in \mathsf{Coh}(C)$ we have 
$$\Hom_S(\calF,(i_C)_*\calG)=\Hom_C((i_C)^*\calF,\calG)=\Hom_C(\calF\otimes \calO_C,\calG),$$
hence any morphism $f\colon \calF \to \calG$ factors through the map $\calF \to \calF \otimes \calO_C$. 
Hence, we have $\calF(-C) \subseteq\ker(f)$, and therefore $\ker(f)=\calF(-C)$ for degree reasons.
\end{proof}

For a future use, we will prove a version of Proposition \ref{prop:PicardisHecke} which is relative over the Chow variety $B_S$ and then take the global sections.
Let us first quote the following easy lemma which follows from the fact that, for any invertible sheaf $\calL$ over $S$,
the automorphism $\otimes\calL$ of $\calM_S$ preserves the support of sheaves over $S$, 
hence the Chow map $\rho_S\colon \calM_S\to B_S$ is $\otimes\calL$-invariant.

\begin{lemma}
Let $\calL$ be an invertible sheaf on $S$. The automorphism $\otimes \calL$ of the stack $\calM_S$ preserves the open substack $\calM_S^{\geqslant 1}$.
The dualizing complexes $\D\Q_{\calM_S}$ and $\D\Q_{\calM_S^{\geqslant 1}}$ are $\otimes\calL$-equivariant. 
The pushforward by $\otimes \calL$ yields an automorphism $\omega_\calL$ of the
pro-complex of mixed Hodge modules $(\rho_S^{\geqslant 1})^{\pro}_*\D\Q_{\calM_S^{\geqslant 1}(\alpha+\Z\delta)}$
over $B_S(\alpha)$.
\end{lemma}

Proposition~$\ref{prop:PicardisHecke}$ follows from the next lemma by taking global sections.

\begin{lemma}\label{lem:PicardisHecke}
Let $C\subset S$ be a smooth connected curve and $\Sigma$ an effective one-cycle on $S$. 
Assume that $C$ is not a component of any element of the linear system $|\Sigma|$. 
Set $\alpha=[\Sigma]$, $\beta=[C]$ and $m= \alpha\cdot\beta$. Assume that $m>0$. 
\begin{enumerate}[label=$\mathrm{(\alph*)}$,leftmargin=8mm,itemsep=2mm]
\item 
The automorphism $\omega_{\calO(C)}$ of $(\rho_S^{\geqslant 1})^{\pro}_*\D\Q_{\calM_S^{\geqslant 1}(\alpha+\Z\delta)}$ is of degree $m\delta$.
\item
The Hecke operator 
$T_{[\calM_C(m\delta)]}$ on $(\rho_S^{\geqslant 1})^{\pro}_*\D\Q_{\calM_S^{\geqslant 1}(\alpha+\Z\delta)}$ is invertible.
\item
We have $T_{[\calM_C(m\delta)]}=\omega_{\calO(C)}$.
\end{enumerate}
\end{lemma}

\begin{proof}[Proof of Lemma~$\ref{lem:PicardisHecke}$]
It is enough to prove part (c).
Let  $i_C\colon  \calM_C(\Z\delta) \to \calM_S(\Z\delta)^\cl$ be the obvious closed immersion. 
To avoid cumbersome notation we omit the classical truncation functors in this proof, hoping it will not create too much confusion.
The commutative diagram \eqref{induction-S-B} upgrades to the following commutative diagram
\begin{align}\label{induction-S-B2}
\begin{split}
\xymatrix{
&B_S(\alpha)&&\\
\calM_S(\Z\delta)\times\calM_S(\alpha+\Z\delta)\ar[ur]^-{a_S\times \rho_S} &
\ar[l]_-{\gr} \calM_S^\ex(\Z\delta,\alpha+\Z\delta) \ar@{=}[r]&
\calM_S^\ex(\Z\delta,\alpha+\Z\delta) \ar[r]^-{\ev}& 
\calM_S(\alpha+\Z\delta)\ar[ull]_-{\rho_S}\\
\calM_S(\Z\delta)\times\calM_S^{\geqslant 1}(\alpha+\Z\delta)\ar[u]^-{\id\times j}&
\ar[l]_-{\gr^\circ} \calM_S^{\ex,\circ}(\Z\delta,\alpha+\Z\delta)\ar[u]^-{j_2}&
\calM_S^{\ex,\geqslant 1}(\Z\delta,\alpha+\Z\delta)\ar[r]^-{\ev^1}\ar[l]_-{j_1}\ar[u]_-j&
\calM_S^{\geqslant 1}(\alpha+\Z\delta)\ar[u]_-{j} \\
\calM_C(\Z\delta)\times\calM_S^{\geqslant 1}(\alpha+\Z\delta)\ar[u]^-{i_C\times\id}&
\calM_C^{\ex,\circ}(\Z\delta,\alpha+\Z\delta)\ar[u]\ar[l]_-{\gr_C^\circ}&
\calM_C^{\ex,\geqslant 1}(\Z\delta,\alpha+\Z\delta)\ar[l]_-{j_3}\ar@/^1.8pc/[ll]_-{h}\ar[u]\ar[r]^-g&
\calM_S^{\geqslant 1}(\alpha+\Z\delta).\ar@{=}[u]
}
\end{split}
\end{align}
The middle left and right squares are Cartesian, as well as the lower left and middle ones.
Set $a_C=a_S\circ i_C$.
We consider the following chain of maps
\begin{align}\label{omega-S}
\begin{split}
(\rho_S^{\geqslant 1})^{\pro}_*\D\Q_{\calM_S^{\geqslant 1}(\alpha+\Z\delta)}
&\to(a_C\times\rho_S^{\geqslant 1})^{\pro}_*(\D\Q_{\calM_C(\Z\delta)}\boxtimes\D\Q_{\calM_S^{\geqslant 1}(\alpha+\Z\delta)})\\
&\to(a_C\times\rho_S^{\geqslant 1})^{\pro}_*h_*h^*(\D\Q_{\calM_C(\Z\delta)}\boxtimes\D\Q_{\calM_S^{\geqslant 1}(\alpha+\Z\delta)})\\
&\to(a_C\times\rho_S^{\geqslant 1})^{\pro}_*h_*\D\Q_{\calM_C^{\ex,\geqslant 1}(\Z\delta,\alpha+\Z\delta)}\\
&\cong(\rho_S^{\geqslant 1})^{\pro}_*g_*g^!\D\Q_{\calM_S^{\geqslant 1}(\alpha+\Z\delta)}\\
&\to(\rho_S^{\geqslant 1})^{\pro}_*\D\Q_{\calM_S^{\geqslant 1}(\alpha+\Z\delta)}.
\end{split}
\end{align}
The first map is the tensor product with the fundamental class $[\calM_C(m\delta)]$, the second one the unit $1\to h_*h^*$,
the third one the purity transformation \eqref{eq-purity-transform} for the map $h$, which is quasi-smooth by base change from $\gr$ because $j_3$ is an open immersion,
the fourth one is the commutativity of the diagram, the last one the properness of the map $g$ and the counit $g_!g^!\to 1$. 
Recall that $\calM_C^{\ex,\geqslant 1}(\Z\delta,\alpha+\Z\delta)$ is the stack of filtrations 
$$0\to\calG\to\calF\to\calE\to 0$$ 
where $\calE\in\calM_C(\Z\delta)$ and $\calF,\calG\in\calM_S^{\geqslant 1}(\alpha+\Z\delta)$.
Hence, Lemma~\ref{lem:keylemma} identifies $\calM_C^{\ex,\geqslant 1}(\Z\delta,\alpha+\Z\delta)$
 with the stack parametrizing the inclusions 
$\calF(-C) \subset \calF$ with $\calF\in\calM_S^{\geqslant 1}(\alpha+\Z\delta)$.
Further, the map $g$ is identified with the morphism of stacks 
$$\calM_C^{\ex,\geqslant 1} \to \calM_S^{\geqslant 1}(\alpha+(m+n)\delta), \quad (\calF(-C) \subset \calF) \mapsto \calF$$
and the composition of the second projection and $h$ with the morphism of stacks 
$$\calM_C^{\ex,\geqslant 1} \to \calM_S^{\geqslant 1}(\alpha+n\delta), \quad (\calF(-C) \subset \calF) \mapsto \calF(-C).$$
We deduce that the morphism
$$(\rho_S^{\geqslant 1})^{\pro}_*\D\Q_{\calM_S^{\geqslant 1}(\alpha+\Z\delta)}
\to(\rho_S^{\geqslant 1})^{\pro}_*\D\Q_{\calM_S^{\geqslant 1}(\alpha+\Z\delta)}$$
given by the composition of the chain of maps \eqref{omega-S}
coincides with the automorphism $\omega_{\calO(C)}$.
We must check it coincides also with the action of the Hecke operator $T_{[\calM_C(m\delta)]}$.
This Hecke operator is the specialization of the map \eqref{Hecke-S} at the class
$(i_C)_*[\calM_C(m\delta)]$ in $\bfH_S^0$. 
It is the composition of the following chain of maps
where the first map is the tensor product with the fundamental class $[\calM_C(m\delta)]$ and all other maps are as in  \eqref{rel-Hecke-S}
and use the purity transformation \eqref{eq-purity-transform} for the quasi-smooth map $\gr^\circ$
\begin{align}\label{rel-Hecke-S2}
\begin{split}
(\rho_S^{\geqslant 1})^{\pro}_*\D\Q_{\calM_S^{\geqslant 1}(\alpha+\Z\delta)}
&\to(a_S\times\rho_S^{\geqslant 1})^{\pro}_*(\D\Q_{\calM_S(\Z\delta)}\boxtimes\D\Q_{\calM_S^{\geqslant 1}(\alpha+\Z\delta)})\\
&\to(a_S\times\rho_S^{\geqslant 1})^{\pro}_*(\id\times j)^*\gr_*\gr^*(\D\Q_{\calM_S(\Z\delta)}\boxtimes\D\Q_{\calM_S(\alpha+\Z\delta)})\\
&\cong(a_S\times\rho_S^{\geqslant 1})^{\pro}_*(\gr^\circ)_*(\gr^\circ)^*(\D\Q_{\calM_S(\Z\delta)}\boxtimes\D\Q_{\calM_S^{\geqslant 1}(\alpha+\Z\delta)})\\
&\to(a_S\times\rho_S^{\geqslant 1})^{\pro}_*(\gr^\circ)_*\D\Q_{\calM_S^{\ex,\circ}(\Z\delta,\alpha+\Z\delta)}\\
&\to(a_S\times\rho_S^{\geqslant 1})^{\pro}_*(\gr^1)_*\D\Q_{\calM_S^{\ex,\geqslant 1}(\Z\delta,\alpha+\Z\delta)}\\
&\cong(\rho_S^{\geqslant 1})^{\pro}_*(\ev^1)_*(\ev^1)^!\D\Q_{\calM_S^{\geqslant 1}(\alpha+\Z\delta)}\\
&\to(\rho_S^{\geqslant 1})^{\pro}_*\D\Q_{\calM_S^{\geqslant 1}(\alpha+\Z\delta)}.
\end{split}
\end{align}
Here, the map $\gr^1$ is defined as in the diagram   \eqref{induction-S-B}.
Hence, we must compare \eqref{omega-S} with \eqref{rel-Hecke-S2}.
This follows from a diagram chase which reduces to comparing the purity transformations for the quasi-smooth maps $\gr_C^\circ$ and $\gr^\circ$ 
under the pushforward by the lower vertical maps in the diagram \eqref{induction-S-B2}.
To do this, we use the following associativity and functoriality properties of the purity transformation:
given quasi-smooth maps
$\xymatrix{X\ar[r]^-a&Y\ar[r]^-b&Z}$ 
the composed map $ba$ is also quasi-smooth and the following 
functoriality diagram commutes 
$$\xymatrix{a^*b^*\ar[r]\ar@{=}[d]&a^*b^!\ar[r]&a^!b^!\ar@{=}[d]\\(ba)^*\ar[rr]&&(ba)^!}$$
where the horizontal maps are purity transformations, and
given a Cartesian square
$$\xymatrix{X\ar[r]^a\ar[d]_-f&Y\ar[d]_-g\\W\ar[r]^-b&Z}$$
such that $a,$ $b$ are quasi-smooth and $f,$ $g$ are proper the following diagram commutes
$$\xymatrix{
b^*g_*\ar[r]^-1\ar[d]_-2&f_*a^*\ar[d]_-3\\
b^!g_*\ar[r]^-4&f_*a^!}$$
where the maps 2 and 3 are the purity transformations for $b$ and $a$, while 1 and 4 are the base changes.
See \S\ref{sssec-purity-transform} 
and \cite{K19} for details.

\end{proof}

\medskip

\subsection{The Hecke action on the BPS sheaf} \label{sec:Hecke-BPS}
In this section, for each class $\alpha\in \H_2(S,\Q)$, we define actions of the algebra $\bfA_S$ on the complex
of mixed Hodge modules
$$\kappa_*(\pi_X)_*\cBPS_{M_X(\alpha+\Z \delta)}\otimes \H^*(\BGM,\Q)_\vir$$ 
using relative Hecke correspondences over the Chow varieties $B_X$ and $B_S$.
These actions preserve the perverse filtration.
Let $i$ be the closed immersion $\Lambda_\snil^{\geqslant 1}\to\calM_X^{\geqslant 1}$.

\begin{proposition}\label{prop:H0-action}\hfill
\begin{enumerate}[label=$\mathrm{(\alph*)}$,leftmargin=8mm,itemsep=2mm]
\item 
There is a commutative diagram
\begin{align}\label{Hecke-BPS}
\begin{split}
\xymatrix{
\bfH_X^0\otimes(\rho_X^{\geqslant 1})^{\pro}_*i_!i^!\varphi_{\calM_X^{\geqslant 1}(\alpha+\Z\delta)}\ar[r]\ar[d]&
(\rho_X^{\geqslant 1})^{\pro}_*i_!i^!\varphi_{\calM_X^{\geqslant 1}(\alpha+\Z\delta)}\ar[d]\\
\bfH_X^0\otimes(\rho_X^{\geqslant 1})^{\pro}_*\varphi_{\calM_X^{\geqslant 1}(\alpha+\Z\delta)}\ar[r]&
(\rho_X^{\geqslant 1})^{\pro}_*\varphi_{\calM_X^{\geqslant 1}(\alpha+\Z\delta).}}
\end{split}
\end{align}
\item
The algebra $\bfH_X^0$ acts on the semisimple complex of mixed Hodge modules
$$(\pi_X)_*\cBPS_{M_X(\alpha+\Z\delta)}\otimes \H^*(\BGM,\Q)_\vir.$$
\item
The algebra $\bfH_S^0$ acts on the semisimple complex of mixed Hodge modules
$$(\pi_S)_*\cBPS_{M_S(\alpha+\Z\delta)}\otimes \H^*(\BGM,\Q).$$
The action of the subalgebra $\bfA_S$ preserves the perverse filtration.
\end{enumerate}
\end{proposition}

\begin{proof}
Let us first concentrate on part (a). 
We first define the upper arrow in the diagram \eqref{Hecke-BPS}. 
The lower one is already defined in \eqref{Hecke-X}.
To do this, we first claim that we have the following commutative diagram with Cartesian squares
\begin{align*}
\begin{split}
\xymatrix{
\calM_X(\Z\delta)_\red\times\calM_X^{\geqslant 1}(\alpha+\Z\delta)_\red&
\ar[l]_-{\tilde\gr^1} \calM_X^{\ex,\geqslant 1}(\Z\delta,\alpha+\Z\delta)_\red
\ar[r]^-{\tilde\ev^1}& \calM_X^{\geqslant 1}(\alpha+\Z\delta)_\red\\
\calM_X(\Z\delta)_\red\times\Lambda_\snil^{\geqslant 1}(\alpha+\Z\delta)
\ar@{^{(}->}[u]^-{\id\times i}&\ar[l]_-{\gr_\snil^1} 
\Lambda_\snil^{\ex,\geqslant 1}(\Z\delta,\alpha+\Z\delta)
\ar[r]^-{\ev_\snil^1}\ar@{^{(}->}[u]^-{i}& 
\Lambda_\snil^{\geqslant 1}(\alpha+\Z\delta).\ar@{^{(}->}[u]_-{i}
}
\end{split}
\end{align*}
Indeed, we first define $\Lambda_\snil^{\ex,\geqslant 1}(\Z\delta,\alpha+\Z\delta)$ such that the right square is Cartesian.
Then, we claim that the left square is also Cartesian. Indeed, if $\calF$ is a pure one-dimensional sheaf on $X$ which fits in an exact sequence of sheaves
over $X$
$$0\to\calG\to\calF\to\calE\to 0$$
with $\calG$ pure one-dimensional with semi-nilpotent support and $\calE$ zero-dimensional, then $\calF$ also has a semi-nilpotent support: 
otherwise, it would break down as a direct sum of sheaves $\calF=\calG' \oplus \calE'$ with $\calG' \supseteq \calG$ and $\calE' \subseteq \calE$, in 
contradiction with the purity of $\calF$.
Set
$$\varphi_{\Lambda^{\geqslant 1}_\snil(\alpha+\Z\delta)}=i^!\varphi_{\calM_X^{\geqslant 1}(\alpha+\Z\delta)}.$$
Applying a base change to the morphism \eqref{hyp-res-2}, we get a morphism of complexes of mixed Hodge modules
\begin{align}\label{hyp-res-3}
\varphi_{\calM_X(\Z\delta)}\boxtimes\varphi_{\Lambda_\snil^{\geqslant 1}(\alpha+\Z\delta)}
\to (\gr_\snil^1)_* (\ev_\snil^1)^! \varphi_{\Lambda_\snil^{\geqslant 1}(\alpha+\Z\delta)}.
\end{align} 
Taking the pushforward by the map $a_X\times\rho_X^{\geqslant 1}$, the properness of the map $\ev_\snil^1$ and the counit
$(\ev_\snil^1)_!(\ev_\snil^1)^!\to 1$, this morphism yields a morphism of complexes of mixed Hodge modules
\begin{align}\label{action-pure2}
\bfH_X^0\otimes (\rho_X^{\geqslant 1})^{\pro}_* 
i_!i^!\varphi_{\calM_X^{\geqslant 1}(\alpha+\Z\delta)} \to 
(\rho_X^{\geqslant 1})^{\pro}_* i_!i^!\varphi_{\calM_X^{\geqslant 1}(\alpha+\Z\delta)}
\end{align}
such that \eqref{Hecke-X} and \eqref{action-pure2} form the commutative square \eqref{Hecke-BPS}, proving part (a) of the proposition.

The map \eqref{Hecke-X} defines an $\bfH_X^0$-action on the complex of mixed Hodge modules
$(\rho_X^{\geqslant 1})^{\pro}_*\varphi_{\calM_X^{\geqslant 1}(\alpha+\Z\delta)}$.
The map \eqref{action-pure2} defines an $\bfH_X^0$-action on 
$(\rho_X^{\geqslant 1})^{\pro}_*i_!i^!\varphi_{\calM_X^{\geqslant 1}(\alpha+\Z\delta)}$.
The counit $i_!i^!\to 1$ yields an $\bfH_X^0$-module homomorphism
\begin{align}\label{morphism2}
(\rho_X^{\geqslant 1})^{\pro}_*i_!i^!\varphi_{\calM_X^{\geqslant 1}(\alpha+\Z\delta)}\to
(\rho_X^{\geqslant 1})^{\pro}_*\varphi_{\calM_X^{\geqslant 1}(\alpha+\Z\delta)}.
\end{align}
This yields an $\bfH_X^0$-action on the image of the associated graded of the map \eqref{morphism2} for the perverse filtration.  This proves part (b). 
Part (c) is similar: Applying the functor $\kappa_*$, by Proposition \ref{prop:coker}, this also yields an $\bfH_S^0$-action on
the complex of mixed Hodge modules
$$(\pi_S)_*\cBPS_{M_S(\alpha+\Z\delta)}\otimes \H^*(\BGM,\Q)_\vir\otimes\L^{-\frac12}.$$
Restricting the $\bfH_S^0$-action to the subalgebra of zero cohomological degree classes yields an $\bfA_S$-action 
which preserves the perverse filtration.
\end{proof}

\bigskip

\section{Proof of Toda's Gopakumar--Vafa conjecture}

\medskip

\subsection{Hecke operators and Toda's Gopakumar--Vafa conjecture}\label{sec:Proof Toda}
Let $\alpha\in \H_2(S,\Q)\setminus\{0\}$ be the class of an effective one-cycle $\Sigma$.
For each integer $n$,
the moduli stack $\calM_S(\alpha+n\delta)$ 
parametrizes one-dimensional coherent sheaves on $S$ with Mukai vector $\alpha+n\delta$, and the open substack
$\calM_S^\ss(\alpha+n\delta)$ parametrizes the semistable sheaves.
Taking the sum over all integers $n$, we get the stack $\calM_S(\alpha+\Z\delta)$  which is equipped with the Hilbert--Chow map from its reduction
$$\rho_S\colon \calM_S(\alpha+\Z\delta)_\red\to B_S(\alpha).$$
Taking the $(-1)$-shifted cotangent space of $\calM_S(\alpha+\Z\delta)$ we get
the moduli stack $\calM_X(\alpha+n\delta)$ 
of one-dimensional coherent sheaves $\calF$ on the CY 3-fold $X=\A^1\times S$ such that $\ch_2(\calF)=\alpha$
and $\chi(\calF)=n$.
The reduction of this stack is also equipped with the Hilbert--Chow map
$$\rho_X\colon \calM_X(\alpha+\Z\delta)_\red\to B_X(\alpha).$$ 
We consider the good quotient  
$p_X\colon \calM_X^\ss\to M_X$ and 
the BPS mixed Hodge module $\cBPS_{M_X}$ on $M_X$ given by
$$\cBPS_{M_X} = {}^{\text{p}}\calH^1((p_X)_* \varphi_{\calM_X^\ss}).$$
Following Toda \cite{T23}, we define a complex $\Psi_X(\alpha+n\delta)$ 
of mixed Hodge modules over $B_X(\alpha)$ by
$$\Psi_X(\alpha+n\delta)={}^{\text{p}}\gr((\pi_X)_*\cBPS_{M_X(\alpha+n\delta)}).$$
Note that this complex is semisimple, by purity of $\cBPS_{M_X}$ and projectivity of $\pi_X$.
We define similarly a semisimple complex $\Psi_S(\alpha+n\delta)$ 
of mixed Hodge modules over $B_S(\alpha)$ by
$$\Psi_S(\alpha+n\delta)={}^{\text{p}}\gr((\pi_S)_*\cBPS_{M_S(\alpha+n\delta)}).$$
Our goal is to prove the following result, which is a stronger form of \cite[conj.~1.2]{T23}. 

\begin{theorem} \label{thm:Toda}
The complex $\Psi_S(\alpha+n\delta)$ is independent of the integer $n$.
\end{theorem}

\begin{proof}
Since the complex $(\pi_S)_*\cBPS_{M_S(\alpha+\Z\delta)}$  on $B_S$ is semisimple, we can write
$$(\pi_S)_*\cBPS_{M_S(\alpha+\Z\delta)}\otimes \H^*(\BGM,\Q)=\bigoplus_P P\otimes V_P$$
where each $P$ is a simple mixed Hodge module on $B_S$ and $V_P$ is the multiplicity $\Z^{\oplus 2}$-graded complex of vector spaces, where the second 
grading is the cohomological degree and the first records the Euler characteristic $\chi$ of the sheaves considered.  Since $(\pi_S)_*\cBPS_{M_S(\alpha+\Z\delta)}$ is semisimple, $(\pi_S)_*\cBPS_{M_S(\alpha+n\delta)}$ is independent of $n$ if and only if $(\pi_S)_*\cBPS_{M_S(\alpha+n\delta)}\otimes \H^*(\BGM,\Q)$ is.
The graded vector space $V_P$ has finite dimensional homogeneous componenents. 
Then Toda's conjecture is a consequence of the following facts:
\begin{enumerate}[label=$\mathrm{(\alph*)}$,leftmargin=8mm,itemsep=2mm]
\item
Proposition \ref{prop:H0-action} yields an $\bfA_S$-action on each graded vector space $V_P$ 
by homogeneous operators of zero cohomological degree.
\item
Let $C\subset S$ be a smooth projective curve of class $\beta\in\H_2(S,\Z)$
which is not a component of any element of the linear 
system $|\Sigma|$ and such that the integer $m= \alpha\cdot\beta$ is positive.
Using the notation of \S\ref{sec:Hecke-curve}, the class $[\calM_C(m\delta)]$ is identified with $[\calM_{(1^m)}]$, which belongs to 
$\bfA_\beta$ by Proposition  \ref{prop:formula}.
By Lemma \ref{lem:PicardisHecke}, this element
 acts isomorphically on each graded vector space $V_P$ by an operator of degree $(m,0)$. 
\item
By Propositions  \ref{prop:formula} and \ref{prop:algebraic trick} the $D_{1,0}(\beta)$-action on $V_P$ is invertible and of degree $(1,0)$.
\end{enumerate}
\end{proof}

\medskip

\subsection{Generalization}\label{sec:Toda generalization} In this section we briefly explain how to adapt the proof of the Toda conjecture 
to cover the case of a smooth quasi-projective surface $S$ and a class
$\alpha \in \H^2(S,\Z)$ of a properly supported effective one-cycle.
We assume that there is a 
smooth compactification $\iota\colon  S\hookrightarrow \overline{S}$ with a polarization $H$, an effective one-cycle 
with class $\overline{\alpha} \in \H^2(\overline{S},\Z)$ such that $\iota^*(\overline{\alpha})=\alpha$ and that for any 1-cycle $C$ in the class $\overline{\alpha}$ with support in $S$ there exists $D \in |K_{\overline{S}}|$ such that $C \cap D=\emptyset$. This is for instance the case for smooth quasi-projective symplectic surfaces, including the cotangent bundle of a curve, and the minimal resolution of a Kleinian singularity.
We consider the CY threefold $\overline{X}=\text{Tot}(K_{\overline{S}})$, the monodromic mixed Hodge module
$\cBPS_{M_{\overline{X}}(\overline{\alpha}+n\delta)}$ over $M_{\overline{X}}(\overline{\alpha}+n\delta)$, and its pushforward 
$(\pi_{\overline{X}})_*\cBPS_{M_{\overline{X}}(\overline{\alpha}+n\delta)}$ to $B_{\overline{X}}(\overline{\alpha})$.
Let $j\colon B_X(\alpha) \to B_{\overline{X}}(\overline{\alpha})$ be the inclusion of the open
subscheme parametrizing cycles supported on $X=\text{Tot}(K_{S})$. 
Set
$$\Psi_X(\alpha+n\delta)=j^*{}^\p\gr(\pi_{\overline{X}})_*\cBPS_{M_{\overline{X}}(\overline{\alpha}+n\delta)}.$$

\begin{corollary} The isomorphism class of the semisimple complex $\Psi_X(\alpha+n\delta)$ is independent of $n$.
\end{corollary}

\begin{proof}
Let $\calM_X(\alpha+n\delta),$ $ \calM^\ss_X(\alpha+n\delta),$ $ M_X(\alpha+n\delta),$ etc., be the substacks of 
$\calM_{\overline{X}}(\overline{\alpha}+n\delta), $ $\calM^\ss_{\overline{X}}(\overline{\alpha}+n\delta),$ 
$ M_{\overline{X}}(\overline{\alpha}+n\delta),$ etc., classifying coherent sheaves supported on $X$.
Let $\rho_X\colon  \calM_X(\alpha+n\delta)_\red \to B_X(\alpha)$ be the support map. 
We will use similar notations for $S$ instead of $X$. 
The arguments in \S\ref{sec:rel Kirwan} and \S\ref{sec:Hecke-action} imply that the complex of mixed Hodge modules
${}^\p\gr((\rho_X)^{\pro}_*\varphi_{\calM_X(\alpha+\Z\delta)})$ contains 
\begin{equation}
\label{aBPS_inside}(\pi_X)_*\cBPS_{M_X(\alpha+\Z\delta)}\otimes \H^*(\BGM,\Q)_\vir
\end{equation} as a subsheaf, and that
the 3D CoHA of zero-dimensional sheaves $\bfH^0_X$ acts on $(\rho_X)_*\varphi_{\calM_X(\alpha+\Z\delta)}$. 	
We claim that this action preserves the subsheaf \eqref{aBPS_inside}.

Since this can be checked locally, we fix a closed point $C$ of $B_X(\alpha)$.
Then $\kappa(C)$ is a closed point of  $B_S(\alpha)$. 
Fix a divisor $D\in |K_S|$ such that $D \cap\kappa(C)=\emptyset$. 
The same holds for all points in a Zariski neighborhood $U'$ of $\kappa(C)$ in $B_S(\alpha)$. 
Set $U=\kappa^{-1}(U')$ and  $S^\circ=S \smallsetminus D$.
We may thus identify 
$\rho_X^{-1}(U)$ with an open subset of the derived stack of properly supported coherent sheaves on 
$X^\circ=S^\circ \times \mathbb{A}^1$. 
The definition of the BPS sheaf and its main properties both hold for 
the quasi-projective symplectic surface $S^\circ$, see \cite[\S 7.2]{D24}.
We can define as in 
\S\ref{sec:Hecke-BPS} an action of $\mathbf{H}^0_{\overline{X}}$ on 
$$(\pi_X)_*\cBPS_{M_X(\alpha+\Z\delta)}|_U\otimes \H^*(\BGM,\Q)_\vir,$$
factoring through a $\bfH^0_X$-action, proving the claim. 
	
Since we have defined an action of $\bfH^0_X$ on $(\pi_X)_*\cBPS_{M_X(\alpha+\Z\delta)}\otimes \H^*(\BGM,\Q)_\vir$, it is 
again a local problem to check that the operator corresponding by dimensional reduction to $D_{1,0}(\beta)$ is an isomorphism. By 
our arguments above and \S\ref{sec:Proof Toda}, this holds on each open subset $U \subset B_X(\alpha)$ constructed as above.
\end{proof}

\section{The coproduct and the global affinized BPS Lie algebra}
In this section we construct a topological coproduct $\Delta$ on the 3D CoHA $\bfH_X$.
Via dimensional reduction, we use it to define an affinized BPS Lie algebra for the whole stack $\calM_S$. 
We compare this Lie algebra to the BPS Lie algebras introduced in \cite{DHSM2} for stacks of semistable sheaves.

\medskip

\subsection{Completed tensor products}\label{ssec-complete-tensor}

\subsubsection{Setting}\label{sssec:setting}

Throughout \S \ref{ssec-complete-tensor}-\ref{ssec-coprod}, we work with a general $2$CY category $\calC$.
Let $\calM$ be an open substack of $\calM_{\calC}$ satisfying 
the conditions \eqref{item-M-1-Artin}-\eqref{item-M-theta-reductive}
in \S\ref{para-assumptions-M}.
We further assume that there exists a monoid $\Gamma$ and a monoid homomorphism $\calM \to \Gamma$.
For each $\gamma \in \Gamma$ and $\uugamma \in \Gamma^n$, we define $\calM_{\gamma}$, $\calM_{\uugamma}$ and $\Filt_{\uugamma}$ in a manner similar to that in \S \ref{sec:int-identity}.
We assume the following variation of the condition \eqref{item-M-theta-reductive} in \S\ref{para-assumptions-M}:
\begin{enumerate}[label=(4'), ref=(4')]
	\item For each $\uugamma = (\gamma_1, \ldots, \gamma_n)$ with $\gamma = \sum \gamma_i$, the map $\Filt_{\uugamma} \to \calM_{\gamma}$ is proper.  \label{item-Gamma-theta-reductive}
\end{enumerate}
This condition is nothing but the condition \eqref{item-M-theta-reductive} if $\Gamma = \pi_0(\calM)$. 
We also assume that the Euler form 
\[
 \calM \times \calM \to \mathbb{Z} / 2 \mathbb{Z}, \quad (E, F) \mapsto \chi(E, F) \pmod  2 
\]
descends to a map $\chi \colon \Gamma \times \Gamma \to \mathbb{Z} / 2 \mathbb{Z}$.
These conditions are satisfied when $\calM = \calM_S$ and $\Gamma$ is the lattice of Mukai vectors.

Let $\Pro(\Vect)_{\Gamma}$ denote the $\infty$-category of $\Gamma$-graded pro-dg-vector spaces over $\mathbb{Q}$.
We define a symmetric monoidal structure $\hat{\otimes}_{\Gamma}$ on $\Pro(\Vect)_{\Gamma}$ as follows.
For $V = (V_{\gamma})_{\gamma \in \Gamma}$ and $W = (W_{\gamma})_{\gamma \in \Gamma}$, the tensor product is defined as
\[
 (V \hat{\otimes}_{\Gamma} W)_{\gamma} = \prod_{\gamma_1 + \gamma_2 = \gamma} V_{\gamma_1} \hat{\otimes} W_{\gamma_2}.	
\]
The braiding isomorphism 
\begin{equation}\label{eq-braiding-sign}
	V \hat{\otimes}_{\Gamma} W \cong W \hat{\otimes}_{\Gamma} V
\end{equation}
is twisted by the sign $(-1)^{\chi(\gamma_1, \gamma_2)}$ on the $(\gamma_1, \gamma_2)$-graded piece.
By symmetry of $\chi(-,-)$, these data defines a symmetric monoidal structure.
We will also use a locally finite version of the tensor product defined as 
\[
	(V \hat{\otimes}_{\Gamma}' W)_{\gamma} = \bigoplus_{\gamma_1 + \gamma_2 = \gamma} V_{\gamma_1} \hat{\otimes} W_{\gamma_2}.	
\]
Similarly to \eqref{AMpro}, define $\calA_{\calM,\pro } \in \Pro(\Vect)_{\Gamma}$ by
\[
	\calA_{\calM,\pro, \gamma } = \RGproBM(\calM _{\gamma},\Q)[- \vdim \calM_{\gamma}].
\]
Then, arguing as in \S\ref{sec:COHA}, we obtain an associative multiplication
\begin{equation}\label{eq-Gamma-CoHA}
	\calA_{\calM,\pro } \hat{\otimes}_{\Gamma}' \calA_{\calM,\pro }  \to \calA_{\calM,\pro}.	
\end{equation}
In general, we cannot define a multiplication if we use the tensor product $\hat{\otimes}_{\Gamma}$ because of possible infinite non-convergent sums.

\subsubsection{Extending multiplications}

Since $\hat{\otimes}_{\Gamma}'$ determines a symmetric monoidal structure, 
the map \eqref{eq-Gamma-CoHA} determines a multiplication map
\begin{equation}\label{eq-square-multiplication}
  (\calA_{\calM,\pro } \hat{\otimes}_{\Gamma}' \calA_{\calM,\pro }) \hat{\otimes}_{\Gamma}' (\calA_{\calM,\pro } \hat{\otimes}_{\Gamma}' \calA_{\calM,\pro }) \to \calA_{\calM,\pro } \hat{\otimes}_{\Gamma}' \calA_{\calM,\pro }
\end{equation}
defining an associative algebra structure on $\calA_{\calM,\pro } \hat{\otimes}_{\Gamma}' \calA_{\calM,\pro }$.
For later use, we extend this multiplication map.

\begin{proposition}\label{prop:coprod prod well defined}
	The multiplication map \eqref{eq-square-multiplication} naturally extends to a map
	\begin{equation}\label{eq-complete-square-multiplication}
		(\calA_{\calM,\pro } \hat{\otimes}_{\Gamma} \calA_{\calM,\pro }) \hat{\otimes}_{\Gamma}' (\calA_{\calM,\pro } \hat{\otimes}_{\Gamma} \calA_{\calM,\pro }) \to \calA_{\calM,\pro } \hat{\otimes}_{\Gamma} \calA_{\calM,\pro }.
	\end{equation}
\end{proposition}

\begin{proof}
Take $\alpha, \beta \in \Gamma$ with $\alpha + \beta = \gamma$ and quasi-compact open substacks $\calU \subset \calM_{\alpha}$ and $\calV \subset \calM_{\beta}$.
It is enough to show that the composite 
\[
	\displaystyle (\calA_{\calM,\pro } \hat{\otimes}_{\Gamma} \calA_{\calM,\pro }) \hat{\otimes}_{\Gamma}' (\calA_{\calM,\pro } \hat{\otimes}_{\Gamma} \calA_{\calM,\pro }) \to \calA_{\calM,\pro } \hat{\otimes}_{\Gamma} \calA_{\calM,\pro } \to R\Gamma^{\BM}(\calU, \mathbb{Q}) \otimes R\Gamma^{\BM}(\calV, \mathbb{Q})[-d -d']
\]
is well-defined, 
where we set $d \coloneqq \vdim \calU$, $d' \coloneqq \vdim \calV$ and the last map is given by the restriction.
Consider the following map
\[\begin{tikzcd}
	{\left(\coprod_{\alpha_1 + \alpha_2 = \alpha} \Filt_{(\alpha_1, \alpha_2)} \right) \times \left(\coprod_{\beta_1 + \beta_2 = \beta} \Filt_{(\beta_1, \beta_2)} \right) } & {\calM_{\alpha} \times \calM_{\beta}.}
	\arrow["{\ev_{\alpha} \times \ev_{\beta}}", from=1-1, to=1-2]
\end{tikzcd}\]
By construction, it is enough to show that the intersection
\[
	\displaystyle \left( \ev_{\alpha}^{-1}(\calU) \times \ev_{\beta}^{-1}(\calV)  \right)	\cap \left( \coprod_{\substack{\alpha_1 + \alpha_2 = \alpha, \\ \beta_1 + \beta_2 = \beta, \\ \alpha_1 + \beta_1 = \gamma_1}} \Filt_{(\alpha_1, \alpha_2)} \times \Filt_{(\beta_1, \beta_2)} \right)
\]
is quasi-compact for each $\gamma_1, \gamma_2 \in \Gamma$ with $\gamma_1 + \gamma_2 = \gamma$.
Consider the following commutative diagram:
\[\begin{tikzcd}
	\begin{array}{c}
		\displaystyle \left( \coprod_{\substack{\alpha_1 + \alpha_2 = \alpha, \\ \beta_1 + \beta_2 = \beta, \\ \alpha_1 + \beta_1 = \gamma_1}} \Filt_{(\alpha_1, \alpha_2)} \times \Filt_{(\beta_1, \beta_2)} \right) \end{array} & {\calM_{\alpha} \times \calM_{\beta}} \\
	{ \Filt_{(\gamma_1, \gamma_2)}} & {\calM_{\gamma}.}
	\arrow["{\ev_{\alpha} \times \ev_{\beta}}", from=1-1, to=1-2]
	\arrow["\oplus"', from=1-1, to=2-1]
	\arrow["\oplus", from=1-2, to=2-2]
	\arrow["{\ev_{(\gamma_1, \gamma_2)}}"', from=2-1, to=2-2]
\end{tikzcd}\]
By assumption \eqref{item-Gamma-theta-reductive}, the bottom horizontal map is quasi-compact.
Since we assume that $\calM$ has affine diagonal, it is quasi-separated, and thus so is the rightmost $\oplus$.
Also, by the argument in \cite[\S~10.2.3]{Bu25b}, the leftmost $\oplus$ is quasi-compact. 
So $\oplus \circ (\ev_{\alpha} \times \ev_{\beta})$ is quasi-compact, with $\oplus$ quasi-separated, and so $\ev_{\alpha} \times \ev_{\beta}$ is quasi-compact as desired.
\end{proof}

\medskip

\subsection{Construction of the coproduct}\label{ssec-coprod}

\smallskip

\subsubsection{The factorization coalgebra associated with $\calM$}
Our first goal is to construct a coproduct on the Borel--Moore homology of $\widetilde\calM$.
The key ingredient is a locally constant factorization coalgebra on the complex plane $\C$ whose space of global sections 
is the Borel--Moore homology of $\widetilde\calM$.

Recall that an analytic stack is a stack of $\infty$-groupoids on the category of (possibly singular) Stein
analytic spaces over $\C$, equipped with the Grothendieck topology consisting of open 
covers in the usual sense. Analytic stacks form a 2-category.
For every scheme $X$ of finite type over $\C$ we have a canonical analytic space $X^\an$ called
the analytification of $X$. 
Similarly, for any Artin stack $\calX$ over $\C$ we have a canonical analytic stack $\calX^\an$
called the
analytification, defined as the right Kan extension of $\calX$ from the category of affine schemes
of finite type to the category of Stein analytic spaces, see \cite[\S 4]{P19} or \cite[\S 3]{HP18}. 
The analytification yields a fully faithful functor
$$\mathsf{D_c^+}(\calX,\Q)\to \mathsf{D_c^+}(\calX^\an,\Q),\quad\calE\mapsto\calE^\an$$
between the derived categories of algebraically constructible sheaves,
which commutes with direct images and tensor products.
In particular, it yields an isomorphism
$$\H^*(\calX,\calE)\cong \H^*(\calX^\an,\calE^\an).$$
In order to unburden the notation, we will abbreviate $\calX=\calX^\an$ and $\calE=\calE^\an$.

Now, we return to the stacks $\calM$ and $\widetilde\calM$.
A point in $\widetilde\calM$ is  a pair of a point $[E] \in \calM$ together with an endomorphism $\phi$ of $E$,
see \S\ref{sec:nilp-micro}. For an open subset $U \subset \C$ with respect to the analytic topology, let 
$\widetilde\calM_{U} \subset \widetilde\calM$
be the analytic substack consisting of pairs $(E, \phi)$ such that all eigenvalues of $\phi$ are contained in $U$
(recall that we abbreviate $\widetilde\calM=\widetilde\calM^\an$).

\begin{lemma}
	The analytic stack $\widetilde\calM_{U}$ is open in $\widetilde\calM$.
\end{lemma}

\begin{proof}
	Let $T = \Spec R$ be an affine scheme, $f \colon T \to \widetilde\calM$ be a morphism corresponding to an object $E_f \in \calC_{R}$ 
	and an endomorphism $\phi_f \colon E_{f} \to E_{f}$.
	It is enough to show that $f^{- 1}(\widetilde\calM_{U})$ is an analytic open subset of $T$.
	Consider the endomorphism module $M_f = R \Hom (E_f, E_f)$ defined as a perfect complex over $R$.
	The precomposition by $\phi_f$ defines an endomorphism
	$\phi_{f}^* \colon  M_f \to M_f$.
	For each point $x \in T$, the set of eigenvalues for $E_f |_{x}$ coincides with the set of eigenvalues for $\H^0 (\phi_f^* |_{x})$, 
	since the image of the identity section $\id \in \H^0(M_{f})$ under the map $\phi_f^*$ is identified with $\phi_f$.
	We let $\widetilde{M}_f$ denote the perfect complex over $T \times \C$ associated with $M_f$ and the endomorphism $\phi_f^*$.
	We let $\widetilde{T}_f$ denote the support of $\widetilde{M}_f$ equipped with the reduced scheme structure.
	We claim that the projection map $\pi_f \colon \widetilde{T}_f \to T$ is a finite morphism.
	Since we have an identification
	\[f^{-1} (\widetilde\calM_U) =  T \smallsetminus  \pi_f(\widetilde{T}_f \cap (T \times (\C \smallsetminus U)) ),\]
	this implies the desired statement. 
	To prove the finiteness of $\pi_f$, it is enough to show that the reduced support of the cohomology sheaf $\H^i(\widetilde{M}_f)$ denoted as 
	$\widetilde{T}_{f, i}$ is finite over $T$ for any $i \in \mathbb{Z}$.
	For each irreducible component $Z \subset \widetilde{T}_{f, i}$, there exists an element $m \in \H^i(\widetilde{M}_f)$ 
	such that the support of 
	the annihilator of $m$ is $Z$.
	This implies that there exists an injection
	$R[x] /I_{Z} \hookrightarrow   \H^i(\widetilde{M}_f)$
	where $I_{Z} \subset R[x]$ is the prime ideal corresponding to $Z$. Since $\H^i(\widetilde{M}_f) $ is finite over $R$,
	we see that $R[x] /I_{Z}$ is finite over $R$. In particular $Z \to T$ is a finite morphism.
	Therefore $\pi_f \colon \widetilde{T}_f \to T$ is finite as desired.
\end{proof}

For each analytic open subset $U \subset \C$, we set 
$$\varphi_{\widetilde\calM_{U}} = \varphi_{\widetilde\calM} |_{\widetilde\calM_U} \in \mathsf{Perv}(\widetilde\calM_U).$$

\medskip

\subsubsection{The sheaf $\DT_{\calM,\pro}$}
We define a presheaf $\DT_{\calM,\pro}$ over $\C$ with values in $\Pro(\mathrm{Vect})_{\Gamma}$ such that
\begin{align}\label{DTpro}
	\DT_{\calM,\pro}(U) = R  \Gamma_\pro( \widetilde\calM_{U}, \varphi_{\widetilde\calM_{U}})
\end{align}
for each analytic open subset $U\subset\C$. 
Here the pro-object $ R  \Gamma_\pro( \widetilde\calM_{U}, \varphi_{\widetilde\calM_{U}} )$ is defined using a 
covering of $\calM$ by quasi-compact open substacks.

\begin{lemma}\label{lem-DT-factorization}
	\hfill
	\begin{enumerate}[label=$\mathrm{(\alph*)}$,leftmargin=8mm,itemsep=2mm]
		\item There is an isomorphism $\DT_{\calM,\pro}(\C) \cong R \Gamma^{\BM}_\pro(\calM,\Q)[- \vd_\calM]$.
		\item Let $U \subset \C$ be any convex open subset. 
		The restriction gives an isomorphism 
		$$\DT_{\calM,\pro}(\C)\cong\DT_{\calM,\pro}(U)\cong R \Gamma^{\BM}_\pro(\calM,\Q)[-\vd_\calM].$$
		\item For disjoint open subsets $U_1, \ldots, U_l \subset \C$, there is a map
		\begin{equation}\label{eq-DT-kunnet-map}
			\DT_{\calM,\pro}(U_1) \,\widehat\otimes_{\Gamma}\, \cdots \widehat\otimes_{\Gamma} \DT_{\calM,\pro}(U_l) \to \DT_{\calM,\pro}(U_1 \sqcup \cdots \sqcup U_l)
		\end{equation}
		which is independent of the order of the open subsets up to the braiding isomorphism defined in \eqref{eq-braiding-sign}. If each $U_i$ is convex, this map is an isomorphism.
	\end{enumerate}
\end{lemma}

\begin{proof}
Statement (a) follows from Theorem \ref{thm-dimensional-reduction}.
The second statement (b) follows from Proposition \ref{prop:equivariance} and \cite[cor.~3.7.3]{KS13}.
We turn to (c). Note that we have an identification
\[\widetilde\calM_{U_1 \sqcup \cdots \sqcup U_l} \cong \widetilde\calM_{U_1} \times \cdots \times \widetilde\calM_{U_l}.\]
Further, using the fact that the correspondence \eqref{attractor-explicit} on the right is an oriented Lagrangian correspondence,
we see that the above isomorphism preserves the d-critical structure and there is a natural identification of the orientations
\begin{equation}\label{eq-identify-orientation}
	o^{\mathrm{sta}}_{\widetilde\calM_{U_1 \sqcup \cdots \sqcup U_l}} \cong o^{\mathrm{sta}}_{\widetilde\calM_{U_1 }} \boxtimes \cdots \boxtimes o^{\mathrm{sta}}_{\widetilde\calM_{U_l}}		
\end{equation}
which satisfies the associativity property.
In particular, the Donaldson--Thomas perverse sheaves on them are identified.
Therefore the map \eqref{eq-DT-kunnet-map} is constructed using the Thom--Sebastiani isomorphism \cite[prop.~4.1]{KPS}.
We claim that this map does not depend on the order of the open subsets up to the sign $(-1)^{\chi}$. 
For this, by the associativity, we may assume $l = 2$.
First, by Lemma~\ref{lem-ori-differ-by-Euler-form}, if we let $\sigma \colon \widetilde\calM_{U_1} \times \widetilde\calM_{U_2} \cong \widetilde\calM_{U_2} \times \widetilde\calM_{U_1}$ be the swapping isomorphism, the following diagram commutes:
\begin{equation}\label{eq-sta-ori-swap-localize}
	\begin{tikzcd}
	{o^{\mathrm{sta}}_{\widetilde\calM_{U_1  \sqcup U_2}}} &[50pt] {o^{\mathrm{sta}}_{\widetilde\calM_{U_2  \sqcup U_1}}} \\
	{o^{\mathrm{sta}}_{\widetilde\calM_{U_1 }}\boxtimes o^{\mathrm{sta}}_{\widetilde\calM_{U_2}}} & {\sigma^{\star} ( o^{\mathrm{sta}}_{\widetilde\calM_{U_2 }}\boxtimes o^{\mathrm{sta}}_{\widetilde\calM_{U_1}})}
	\arrow["\cong", from=1-1, to=1-2]
	\arrow["{\eqref{eq-identify-orientation}}"', from=1-1, to=2-1]
	\arrow["{\eqref{eq-identify-orientation}}", from=1-2, to=2-2]
	\arrow["{(-1)^{\chi} \cdot \mathrm{swap}}"', from=2-1, to=2-2]
\end{tikzcd}
\end{equation}
where $\mathrm{swap}$ denotes the braiding isomorphism involving the Koszul sign.
Since the Thom--Sebastiani isomorphism is symmetric by \eqref{eq-TS-commutative} and the evenness of the virtual dimension of $\calM$, we conclude the independence of the order as desired. 
To prove the second claim in (c), we consider the open inclusion
\[\widetilde\calM_{U_1} \times \cdots \times \widetilde\calM_{U_l} \hookrightarrow \widetilde\calM \times \cdots \times \widetilde\calM.\]
Using \cite[cor.~3.7.3]{KS13} again, we see that the following restriction map is an isomorphism
\[R \Gamma_\pro^\BM(\widetilde\calM \times \cdots \times\widetilde\calM, \varphi_{\widetilde\calM \times \cdots \times\widetilde\calM}) \to 
R \Gamma_\pro^\BM(\widetilde\calM_{U_1} \times \cdots \times \widetilde\calM_{U_l}, \varphi_{\widetilde\calM_{U_1} \times 
\cdots \times \widetilde\calM_{U_l}}).\]
Therefore the claim follows from Proposition \ref{prop-kunneth}.
\end{proof}

\medskip

\subsubsection{The definition of the coproduct}\label{para-construction-coproduct}
Using Lemma \ref{lem-DT-factorization}, we construct a coproduct on the $\Gamma$-graded pro-complex 
$$\calA_{\calM,\pro} = R \Gamma^{\BM}_\pro(\calM,\Q)[- \vd_\calM].$$
First, take a convex open subset $U \subset \C$ and consider an inclusion from the disjoint union of convex open subsets 
$V_1 \sqcup \cdots \sqcup V_l \subset U$.
Using Lemma \ref{lem-DT-factorization}, the restriction map for the presheaf $\DT_{\calM,\pro}$ gives rise to a map
\[\Delta^l_{ V_1 \sqcup \cdots \sqcup V_l \subset U} \colon \calA_{\calM,\pro} \to ( \calA_{\calM,\pro} )^{\widehat\otimes_{\Gamma} l}.
\]
Here the completed tensor product on the right is in the sense of \S \ref{sssec:setting}. We claim that this map does not depend on the choice of $V_1, \ldots, V_l, U$.
To prove this, it is enough to consider the case $U = \C$.  
If there is an inclusion $W_i \subset V_i$ for each $i$, then the maps
$$\Delta^l_{V_1 \sqcup \cdots \sqcup V_l \subset \C}
,\quad
\Delta^l_{W_1 \sqcup \cdots \sqcup W_l \subset \C}$$
are equivalent.
The claim follows, because any two sets of disjoint convex opens $(V_1, \ldots, V_l)$ and $(W_1, \ldots, W_l)$ 
are connected by a finite zigzag of inclusions of disjoint convex open subsets. 
We abbreviate 
$$\Delta^l = \Delta^l_{V_1 \sqcup \cdots \sqcup V_l \subset U}
,\quad
\Delta = \Delta^2.$$
The above discussion applied to the pairs $(V_1, V_2)$ and $(V_2, V_1)$ together with Lemma~\ref{lem-DT-factorization} (c) yields the commutativity of the following diagram
\[\begin{tikzcd}
	{\calA_{\calM,\pro}} & {\calA_{\calM,\pro} \widehat\otimes_{\Gamma} \calA_{\calM,\pro}} \\
	& {\calA_{\calM,\pro} \widehat\otimes_{\Gamma} \calA_{\calM,\pro}.}
	\arrow["{\Delta}", from=1-1, to=1-2]
	\arrow["{\Delta}"', from=1-1, to=2-2]
	\arrow["{\sigma}", from=1-2, to=2-2]
\end{tikzcd}\]
where $\sigma$ is the braiding isomorphism \eqref{eq-braiding-sign}. Likewise, by definition, we have the following commutative diagram
\[\begin{tikzcd}
	{\calA_{\calM,\pro}} & {\calA_{\calM,\pro} \widehat\otimes_{\Gamma} \calA_{\calM,\pro}} \\
	{\calA_{\calM,\pro} \widehat\otimes_{\Gamma} \calA_{\calM,\pro}} & {\calA_{\calM,\pro} \widehat\otimes_{\Gamma} \calA_{\calM,\pro} \widehat\otimes_{\Gamma} \calA_{\calM,\pro}.}
	\arrow["{\Delta}", from=1-1, to=1-2]
	\arrow["{\Delta}"', from=1-1, to=2-1]
	\arrow["{\Delta^3}"{description}, dashed, from=1-1, to=2-2]
	\arrow["{\Delta \widehat\otimes_{\Gamma} \id}", from=1-2, to=2-2]
	\arrow["{\id \widehat\otimes_{\Gamma} \Delta}"', from=2-1, to=2-2]
\end{tikzcd}\]
This shows that $\Delta$ equips $\calA_{\calM,\pro}$ with the structure of a coassociative, cocommutative coalgebra.

\medskip

\subsubsection{The bialgebra structure}
We next show that $\calA_{\calM,\pro}$ has a bialgebra structure.

\begin{proposition}\label{porp-prod-coprod-commutes}
	The comultiplication $\Delta$ and the multiplication $*$ defined in \eqref{eq-COHA-multi} and \eqref{eq-complete-square-multiplication} are compatible, i.e. we have 
	$$\Delta \circ (-*-)=  (\Delta(-)* \Delta(-))\colon  \calA_{\calM,\pro} \widehat{\otimes}_{\Gamma}' \calA_{\calM,\pro} \to 
	\calA_{\calM,\pro} \widehat{\otimes}_{\Gamma} \calA_{\calM,\pro},$$
	where $*$ on the right hand side stands for the multiplication in $\calA_{\calM,\pro} \widehat{\otimes} \calA_{\calM,\pro}$ and $\chi$ denotes the Euler form.
	
\end{proposition}

\begin{proof}
	Using the same notation as in \S\ref{sec:int-identity}, we write 
	$$\widetilde\calM^{\mathrm{ex}} = \bigsqcup_{\gamma, \gamma' \in \Gamma} \widetilde{\Filt}_{\gamma, \gamma'}.$$
	We define an analytic open substack $\widetilde\calM^{\mathrm{ex}}_U \subset \widetilde\calM^{\mathrm{ex}}$ in a similar manner as 
	$\widetilde{\calM}_U$.
	The restriction of \eqref{attractor-explicit} yields the correspondence
	\[\widetilde\calM_{U} \times \widetilde\calM_{U} \xleftarrow[]{\wgr_U} 
	\widetilde\calM_{U}^{\mathrm{ex}} \xrightarrow[]{\wev_U} \widetilde\calM_{U}\]
	and the integral isomorphism \eqref{eq-integral-identity} restricts to an isomorphism
	\begin{equation}\label{eq-localized-integral-identity}
		\varphi_{\widetilde\calM_{U} } \boxtimes \varphi_{\widetilde\calM_{U} }   \cong (\wgr_U)_* (\wev_{U})^! \varphi_{\widetilde\calM_U}.    
	\end{equation}
	Taking the left adjoint and  the global sections, we obtain a localized version of the Hall multiplication in \eqref{mult0}
	\[*_{U} \colon \DT_{\calM,\pro}(U) \,\widehat\otimes_{\Gamma} \, \DT_{\calM,\pro}(U) \to \DT_{\calM,\pro}(U)\]
	where 
	$$\DT_{\calM,\pro}(U)=\bigoplus_{\gamma \in \Gamma} R\Gamma_{\pro}\left(\widetilde{\calM}_{U,\gamma}\,,\,\varphi_{\widetilde{\calM}_{U,\gamma}}\right) \in \Pro(\Vect)_{\Gamma}.$$
	By construction, this map is compatible with the restriction map.
	In particular, for any inclusion of disjoint open subsets $U_1 \sqcup U_2 \subset U$, the following diagram commutes
	\[\begin{tikzcd}
		{\DT_{\calM,\pro}(U) \,\widehat\otimes_{\Gamma}\, \DT_{\calM,\pro}(U)} &[30pt] {\DT_{\calM,\pro}(U)} \\
		{\DT_{\calM,\pro}(U_1 \sqcup U_2) \,\widehat\otimes_{\Gamma}\, \DT_{\calM,\pro}(U_1 \sqcup U_2) } & {\DT_{\calM,\pro}(U_1 \sqcup U_2).}
		\arrow["{*_{U}}", from=1-1, to=1-2]
		\arrow[from=1-1, to=2-1]
		\arrow[from=1-2, to=2-2]
		\arrow["{*_{U_1 \sqcup U_2}}", from=2-1, to=2-2]
	\end{tikzcd}\]
	Therefore it is enough to prove the commutativity of the following diagram
	\begin{equation}\label{eq-prod-coprod-absolute}
		\begin{tikzcd}
			{\substack{\displaystyle  \DT_{\calM,\pro}(U_1 \sqcup U_2) \\ \displaystyle  \widehat\otimes_{\Gamma}\, \DT_{\calM,\pro}(U_1 \sqcup U_2) }} &[-8pt]& {\DT_{\calM,\pro}(U_1 \sqcup U_2) } \\
			\begin{array}{c} \substack{\displaystyle \DT_{\calM,\pro}(U_1) \,\widehat\otimes_{\Gamma}\, \DT_{\calM,\pro}(U_2) \\ 
					\displaystyle \widehat\otimes_{\Gamma} \DT_{\calM,\pro}(U_1) \,\widehat\otimes_{\Gamma}\, \DT_{\calM,\pro}(U_2)} \end{array} & \begin{array}{c} \substack{\displaystyle \DT_{\calM,\pro}(U_1) \,\widehat\otimes_{\Gamma}\, \DT_{\calM,\pro}(U_1) \\ 
					\displaystyle \widehat\otimes_{\Gamma} \DT_{\calM,\pro}(U_2) \,\widehat\otimes_{\Gamma}\, \DT_{\calM,\pro}(U_2)} \end{array} & {\DT_{\calM,\pro}(U_1) \,\widehat\otimes_{\Gamma}\, \DT_{\calM,\pro}(U_2). }
			\arrow["{{*_{U_1 \sqcup U_2}}}", from=1-1, to=1-3]
			\arrow[from=2-1, to=1-1]
			\arrow["{\eqref{eq-braiding-sign}}"', from=2-1, to=2-2]
			\arrow["{*_{U_1} \widehat\otimes *_{U_2}}"', from=2-2, to=2-3]
			\arrow[from=2-3, to=1-3]
		\end{tikzcd}
	\end{equation}
	To see this, consider the following diagram
	\begin{equation}\label{eq-prod-coprod-relative}
		\begin{tikzcd}
			{\varphi_{\widetilde\calM_{U_1 \sqcup U_2}} \boxtimes \varphi_{\widetilde\calM_{U_1 \sqcup U_2}}} &[15pt]& {(\wgr_{U_1 \sqcup U_2})_* (\wev_{U_1 \sqcup U_2})^! \varphi_{\widetilde\calM_{U_1 \sqcup U_2}}} \\
			{\substack{\displaystyle (\varphi_{\widetilde\calM_{U_1}} \boxtimes \varphi_{\widetilde\calM_{U_2}}) \\ \displaystyle \boxtimes  (\varphi_{\widetilde\calM_{U_1}} \boxtimes \varphi_{\widetilde\calM_{U_2}}) }} & {\substack{\displaystyle (\id \times\sigma \times \id)^*  \Big((\varphi_{\widetilde\calM_{U_1}} \boxtimes \varphi_{\widetilde\calM_{U_1}}) \\ \displaystyle \boxtimes  (\varphi_{\widetilde\calM_{U_2}} \boxtimes \varphi_{\widetilde\calM_{U_2}} )\Big)}} & {\substack{ \displaystyle (\id \times\sigma \times \id)^*  \Big((\wgr_{U_1})_*( \wev_{U_1 })^! \varphi_{\widetilde\calM_{U_1 }}) \\ \displaystyle \boxtimes (\wgr_{U_2})_* (\wev_{U_2 })^! \varphi_{\widetilde\calM_{U_2}}\Big).}}
			\arrow["{\eqref{eq-localized-integral-identity}}", from=1-1, to=1-3]
			\arrow["\cong"', from=1-1, to=1-3]
			\arrow["\cong", from=2-1, to=1-1]
			\arrow["{(-1)^{\chi} \cdot \mathrm{swap}}"', from=2-1, to=2-2]
			\arrow[""{name=0, anchor=center, inner sep=0}, "\cong", "{{\eqref{eq-localized-integral-identity}}}"', from=2-2, to=2-3]
			\arrow["\cong"', from=2-3, to=1-3]
		\end{tikzcd}
	\end{equation}
	Here, the isomorphism $\sigma \colon \mathcal{M}_{U_1} \times \mathcal{M}_{U_2} \cong \mathcal{M}_{U_2} \times \mathcal{M}_{U_1}$ is the swapping isomorphism and also we use the identification of the oriented d-critical stacks 
	$$\widetilde\calM_{U_1} \times \widetilde\calM_{U_2} \cong \widetilde\calM_{U_1 \sqcup U_2}.$$
	By Remark \ref{rmk-etale-integral-identity}, the isomorphism 
	$\varphi_{\widetilde\calM_{U_1}} \boxtimes \varphi_{\widetilde\calM_{U_2}} \cong \varphi_{\widetilde\calM_{U_1 \sqcup U_2}}$ 
	is a restriction of the map \eqref{eq-localized-integral-identity} for $U = U_1 \sqcup U_2$ to 
	$$\widetilde{\calM}_{U_1} \times \widetilde{\calM}_{U_2} \hookrightarrow \widetilde{\calM}_{U_1 \sqcup U_2} \times \widetilde{\calM}_{U_1 \sqcup U_2}.$$
	Therefore the commutativity of the diagram \eqref{eq-prod-coprod-relative} follows from the associativity and the multiplicativity of the integral isomorphisms in Theorem~\ref{thm-shifted-Lag-class}
	 together with the commutativity of the diagram of orientations \eqref{eq-sta-ori-swap-localize}.
	Taking the left adjoint and the global sections, we conclude that the diagram \eqref{eq-prod-coprod-absolute} commutes.
\end{proof}

Taking the cohomology we may apply the above construction in two contexts: 
either we set $$(\calM\,,\,\Gamma)=(\calM_S\,,\,\H^{\ev}(S,\Z))$$ or, 
fixing a reduced polynomial $p \in \Q[x]$, we set $$(\calM\,,\,\Gamma)=(\calM^\ss_S(p)\,,\,\Gamma_p),$$
see \S\ref{Abelian_subcats_sec}. 
In the latter case, the set $\Gamma$ is a strictly convex cone in a lattice of finite rank. 
Let $\Delta$ and $\Delta^\ss$ denote these coproducts.

\begin{corollary}\label{cor:exists coproduct} Let $p$ be any reduced polynomial.\hfill
\begin{enumerate}[label=$\mathrm{(\alph*)}$,leftmargin=8mm,itemsep=2mm]
\item The coproduct $\Delta$ equips $\bfH_S$ with a graded cocommutative bialgebra structure.
\item The coproduct $\Delta^\ss$ equips $\bfH^{\ss}_S(p)$ with a graded cocommutative bialgebra structure.
\item The restriction morphism $\res^{p}\colon  \bfH_S(p) \to \bfH_S^{\ss}(p)$ is compatible with the 
coproduct, i.e., we have 
$$(\res^{p} \otimes \res^{p}) \circ \Delta=\Delta^{\ss}\circ \res^{p}.$$
\qed
\end{enumerate}	
\end{corollary}

The compatibility of the coproduct with the restriction is a direct consequence of the construction.

\medskip

\subsection{Definition of the affinized BPS Lie algebra}

\subsubsection{The affinized BPS Lie algebras}
\begin{definition} Let $\nu \in \K_0(S)^+$ be a Mukai vector. 
We define ${\g}_S(\nu)$ to be the subspace of primitive elements in $\bfH_S(\nu)$, 
i.e., the subspace of elements satisfying $\Delta(x)=x \otimes 1 + 1 \otimes x$.  
The  sum ${\g}_S= \bigoplus_\nu {\g}_S(\nu)$ is a $\K_0(S)^+$-graded Lie algebra. 
We call it the affinized BPS Lie algebra.
\end{definition}

The group $\Pic(S)$ acts on $\bfH_S$ by bialgebra automorphisms. In particular, it restricts to an automorphism of $\g_S$. 
Fix a reduced Hilbert polynomial $p$ and consider the cone $\Gamma_p$. 

\begin{definition} For any $\nu \in \Gamma_p$,
let $\g_S^\ss(\nu)$ be the space of primitive elements in $\bfH_S^{\ss}(\nu)$, i.e., 
the elements satisfying $\Delta^{\ss}(x)=x \otimes 1 + 1 \otimes x$. 
The semistable affinized BPS Lie algebra associated to the polynomial $p$ is the sum
$\g_S^{\ss}(p)= \bigoplus_{\nu \in \Gamma_p} {\g}^{\ss}_S(\nu)$.
It is a $\Z \times\Gamma_p$-graded Lie algebra. The first grading is the cohomological degree.
\end{definition}

We say that a Mukai vector $\nu \in \Gamma_p$ is primitive with respect to the polarization $H$ if it cannot be written as a nontrivial sum 
$\sigma+\sigma'$ with $\sigma,\sigma' \in \Gamma_p$. 
When $H$ is not generic with respect to $\nu$, this condition is stronger than the usual notion of indivisibility.

Recall that $\bfH_S^{\ss}(p)$ has finite dimensional $\Z\times\Gamma_p$-graded pieces. 
The Milnor--Moore theorem thus yields a graded algebra isomorphism
\begin{align}\label{isomss}\U(\g_S^{\ss}(p)) \cong \bfH_S^{\ss}(p).\end{align}
In particular, we have $\frakg_S^\ss(\nu)=\bfH^\ss_S(\nu)$ for any Mukai vector 
$\nu \in \Gamma_p$ which is primitive with respect to $H$.

\medskip

\subsubsection{Comparison of $\frakg$, $\frakg^\ss$ and the BPS sheaf}

For any reduced polynomial $p$, set
$$\g_S(p)=\bigoplus_{\nu \in \Gamma_p}{\g}_S(\nu).$$
It is  a $\Z \times \Gamma_p$-graded Lie algebra. 
The following theorem is proved in \S\ref{sec:proofthmss}.

\begin{theorem}\label{thm:restriction-isom} Fix a reduced polynomial $p$.\hfill
\begin{enumerate}[label=$\mathrm{(\alph*)}$,leftmargin=8mm,itemsep=2mm]
\item
For any $\nu$ we have $\H^*_{\BPS}(M_S(\nu),\Q)\otimes \H^*(\BGM,\Q) = \frakg_S^\ss(\nu)$.
\item 
The restriction to the open subset $\calM_S^{\ss}(p)$ of $\calM_S(p)$ 
induces a Lie algebra isomorphism
$\mathrm{res}^{p}\colon {\g}_S(p) \to \g_S^{\ss}(p).$
\item
The canonical morphism $\U({\g}_S) \to \bfH_S$ is a dense embedding.
\end{enumerate}
\end{theorem}

\begin{remark}
Theorem~\ref{thm:restriction-isom} may be interpreted as saying that given an element $x\in \bfH_S^\ss(\nu)$ 
which is primitive for $\Delta^{\ss}$ there exists a unique way to extend $x$ to a class in $\bfH_S(\nu)$ which is primitive.
By \eqref{isomss}, this provides a canonical section of the restriction morphism $\mathrm{res}^{p}$.
\end{remark}

\begin{remark}
Parts (a) and (b) of Theorem~\ref{thm:restriction-isom} imply that the BPS cohomology is independent of the polarization $H'$.  Using Proposition \ref{prop:coker}, the same argument implies that $\H^*_{\BPS}(M_S(\nu),\Q)$, along with its perverse filtration induced by the support morphism $\pi_S$, is independent of polarization.
\end{remark}

\medskip

\subsection{Proof of Theorem~\ref{thm:restriction-isom}, part I}

In this section we prove simultaneously parts (a) and (b). Let us explain the strategy of the proof.
Let $\bfH_S(\nu)_\snilp$ be the image of the map 
$\H^{\Lambda_{\snilp}}_{*}(\calM_S(\nu),\Q) \to \bfH_S(\nu)$.
In a similar manner we define the subspace
$\bfH_S^\ss(\nu)_{\snilp}$ of $\bfH_S^\ss(\nu)$.
We first prove the following.

\begin{lemma}\label{lem-snilp-prim-include}
We have the inclusions
$\bfH_S(\nu)_{\snilp} \subseteq \frakg_S(\nu)$
and
$\bfH_S^\ss(\nu)_{\snilp} \subseteq  \frakg^\ss_S(\nu).$
\end{lemma}

\begin{proof}
We only prove the statement for $\calM_{S}$ since the semistable case is proved in an analogous manner.
Define a presheaf $\DT_{\calM,\pro}^\snilp$ on the complex plane $\mathbb{C}$ such that
\[\DT_{\calM,\pro}^\snilp(U)= R  \Gamma_\pro( \widetilde\calM_{U} \cap \Lambda^{\snilp}\,,\,\varphi_{\widetilde\calM} |^!_{\widetilde\calM_{U}  \cap \Lambda^{\snilp}}) \in \Pro(\mathrm{Vect})_{\Gamma}.\] 
We let $\DT_{\calM}^\snilp(U)$ be the realization of $\DT_{\calM,\pro}^\snilp(U)$.
There is a natural map of presheaves 
\begin{align}\label{DTsnilp}\DT_{\calM,\pro}^\snilp \to \DT_{\calM,\pro}\end{align}
where the right hand side is as in \eqref{DTpro}.
Let $U_1, U_2 \subset \mathbb{C}$ be disjoint convex open subsets.
We have 
\[
\widetilde\calM_{U_1 \sqcup U_2 } \cap \Lambda^{\snilp}  = 
(\widetilde\calM_{U_1} \cap \Lambda^{\snilp}) \sqcup  (\widetilde\calM_{U_2} \cap \Lambda^{\snilp}).
\]
In particular, we have a natural isomorphism
\[
\DT_{\calM,\pro}^\snilp(U_1 \sqcup U_2) \cong  \DT_{\calM,\pro}^\snilp(U_1)    \oplus  \DT_{\calM,\pro}^\snilp(U_2).
\]
Let $U_1 \sqcup U_2 \hookrightarrow U$ be an open inclusion to a convex open subset. Consider the following diagram
\[\begin{tikzcd}
{\DT_{\calM,\pro}^\snilp(U) } & {\DT_{\calM,\pro}(U) } \\
{\DT_{\calM,\pro}^\snilp(U_1 \sqcup U_2) } & {\DT_{\calM,\pro}(U_1 \sqcup U_2) } \\
{\DT_{\calM,\pro}^\snilp(U_1 ) \oplus \DT_{\calM,\pro}^\snilp( U_2)  } & {\DT_{\calM,\pro}(U_1 ) \,\widehat\otimes\, \DT_{\calM,\pro}( U_2).  }
\arrow[from=1-1, to=1-2]
\arrow[from=1-1, to=2-1]
\arrow[from=1-2, to=2-2]
\arrow[from=2-1, to=2-2]
\arrow[from=2-1, to=3-1]
\arrow[from=2-2, to=3-2]
\arrow[from=3-1, to=3-2]
\end{tikzcd}\]
The commutativity of this diagram shows that the image of the map 
\eqref{DTsnilp} consists of primitive elements.
The lemma follows.
\end{proof}

Next, we consider the following commutative diagram
\begin{equation}\label{diag-res}
\begin{tikzcd}
		{ \bfH_S(\nu)_{\snilp} } & {\frakg_S(\nu) } \\
		{ \bfH_S^{\ss}(\nu)_{\snilp} } & { \frakg_S^\ss(\nu).}
		\arrow["\kappa", hook, from=1-1, to=1-2]
		\arrow["{\res_\snil^\ss}", from=1-1, to=2-1]
		\arrow["{\res_\prim^\ss}", from=1-2, to=2-2]
		\arrow["{\kappa^{\mathrm{ss}}}", hook, from=2-1, to=2-2]
\end{tikzcd}
\end{equation}
Here $\res_\snil^\ss$ and $\res_\prim^\ss$ are the restrictions to the semistable locus.
We have 
\begin{enumerate}[label=$\mathrm{(\alph*)}$,leftmargin=8mm,itemsep=2mm]
\item $\res_\snil^\ss$ is surjective by Corollary \ref{cor:Kirwan-surj},
\item $\mathrm{\kappa}^{\mathrm{ss}}$ is an isomorphism by Lemma \ref{lem-kappa-ss-isom} below,
\item $\res_\prim^\ss$ is injective by Lemma \ref{lem-r-prim-inj}.
\end{enumerate}
Combining these three statements, we conclude that all morphisms in the above diagram are isomorphisms.
Now, let us prove the steps (b) and (c).

\begin{lemma}\label{lem-kappa-ss-isom}
The map $\kappa^{\mathrm{ss}}$ is an isomorphism $\bfH_S^{\ss}(\nu)_{\snilp} \cong  \frakg^\ss_S(\nu)$.
\end{lemma}

\begin{proof}
Let $p=p(\nu)$ so that $\nu \in \Gamma_p$.
Since $\H^{\BM}_i(\calM_S^\ss(\nu),\Q)$ is finite dimensional for each $i$, it is enough to prove the equality
\begin{equation}\label{eq-dim-equal}
\dim \H^{\BM}_i(\calM_S^\ss(\nu),\Q)_{\snilp} = \dim \H^{\BM}_i(\calM_S^\ss(\nu),\Q) \cap \frakg^\ss_S(\nu)
\end{equation}
for any $i \in \Z$.
To see this, we consider the symmetric product operation of $\mathbb{Z} \times \Gamma_p$-graded $\Q$-vector spaces
\[
\Sym \colon \Vect_{\mathbb{Q}}^{\mathbb{Z} \times \Gamma_p} \to \Vect_{\mathbb{Q}}^{\mathbb{Z} \times \Gamma_p}. 
\]
Here we use the Koszul sign rule for the first grading and no sign intervention for the second one.
The sum
\[
\bfH_S^{\ss}(p)_{\snilp}=\bigoplus_{\nu\in\Gamma_p} \bfH_S^\ss(\nu)_{\snilp}
\]
is an object in $\Vect_{\mathbb{Q}}^{\mathbb{Z} \times \Gamma_p}$.
Here the first grading is the cohomological one and the second grading is the grading by the Mukai vector $\nu$.
The equality \eqref{eq-dim-equal} is equivalent to
\[
\dim\, \Sym\, \bfH_S^{\ss}(p)_{\snilp} = 
	\dim\, \Sym \,\frakg^{\ss}_S(p)
	\]
where $\dim\,$ is the $\mathbb{Z} \times \Gamma_p$-graded dimension. 
Corollaries \ref{cor:absolute-integrality} and \ref{cor:snilp-BPS} imply that
\[\dim\, \Sym\, \bfH_S^{\ss}(p)_{\snilp} = 
\dim\,  \bfH_S^{\ss}(p).\]
The Milnor--Moore theorem and the PBW theorem together imply
$$
\dim\, \Sym \; \frakg^{\ss}_S(p) = \dim\,  \bfH_S^{\ss}(p).$$
The lemma is proved.
\end{proof}

\begin{lemma}\label{lem-r-prim-inj}
The map $\res_\prim^\ss \colon  \frakg_S(\nu) \to \frakg^\ss_S(\nu)$ is injective.
\end{lemma}

\begin{proof} 
	For each $\uunu \in \HN(\nu)$, we set 
	\[
	\H\calA_{\calM, \uunu} = \H^{\BM}_{ - * + \vd_\calM}( \calM_S^+(\uunu),\Q), \quad
	\H\calA_{\calM, \leqslant\uunu} = \H^{\BM}_{ - * + \vd_\calM}( \calM_S^+(\leqslant\uunu),\Q).
	\]
	Consider the following presheaves on the complex plane
	\begin{align*}
		\HDT_{\calM, \uunu}  &\colon  \mathbb{C} \supset U \mapsto \H^{*}(\widetilde\calM_{U} \cap \calM_X^+(\uunu), \varphi_{\widetilde\calM}|_{\widetilde\calM_{U} \cap \calM_X^+(\uunu)}),\\
		\HDT_{\calM, \leqslant\uunu}  &\colon  \mathbb{C} \supset U \mapsto \H^{*}(\widetilde\calM_{U} \cap \calM_X^+(\leqslant\uunu), \varphi_{\widetilde\calM}|_{\widetilde\calM_{U} \cap \calM_X^+(\leqslant\uunu)}).
	\end{align*}
	Arguing as the proof of Lemma \ref{lem-DT-factorization}, we obtain 
	\[
	\HDT_{\calM, \uunu}(\mathbb{C}) \cong  \H\calA_{\calM, \uunu}
	, \quad 
	\HDT_{\calM, \leqslant\uunu}(\mathbb{C}) \cong  \H\calA_{\calM, \leqslant\uunu}.
	\]
	We fix non-empty disjoint convex open subsets $U_1, U_2 \subset \mathbb{C}$.
	Let $(\H\calA_{\calM, \uunu})_\prim \subset \H\calA_{\calM, \uunu}$ be the kernel of the composition
	\[
	\H\calA_{\calM, \uunu} \cong \HDT_{\calM, \uunu}(\mathbb{C}) \to \HDT_{\calM, \uunu}(U_1 \sqcup U_2) \to
	\HDT_{\calM, \uunu}(U_1 \sqcup U_2)_{\mathrm{mult}}
	\]
	where
	$$\HDT_{\calM, \uunu}(U_1 \sqcup U_2)_{\mathrm{mult}} = \HDT_{\calM, \uunu}(U_1 \sqcup U_2) / (\HDT_{\calM, \uunu}(U_1) + \HDT_{\calM, \uunu}(U_2) ).$$
	We define $(\H\calA_{\calM, \leqslant\uunu})_\prim \subset \H\calA_{\calM, \leqslant\uunu}$ 
	in a similar manner.
	Let $\uunu'$ be the predecessor of $\uunu$ in $\HN(\nu)$. 
	Thanks to Corollary~\ref{cor-BM-limit}, to prove the Lemma it is enough to show that the following restriction map is injective
	\[
	(\H\calA_{\calM, \leqslant\uunu})_\prim \to (\H\calA_{\calM, \leqslant\uunu'})_\prim
	\]
	unless $\uunu=(\nu)$. We will prove this by the following two steps: first, we will show that the sequence
	\begin{equation}\label{eq-primitive-left-exact}
		0 \to  (\H\calA_{\calM, \uunu})_\prim \to  (\H\calA_{\calM, \leqslant\uunu})_\prim \to(\H\calA_{\calM, \leqslant\uunu'})_\prim .
	\end{equation}
	is exact, and then we will prove the vanishing $  (\H\calA_{\calM, \uunu})_\prim = 0$ if $\uunu\neq(\nu)$.
	To prove the exactness of \eqref{eq-primitive-left-exact}, consider the following diagram:
	\[\begin{tikzcd}
		& 0 & 0 & 0 \\
		0 & { (\H\calA_{\calM,\uunu})_\prim } & { (\H\calA_{\calM, \leqslant\uunu})_\prim } & {(\H\calA_{\calM, \leqslant\uunu'})_\prim } \\
		0 & { \H\calA_{\calM, \uunu}} & {\H\calA_{\calM, \leqslant\uunu}} & { \H\calA_{\calM, \leqslant\uunu'}} \\
		& {\HDT_{\calM, \uunu}(U_1 \sqcup U_2)_{\mathrm{mult}}} & {\HDT_{\calM, \leqslant\uunu}(U_1 \sqcup U_2)_{\mathrm{mult}}} & {\HDT_{\calM, \leqslant\uunu'}(U_1 \sqcup U_2)_{\mathrm{mult}}.}
		\arrow[from=1-2, to=2-2]
		\arrow[from=1-3, to=2-3]
		\arrow[from=1-4, to=2-4]
		\arrow[from=2-1, to=2-2]
		\arrow[from=2-2, to=2-3]
		\arrow[from=2-2, to=3-2]
		\arrow[from=2-3, to=2-4]
		\arrow[from=2-3, to=3-3]
		\arrow[from=2-4, to=3-4]
		\arrow[from=3-1, to=3-2]
		\arrow[from=3-2, to=3-3]
		\arrow[from=3-2, to=4-2]
		\arrow[from=3-3, to=3-4]
		\arrow[from=3-3, to=4-3]
		\arrow[from=3-4, to=4-4]
		\arrow[from=4-2, to=4-3]
		\arrow[from=4-3, to=4-4]
	\end{tikzcd}\]
	Here, the vertical sequences and the lower horizontal sequence are exact. 
	Also, Theorem \ref{thm-transition-surjective} implies that the middle horizontal sequence is exact.
	Therefore it is enough to show that the map 
	\[
	\HDT_{\calM, \uunu}(U_1 \sqcup U_2)_{} \to \HDT_{\calM, \leqslant\uunu}(U_1 \sqcup U_2)_{}    
	\]
	is injective (as the induced map $	\HDT_{\calM, \uunu}(U_1 \sqcup U_2)_{\mathrm{mult}} \to \HDT_{\calM, \leqslant\uunu}(U_1 \sqcup U_2)_{\mathrm{mult}} $ will then also be injective).
	Arguing as the proof of Lemma \ref{lem-DT-factorization}, we have
	\begin{align*}
		& \HDT_{\calM, \uunu}(U_1 \sqcup U_2)_{} \cong \H^{\BM}_{-* + 2\vd_\calM}(\oplus^{-1}(\calM^+_{S}(\uunu))), \\
		& \HDT_{\calM, \leqslant\uunu}(U_1 \sqcup U_2)_{} \cong \H^{\BM}_{-* + 2\vd_\calM}(\oplus^{-1}(\calM^+_{S}( \leqslant\uunu)))
	\end{align*}
	where $\oplus\colon  \calM_S \times \calM_S \to \calM_S$ is the direct sum map.
	Therefore we are reduced to proving that the map
	\begin{equation}\label{eq-strata-product}
		\H^{\BM}_*(\oplus^{-1}(\calM^+_{S}( \uunu)),\Q)  \to \H^{\BM}_*(\oplus^{-1}(\calM^+_{S}(\leqslant\uunu)),\Q)
	\end{equation}
	is injective. We have a stratification
	\[\bigsqcup_{\uunu_1 \cup \uunu_2 = \uunu} \calM^+_{S}( \uunu_1) \times \calM^+_{S}(\uunu_2)  =  \oplus^{-1}(\calM_{S}^+(\uunu)), \quad  \bigsqcup_{\uunu_1 \cup \uunu_2 \leqslant \uunu} \calM^+_{S}(\uunu_1) \times \calM_{S}^+(\uunu_2)  =  \oplus^{-1}(\calM^+_{S} (\leqslant\uunu))
	\]
	where $\uunu_1 \cup \uunu_2$ is the reordering according to dimension and slope of the concatenation of $\uunu_1$ and $\uunu_2$.
	Thus, the argument of the proof of Theorem \ref{thm-transition-surjective} implies that 
	$$\H^{\BM}_*(\oplus^{-1}(\calM^+_{S}(\uunu)),\Q)
	,\quad
	\H^{\BM}_*(\oplus^{-1}(\calM^+_{S}(\leqslant\uunu')),\Q)$$
	underlie pure Hodge structures. Therefore the map \eqref{eq-strata-product} is injective.
	
	Now we prove that $(\H\calA_{\calM, \uunu})_\prim = 0$ for $\uunu = (\nu_1, \ldots, \nu_\ell)$ with $\ell \geqslant 2$.
	Set $p_j=p(\nu_j)$, so that $\nu_{j}\in \Gamma_{p_j}$. Note that by construction the elements $p_1, \ldots, p_\ell$ are all distinct and ordered according to the dimension and the order among polynomials of the same degree.
	Let $\underline{\Gamma}=\Gamma_1 \times \cdots \Gamma_\ell$, where $\Gamma_i=\Gamma_{p_i}$ for all $i$. Consider the stack $$\calM^+_S(\underline{p})\coloneqq\bigsqcup_{\underline{\sigma}} \calM^+_S(\underline{\sigma})$$
	where the union ranges over all HN-types $\underline{\sigma}=(\sigma_1, \ldots, \sigma_\ell)$ such that $\sigma_i \in \Gamma_{p_i}$ for all $i$.
	We also set
	$$\H\calA_{\calM, \underline{\Gamma}} = \H^{\BM}_{ - * + \vd_\calM}( \calM_S^+(\underline{p}),\Q), \quad \H\calA_{\calM, \Gamma_j} = \H^{\BM}_{- * + \vd_\calM}(\calM_S^{\ss}(p_j)). $$
	The integral isomorphism \eqref{eq-integrality-HN-strata} implies
	\begin{equation}\label{eq:integral uutheta}
		\H\calA_{\calM, \underline{\Gamma}}\cong  \H\calA_{\calM, \Gamma_1} \otimes \cdots \otimes  \H\calA_{\calM, \Gamma_\ell}.
	\end{equation}
	By repeating the construction of the coproduct in \S\ref{para-construction-coproduct}, we obtain a coproduct
	\begin{equation}\label{eq:coproduct uutheta}
		\H\calA_{\calM, \underline{\Gamma}}  \to  \H\calA_{\calM, \underline{\Gamma}}  \otimes  \H\calA_{\calM, \underline{\Gamma}}.
	\end{equation}
	Let $(\H\calA_{\calM,\underline{\Gamma}})_\prim  \subset  \H\calA_{\calM, \underline{\Gamma}} $ denote the primitive part.
	We have 
	\[(\H\calA_{\calM, \uunu})_\prim =  (\H\calA_{\calM,\underline{\Gamma}})_\prim  \cap  \H\calA_{\calM, \uunu}.\]
	Now we will identify the primitive part $  (\H\calA_{\calM,\underline{\Gamma}})_\prim $. The coproduct~\eqref{eq:coproduct uutheta} is compatible with the isomorphism~\eqref{eq:integral uutheta} and the coproduct on each $\H\calA_{\calM, \Gamma_j}$.
	The proof of Lemma \ref{lem-kappa-ss-isom} yields isomorphisms
	\[\H\calA_{\calM, \Gamma_j} \cong \U( \frakg^{\ss}_S(p_j)), \quad j=1, \ldots, \ell\]
	and hence an isomorphism of coalgebras
	\[\H\calA_{\calM, \underline{\Gamma}} \cong \U( \frakg^{\ss}_S(p_1) ) \otimes \cdots \otimes 
	\U(\frakg^{\ss}_S(p_\ell)) \cong \U\Big(\bigoplus_{j} \frakg^{\ss}_S(p_j) \Big).\]
	We deduce that under the above isomorphism
	\[(\H\calA_{\calM,\underline{\Gamma}})_\prim   = \bigoplus_{j} \frakg^{\ss}_S(p_j)\]
	which implies the desired vanishing $  (\H\calA_{\calM, \uunu})_\prim = 0$ because $\ell\geqslant 2$. This finishes the proof of Lemma~\ref{lem-r-prim-inj}.
\end{proof}

We can now finish the proof of Theorem \ref{thm:restriction-isom}, parts (a) and (b). Lemma \ref{lem-kappa-ss-isom} and 
Lemma \ref{lem-r-prim-inj} together imply that the map $\res_\prim^\ss$ in \eqref{diag-res} is invertible. Along the way, we have also shown that
$$\res_\snil^\ss \colon  \bfH_S(\nu)_{\snilp} \to \bfH_S^\ss(\nu)_{\snilp}$$ 
is an isomorphism. Combining this with Corollary \ref{cor:snilp-BPS}, we obtain an isomorphism of $\H^\ev(S,\Z)\times\Z$-graded 
vector spaces
\[\HH_{\mathrm{BPS}}^{*}(M_S,\Q) \otimes \H^*(\BGM,\Q)_{\vir} \cong \frakg_S.\]

\qed

\medskip

\subsection{Proof of Theorem~\ref{thm:restriction-isom}, part II}\label{sec:proofthmss}
We now prove part (c) of Theorem~\ref{thm:restriction-isom}. For each $p$ there is an isomorphism 
$$\res^\ss\colon \U(\frakg_S(p)) \cong\U(\frakg_S^{\ss}(p)) = \bfH_S^{\ss}(p).$$ 

By the PBW theorem the multiplication map induces a graded vector space isomorphism
$$\vec{\bigotimes_{(d,p)} }\U(\frakg^d_S(p))\cong \U(\frakg_S).$$
Here $\vec{\bigotimes}$ is the restricted tensor product, ordered increasingly according to the dimension and the order among polynomials. 
We must prove that the multiplication map is a dense embedding 
$$m\colon \vec{\bigotimes_{(d,p)}} \U(\frakg^d_S(p)) \to \bfH_S.$$
An element $x \in \text{Ker}(m)$ may be written as 
$$x=\sum_{\underline{d},\underline{p}}x_{i,1} \otimes \cdots \otimes x_{i,s_i}, \quad x_{i,j} \in \U(\frakg_S^{d_{i,j}}(p_{i,j})),$$
where for each fixed $i$, the sequence of pairs $(d_{i,1},p_{1,i}), \ldots, (d_{i,s_i},p_{i,s_i})$ is 
increasing. Assume that $x$ is homogeneous of weight $\nu$.
Let $\nu_{i,j} \in \Gamma_{p_{i,j}}$ be the degree of $x_{i,j}$.
Hence $\sum_j \nu_{i,j}=\nu$ for all $i$. 
Assume $x \neq0$.
There exits $i$ 
such that the HN-type $\underline{\nu}_i=(\nu_{i,1}, \ldots, \nu_{i,s_i})$ is minimal for the total order $<$. 
Let us consider the composition
$$\pi_{\underline{\nu}_i}\colon  \bfH_S(\nu) \xrightarrow{\Delta^{s_i-1}} \bfH_S(\nu_{i,1}) \wotimes \cdots \wotimes \bfH_S(\nu_{i,s_i}) \xrightarrow{\res^\ss} \bfH_S^\ss(\nu_{i,1}) \otimes \cdots \otimes \bfH_S^\ss(\nu_{i,s_i}).$$ 
We abbreviate $x^\ss=\res^\ss(x)$.
By Lemma~\ref{lem:ev and HN order} we have
$$\pi_{\underline{\nu}_i}(x)=\pi_{\underline{\nu}_i}(x_{i,1} \cdots x_{i,s_i})=x_{i,1}^\ss \otimes \cdots \otimes x_{i,s_i}^\ss \neq 0.$$
This shows that $m$ is injective. The proof that $m$ has a dense image is similar. By the 
preceding argument, for any $\underline{\nu}=(\nu_1, \ldots, \nu_s) \in \HN(\nu)$ and any 
$z \in \H_*^\BM(\calM^+_S(\underline{\nu}),\Q)$ there exists an element
$$\overline{z} \in \U(\frakg_S(\nu_{1})) \otimes \cdots \otimes \U(\frakg_S(\nu_s))$$ such that 
$$m(\overline{z}) \in \H_*^\BM(\calM_S^+(\geqslant \!\underline{\nu}),\Q)
,\quad
m(\overline{z}) \in z + \H_*^\BM(\calM_S^+(>\!\underline{\nu}),\Q).$$
We deduce that for any $y\in \bfH_S(\nu)$ and any $\underline{\nu}' \in \HN(\nu)$ there exists an element 
$x \in \bigotimes_{(d,p)} \U(\frakg_S^d(p))$ such that $y$ and $m(x)$ have the same image in 
$\H_*^\BM(\calM^+_S(\leqslant\! \underline{\nu}'))$.
The density of $\text{Im}(m)$ follows. This concludes the proof of Theorem~\ref{thm:restriction-isom}. 
\qed

\medskip

\subsection{Action of Hecke operators on $\frakg_S$}
We define the vector space 
$$\frakg_S^{\geqslant 1}\subset  \bfH_S^{\geqslant 1}$$
to  be the image of the affinized BPS Lie algebra $\frakg_S$ under the restriction morphism 
$\res^{\geqslant 1}\colon \bfH_S \to \bfH_S^{\geqslant 1}$.
There is an obvious decomposition $\frakg_S^{\geqslant 1}=\bigoplus_\nu \frakg_S^{\geqslant 1}(\nu)$.
The elements of $\frakg^{\geqslant 1}_S$ are called quasi-primitive.

\begin{proposition}\label{prop:hecke on BPS Lie} Let $\nu$ be a Mukai vector of dimension $d>0$. \hfill
\begin{enumerate}[label=$\mathrm{(\alph*)}$,leftmargin=8mm,itemsep=2mm]
\item The kernel of the restriction
$\res^{\geqslant 1}\colon \frakg_S \to \frakg_S^{\geqslant 1}$ is equal to $\frakg_S(\N\delta)$. 
\item The map
$\res^{\geqslant 1}$ gives an isomorphism of graded vector spaces $\frakg_S(\nu) \cong \frakg^{\geqslant 1}_S(\nu)$.
\item The subspace $\frakg^{\geqslant 1}_S$  of $\bfH^{\geqslant 1}_S$ is preserved under the $\bfH_S^0$-action on $\bfH^{\geqslant 1}_S$.
\end{enumerate}
\end{proposition}

\begin{proof}
Part (a) is a consequence of the factorization
\begin{equation}\label{eq:ssthroughgeqone}
\res^{\ss}\colon \frakg_S \xrightarrow{\res^{\geqslant 1}} \frakg_S^{\geqslant 1} \longrightarrow \frakg_S^{\ss}
\end{equation}
and the fact that $\res^\ss\colon \frakg_S(\nu) \to \frakg^\ss_S(\nu)$ is an isomorphism if $\nu$ is of dimension $d>0$ by 
Theorem~\ref{thm:restriction-isom}.
Part (b) follows from (a).
To prove (c), observe that for any integer $n>0$ we have $\res^{\geqslant 1}(\bfH_S \cdot \bfH_S(n\delta))=0$ 
for support reasons. Thus for any $x \in \frakg_S(n\delta)$ and any $y \in \frakg_S$ we have
$$T_x (\res^{\geqslant 1}(y))=\res^{\geqslant 1}(x\cdot y)=\res^{\geqslant 1}([x,y]) \in \res^{\geqslant 1}(\frakg_S)=\frakg_S^{\geqslant 1}.$$
\end{proof}

\medskip

\subsection{Proof of the $\chi$-independence of $\frakg_S$}
Taking global sections in Theorem \ref{thm:Toda} yields the following:

\begin{theorem}\label{thm:chi-indep-Lie-algebra}
Let $\alpha$ be the class of an effective one-cycle. For any ample $\beta \in \H^2(S,\Z)$ we have 
\begin{enumerate}[label=$\mathrm{(\alph*)}$,leftmargin=8mm,itemsep=2mm]
		\item the Hecke operator
		$T_{D_{1,0}(\beta)}\colon  \frakg^{\geqslant 1}_S(\alpha+n\delta)\to \frakg^{\geqslant 1}_S(\alpha+(n+1)\delta)$
		is an isomorphism of graded vector spaces,
		\item the adjoint operator
		$\ad_{D_{1,0}(\beta)}\colon  \frakg_S(\alpha+n\delta)\to \frakg_S(\alpha+(n+1)\delta)$
		is an isomorphism of graded vector spaces.
	\end{enumerate}
\end{theorem}

\begin{proof}
We give a proof which does not depend on the comparison between 2D and 3D CoHAs. 
By Proposition~\ref{prop:hecke on BPS Lie} there is an action of $\bfH_{S}^0$ by Hecke operators on the vector space 
$$\frakg^{\geqslant 1}_S(\alpha+\Z\delta)=\bigoplus_{n \in\Z}\frakg^{\geqslant 1}_S(\alpha+n\delta).$$
This action restricts to an action of $\bfA_S$ on $\frakg^{\geqslant 1}_S(\alpha+\Z\delta)$.
It is of degree zero for the cohomological grading. 
For each $\nu$, the subspace $\frakg_{S,i}^{\geqslant 1}(\nu)\subset \frakg^{\geqslant 1}_S(\nu)$ of 
elements of cohomological degree $i$ is finite dimensional. The space 
$$\frakg_{S,i}^{\geqslant 1}(\alpha+\Z\delta)=\bigoplus_{n \in \Z} \frakg_{S,i}^{\geqslant 1}(\alpha +n\delta)$$ 
is stable under the automorphism $\otimes \calL$ for any line bundle $\calL \in \Pic(S)$. 
We may now apply Proposition~\ref{prop:algebraic trick} to prove part (a). 
Part (b) follows from Proposition~\ref{prop:hecke on BPS Lie} (a).
\end{proof}

\begin{remark} The theorem (and its proof) are valid more generally whenever $\beta$ is an effective curve 
class such that $\alpha \cdot \beta >0$ and $|n\beta|$ contains a smooth irreducible curve for some $n \geqslant 1$.
This holds for ample curve classes thanks to Bertini's connectedness theorem. 
\end{remark}

\subsection{Calculation of Hodge numbers}\label{formulae_sec}
Combined with earlier work of G\"ottsche, Huybrechts, Soergel and Yoshioka our $\chi$-independence result enables us to write down explicit generating sequences determining the mixed Hodge polynomials of the Borel--Moore
homology of stacks of one-dimensional semistable coherent sheaves on K3 surfaces and Abelian surfaces, as well as intersection mixed Hodge polynomials for their coarse moduli spaces.

Let $V$ be a $\Gamma_p\oplus\Z$-graded mixed $\Q$-Hodge structure.  We define its mixed Hodge series $\hdg^{\mathrm{gr}}(V)\coloneqq \sum_{\nu\in\Gamma_p}\hdg(V_\nu)T^\nu$, where for a cohomologically graded mixed Hodge structure $U$ we define $$\hdg(U)=\sum_{p,q,i\in\Z}\dim_{\C}(\mathrm{Gr}_\mathrm{W}^{p+q}(U^i)_{\C}^{p,q})u^pv^q(-t)^i\ .$$
Fix a K3 surface S.  We consider the $\mathbb{N}\oplus\Z$-graded (pure) mixed Hodge structure $\mathcal{H}_S=\bigoplus_{n\in\mathbb{N}} \H^*(\Hilb^n(S),\Q)\otimes \L^{-\otimes n}$ where $\L$ is defined in \S \ref{conventions_sec}; it is a pure weight 2 mixed Hodge structure concentrated in cohomological degree $2$.  It follows from \cite{Gottsche-Soergel} that there is an isomorphism of $\mathbb{N}\oplus\Z$-graded mixed Hodge structures $\mathcal{H}_S\cong \Sym\left(\bigoplus_{n\geqslant 1} \H^*(S,\Q)\otimes\L^{-1}\right)$, and so an identity of generating series
\begin{align*}
\mathcal{Z}^{\Kthree}(u,v,t,T)\coloneqq \hdg^{\mathrm{gr}}(\mathcal{H}_S)=&\Exp\left(\sum_{n\geqslant 1} \hdg(\H^*(S,\Q))T^nt^{-2}u^{-1}v^{-1}\right)\\=&\Exp\left(\sum_{n\geqslant 1} (t^{-2}u^{-1}v^{-1}+uv^{-1}+20+u^{-1}v+uvt^2)T^n\right)\eqqcolon&\sum_{n\geqslant 0}\mathcal{Z}^{\Kthree}_{n-1}T^{n}
\end{align*}
where $\Exp$ denotes the plethystic exponential in the variables $u,v,t,T$, and the final equation serves to define Laurent polynomials $\mathcal{Z}^{\Kthree}_n$ in $t,u,v$ for $n\geqslant -1$.  We similarly define Laurent polynomials $\mathcal{Z}^{\Abelian}_n$ by setting $$\mathcal{Z}^{\Abelian}(u,v,t,T)=\sum_{n\geqslant 0}\mathcal{Z}^{\Abelian}_nT^n=(1-tu)^2(1-tv)^2t^{-2}u^{-1}v^{-1}\Exp(\sum_{n \geqslant 1}(1-tu)^2(1-tv)^2t^{-2}u^{-1}v^{-1}T^n).$$

Now let $\nu=(0,\alpha,n\delta)$ be a one-dimensional Mukai vector.  We write $\heartsuit=\Kthree$ if $S$ is a K3 surface and $\heartsuit=\Abelian$ if $S$ is Abelian.  We assume that $ \nu\cdot\nu >0$ in order to derive new results below.  By Theorem \ref{thm:Toda} there is an isomorphism of mixed Hodge structures $\g^{\mathrm{ss}}_S(\nu)\cong \g^{\mathrm{ss}}_S(\nu')$ where $\nu'=\nu+l\delta$ is a primitive Mukai vector for some sufficiently generic polarization $H'$.  Since $\nu'$ is primitive, by polarization independence of BPS cohomology (a consequence of Theorem \ref{thm:restriction-isom}) we have $\g^{\mathrm{ss}}_S(\nu')\cong \H^*({\calM'}_S^\ss(\nu'),\Q)\otimes \L^{-\dim M'_S(\nu')/2}\cong \H^*(M'_S(\nu'),\Q)\otimes\L^{-\dim M'_S(\nu')/2}\otimes \H^*(\mathrm{B}\C^*,\Q)$, where $M'_S(\nu')$ and ${\calM'}_S^\ss(\nu')$ are defined with respect to $H'$.  By \cite{Yoshioka3,Huyb} the coarse moduli space $M'_S(\nu')$ is deformation equivalent to $\Hilb^{\nu'\cdot\nu'/2+1}(S)$ if $S$ is a K3 surface, and $\widehat{S}\times \text{Hilb}^{\nu' \cdot \nu' / 2}(S)$ if it is an Abelian surface, where $\widehat{S}$ is the dual surface.  Note that since $\nu$ is one-dimensional we have $\nu\cdot \nu=\nu'\cdot\nu'$ and we deduce $\hdg(\g^{\mathrm{ss}}_S(\nu))=\mathcal{Z}^{\heartsuit}_{\nu\cdot\nu/2}(1-uvt^2)^{-1}$.  Combining this equality with the generating function identity induced by \eqref{intro1} after passing to derived global sections, we deduce the identities
\begin{align}
\label{Hodge_identities1}
\hdg^{\mathrm{gr}}(\bfH^{\mathrm{ss}}_S(p))=&\Exp\left(\sum_{\nu\in \Gamma_p^\times} \mathcal{Z}^{\heartsuit}_{\nu\cdot\nu/2}T^{\nu}(1-uvt^2)^{-1}\right),
\\
\label{Hodge_identities2}
 \hdg^{\mathrm{gr}}({}^{0}\bfH^{\mathrm{ss}}_S(p))=&\Exp\left(\sum_{\nu\in \Gamma_p^\times} \mathcal{Z}^{\heartsuit}_{\nu\cdot\nu/2}T^{\nu}\right)
\end{align}
where we denote by ${}^{0}\bfH^{\mathrm{ss}}_S(p)$ the zeroth piece of the perverse filtration on $\bfH^{\mathrm{ss}}_S(p)$, which is defined by taking derived global sections of $\bigoplus_{\nu\in\Gamma_p} {}^{\text{p}}\calH^0\big(({p}_S)_* \mathbb{DQ}_{\calM^\ss_S(\nu)}\big)$, considered as a mixed Hodge module.  

Equation \eqref{Hodge_identities1} determines the mixed Hodge series of the Borel--Moore homology of $\calM_S^\ss(\nu)$ for every one-dimensional $\nu$.  By \eqref{intro2} we have an isomorphism of algebras ${}^{0}\bfH^{\mathrm{ss}}_S(p)\cong \mathrm{T}\left(\bigoplus_{\nu\in\Gamma^\times_p} \IH^*(M_S(\nu))\right)$, where the target is the free associative algebra generated by the intersection cohomology of the coarse moduli spaces of polystable sheaves.  Combining with the second equation \eqref{Hodge_identities2} we deduce
\begin{equation}\label{IC_calc}
\hdg^{\mathrm{gr}}\left(\bigoplus_{\nu\in\Gamma_p^\times} \IH^*(M_S(\nu))\right)=1-\Exp\left(-\sum_{\nu\in \Gamma_p^\times} \mathcal{Z}^{\heartsuit}_{\nu\cdot\nu/2}T^{\nu}\right).
\end{equation}
We use here the identities $\Exp(F)^{-1}=\Exp(-F)$ and $\hdg^{\mathrm{gr}}(\mathrm{T}(U))=(1-\hdg^{\mathrm{gr}}(U))^{-1}$.  Equation \eqref{IC_calc} fully determines the intersection Hodge polynomials of all coarse moduli spaces of polystable sheaves with one-dimensional Mukai vector.

\bigskip

\section{Tautological classes and a generalization of Markman's theorem}

\medskip

In this section, we apply a variation of the ideas in the previous sections to give a characterization of the affinized BPS Lie algebra 
$\frakg_S(\nu)$ in terms of tautological classes, see Theorem~\ref{thm:Markman}.
In order to state this result, we first introduce tautological classes, in cohomology and in Borel--Moore homology.

\medskip

\subsection{Tautological classes and their action on the CoHA}
The aim of this section is to describe some compatibility between the operators of multiplication by tautological classes on $\calM_S(\nu)$ and 
the CoHA.

\medskip

\subsubsection{Tautological classes} Let $\calE_\nu \in \Coh(\calM_S(\nu) \times S)$ be the tautological sheaf. 
Let $\ch(\calE_\nu)=\sum_i \ch_i(\calE_\nu)$ be its Chern character. 
Let $\Lambda=\Q[p_1, p_2, \ldots]$ be the ring of symmetric functions. Consider the morphism 
$$\text{a}_\nu\colon  \Lambda \to\H^*(\calM_S(\nu) \times S,\Q), \quad \frac{p_i}{i!} \mapsto \ch_i(\calE_\nu), \quad i \in \N.$$
To each cohomologically homogeneous $\gamma \in \H^*(S,\Q)$ and degree $i$ polynomial $f \in \Lambda$ we associate the cohomology class
$$f(\gamma) = \int_S \text{a}_\nu(f) \cup \gamma \in \H^{2i+|\gamma|-4}(\calM_S(\nu)).$$
By definition, the tautological cohomology ring $\H^*_\taut(\calM_S(\nu),\Q)$ of $\calM_S(\nu)$ is the subring of 
$\H^*(\calM_S(\nu),\Q)$ generated by the classes $f(\gamma)$ for $f \in \Lambda$ and $\gamma \in \H^*(S,\Q)$. 


\medskip

\subsubsection{Universal ring of tautological classes} We consider the graded algebra
$$\Lambda_S= \Q[\,\overline{p}_n(\gamma)\;;\; n \geqslant 0, \gamma \in \H^*(S,\Q)]$$
where $\deg(\overline{p}_n(\gamma))=2n-4+|\gamma|$,
in which we impose the following relations
$$\overline{p}_n(\gamma+\eta)=\overline{p}_n(\gamma) + \overline{p}_n(\eta)
, \quad 
\overline{p}_n(a\gamma)=a\,\overline{p}_n(\gamma)$$
for  $a \in\Q$ and $\gamma,\eta \in \H^*(S,\Q)$.
Further, we set
$$\overline{p}_n(\gamma)=0\ \text{if}\ 2n + |\gamma| <4.$$
Thus, the ring $\Lambda_S$ is a non-negatively graded supercommutative polynomial ring.
It is convenient to introduce a generating series for the elements $\overline{p}_n(\gamma)$.
Let $\{\pi_i\}$ be a homogeneous basis of $\H^*(S,\Q)$, and let $\{\pi^i\}$ 
be the dual basis with respect to the intersection pairing.
Hence,  we have
$$[\Delta_S]=\sum_i \pi_i \otimes \pi^i\in \H^*(S,\Q) \otimes \H^*(S,\Q).$$
We set
\begin{align}\label{ch}
	\overline{\ch}(s)= \sum_{i,n} \overline{p}_n(\pi_i)\otimes \pi^i \otimes \frac{s^n}{n!}\in \Lambda_S\otimes \H^*(S,\Q)[[s]].
\end{align}
For any Mukai vector $\nu$ there is an algebra morphism
\begin{align}\label{ev}
	\text{e}_\nu\colon \Lambda_S \to \H^*(\calM_S(\nu),\Q), \quad \overline{\ch}(s) \mapsto \ch(\calE_\nu,s)
	, \quad 
	\overline{p}_n(\gamma) \mapsto p_n(\gamma).
\end{align}
Note that the classes $\overline{p}_0(\gamma),$ $\overline{p}_1(\eta)$ and $ \overline{p}_2(1)$
evaluate to scalars in $\H^0(\calM_S(\nu),\Q)$ for any $\gamma \in \H^4(S,\Q)$, $\eta \in \H^2(S,\Q)$ and 
any Mukai vector $\nu$.
This scalar depends on $\nu$. 
In fact, these classes provide a way to encode the class $\nu$ itself. 
For instance, for $S$ a K3 surface, we have
$$\text{e}_\nu(\overline{p}_0(\delta))=\text{rk}(\nu), \quad \text{e}_\nu(\overline{p}_2(1))=2(\nu_2-\rk(\nu)).$$

\medskip

\subsubsection{Compatibility with tautological classes}

The ring $\Lambda_S$ acts on the algebra $\bfH_S$ as follows
\begin{align}\label{bullet}
	c \bullet x=\text{e}_\nu(c) \cap x, \quad c \in \Lambda_S, \quad x \in \bfH_S(\nu).
\end{align}
This action is graded in the sense that we have
$\Lambda_{S,n} \bullet \bfH_{S,m}(\nu) \subseteq \bfH_{S,m-n}(\nu).$ 
Since the cap product commutes with open restriction, 
for any reduced polynomial $p$, the $\Lambda_S$-action given by $\bullet$ on $\bfH_S$ 
intertwines with the $\Lambda_S$-action on the quotient algebra $\bfH_S^\ss(p)$.
We also denote this action by  
$$\bullet\colon \Lambda_S\otimes\bfH_S^\ss(p)\to\bfH_S^\ss(p).$$
We equip the graded ring $\Lambda_S$ with the cocommutative comultiplication 
$\Delta\colon \Lambda_S \to \Lambda_S \otimes \Lambda_S$ 
such that
\begin{equation}\label{def:coprod BBHS}
	(\Delta\otimes \id)(\overline{\ch}(s))=\overline{\ch}(s)\otimes 1 + 1 \otimes \overline{\ch}(s)
\end{equation}
in terms of the formal series $\overline{\ch}$ introduced in \eqref{ch}.
The additivity of the Chern character implies that, for each $c \in \Lambda_S$ and $\nu,\nu_1,\nu_2$ with $\nu=\nu_1+\nu_2$, we have
\begin{equation}\label{eq:taut classes coprod}
	\ev^*(\text{e}_\nu(c))=\gr^*((\text{e}_{\nu_1} \otimes \text{e}_{\nu_2})(\Delta(c)))
\end{equation}
where the evaluation morphism $\text{e}_\nu$ is as in \eqref{ev} and the maps
$\ev$, $\gr$ are as in \eqref{induction-diagram-S}.
Equivalently, the comultiplication $\Delta$ of $\Lambda_S$ is characterized by the commutativity of the following diagram
$$\xymatrix{\Lambda_S \ar[d]_-{\text{e}_\nu} \ar[r]^-{\Delta} & \Lambda_S \otimes \Lambda_S \ar[d]^-{\text{e}_{\nu_1} \otimes \text{e}_{\nu_2}}\\ 
	\text{H}^*_S(\nu)\ar[r]^-{\bigoplus^*} & \text{H}^*_S(\nu_1)\otimes \text{H}^*_S(\nu_2)}.$$

The action of $\Lambda_S$ on $\bfH_S$ gives an action of
$\Lambda_S\otimes\Lambda_S$ on $\bfH_S\,\widehat\otimes\,\bfH_S$ satisfying the usual graded-commutative rules.
For any $c \in \Lambda_S$ we write $\Delta(c)=\sum_i c_{i,1}\otimes c_{i,2}$, and for any $x,y \in \bfH_S$ we define
\begin{equation}\label{prod compat0}
	\Delta(c)\bullet (x\otimes y)=\sum_i (-1)^{\deg(x)\cdot\deg(c_{i,2})}(c_{i,1}\bullet x) \cdot (c_{i,2}\bullet y).
\end{equation}

\begin{proposition}\label{prop:compat} 
\hfill
\begin{enumerate}[label=$\mathrm{(\alph*)}$,leftmargin=8mm,itemsep=2mm]
\item
$\bfH_S$ is a Hopf-module algebra over $\Lambda_S$, i.e., 
for any $c \in \Lambda_S$ and $u,v \in \bfH_S$ we have
\begin{equation}\label{prod compat}
c \bullet (u \cdot v)=\Delta(c)\bullet (u\otimes v)=\sum_i (-1)^{|u|\cdot|c_{i,2}|}(c_{i,1}\bullet u) \cdot (c_{i,2} \bullet v).
\end{equation}
\item
$\bfH_S$ is a Hopf-module coalgebra over $\Lambda_S$, i.e.,
for any $c \in \Lambda_S$ and $u,v \in \bfH_S$ we have
\begin{equation}\label{coprod compat}
\Delta (c \bullet u)=\Delta(c) \bullet \Delta (u).
\end{equation}
\item
The same holds for $\bfH_S^{\ss}(p)$ for any reduced polynomial $p$.
\end{enumerate}
\end{proposition}

\begin{proof} The proof of (a) appears in \cite[prop.~1.18]{MMSV}, we provide it for the reader's convenience. 
Starting from the induction diagram~\eqref{induction-diagram-S} we have, by the projection formula 
\begin{align*}
		c \bullet (u \cdot v)&=\text{e}_{\nu}(c) \cap \ev_*\gr^!(u \otimes v)\\
		&=\ev_*(\ev^*(\text{e}_{\nu}(c) )\cap \gr^!(u \otimes v)).
\end{align*}
By \eqref{eq:taut classes coprod}, we have $\ev^*(\text{e}_{\nu}(c))=\gr^*((\text{e}_{\nu_1} \otimes \text{e}_{\nu_2})(\Delta(c)))$, 
from which we get
$$\ev_*(\ev^*(\text{e}_{\nu}(c)) \cap \gr^!(u \otimes v))=\ev_*\gr^!((\text{e}_{\nu_1} \otimes \text{e}_{\nu_2})(\Delta(c)) \cap u \otimes v)$$
as wanted. The same proof works in the case of $\bfH_S^{\ss}(p)$.
	
Next, we give the proof of (b).
By the projection formula, under the dimensional reduction isomorphism \eqref{eq:dimensional-reduction}, the cap-product 
$\bfH_S(\nu)\to\bfH_S(\nu)$ with the Chern character
$\ch(\calE_\nu)$ in $\H^*(\calM_S(\nu),\Q)$ is identified with the cap-product 
$\bfH_X(\nu)\to\bfH_X(\nu)$ with the cohomology class 
$\kappa^*\ch(\calE_\nu)$ in $\H^*(\calM_X(\nu),\Q)$.
Here the map $\kappa$ is as in \eqref{kappa}.
Hence, the claim follows from the additivity of the Chern character.
The proof for (c) is identical.
\end{proof}

\medskip

\subsection{The space of tautological classes in the CoHA}
We next introduce a subspace of tautological classes in $\bfH_S$.
Following Schürg--Toën--Vezzosi, see \cite{STV} to which we refer for details, we first briefly recall the definition of the reduced virtual 
fundamental class $[\calM_S(\nu)^\rd]$. 
The notation $[\mathcal M_S(\nu)^{\mathrm{rd}}]$ should not be confused with the reduced classical stack 
$\mathcal M_S(\nu)_{\mathrm{red}}$ appearing in \eqref{reduced-stack}.
Let $\Pic(S)$ be the derived Picard stack of $S$ and let $\Pic(S)^\cl$ be its classical truncation. For 
any Mukai vector $\nu$ there is a canonical map of derived stacks
\begin{equation}\label{eq:redclassone}
\det_S\colon  \calM_S(\nu) \to \Pic(S)
\end{equation}
which sends a coherent sheaf $\calF$ to its determinant line bundle.
The tangent morphism at a point $\calF$ is given by the trace map
$$\mathbb{T}_\calF \det_S=\mathrm{tr}_\calF[1]\colon  R\Hom(\calF,\calF)[1] \to R\Hom(\calO_S,\calO_S)[1].$$
By \cite[prop.~4.6]{STV} there is a chain of morphisms
\begin{equation}\label{eq:redclasstwo}
\mathrm{pr}\colon \Pic(S) \stackrel{\sim}{\to} \Pic(S)^\cl \times \Spec\,\Sym(\H^0(S,K_S)[1]) \to \Spec\,\Sym(\H^0(S,K_S)[1]).
\end{equation}
On tangent spaces it corresponds to the projection
\begin{equation*}\label{eq:redclasstwo2}
\mathbb{T}_{\mathrm{pr}}\colon \mathbb{T}_eR\Hom(\calO_S,\calO_S)[1] \to 
\mathbb{T}_0\Spec\,\Sym(\H^0(S,K_S)[1])\cong\Ext^2(\calO_S,\calO_S).
\end{equation*}
Composing \eqref{eq:redclassone} and \eqref{eq:redclasstwo} we get a canonical morphism
$$\pi\colon \calM_S(\nu) \to \Spec\,\Sym(\H^0(S,K_S)[1]).$$
The reduced stack $\calM_S(\nu)^\rd$ is defined by the homotopy Cartesian square
$$\xymatrix{\calM_S(\nu)^\rd\ar[d] \ar[r]& \calM_S(\nu) \ar[d]^-{\pi}\\ \Spec \C \ar[r] &\Spec\,\Sym(\H^0(S,K_S)[1]).}$$
Equivalently, we have
$$\calM_S(\nu)^\rd=\calM_S(\nu) \underset{\Pic(S)}{\times}\Pic(S)^\cl.$$
By construction, we have
$$(\calM_S(\nu)^\rd)^\cl=\calM_S(\nu)^\cl, \quad \dim(\calM_S(\nu)^\rd)=\dim(\calM_S(\nu))+1$$
and the derived stack $\calM_S(\nu)^\rd$ is quasi-smooth. 
We denote its fundamental class by 
$$[\calM_S(\nu)^\rd]\in \H_{2\dim(\calM_S(\nu))+2}^\BM(\calM_S(\nu),\Q).$$

\begin{definition}
We define
$\bfH_S^\taut(\nu)=\Lambda_S \bullet [\calM_S(\nu)^\rd]$ and $\bfH_S^\taut=\bigoplus_\nu \bfH_S^\taut(\nu).$
We define
$\bfH_S^{\taut,\geqslant 1}(\nu)$ and $ \bfH_S^{\taut,\geqslant 1}$ likewise.
\end{definition}

\begin{remark}\label{rem:reduce class is classical class}
Assume that $\calM^{\geqslant 1}_S(\nu)$ is non-empty, irreducible and that $\calM^\ss_S(\nu)^\cl$ 
is nonempty and generically a $\GM$-gerbe over $M_S(\nu)$. Then, we have
$$\dim(\calM^{\geqslant 1}_S(\nu))^\cl=\dim(\calM^{\geqslant 1}_S(\nu))+1=\dim(\calM^{\geqslant 1}_S(\nu)^\rd).$$
It follows that the derived stack $\calM_S^{\geqslant 1}(\nu)^\rd$ is generically underived.
Thus we have
\begin{equation}\label{eq:reduced class is truncated class}
[\calM_S^{\geqslant 1}(\nu)^\cl]=[\calM_S^{\geqslant 1}(\nu)^\rd].
\end{equation}
In this case the space $\bfH_S^{\taut,\geqslant 1}(\nu)$ coincides with the space of classical tautological classes 
$\Lambda_S \bullet [\calM^{\geqslant 1}_S(\nu)^\cl]$.
\end{remark}

\begin{proposition}\label{prop:hecke acts on reduced tautological} The subspace 
$\bfH_S^{\taut,\geqslant 1}$ of $ \bfH_S^{\geqslant 1}$ is preserved by Hecke operators.
\end{proposition}

\begin{proof}
A similar result is proved in \cite[thm.~C, \S2]{MMSV} for the space of non-reduced tautological classes. 
The proof of Proposition~\ref{prop:hecke acts on reduced tautological} is a variation. We sketch it in \S\ref{app:F}.
\end{proof}

\medskip

\subsection{Markman's Theorem for curve classes}

The following result, which relates tautological classes in Borel--Moore homology with the affinized BPS Lie algebra is proved in
\S\ref{sec:proofmarkman}.

\begin{theorem}\label{thm:Markman}
Let $\alpha \in \H^2(S,\Z)$ be the class of an effective one-cycle. For any integer $n$ we have a commutative diagram of isomorphisms
$$\xymatrix{\bfH_S^{\taut}(\alpha+n\delta) \ar[r]^-\cong \ar[d]_-\cong & \frakg_S(\alpha+n\delta) \ar[d]^-\cong\\ 
\bfH_S^{\taut,\ss}(\alpha+n\delta)\ar[r]^-\cong & \frakg^{\ss}_S(\alpha+n\delta)}$$
where the vertical maps are induced by the restrictions to the open subset $\calM^\ss_S$ of $\calM_S$. 
\end{theorem} 

\begin{remark} A corollary of Theorem~\ref{thm:Markman} is that the relations satisfied by tautological classes in 
$\bfH_S(\nu)$ and in $\bfH_S^\ss(\nu)$ are the same. In other words, a tautological class on 
$\calM_S(\nu)$ whose restriction to $\calM_S^\ss(\nu)$ is zero in fact vanishes on the whole of $\calM_S(\nu)$. This is in stark 
contrast to the case of the moduli stack of vector bundles on a curve $C$, where by the Atiyah--Bott theorem $\H_*^\BM(\Bun_C(\nu),\Q)$ is 
a free module of rank one over the ring $\BBH^{\geqslant 1}_C$ of tautological classes (generated by Chern classes of degree at most the 
rank). 
\end{remark}

\medskip

\subsection{Stacky version of Markman's theorem} We will need the following slight variant of Markman's `generation of cohomology' theorem:

\begin{proposition}\label{prop:stackymarkman} Let $\nu\neq \delta$ be an effective primitive Mukai vector which is generic with respect to a polarization $H$. Then the canonical restriction map $\res^\ss\colon \bfH_S^{\taut,\geqslant 1}(\nu) \to \bfH_S^{\ss}(\nu)$ is surjective.
\end{proposition} 
\begin{proof} By our assumptions, the canonical morphism $p\colon \calM^\ss_S(\nu) \to M_S(\nu)$ is a $\mathbb{G}_m$-gerbe. This gerbe may be nontrivial, but there always exists a quasi-universal sheaf $U \in \Coh(M_S(\nu) \times S)$, whose construction we briefly recall. Let $V=((\pi_{\calM})_*(\calE_\nu\otimes \pi_S^*\mathcal{L}))^{\vee}$ where $\mathcal{L}$ is a sufficiently ample line bundle on $S$. This is a vector bundle on $\calM_S^\ss(\nu)$ with weight $-1$ with respect to the $\BGM$-action and the descent $U$ to $M_S(\nu) \times S$ of $\calE_\nu \otimes V$ is a quasi-universal sheaf. 

As an algebra, $\H^*(\calM^\ss_S(\nu),\Q)$ is generated by $p^*\H^*(M_S(\nu),\Q)$ and $\ch_1(V)$. Note that $\ch_i(V) \in \H^*_\taut(\calM^\ss_S(\nu),\Q)$ for all $i$ by construction.
By Markman's theorem \cite[Cor.~2, Rem.~3]{Markman} (see also \cite[Thm~1.1, Rem.~3.8]{Bottini}), $\H^*(M_S(\nu),\Q)$ is generated by the components of $\int_S \ch(U) \cup \eta$ for $\eta \in \H^*(S,\Q)$. Thus $p^*(\H^*(M_S(\nu),\Q))$ is generated by the components of $\int_S \ch(\calE_\nu) \cup \ch(V) \cup \eta=\ch(V) \cup \int_S \ch(\calE_\nu) \cup \eta$ for $\eta \in \H^*(S,\Q)$. As $\ch(V) \in \H^*_\taut(\calM^\ss_S(\nu),\Q)$ it follows that $p^*(\H^*_\taut(M_S(\nu),\Q)) \subset \H^*_\taut(\calM^\ss_S(\nu),\Q)$, and hence $\H^*_\taut(\calM^\ss_S(\nu),\Q)=\H^*(\calM^\ss_S(\nu),\Q)$ as wanted.
\end{proof}

\medskip

\subsection{Proof of Theorem~\ref{thm:Markman}}\label{sec:proofmarkman}
 
\medskip

Let $\alpha$ satisfy the hypotheses of Theorem~\ref{thm:Markman}. 
We first prove the inclusion
\begin{equation}\label{eq:proof markman 1}
\bfH_S^{\taut}(\alpha+\Z\delta) \subseteq \frakg_S(\alpha+\Z\delta).
\end{equation}
Recall from \eqref{eq-def-Lambda-scal} that we have defined a subcone $\Lambda_{\mathrm{scal}} \subset \widetilde{\calM_S}$ consisting of pairs with scalar endomorphisms.
We define a presheaf $\DT_{\calM_S,\pro}^\mathrm{scal}$ over the complex plane by the following assignment for $U \subset \mathbb{C}$
\[\DT_{\calM_S}^\mathrm{scal}(U)= R  \Gamma( \widetilde\calM_{S, U} \cap \Lambda_{\mathrm{scal}}\,,\,\varphi_{\widetilde\calM_S} |^!_{\widetilde\calM_{S, U}  \cap \Lambda_{\mathrm{scal}}}) \in \Pro(\mathrm{Vect})_{\Gamma}.\] 
By Proposition~\ref{prop-scal-microlocal}\ref{item-scal-compute}, there is an isomorphism
$\DT_{\calM_S}^\mathrm{scal}(U) \cong \mathrm{H}^*(\calM_S)$
for convex $U$, which shifts the grading by $\vdim \calM_S + 2$. Now take two disjoint convex open subsets $U_1, U_2 \subset U$ and consider the following commutative diagram (we omit the shifts for simplicity)
\[\begin{tikzcd}
	{\mathrm{H}^*(\calM_S)} & {\mathrm{DT}^{\mathrm{scal}}_{\calM_S}(U)} & {\mathrm{DT}_{\calM_S}(U)} & {\bfH_S} \\
	{\mathrm{H}^*(\calM_S) \oplus \mathrm{H}^*(\calM_S)} & {\mathrm{DT}^{\mathrm{scal}}_{\calM_S}(U_1 \coprod U_2)} & {\mathrm{DT}_{\calM_S}(U_1 \coprod U_2)} & {\bfH_S \hat{\otimes} \bfH_S.}
	\arrow["\cong", from=1-1, to=1-2]
	\arrow["{(\id, \id)}"', from=1-1, to=2-1]
	\arrow[from=1-2, to=1-3]
	\arrow["{\mathrm{res}}", from=1-2, to=2-2]
	\arrow["\cong", from=1-3, to=1-4]
	\arrow["{\mathrm{res}}", from=1-3, to=2-3]
	\arrow["\Delta", from=1-4, to=2-4]
	\arrow["\cong"', from=2-1, to=2-2]
	\arrow[from=2-2, to=2-3]
	\arrow["\cong"', from=2-3, to=2-4]
\end{tikzcd}\]
The commutativity of this diagram shows that the image of the composite of the upper horizontal maps is contained in $\mathfrak{g}_S$.
On the other hand, Proposition~\ref{prop-scal-microlocal}\ref{item-scal-VFC} shows that the image of the composite of the upper horizontal maps is equal to $\mathrm{H}^*(\calM_S) \cap [\mathcal{M}_S]_{\mathrm{rd}} \subset \bfH_S$.
In particular, we obtain the desired inclusion \eqref{eq:proof markman 1}.

It remains to check the inverse inclusion of \eqref{eq:proof markman 1}. 
By Proposition~\ref{prop:hecke acts on reduced tautological}, the graded vector space 
$\bfH_S^{\taut,\geqslant 1}(\alpha+\Z\delta)$ is stable under all Hecke operators. It is also preserved by 
$\otimes \calL$ for any line bundle $\calL$ over $S$, because 
$$[\calM^{\geqslant 1}_S(\nu\otimes \calL)^\rd]=\omega_\calL ([\calM^{\geqslant 1}_S(\nu)^\rd]).$$ 
As in the proof of 
Theorem~\ref{thm:chi-indep-Lie-algebra}, it follows by application of Proposition~\ref{prop:algebraic trick} that for any 
$\beta\in\H^2(S;\Z)$ such that $\alpha \cdot \beta \neq 0$, the map $T_{D_{1,0}(\beta)}$ restricts to isomorphisms 
$$T_{D_{1,0}(\beta)}\colon \bfH_S^{\taut,\geqslant 1}(\alpha+n\delta) \cong \bfH_S^{\taut,\geqslant 1}(\alpha+(n+1)\delta).$$
The Mukai vector $\alpha+\delta$ is primitive and there exists a polarization $H$ with respect to which it is generic. 
It follows from Yoshioka's theorem that $\calM^\ss_S(\alpha+\delta)=\calM_S^\st(\alpha+\delta)$ is non-empty and is a $\GM$-gerbe over a smooth scheme. 
From this we deduce by Remark \ref{rem:reduce class is classical class}
that $[\calM_S^{\geqslant 1}(\alpha+\delta)^\rd]=[\calM_S^{\geqslant 1}(\alpha+\delta)^\cl]$. 
Moreover, by the stacky version of Markman's theorem (see Proposition~\ref{prop:stackymarkman}) we have a surjection
$$\bfH_S^{\taut,\geqslant 1}(\alpha+\delta) \to \bfH_S^\ss(\alpha+\delta)=\frakg^\ss_S(\alpha+\delta).$$
Using Theorem \ref{thm:restriction-isom}, Proposition \ref{prop:hecke on BPS Lie} and \eqref{eq:proof markman 1} we get the chain of maps
$$\bfH_S^{\taut,\geqslant 1}(\alpha+\delta) \subseteq \frakg^{\geqslant 1}_S(\alpha+\delta) 
\xrightarrow{\sim}  \frakg^\ss_S(\alpha+\delta).$$
We deduce that 
$$\bfH_S^{\taut,\geqslant 1}(\alpha+\delta) = \frakg^{\geqslant 1}_S(\alpha+\delta).
$$
Then, we use Proposition \ref{prop:hecke on BPS Lie} and Theorem \ref{thm:chi-indep-Lie-algebra}.
Applying the Hecke operator $T_{D_{1,0}(\beta)}$ and its inverse successively, we obtain the equality
$$\bfH_S^{\taut,\geqslant 1}(\alpha+n\delta) = \frakg^{\geqslant 1}_S(\alpha+n\delta).$$ 
Hence, using Proposition \ref{prop:hecke on BPS Lie} and \eqref{eq:proof markman 1} we deduce that
$$\bfH_S^{\taut}(\alpha+n\delta) = \frakg_S(\alpha+n\delta)$$
for any $n \in \Z$ as wanted. 
This also proves that the restriction morphism 
$$\bfH_S^{\taut}(\alpha+n\delta) \to \bfH_S^{\taut,\ss}(\alpha+n\delta)$$ 
is an isomorphism. We are done. \qed

\medskip

We provide another proof of Theorem~\ref{thm:Markman} in Appendix~\ref{app:G} (valid under a weak positivity assumption on the class $\alpha$) which does not depend on \cite{KKPS}.

\bigskip

\section{The case of Higgs bundles}

In this section, we state the analogs of our results for an open surface of the form $S=T^*C$ for $C$ a smooth, 
connected projective curve of genus $g$. In that case, a pure, properly supported coherent sheaf is either of dimension zero or of dimension 
one, so that it may be identified with a Higgs bundle on $C$ via the Beauville--Narasimhan--Ramanan correspondence \cite{BNR_correspondence}. 
The case of zero-dimensional sheaves is described in \cite{MMSV}.
Here we focus on the case of Higgs bundles.
We have
$$\H^0(T^*C,\Q)=\Q, \quad \H^1(T^*C,\Q)=\H^1(C,\Q)=\Q^{2g}, \quad \H^2(T^*C,\Q)=\Q\eta, \quad \H^{>2}(T^*C,\Q)=\{0\}$$
where $\eta$ is the class of a fiber of the projection map $T^*C \to C$. Any effective $1$-cycle on $S$ with proper support 
is equivalent to $n [C]$ for some integer $n \geqslant 0$, where $[C]$ is the class of the zero section. 
For any pair $(r,d)\in \N \times \Z$, we let 
$\calM_{T^*C}(r,d)$ and $\calM^\ss_{T^*C}(r,d)$ be the moduli stack of rank $r$ and degree $d$ Higgs sheaves and semistable Higgs 
sheaves on $C$. They are both quasi-smooth derived Artin stacks of virtual dimension $\vd_\calM=2(g-1)r^2$. 
There is a good moduli space 
$p\colon  \calM^\ss_{T^*C}(r,d) \to M_{T^*C}(r,d)$ and a BPS sheaf $\mathcal{BPS}_{M_{T^*C}(r,d)}$ in $ \MHM(M_{T^*C}(r,d))$.  
The Hitchin map is
$$\pi\colon  M_{T^*C}(r,d) \xrightarrow{} \A_r\cong \A^{(g-1)r^2+1}.$$
The CoHA  is defined as in \S\ref{sec:COHA-S}
$$\bfH_C=\bigoplus_{r,d} \H_{-*+\vd_\calM}^\BM(\calM_{T^*C}(r,d),\Q).$$ 
For any slope $\mu \in \Q \cup \{\infty\}$ we also have the semistable COHA 
$$\bfH^\ss_C(\mu)=\bigoplus_{d/r=\mu} \H_{-*+\vd_\calM}^\BM(\calM^\ss_{T^*C}(r,d),\Q).$$

\smallskip
\begin{theorem}\label{thm: Higgs bundles case}
\hfill
\begin{enumerate}[label=$\mathrm{(\alph*)}$,leftmargin=8mm,itemsep=2mm]
\item The algebras $\bfH_{C}$ and $\bfH^\ss_C(\mu)$ have compatible cocommutative coproducts 
$$\Delta\colon  \bfH_C \to \bfH_C\, \wotimes \,\bfH_C
,\quad
\Delta^\ss\colon  \bfH^\ss_C(\mu) \to \bfH^\ss_C(\mu)\, \wotimes\, \bfH^\ss_C(\mu).$$
\item
The Lie algebras of 
primitive elements of $\bfH_C$ and $\bfH^\ss_C(\mu)$ are
$$\frakg_C=\bigoplus_{r,d}\frakg_C(r,d)
,\quad
\frakg_C^\ss(\mu)=\bigoplus_{d/r=\mu}\frakg^\ss_C(r,d).$$ 
The multiplication gives a dense embedding and an isomorphism of graded 
vector spaces
$$\U(\frakg_C) \to \bfH_{C}, \quad \U(\frakg_C^\ss(\mu)) \cong \bfH^\ss_C(\mu).$$
The restriction yields an isomorphism of graded Lie algebras
$$\res^\ss\colon \bigoplus_{d/r=\mu}\frakg_C(r,d) \cong \frakg_C^\ss(\mu).$$
\item For any pair $(r,d)$, the adjoint action of $D_{1,0}(\eta)\in\frakg_C(0,1)$ yields a graded vector space isomorphism
$$\frakg_C(r,d) \cong \frakg_C(r,d+1).$$
\item For any pair $(r,d)$, the Hecke operator $T_{D_{1,0}(\eta)}$ gives an isomorphism in $\MHM(\A_r)$
$$\pi_*\mathcal{BPS}_{M_{T^*C}(r,d)} \cong \pi_*\mathcal{BPS}_{M_{T^*C}(r,d+1)}.$$
\item Assume that $g>1$. Then, we have
$$\frakg_C(r,d)\cong \frakg^{\geqslant 1}_C(r,d)=\H_*^\taut(\calM_{T^*C}^{\geqslant 1}(r,d),\Q).$$
\end{enumerate}	
\qed
\end{theorem}

The proofs of all of the above statements are either simpler versions or straightforward modifications of the arguments used in the case of 
projective surfaces. The fact that $\pi_*\mathcal{BPS}_{M_{T^*C}(r,d)}$ is independent of $d$ was proved in \cite{KinjoKoseki} by 
other means. Theorem~\ref{thm: Higgs bundles case} provides a more conceptual explanation of the independence of $d$ phenomenon 
observed by Hausel and Thaddeus \cite{HauselThaddeus} in the context of moduli spaces of stable Higgs bundles.

\bigskip

{\centerline{\textbf{Acknowledgments}}}

\medskip

We would like to thank the following people for enlightening discussions and correspondence:
M.~Aprodu, A.~Bayer, J.~Bi, P.~Descombes, D.~Fratila, B.~Hennion, E.~Macr\`i, D.~Maulik, T.~P\u{a}durariu, M.~Porta, F.~Sala and N. Seroux. We are especially grateful to J.~Shen and Y.~Toda for discussions and for encouraging us to consider the relative version of $\chi$-independence, and to K.~Costello for conversations regarding the coproduct that inspired one of the main constructions in this paper.
B. D is supported by the EPSRC grant UKRI3202 {`BPS cohomology in geometry and representation theory'}.
T. K is supported by JSPS KAKENHI Grant Number 25K17229.
O.S is in part supported by the PNRR grant  CF 44/14.11.2022 {`COHAs of smooth surfaces and applications'} 
and would like to thank the Simion Stoilow Institute of Mathematics for great working conditions.
E.V. is partially supported by the ANR grant ANR-24-CE40-3389 (GRAW).

\bigskip

\appendix

\section{Pro-objects and K\"unneth theorem}\label{app:A}

Since we work with non-quasi-compact stacks, it is useful to use pro-objects to keep track of the filtrations induced from the system of quasi-compact open substacks.
For this purpose, we will collect some basic facts about pro-objects in $\infty$-categories that were used in the main body of the text.
See \cite[Appendix A]{PortaSala} for a similar discussion.
One of the main goals of this appendix is to prove the K\"unneth theorem for non-quasi-compact stacks using the complete tensor product on pro-objects.

\subsection{Pro-objects}\label{sec-pro}

Let $\calC$ be an accessible $\infty$-category which admits finite limits.
Let $\mathsf{S}$ be the $\infty$-category of spaces, in the sense of \cite[\S 1.2.16]{lurie2009higher}.
Following \cite[def.~3.1.1]{lurie2011homotopy}, we define the 
$\infty$-category of pro-objects $\mathsf{Pro}(\calC)$  in $\calC$ (=the pro-category of $\calC$) to be the full-subcategory of 
$\mathsf{Fun}(\calC, \mathsf{S})^{\mathrm{op}}$ consisting of accessible functors preserving finite limits.
It follows from the proof of \cite[prop.~3.1.6]{lurie2011homotopy} that an object $F$ in $\mathsf{Fun}(\calC, \mathsf{S})^{\mathrm{op}}$ belongs to $\mathsf{Pro}(\calC)$ if and only if $F$ is a filtered limit of representable functors.
For such a filtered system $F \colon I \to \calC$, let $\Lim{I} F(i)$ denote the limit in $\mathsf{Pro}(\calC)$.
As in \cite[rem.~3.1.7]{lurie2011homotopy}, for a functor $f \colon \mathcal{C} \to \mathcal{D}$ between accessible $\infty$-categories with finite limits, 
one can naturally define $\mathsf{Pro}(f) \colon \mathsf{Pro}(\mathcal{C}) \to \mathsf{Pro}(\mathcal{D})$.
If there is no risk of confusion, $\mathsf{Pro}(f)$ will be simply denoted by $f$.

Assume that $\calC$ is a stable $\infty$-category equipped with a t-structure and $\calC^{\heartsuit}$ be the heart of the t-structure.
Then the $i$-th cohomology functor $\mathrm{H}^{i}(-) \colon \calC \to \calC^{\heartsuit}$ extends to 
\[
	\mathrm{H}^{i}(-) \colon \Pro(\calC) \to \Pro(\calC^{\heartsuit}).
\]
By \cite[thm~8.6.5]{KS06}, the category $\Pro(\calC^{\heartsuit})$ is abelian. We define
\[
\mathrm{H}^{*}(-) \coloneqq\bigoplus_{i\in\Z} \mathrm{H}^{i}(-)[-i] \colon \Pro(\calC) \to \Pro(\calC^{\heartsuit})^{\mathbb{Z}-\mathrm{gr}},
\]
where $\Pro(\calC^{\heartsuit})^{\mathbb{Z}-\mathrm{gr}}$ denotes the category of $\mathbb{Z}$-graded objects in $\Pro(\calC^{\heartsuit})$.

If $\calC$ is equipped with a symmetric monoidal product $\otimes$,
by \cite[rem.~2.4.2.7, prop.~4.8.1.10]{lurie2017higher}, the $\infty$-category  $\mathsf{Pro}(\calC)$ is equipped with a symmetric monoidal 
product $\widehat{\otimes}$ which extends the monoidal product $\otimes$ under the Yoneda embedding 
$\calC\to\mathsf{Pro}(\calC)$ and which preserves the small filtered limits separately in each variable.
We refer to $\widehat{\otimes}$ as the {complete tensor product}.

Let $\calX$ be a derived Artin stack and $f \colon \calX \to B$ a morphism to a quasi-compact derived Artin stack.
For any object $\calF$ in $ \Dhcmon(\calX)$ we define the pro-direct image by
\[
 f^{\pro}_* \calF = \Lim{\mathcal{U} } (f |_{\calU})_* (\calF|_{\calU}) \in \Pro(\Dhcmon(B))
\]
where $\calU$ runs over all quasi-compact open substacks of $\calX$.
The realization functor 
$\Pro(\Dhcmon(B)) \to \Dh^{\mathsf{mon}}(\bullet)$ 
applied to the complex $f^{\pro}_* \calF$ recovers the pushforward complex $f_* \calF$.
When $B$ is the point, the pro-direct image will be denoted as
\[
R \Gamma_\pro(\calX, \mathcal{F})    = \Lim{\mathcal{U} } R \Gamma(\calU, \calF |_{\calU}) \in \Pro(\Dhcmon(\bullet)).
\]

Let $\D\Q_{\calX}$ be the dualizing complex. We set 
\begin{align}\label{RGproBM}
	R \Gamma_\pro^\BM(\calX,\Q) = R \Gamma_\pro(\calX, \D \Q_{\calX})
	,\quad
	R \Gamma^\BM(\calX,\Q) = R \Gamma(\calX, \D \Q_{\calX}).
\end{align} 
Let $\H_{*}^\BM(\calX,\Q)_\pro$ and $\H_{*}^\BM(\calX,\Q)$ 
be the cohomology space of the complexes $R \Gamma_\pro^\BM(\calX,\Q) $ and
$R \Gamma^\BM(\calX,\Q) $. The vector space $\H_{*}^\BM(\calX,\Q)$ is equipped with the pro-topology,
which is the coarsest topology such that the projections
$\H_{*}^\BM(\calX,\Q)\to\H_{*}^\BM(\calU,\Q)$ are continuous for each quasi-compact open subset $\calU\subset\calX$.

\subsection{K\"unneth theorem}

We have the following version of the K\"unneth theorem.

\begin{proposition}\label{prop-kunneth}
	Let $\calX_1$ and $\calX_2$ be derived Artin stacks with affine diagonal and $f_1 \colon \calX_1 \to B_1$ and $f_2 \colon \calX_2 \to B_2$ be morphisms to quasi-projective schemes.
	Let $\calF_1 \in \Dhcmonb(\calX_1)$ and $\calF_2 \in \Dhcmonb(\calX_2)$ be monodromic mixed Hodge complexes.
	Then the obvious map yields an isomorphism
	\[
	 (f_1)^{\pro}_* \calF_1 \widehat{\boxtimes} (f_2)^{\pro}_* \calF_2 \cong (f_1 \times f_2)^{\pro}_* (\calF_1 \boxtimes \calF_2).
	\]
\end{proposition}

\begin{proof}
	For the proof, we may assume that $\calX_1$ and $\calX_2$ are quasi-compact.
	Also, it is enough to show that the obvious map is invertible on the underlying dg-vector space after 
	forgetting the monodromic mixed Hodge structure.
	In this case, the statement is proved in \cite[prop.~A.5.19]{gaitsgory2019weil}.
\end{proof}

\begin{corollary}\label{cor:Kunneth}
For derived Artin stacks $\calX_1$ and $\calX_2$ with affine diagonal, there is a natural isomorphism 
\[
R \Gamma^\BM(\calX_1,\Q)_\pro \,\widehat{\otimes}\, R \Gamma_\pro^\BM(\calX_2,\Q) \cong 
R \Gamma^\BM(\calX_1 \times \calX_2,\Q)_\pro.
\]
\qed
\end{corollary}

\bigskip

\section{The HN-stratification of arbitrary sheaves}\label{app:C}

\medskip

\subsection{Proof of Proposition~\ref{prop:finite interval partial}.}\label{app:finite fibers}

We begin with (a). From the construction it is clear that the set of pairs of convex paths $(\rho^2,\rho^1)$ fitting between 
$(\rho^2_{\uunu},\rho^1_{\uunu})$ and $(\rho^2_{\uunu'},\rho^1_{\uunu'})$
is finite. Thus it is enough to show that the map which assigns to $\underline{\sigma}\in \HN(\nu)$ the pair of paths 
$(\rho^2_{\underline{\sigma}},\rho^1_{\underline{\sigma}})$ has finitely many fibers corresponding to effective HN-types. We will deduce this from the Bogomolov inequality. Fix a 
pair $(\rho^2,\rho^1)$ and describe all potential HN-types $\uunu=(\uunu_2,\uunu_1,\nu_0)$ mapping to 
it. We write $\nu_1,\nu_2$ for the weights of $\uunu_1,\uunu_2$ respectively.  
We use the notation $\uunu_i=(\uunu_{i,1},\ldots,\uunu_{i,s})$.

First note that $\nu_0$ is uniquely determined (as the length of the vertical path in $\rho^1$), as is $\ch_2(\nu_1)$. It follows that $\ch_2(\nu_2)$ is also fully determined. Denoting by $\pi\colon  \NS(S)_\mathbb{R} \to \mathbb{R}\omega^\perp$ the orthogonal 
projection, what remains to be determined are the values of $\pi(\uunu_{1,i})$, $\pi(\uunu_{2,j})$ as well as $\ch_2(\uunu_{2,j})$, for all $i,j$. 
Note that by Hodge's theorem, for any $a \in \NS(S)$ we have 
$$a^2 = \frac{(a\cdot \omega)^2}{\omega^2} +\pi(a)^2 \leqslant \frac{(a\cdot \omega)^2}{\omega^2}.$$ 
Using this, Bogomolov's inequality yields upper bounds
$$\frac{(c_1(\uunu_{2,j})\cdot \omega)^2}{2\rk(\uunu_{2,j}) \omega^2} \geqslant \ch_2(\uunu_{2,j})$$
for all $j$ and for the $\ch_2(\uunu_{2,j})$ which, coupled with the relation $\sum_j \ch_2(\uunu_{2,j})=\ch_2(\nu_2)$, cuts out a finite set of 
possible values for the tuple $(\ch_2(\uunu_{2,j}))_j$. Next, for each such choice of tuple, Bogomolov's inequality again implies that the 
elements $\pi(\uunu_{2,j})$ satisfy the following inequality for all $j$
$$\pi(\uunu_{2,j})^2 \geqslant 2\rk(\uunu_{2,j})\ch_2(\uunu_{2,j})-\frac{(c_1(\uunu_{2,j})\cdot\omega)^2}{\omega^2}.$$
As the restriction of the intersection pairing on $\mathbb{R}\omega^\perp$ is negative definite, these inequalities have finitely many 
solutions. Finally, we observe that once the tuple $(\uunu_{2,j})_j$ has been fixed, $\nu_1$, and in particular $c_1(\nu_1)$ is uniquely 
determined. But then, as $\c_1(\uunu_{1,i}) \in \NS(S)^{\mathrm{eff}}$ for all $i$, there remain only finitely many possible choices for each 
$c_1(\uunu_{1,i})$. This proves (a).

\begin{lemma}\label{lem:appC}
	Let $S$ be a countable set, and let $\preccurlyeq$ be a partial order on $S$. Put $S_{\preccurlyeq s}\coloneqq \{t \in S\;|\; t\preccurlyeq s\}$. Assume that for any $s \in S$ the cardinality of $S_{\preccurlyeq s}$ is finite.
	Then there is a total order $\leqslant$ refining $\preccurlyeq$ which is locally finite and contains a smallest element, equivalently there is an order preserving bijective morphism $S \cong \N$.
\end{lemma}

\begin{proof} By our assumption the set of minimal element is nonempty; let $s_0$ be one such minimal element. Let us begin by fixing an arbitrary bijection $\gamma\colon \N \to S$ sending $0$ to $s_0$. We build a total order $\leqslant$ on $S$ by the following 
	inductive process. We begin by extending the partial order by requiring that $\gamma(0) < S^{(0)}\coloneqq S \smallsetminus\{\gamma(0)\}$ and we refine the restriction of $\preccurlyeq$ to $S_{\preccurlyeq\gamma(1)} \cap S^{(0)}$ to a total order. We set 
	$$S^{(1)}= S^{(0)} \smallsetminus S_{\preccurlyeq\gamma(1)}\ .$$
	By construction, for any $t \not\in S^{(1)}$ and $u \in S^{(1)}$ we have $u \not\preccurlyeq t$. Thus we may
	extend the partial order by requiring
	$$\{\gamma(0)\} \cup  S_{\preccurlyeq \gamma(1)} < S^{(1)}.$$
	Next, we refine the partial ordering on the (finite) poset $S_{\preccurlyeq\gamma(2)} \cap S^{(1)}$ into a total order $\leqslant$, and we set
	$$ S^{(2)}= S^{(1)} \smallsetminus \left(S_{\preccurlyeq \gamma(2)} \cap S^{(1)}\right),$$
	$$ \left(S_{\preccurlyeq \gamma(2)} \cap S^{(1)}\right) < S^{(2)}.$$
	Iterating the process we build a total order on $S$ compatible with $\preccurlyeq$. Observe that $\bigcap_i S^{(i)}=\emptyset$ as $\gamma(i) \not\in S^{(i)}$ by construction. In addition, $<$ has finite intervals since for any $i$, 
	$$\{t\;|\;\gamma(0)<t<\gamma(i)\} \subset \bigcup_{j\leqslant i} S_{\preccurlyeq\gamma(j)}.$$
\end{proof}

To prove (b), it now only remains to check that for any fixed HN-type $\uunu \in \text{HN}(\nu)$, the cardinality of the set $\text{HN}(\nu)_{\preccurlyeq \uunu}\coloneqq \{\underline{\sigma} \in \text{HN}(\nu)\;|\; \underline{\sigma} \preccurlyeq \uunu\}$ is finite. This amounts to proving that there are only finitely many HN-types $\underline{\sigma}$ of weight $\nu$ whose associated pair of paths $(\rho^2,\rho^1)$ lies above $(\rho^2_{\uunu},\rho^1_{\uunu})$. For this one may argue as follows. First of all, note that there are only finitely many convex paths $\rho^2$ lying above $\rho^2_\uunu$. Next, as $\rho^1$ lies above $\rho^1_{\uunu}$, $\ch_2(\underline{\sigma}_1)$ is bounded above. By the same argument as in (a) using the Bogomolov inequality, this implies that there are only finitely many possible choices for the collection of Mukai vectors $(\underline{\sigma}_2)_j$, and hence also for $\ch_2(\underline{\sigma}_1)$. It then follows that there are finitely many possible choices for the collection $(\underline{\sigma}_1)_j$ as well, as desired. \qed

\medskip

\subsection{Proof of Proposition $\ref{prop:good HN order}$}
Statement (b) follows from (a) together with \cite[thm.~3.3.7]{HL} and Proposition~\ref{prop:finite interval partial} (a). Statement (c) 
then also follows, hence we only need to prove (a). Since $\calM^+_S(\succcurlyeq\underline{\gamma}) \subseteq \calM^+_S(\succcurlyeq\uunu)$ 
whenever $\underline{\gamma} \succcurlyeq \uunu$, it is enough to prove that $\calM^+_S(\succcurlyeq\uunu)$ contains the Zariski 
closure of $\calM_S(\uunu)$ for any $\uunu$. Consider the morphism $\ev_{\uunu}\colon  \Filt_\uunu \to \calM_S(\nu)$ which is 
representable and proper.

The map $\ev_{\uunu}$ is an isomorphism over $\calM_S^+(\uunu)$, in particular 
$\calM_S^+(\uunu) \subset \text{Im}(\ev_{\uunu})$. Thus it is enough to prove the following

\begin{lemma}\label{lem:ev and HN order} Let $\nu$ be a Mukai vector and $\uunu$ an arbitrary HN-type of weight $\nu$. The map $\ev_{\uunu}\colon  \Filt_\uunu \to \calM_S(\nu)$ factors
	through the inclusion $\calM_S^+(\succcurlyeq\uunu) \to \calM_S(\nu)$.
\end{lemma}
\begin{proof}Let us fix a filtration 
	$\mathcal{F}_\bullet =\calF_S \subset \cdots \subset \calF_1 \in \Filt_\uunu$  and show that $\mathcal{F}_1$ lies in a HN-stratum 
	$\calM_S^+(\underline{\beta})$ satisfying $\underline{\beta} \succcurlyeq \uunu$. 
	We will make use of the following standard observation, which is a direct consequence of Theorem~\ref{prop: purity-ss}. 
	Let $\mathcal{G}, \mathcal{H}$ be two pure two dimensional semistable coherent sheaves, of Mukai vectors $\alpha,\beta$ respectively. 
	Assume that $\mu^2(\alpha) > \mu^2(\beta)$. Then for any extension $ 0 \to \mathcal{H} \to \mathcal{J} \to \mathcal{G} \to 0$, the path  
	$\rho^2(\calJ)$ lies in the parallelogram with sides 
	$\rho^2_{(\alpha)}=(\rk(\alpha), c_1(\alpha)\cdot \omega)$ and $\rho^2_{(\beta)}= (\rk(\beta), c_1(\beta)\cdot \omega)$, 
	see the example below 
	
	
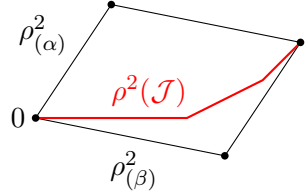
\vspace{.15in}
	
	{ \hspace{2in}
		\begin{tikzpicture}
			\draw[black] (-0.5,1) --(0.5,2.5) -- (3,2) --  (2,0.5) -- (-0.5,1);
			\draw[red, thick] (-0.5,1) -- (1.5,1) -- (2.5,1.5) --(3,2);
			\fill[black] (-0.5,1) circle (0.05cm) node[anchor=east] {$0$};
			\fill[black] (0.5,2.5) circle (0.05cm);
			\fill[black] (-0.4,1.7) circle (0.0cm) node[anchor=south] {$\rho^2_{(\alpha)}$};
			\fill[black] (3,2) circle (0.05cm);
			\fill[black] (2,0.5) circle (0.05cm);
			\fill[black] (0.8,0.7) circle (0.0cm) node[anchor=north] {$\rho^2_{(\beta)}$};
			\fill[black] (1,1.7) circle (0.0cm) node[anchor=north] {${\color{red}{\rho^2(\calJ)}}$};
	\end{tikzpicture}}
	
	The same estimate on the HN-type of an extension of semistable one-dimensional sheaves holds, with the paths $\rho^2$ being replaced by the paths $\rho^1$. The situation is the same for an extension $0 \to \calH \to \calJ \to \calG \to 0$ with $\calH$ being a semistable 
	$d$-dimensional sheaf and $\calG$ a semistable $e$ dimensional sheaf with $d>e$; the only difference is that the $\mu^d$-slope of $\calG$ being infinite, the corresponding path is vertical as in the example in the figure below: 
	
	
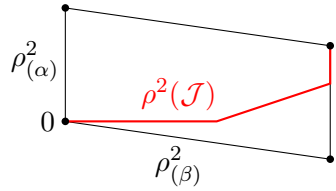
\vspace{.15in}
	
	{ \hspace{2in}
		\begin{tikzpicture}
			\draw[black] (-0.5,1) --(-0.5,2.5) -- (3,2) --  (3,0.5) -- (-0.5,1);
			\draw[red, thick] (-0.5,1) -- (1.5,1) -- (3,1.5) -- (3,1.5) -- (3,2);
			\fill[black] (-0.5,1) circle (0.05cm) node[anchor=east] {$0$};
			\fill[black] (-0.5,2.5) circle (0.05cm);
			\fill[black] (-0.9,1.4) circle (0.0cm) node[anchor=south] {$\rho^2_{(\alpha)}$};
			\fill[black] (3,2) circle (0.05cm);
			\fill[black] (3,0.5) circle (0.05cm);
			\fill[black] (1,0.8) circle (0.0cm) node[anchor=north] {$\rho^2_{(\beta)}$};
			\fill[black] (1,1.7) circle (0.0cm) node[anchor=north] {${\color{red}{\rho^2(\calJ)}}$};
	\end{tikzpicture}}
	
	In order to estimate the HN-type of $\mathcal{F}_1$, we use the above observation repeatedly. We first consider the concatenation of the paths $\rho^2(\mathcal{F}_i / \mathcal{F}_{i+1})$ and denote by $a_1, a_2, \ldots, a_t$ the successive line segments in the path obtained, as in the picture below. To each segment $a_i$ there corresponds a semistable sheaf $\mathcal{G}_i$, (which is a subquotient of $\mathcal{F}_1$), and to each pair $a_i, a_{i+1}$ corresponds an extension between $\mathcal{G}_i$ and $\mathcal{G}_{i+1}$. 
	
	
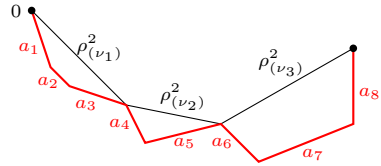
\vspace{.05in}
	
	{ \hspace{2in}
		\begin{tikzpicture}
			\draw[black] (-0.5,2.5) --(0.75,1.25) -- (2,1) --  (3.75,2) ;
			\draw[red, thick] (-0.5,2.5) -- (-0.25,1.75) -- (0,1.5) --(0.75,1.25) -- (1, 0.75) -- (2,1) -- (2.5, 0.5) -- (3.75, 1) -- (3.75, 2);
			\fill[black] (-0.5,2.5) circle (0.05cm) node[anchor=east] {\tiny{$0$}};
			\fill[black] (3.75,2) circle (0.05cm);
			\fill[black] (0.4,1.7) circle (0.0cm) node[anchor=south] {\tiny{$\rho^2_{(\nu_1)}$}};
			\fill[black] (1.5,1.05) circle (0.0cm) node[anchor=south] {\tiny{$\rho^2_{(\nu_2)}$}};
			\fill[black] (2.8,1.5) circle (0.0cm) node[anchor=south] {\tiny{$\rho^2_{(\nu_3)}$}};
			\fill[black] (-0.25,2) circle (0.0cm) node[anchor=east] {\color{red}{\tiny{$a_1$}}};
			\fill[black] (-0,1.55) circle (0.0cm) node[anchor=east] {\color{red}{\tiny{$a_2$}}};
			\fill[black] (0.5,1.25) circle (0.0cm) node[anchor=east] {\color{red}{\tiny{$a_3$}}};
			\fill[black] (0.95,1) circle (0.0cm) node[anchor=east] {\color{red}{\tiny{$a_4$}}};
			\fill[black] (1.8,0.75) circle (0.0cm) node[anchor=east] {\color{red}{\tiny{$a_5$}}};
			\fill[black] (2.3,0.75) circle (0.0cm) node[anchor=east] {\color{red}{\tiny{$a_6$}}};
			\fill[black] (3.5,0.6) circle (0.0cm) node[anchor=east] {\color{red}{\tiny{$a_7$}}};
			\fill[black] (4.25,1.4) circle (0.0cm) node[anchor=east] {\color{red}{\tiny{$a_8$}}};
	\end{tikzpicture}}

	\medskip
	
	If the path $\underline{a}=(a_1, \ldots, a_t)$ is convex then it is necessarily equal to $\rho^2(\mathcal{F}_1)$. Note that this path lies, by construction, entirely below the path $\rho^2_{\uunu}$. If on the other hand, say $\mu^2(a_i) > \mu^2(a_{i+1})$ then by the observation above the HN-polygon of the extension corresponding to $a_i,a_{i+1}$ is a convex path $b_1, \ldots, b_{n}$ lying entirely in the parallelogram with sides $a_i, a_{i+1}$, as in the figure below.
	
	
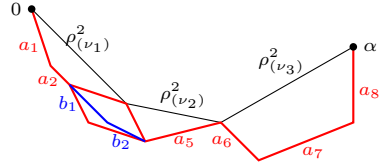
\vspace{.05in}
	
	{ \hspace{2in}
		\begin{tikzpicture}
			\draw[black] (-0.5,2.5) --(0.75,1.25) -- (2,1) --  (3.75,2) ;
			\draw[red, thick] (-0.5,2.5) -- (-0.25,1.75) -- (0,1.5) --(0.25,1) -- (1, 0.75) -- (2,1) -- (2.5, 0.5) -- (3.75, 1) -- (3.75, 2);
			\draw[red, thick] (0,1.5) --(0.75,1.25) -- (1, 0.75);
			\draw[blue, thick] (0,1.5) --(0.5,1) -- (1, 0.75);
			\fill[black] (-0.5,2.5) circle (0.05cm) node[anchor=east] {\tiny{$0$}};
			\fill[black] (3.75,2) circle (0.05cm) node[anchor=west] {\tiny{$\alpha$}};
			\fill[black] (0.25,1.8) circle (0.0cm) node[anchor=south] {\tiny{$\rho^2_{(\nu_1)}$}};
			\fill[black] (1.5,1.05) circle (0.0cm) node[anchor=south] {\tiny{$\rho^2_{(\nu_2)}$}};
			\fill[black] (2.8,1.5) circle (0.0cm) node[anchor=south] {\tiny{$\rho^2_{(\nu_3)}$}};
			\fill[black] (-0.25,2) circle (0.0cm) node[anchor=east] {\color{red}{\tiny{$a_1$}}};
			\fill[black] (-0,1.55) circle (0.0cm) node[anchor=east] {\color{red}{\tiny{$a_2$}}};
			\fill[black] (0.25,1.25) circle (0.0cm) node[anchor=east] {\color{blue}{\tiny{$b_1$}}};
			\fill[black] (0.95,0.75) circle (0.0cm) node[anchor=east] {\color{blue}{\tiny{$b_2$}}};
			\fill[black] (1.8,0.75) circle (0.0cm) node[anchor=east] {\color{red}{\tiny{$a_5$}}};
			\fill[black] (2.3,0.75) circle (0.0cm) node[anchor=east] {\color{red}{\tiny{$a_6$}}};
			\fill[black] (3.5,0.6) circle (0.0cm) node[anchor=east] {\color{red}{\tiny{$a_7$}}};
			\fill[black] (4.25,1.4) circle (0.0cm) node[anchor=east] {\color{red}{\tiny{$a_8$}}};
	\end{tikzpicture}}

	\medskip
	
	Since this parallelogram lies below the path $\uunu$ it follows that the new path $\underline{a}'$ obtained by
	replacing in $(a_1, \ldots, a_t)$ the subpath $(a_i,a_{i+1})$ with $(b_1, \ldots, b_n)$ also lies below $\uunu$. 
	This new path $\underline{a}'$ is a better approximation of $\rho^2(\mathcal{F}_1)$ in the sense that it is more convex than 
	$(a_1, \ldots, a_t)$. This can be made precise as follows. For any path $\underline{u}=(u_1, \ldots, u_k)$ let us denote by 
	$\#\underline{u}$ the path obtained by rearranging the segments of $\underline{u}$ in increasing order of slope 
	(so that $\#\underline{u}$ lies below $\underline{u}$), and let $d(\underline{u})$ be the area of the polygon bounded by 
	$\underline{u}$ and $\#\underline{u}$. Note that 
	$d(\underline{u})=0$ if and only if $\underline{u}$ is convex. It is easy to see that if $\underline{a}$ is not convex then the polygon 
	bounded by $\#\underline{a}'$ and $\underline{a}'$ lies inside the polygon bounded by $\underline{a}$ and $\#\underline{a}$. It 
	follows that $d(\underline{a}') < d(\underline{a})$. Iterating this process yields a sequence of nested rational polygons whose areas 
	strictly decrease. Note that as $\omega$ is ample, the possible endpoints of paths forms a discrete set. Hence the above straightening process has to terminate after a finite number of steps -- that is, results in a convex path --- and converges to $\rho^2(\mathcal{F}_1)$, 
	which therefore lies below $\rho^2_{\uunu}$ as wanted. Note that the resulting path may have a longer vertical 
	segment at the end than $\rho^2_{\uunu}$, but it contains that segment. We now apply the exact same straightening 
	process to compute $\rho^1(\calF_1)$, and deduce that $\rho^1_{\uunu}$ lies above $\rho^1(\calF_1)$. 
	Note that if $\rho^i_{\uunu}=\rho^i(\calF_1)$ for $i=1,2$ then necessarily $\calF_1 \in \calM_S^+(\uunu)$. 
	This finishes the proof of Lemma~\ref{lem:ev and HN order} and Proposition~\ref{prop:good HN order}.
\end{proof}

\section{Proof of Proposition~\ref{prop:hecke acts on reduced tautological}}\label{app:F}

Consider the following large commutative diagram, which involves several partial classical truncations of the diagram defining length one Hecke operators

\begin{equation}\label{diag:partialtruncations}
\xymatrix{
\overline\calM_S^\ex(\delta,\nu)^\rd \ar[r] \ar[d] & \calM_S(\delta)^\cl \times \calM_S(\nu)^\rd \ar[d] \ar[r] & \Pic(S)^\cl \times \Pic(S)^\cl \ar[r]^-{*} \ar[d] & \Pic(S)^\cl \ar[d]\\
\overline\calM_S^\ex(\delta,\nu) \ar[r] \ar[d] & \calM_S(\delta)^\cl \times \calM_S(\nu) \ar[d] \ar[r] & \Pic(S)^\cl \times \Pic(S) \ar[r]^-{*} \ar[d] & \Pic(S) \ar@{=}[d]\\
{\calM_S}^\ex(\delta,\nu) \ar[r]^-{\gr} \ar[d]^-{\ev} & \calM_S(\delta) \times \calM_S(\nu) \ar[r] & \Pic(S) \times \Pic(S) \ar[r]^-{*}  & \Pic(S) \ar@{=}[d]\\
\calM_S(\gamma) \ar[rrr] &&& \Pic(s)\\
\calM_S(\gamma)^\rd \ar[u] \ar[rrr] &&& \Pic(s)^\cl \ar[u]
}
\end{equation}

where
\begin{itemize}
	\item $\gamma=\nu+\delta$ and $\gr, \ev$ are defined as in \eqref{induction-diagram-S},
	\item all the squares in the top row, the leftmost square in the second row and the bottom square are Cartesian (this defines $\overline\calM_S^\ex(\delta,\nu)$ and $\overline\calM_S^\ex(\delta,\nu)^\rd$),
	\item the morphism $*$ is the group stack structure on $\Pic(S),$ $\Pic(S)^\cl$.
\end{itemize}
The Cartesianity of the upper rightmost square comes from the fact that $\Pic(S)^\cl$ is a closed group substack of $\Pic(S)$.

From \eqref{diag:partialtruncations} we extract a reduced\footnote{Here the word `reduced' is used in relation to the {reduced fundamental class}, not in the usual sense of `reduced' in derived algebraic geometry.} and (partially) truncated version of the induction diagram \eqref{induction-diagram-S}
$$\xymatrix{\calM_S(\delta)^\cl \times \calM_S(\nu)^\rd & \overline\calM_S^\ex(\delta,\nu)^\rd \ar[r]^-{\ev^\rd} \ar[l]_-{\gr^\rd}& \calM_S(\gamma)^\rd}.$$
Restricting to the open substack of sheaves containing no zero-dimensional subsheaf as in \cite[Sec.2]{MMSV} yields a diagram
\begin{equation}\label{diag:restrictedinducdimone}
	\xymatrix{\calM_S(\delta)^\cl \times \calM_S^{\geqslant 1}(\nu)^\rd& X(\delta,\nu) \ar[r]^-{\ev'} \ar[l]_-{\gr'}& \calM_S^{\geqslant 1}(\gamma)^\rd}.
\end{equation}
As in [{loc.cit.}, Prop. 2.2], $X(\delta,\nu)$ is canonically isomorphic to (the base change of) the derived projectivizations of tautological sheaves $\calF_\nu=\calE_\nu^\vee \otimes K_S [1]$ and $\calE_\gamma$ on $\calM_S^{\geqslant 1}(\nu) \times S$, which are both of perfect amplitude $[-1,0]$; in particular $\gr',\ev'$ are quasi-smooth. 
The commutativity of \eqref{diag:partialtruncations} thus implies that \begin{equation}\label{eq:reducedclasses agree}
	(\gr')^!([\calM_S(\delta)^\cl] \boxtimes [\calM_S^{\geqslant 1}(\nu)^\rd])=(\ev')^!([\calM_S^{\geqslant 1}(\gamma)^\rd]).
\end{equation}

The computation of the action of length one Hecke operators on tautological classes conducted in [{loc.cit}, Sec.2] may now be carried out {verbatim} in the case of reduced tautological classes $\bigoplus_{n \in \Z} \Lambda_S \bullet [\calM_S^{\geqslant 1}(\alpha+n\delta)^\rd]$, yielding the desired statement.\qed

\bigskip

\section{Generalization of Markman's theorem for positive curve classes}\label{app:G}

\medskip
In this appendix, we give a proof of the following slightly weaker form of Theorem~\ref{thm:Markman}, which is independent of \cite{KKPS}. Denote by $\g_S^{\geqslant 1}(\sigma) \subseteq \bfH_S^{\geqslant 1}(\sigma)$ the image of $\g_S(\sigma)$ under the restriction morphism $\bfH_S \to \bfH_S^{\geqslant 1}$. Elements of $\g_S^{\geqslant 1}=\bigoplus_\sigma \g_S^{\geqslant 1}(\sigma)$ will be called \textit{quasi-primitive}.

\begin{theorem}\label{thm:weakMarkman} Let $\alpha \in \H^2(S,\Z)$ be the class of an effective $1$-cycle such that $\alpha^2>0$ and the linear system $|\alpha|$ contains an irreducible curve. Then
	$$\bfH_S^{\taut,\geqslant 1}(\alpha+n\delta) = \frakg^{\geqslant 1}_S(\alpha+n\delta)$$ 
	for any $n \in \Z$.
\end{theorem}

The proof of Theorem~\ref{thm:Markman} applies verbatim, once we have the following result:

\begin{proposition}\label{prop:quasiprimitivity fund class}
	Let $\alpha$ be the class of an effective one-cycle such that $\alpha^2>0$.
	Assume that $|\alpha|$ contains an irreducible curve. Then for any integer $n$, 
	the reduced virtual fundamental class $[\calM_S^{\geqslant 1}(\alpha+n\delta)^\rd]$ is quasi-primitive and coincides with 
	$[\calM^{\geqslant 1}_S(\alpha+n\delta)^\cl]$.
\end{proposition}

\medskip

\subsection{A criterion of quasi-primitivity.} We begin with the following simple observation.

\begin{lemma}\label{lem:quasiprimitive} Let $\nu \in \mathrm{H}^{\ev}(S,\Z)$ be an effective class.
	\begin{enumerate}[label=$\mathrm{(\alph*)}$,leftmargin=8mm,itemsep=2mm]
		\item The spaces 
		$\frakg^{\geqslant 1}_S(\nu)$ and $\frakg^{}_S(\nu)$ are preserved under multiplication by tautological classes.
		\item
		Let $y \in \bfH_S$ satisfy 
		\begin{equation}\label{eq:quasiprimcriteria}
			(\res^{\geqslant 1} \otimes \res^{\geqslant 1})\Delta(y)=(\res^{\geqslant 1}(y) \otimes 1) + (1 \otimes \res^{\geqslant 1}(y)).
		\end{equation}
		Then $\res^{\geqslant 1}(y)$ is quasi-primitive.
	\end{enumerate} 
\end{lemma}

\begin{proof}
	Part (a) follows from the compatibility of the coproduct and the operators of multiplication by tautological classes, see 
	Proposition~\ref{prop:compat} (b). Next, let $y\in \bfH_S(\nu)$ satisfy \eqref{eq:quasiprimcriteria}. We may adapt the inductive construction 
	used in the proof of Theorem~\ref{thm:restriction-isom}, part (c) to modify $y$ by elements supported on 
	$\calM_S(\nu) \smallsetminus \calM_S^{\geqslant 1}(\nu)$ to obtain a primitive element $y'\in \frakg_S(\nu)$ satisfying 
	$\res^{\geqslant 1}(y')=\res^{\geqslant 1}(y)$. This proves (b).
\end{proof}

\medskip

\subsection{Classical dimension estimates} We next provide some estimates of classical dimension for the stacks $\calM_S(\nu)$. Recall that for any Mukai vector $\nu$ such that the derived stack $\calM_S(\nu)$ is non empty we have
$$\dim\calM_S(\nu)=-(\nu,\nu).$$ 
If $\rk(\nu)=0$ then this simplifies to 
$$\dim\calM_S(\nu)=\nu^2.$$  As this is all that we need in this paper, we restrict our attention to one-dimensional sheaves.

\begin{proposition}\label{prop:estimateclassicaldim} Let $\nu  \in \NS(S)^{\mathrm{eff}}\oplus \Z\delta$ and assume that $\omega$ and $\nu$ are in general position.
\begin{enumerate}[label=$\mathrm{(\alph*)}$,leftmargin=8mm,itemsep=2mm]
\item If $\nu=n\nu_\pr$ with $\nu_\pr$ primitive and $\nu_\pr^2=-2$ then $\calM_S^\ss(\nu)$ has a single closed point corresponding to 
$\calE^{\oplus n}$ where $\calE$ is the unique stable sheaf of Mukai vector $\nu_\pr$. Thus $\dim \calM_S^{\ss}(\nu)^\cl=-n^2.$
\item If $\nu=n\nu_\pr$ with $\nu_\pr$ primitive and $\nu_\pr^2=0$ then $\calM_S^\ss(\nu)$ is irreducible and the generic element is a direct 
sum $\calE_1 \oplus \cdots \oplus \calE_n$ of distinct stable sheaves of class $\nu_\pr$. Thus $\dim \calM_S^{\ss}(\nu)^\cl=n.$
\item If $\nu^2>0$ then $\dim(t_0\calM_S^{ss}(\nu))=\nu^2+1$ and $\calM_S^{\ss}(\nu)$ is irreducible. 
If moreover $\nu=\alpha+l\delta$ with $\alpha\in\NS(S)^{\mathrm{eff}}$ and the linear system $\lvert \alpha\lvert$ contains an irreducible curve, then the stack $\calM_S(\nu)$ is irreducible. 
\end{enumerate}
\end{proposition}

\begin{proof}Statements (a) and (b) are well-known and due to Mukai and Yoshioka, see \cite{MukaiBombay, Yoshioka}.
We concentrate on (c).  
Let $\nu \in \NS(S)^{\mathrm{eff}}\oplus \Z\delta$ satisfy $\nu^2>0$. Assume first that $\nu=\alpha+l\delta$ is $\omega$-general and $l \neq 0$. 
We have $\dim\Ext^2(\calE,\calE)=\dim\Hom(\calE,\calE)=1$ for any stable sheaf $\calE$ hence 
$\dim \calM_S^{\mathrm{st}}(\nu)^\cl=\dim\calM_S^{\mathrm{st}}(\nu)+1=\nu^2+1$.
The irreducibility of $\calM_S^{\mathrm{st}}(\nu)$ is then given by 
\cite[thm.~4.4]{KaledinLehnSorger}. 
The local description of $\calM_S^\ss(\nu)$ over its coarse moduli space given in \cite[thm. 5.11]{D24} in terms of 
representations of preprojective algebras shows that  $\calM_S^{\mathrm{st}}(\nu)$ is dense in  $\calM_S^{\mathrm{ss}}(\nu)$, from which we 
deduce that $\calM_S^\ss(\nu)$ is itself irreducible. We turn to the second statement of (c). Let $\calK \subset \calM_S(\nu)$ be an irreducible 
component and let $\calU \subset \calK$ be open in $\calM_S(\nu)$. Then $\calU$ is quasi-smooth of dimension $\nu^2$. 
For any $\calF \in \calU$ we have $\H^1(\mathbb{T}_{\calM_S(\nu)})|_{\calF}=\Ext^2(\calF,\calF) \cong \mathrm{End}(\calF)^\vee$. 
In particular, $\dim \calK^\cl =\dim \calU^\cl \geqslant \dim \calU +1$. To finish the proof, it suffices, given the irreducibility of $\calM_S^\ss(\nu)$, 
to show that for any HN-type $\uunu \in \HN(\nu)$ of length greater than one we have $\dim \calM_S^\ss(\uunu)^\cl \leqslant \nu^2$. 
In light of Proposition~\ref{prop: purity-ss2} (b) this boils down to the inequality
\begin{equation}\label{eq:irredineq}
\nu^2 \geqslant \sum_{i<j} \nu_i \cdot \nu_j + \sum_i \dim \calM_S^\ss(\nu_i)^\cl.
\end{equation}
Setting $\partial(\gamma)= \dim \calM_S^\ss(\gamma)^\cl-\dim \calM_S^\ss(\gamma)$ 
for an effective Mukai vector $\gamma$, \eqref{eq:irredineq} is easily seen to be a special case of the following.
	
\begin{lemma}\label{lem:twoconnec}
Let $\nu=\alpha+l\delta \in \NS(S)^{\mathrm{eff}} \oplus \Z\delta$ satisfy $\nu^2>0$ and assume that the linear system $|\alpha|$ contains an 
irreducible curve. Let $\nu=\sum_i \alpha_i$ be a decomposition into a sum of at least two effective Mukai vectors. Then $\sum_i \partial(\alpha_i) \leqslant \sum_{i<j} \alpha_i \cdot \alpha_j$.
\end{lemma}

\begin{proof} 
Observe that for $\alpha =\alpha_1+\alpha_0$ with $\alpha_1 \in \NS(S)^{\mathrm{eff}}$ and $\alpha_0\in \Z\delta$ we have 
$\partial(\alpha_1) \geqslant \partial(\alpha)$; by Proposition \ref{prop:estimateclassicaldim} the number $\partial(\alpha)$ decreases as we decrease the divisibility of $\alpha$ . We may thus assume that $\nu$ and each of the $\alpha_i$ belongs to 
$\NS(S)^{\mathrm{eff}}$.
Let us order the classes $\alpha_i$ in such a way that $\alpha_i^2<0$ for $1 \leqslant i \leqslant l$, $\alpha_i^2=0$ for $l<i\leqslant n$ and 
$\alpha_i^2>0$ for $i>n$. Let us write $\alpha_i=n_i\beta_i$ with $\beta_i$ primitive for $i \leqslant n$. Thus,
$$\sum_i \partial(\alpha_i)=\sum_{i=1}^l n_i^2 + \sum_{i=l+1}^n n_i+\sum_{i >n}1.$$
We will now use the fact that under our hypotheses, $\nu$ is $2$-connected, i.e. for any decomposition $\nu=\alpha+\beta$ into nonzero 
effective classes $\alpha,\beta \in \NS(S)^{\mathrm{eff}}$ we have $\alpha \cdot \beta\geqslant 2$.  If $S$ is a K3 surface this is \cite[lem.~3.7]{SaintDonat}, while if $S$ is an Abelian surface it follows from \cite[lem.~10.4.6]{CAV}.
Put $\alpha_{\neq i}=\sum_{j \neq i} \alpha_j$. Fix $i \leqslant n$. Writing $\nu=\beta_i + \left((n_i-1)\beta_i + \alpha_{\neq i}\right)$ we get
$$2 \leqslant \beta_i \cdot \left((n_i-1)\beta_i + \alpha_{\neq i}\right)=-2(n_i-1)+\beta_i\cdot \alpha_{\neq i}$$
hence $\beta_i \cdot \alpha_{\neq i} \geqslant 2n_i$ and $\alpha_i\cdot \alpha_{\neq i} \geqslant 2n_i^2$. 
A similar reasoning yields $\alpha_i \cdot \alpha_{\neq i} \geqslant 2n_i$ for $l<i\leqslant n$. Thus
$$2\sum_{i <j} \alpha_i\cdot \alpha_j=\sum_i \alpha_i \cdot \alpha_{\neq i} \geqslant \sum_{i=1}^l 2n_i^2 + \sum_{i=l+1}^n 2n_i + \sum_{i>l} 2$$
which is what we wanted. This concludes the proof of Lemma~\ref{lem:twoconnec}. 
\end{proof}

To finish the proof of (c), it remains to observe that the irreducibility of $\calM_S(\nu)$ is independent of the choice of $\omega$ (which can 
hence be picked to be $\nu$-general), and obviously invariant under tensoring $\nu$ by any line bundle so that we may remove the condition 
$l\neq 0$ (notice also that the irreducibility of $\calM_S(\nu)$ implies that of $\calM^\ss(\nu)$ for any $\omega$). 
Proposition~\ref{prop:estimateclassicaldim} is proved. 
	
\end{proof}

\begin{remark}\hfill
	\begin{enumerate}[label=$\mathrm{(\alph*)}$,leftmargin=8mm,itemsep=2mm]
		\item Any ample class has an irreducible (even smooth) curve in its associated linear system hence Proposition~\ref{prop:estimateclassicaldim} (c) holds for ample classes.
		\item
		One can show that $\calM_S^{\geqslant 1}(\nu)$ is in fact irreducible whenever $\nu^2>0$. This is false in general when $\nu^2 \leqslant 0$. This is similar to the statement that the stack of Higgs bundles of fixed rank and degree on a curve of genus $g>1$ is irreducible, see, e.g., \cite[\S8]{MMSV}. We will not need this. 
	\end{enumerate}
\end{remark}

\medskip

\subsection{Proof of Proposition~\ref{prop:quasiprimitivity fund class}.} Finally, let $\gamma\in \NS(S)^{\mathrm{eff}}$ be a class satisfying 
$\gamma^2>0$ and whose associated linear system contains an irreducible curve. Let $d \in \Z$ and set $\nu=\gamma+d\delta$. 
By Proposition~\ref{prop:estimateclassicaldim} (c), the fundamental class of $\overline{\calM_S^{\geqslant 1}(\nu)}^\cl$ is of degree
$$\deg \left[\overline{\calM_S^{\geqslant 1}(\nu)}^\cl\right]=2\nu^2+2-\nu^2=\nu^2+2$$
(recall our shift convention \eqref{eq:degreeshiftin defCOHA} on the homological grading of $\bfH_S$). Let $\alpha,$ $\beta$ be dimension one 
effective classes such that $\nu=\alpha+\beta$. We will prove that $\bfH_S(\alpha) \wotimes \bfH_S(\beta)$ has no graded piece of degree 
$\nu^2+2$ hence the component of $\Delta\left(\left[\overline{\calM_S^{\geqslant 1}(\nu)}^\cl \right]\right)$ in 
$\bfH_S(\alpha) \wotimes \bfH_S(\beta)$ must be zero. This will imply that $\left[\overline{\calM_S^{\geqslant 1}(\nu)}^\cl \right]$ satisfies the 
conditions of Lemma~\ref{lem:quasiprimitive} (b), hence that $\left[\calM_S^{\geqslant 1}(\nu)^\cl \right]$ is quasi-primitive as wanted. It suffices to show that for any pair of purely one-dimensional HN-types $(\underline{\alpha},\underline{\beta})$ in $\HN(\alpha) \times \HN(\beta)$ we have
$$\nu^2+2 > \sum_i 2 \dim \calM_S^\ss(\alpha_i)^\cl+ 2\sum_{i<j} \alpha_i\cdot \alpha_j -\alpha^2+ \sum_k 2 \dim \calM_S^\ss(\beta_k)^\cl + 2\sum_{k<l} \beta_k\cdot \beta_l -\beta^2$$
which reduces to
$$\nu^2+2 > \sum_i \alpha_i^2 + 2\sum_i \partial(\alpha_i) + \sum_k \beta_k^2 + 2\sum_k \partial(\beta_k)$$
where we have put $\partial(\gamma)=\dim \calM_S^\ss(\gamma)^\cl -\dim\calM_S^\ss(\gamma)$. 
This last inequality is a direct application of Lemma~\ref{lem:twoconnec} for the decomposition $\nu=\sum_i \alpha_i+\sum_k\beta_k$. 
This finishes the proof of Proposition~\ref{prop:quasiprimitivity fund class}, and hence also of Theorem~\ref{thm:weakMarkman}.


\bigskip


\begin{thebibliography}{99}
	
	
	\bibitem{AHPS23}
	E. Ahlqvist, J. Hekking, M. Pernice, M. Savvas, \emph{Good moduli spaces in derived algebraic geometry}, arXiv:2309.16574 (2023).
	\bibitem{alper2023existence}
	J. Alper,  D. Halpern-Leistner, J. Heinloth,
	\emph{Existence of moduli spaces for algebraic stacks},
	Invent. Math.
	\textbf{234} (2023), 949--1038.
	
	\bibitem{BM14}
	A.~Bayer, E.~Macr\`i,
	\emph{MMP for moduli of sheaves on K3s via wall-crossing: nef
              and movable cones, Lagrangian fibrations}, Invent. Math. \textbf{198} (2014), 505--590.

	\bibitem{BNR_correspondence}
	A.~Beauville, M.~Narasimhan, S.~Ramanan,
	\emph{Spectral curves and the generalised theta divisor}, J. Reine Angew. Math. \textbf{398} (1989), 169--179.
	
	\bibitem{BBDG} 
	A. A. Beilinson, J. Bernstein, P. Deligne, O. Gabber,
	\emph{Faisceaux pervers}, Astérisque, \textbf{100},
	Société Mathématique de France, Paris, 1982, 5--171.
	
	\bibitem{CAV}
	C. Birkenhake, H. Lange,
	\emph{Complex Abelian Varieties},
	Grundlehren Math. Wiss., \textbf{302}, Springer Verlag, (1992).
	
	\bibitem{BottaDavison}
	T. Botta, B. Davison, \emph{Okounkov's conjecture via BPS Lie algebras}, arXiv:2312.14008 (2023).
	
	\bibitem{Bottini}
	A. Bottini, \emph{Stable sheaves on K3 surfaces via wall-crossing}, Kyoto J. Math. \textbf{64} (2024), no. 2, 459–499.
	
	\bibitem{bozec2020relative}
	T. Bozec, D. Calaque, S. Scherotzke, \emph{Relative critical loci and quiver moduli},
	Ann. Sci. Éc. Norm. Supér. (4) \textbf{57} (2024), no. 2, 553–614.
	
	\bibitem{BBDJS15}
	C. Brav, V. Bussi, D. Dupont, D. Joyce, B. Szendroi, \emph{Symmetries and stabilization for sheaves of vanishing cycles, with an appendix by J\"org Sch\"urmann}, J. Singul.
	\textbf{11} (2015) 85--151.
	
	\bibitem{BD19}
	C. Brav, T. Dyckerhoff, \emph{Relative Calabi-Yau structures}, Compos. Math.
	\textbf{155} (2019), 372--412.
	
	\bibitem{BD21}
	C. Brav, T. Dyckerhoff, \emph{Relative Calabi-Yau structures II: shifted Lagrangians
	in the moduli of objects}, Selecta Math. (N.S.) \textbf{27} (2021).
	
	\bibitem{BrThesis}
	T. Bridgeland, \emph{Fourier--Mukai transforms for elliptic surfaces}, J. Reine Angew. Math. \textbf{498} (1998), 115--133.
	
	\bibitem{Bu25b}
	C. Bu, B. Davison, A. Ib\'a\~nez N\'u\~nez, T. Kinjo, T. P\u{a}durariu, \emph{Cohomology of symmetric stacks},  arXiv:2502.04253,
	(2025).

	\bibitem{Bu25a}
	C. Bu, D. Halpern-Leistner, A. Ib\'a\~nez N\'u\~nez, T. Kinjo, \emph{Intrinsic Donaldson--Thomas theory. I. Component lattices of stacks}, arXiv:2502.13892, (2025).
	
	
	\bibitem{Cal15}
	D. Calaque, \emph{Lagrangian structures on mapping stacks and semi-classical TFTs}, Contemp. Math. \textbf{643} (2015), 1--23.

	\bibitem{Cal19} D. Calaque, \emph{Shifted cotangent stacks are shifted symplectic}, Ann. Fac. Sci. Toulouse Math. (6)  28 (2019),
67--90.

	\bibitem{CS}
	D. Calaque, P. Safronov, \emph{Shifted cotangent bundles, symplectic groupoids and deformation to the normal cone}, arXiv:2407.08622, (2024).
	
	\bibitem{CG}
	K. Costello, O. Gwilliam, \emph{Factorization algebras in quantum field theory, Vol. 1}, New Mathematical Monographs
		\textbf{31}, Cambridge Univ. Press, Cambridge,
		(2017).
	
	\bibitem{BenDimRed}
	B. Davison, \emph{The critical CoHA of a quiver with potential}, Q. J. Math. \textbf{68} (2017), no. 2, 635–703.
	
	\bibitem{DM}
	B. Davison, S. Meinhardt, \emph{Cohomological Donaldson--Thomas theory of a quiver with potential and quantum enveloping algebras}, Invent. Math. \textbf{221} (2020), 777--871.
	
	\bibitem{DavisonKac}
	B. Davison, \emph{The integrality conjecture and the cohomology of preprojective stacks}, J. Reine Angew. Math. \textbf{804} (2023), 105–154.
	
	\bibitem{D24}
	B. Davison, \emph{Purity and 2-Calabi-Yau categories},
	Invent. Math. \textbf{238} (2024), no. 1, 69--173.
	
	\bibitem{Davison_affine_gl_1}
	B. Davison, \emph{Affine BPS algebras, W algebras, and the cohomological Hall algebra of $\mathbb{A}^2$}, in \textit{Homotopical methods in Geometry and Physics}, of Contemp. Math., 841, Amer. Math. Soc., Providence (2026), 163--199.
	
	\bibitem{DHSM1}
	B. Davison, L. Hennecart, S. Schlegel Mejia, \emph{BPS Lie algebras for totally negative 2-Calabi-Yau categories and nonabelian Hodge theory for stacks}, arXiv:2212.07668, (2022).
	
	\bibitem{DHSM2}
	B. Davison, L. Hennecart, S. Schlegel Mejia, \emph{BPS algebras and generalised Kac-Moody algebras from 2-Calabi-Yau categories}, arXiv:2303.12592, (2023).
	

	\bibitem{Des}
	P. Descombes, \emph{Hyperbolic localization in Donaldson--Thomas theory}, arXiv:2506.22400, (2025).	
	
	\bibitem{DPSSV}
	D.-E. Diaconescu, M. Porta, F. Sala, O. Schiffmann, E. Vasserot, \emph{Cohomological Hall algebras of one-dimensional sheaves on surfaces and Yangians}, arXiv:2502.19445, (2025).
	
	\bibitem{gaitsgory2019weil}
	D. Gaitsgory, J. Lurie, \emph{Weil's conjecture for function fields. Vol. 1}, Ann. of Math. Stud., \textbf{199}
	Princeton University Press, Princeton, NJ, 2019.
	
	\bibitem{Gottsche-Soergel}
	L. Göttsche, W. Soergel, \emph{Perverse sheaves and the cohomology of Hilbert schemes of smooth algebraic surfaces}, Math. Ann. \textbf{296} (1993), no. 2, 235–245.
	
	\bibitem{GWZ}
	M. Groechenig, D. Wyss, P. Ziegler, \emph{$\chi$-independence for K3-surfaces via $p$-adic integration}, (2026).

	\bibitem{halpern2014structure}
	D. Halpern-Leistner, \emph{On the structure of instability in moduli theory}, arXiv:1411.0627, (2014).
	
	\bibitem{halpern2016theta}
	D. Halpern-Leistner, \textit{Theta-stratifications, {T}heta-reductive stacks, and applications}, arXiv:1608.04797, (2016).
	
	\bibitem{HauselThaddeus}
	T. Hausel, M. Thaddeus, \emph{ Mirror symmetry, Langlands duality, and the Hitchin system}, Invent. Math. \textbf{153} (2003), no. 1, 197–229.
	
	\bibitem{HK_BPS}
	L. Hennecart, T. Kinjo, \emph{The BPS decomposition theorem}, arXiv:2509.21298, (2025).

	\bibitem{HP18}
	J. Holstein, M. Porta, \emph{Analytification of mapping stacks}, Algebr. Geom. \textbf{12} (2025), 13--83.
	
	\bibitem{Huyb}
	D. Huybrechts, \emph{Compact hyper-K\"ahler manifolds: basic results}, Invent. Math. \textbf{135} (1999), no. 1, 63--113.
	
	\bibitem{K3Global}
	D. Huybrechts, \emph{Lectures on $K3$ surfaces}, Cambridge Stud. Adv. Math. \textbf{158}
	Cambridge University Press, Cambridge, (2016).
	
	\bibitem{HL}
	D. Huybrechts, M. Lehn, \emph{The geometry of moduli spaces of sheaves}, Aspects Math. E31
	Friedr. Vieweg \& Sohn, Braunschweig, (1997).
	
	\bibitem{IN23}
	A. Ib\'anez N\'unez, \emph{Refined Harder-Narasimhan filtrations in moduli theory}, arXiv:2311.18050, (2023).
	
		\bibitem{Shivang}
	S. Jindal, \emph{CoHA of cyclic quivers and an integral form of affine Yangians}, arXiv:2408.02618, (2024).

	
	\bibitem{Jindal-Latyntsev-Sarunas}
	S. Jindal, S. Kaubrys, A. Latyntsev,  \emph{Critical CoHAs, vertex coalgebras and deformed Drinfeld coproducts}, arXiv:2603.21707, (2026).

	\bibitem{Joyce-Safronov}
	D. Joyce, P. Safronov, \emph{A Lagrangian neighbourhood theorem for shifted symplectic derived schemes}, Ann. Fac. Sci. Toulouse Math. (6) 28 (2019), no. 5, 831--908.
	
	\bibitem{KaledinLehnSorger}
	D. Kaledin, M. Lehn, C. Sorger, \emph{Singular symplectic moduli spaces}, Invent. Math. \textbf{164} (2006),  591--614.
	
	\bibitem{KV19}
	M. Kapranov, E. Vasserot, \emph{The cohomological Hall algebra of a surface and factorization cohomology}, 
	J. Eur. Math. Soc. (JEMS) \textbf{25} (2023), 4221--4289.
	
	\bibitem{KS06}
	M. Kashiwara, P. Schapira, \emph{Categories and sheaves}, Grundlehren Math. Wiss., \textbf{352}
	Springer-Verlag, Berlin, (2006).

	\bibitem{KS13}
	M. Kashiwara, P. Schapira, \emph{Sheaves on Manifolds}, Grundlehren Math. Wiss., \textbf{292}
	Springer-Verlag, Berlin, (1990).
	
	\bibitem{keller2011deformed}
	B. Keller, \emph{Deformed Calabi-Yau completions},
	J. Reine Angew. Math. \textbf{654} (2011), 125–180.
	
	\bibitem{K19} 
	A. A. Khan, \emph{Virtual fundamental classes of derived stacks I}, arXiv:1909.01332, (2019).
	
	\bibitem{KhanRavi}
	A. A. Khan, C. Ravi, \emph{Generalized cohomology theories for algebraic stacks}, Adv. Math. \textbf{458} (2024), article no. 109975, 104 pp.
	
	\bibitem{KK}
	A. A. Khan, T. Kinjo, \emph{3D cohomological Hall algebras for local surfaces}, \url{https://www.preschema.com/papers/dimredcoha.pdf}.

	\bibitem{KKPS}
	A. A. Khan, T. Kinjo, H. Park, P. Safronov, \emph{Shifted Lagrangian classes}, to appear.

	\bibitem{KKPSpervpull}
	A. A. Khan, T. Kinjo, H. Park, P. Safronov, \emph{Perverse pullbacks}, arXiv:2510.16563 (2025).

	\bibitem{K}
	T. Kinjo, \emph{Dimensional reduction in cohomological Donaldson--Thomas theory}, Compos. Math. \textbf{158} (2022), 123--167.
	
	\bibitem{kinjo2021virtual}
	T. Kinjo, \emph{Virtual classes via vanishing cycles}, arXiv:2109.06468, (2021).
	
	\bibitem{KinjoKoseki}
	T. Kinjo, N. Koseki, \emph{Cohomological $\chi$-independence for Higgs bundles and Gopakumar--Vafa invariants}, J. Eur. Math. Soc. (JEMS) \textbf{28} (2026), no. 2, 619--671.
	
	\bibitem{KPS}
	T. Kinjo, H. Park, P. Safronov, \emph{Cohomological Hall algebras for 3-Calabi-Yau categories}, arXiv:2406.12838 (2024).	
	
	
	\bibitem{lurie2017higher} 
	J. Lurie, \emph{Higher Algebra}, available at \url{https://www.math.ias.edu/~lurie/} (2017).
	
	\bibitem{lurie2009higher}
	J. Lurie, \emph{Higher Topos Theory}, Ann. of Math. Stud., \textbf{170}, Princeton University Press, Princeton, NJ, (2009). 

	\bibitem{lurie2011homotopy}
	J. Lurie, \emph{Derived Algebraic Geometry XIII: Rational and $p$-adic Homotopy Theory}, (2011).
	
	\bibitem{Markman}
	E. Markman, \emph{Generators of the cohomology ring of moduli spaces of sheaves on symplectic surfaces},	
	J. Reine Angew. Math. \textbf{544} (2002), 61--82.

	\bibitem{MSP=W}
	D. Maulik, J. Shen, \emph{The  $P=W$  conjecture for $GL_n$}, Ann. of Math. (2) \textbf{200} (2024), no. 2, 529–556.
		
	\bibitem{MT19}
	D.~Maulik, R.~Thomas, \emph{Sheaf counting on local K3 surfaces}, Pure Appl. Math. Q. \textbf{14} (2018), no. 3-4, 419--441.
	
	\bibitem{MT18}
	D.~Maulik, Y.~Toda, \emph{Gopakumar--Vafa invariants via vanishing cycles}, Invent. Math. \textbf{213} (2018), 1017--1097.
	
	\bibitem{MMHS}
	A. Mellit, A. Minets, O. Schiffmann, T. Hausel, \emph{$P=W$ via $\mathcal{H}_2$}, arXiv:2209.05429 (2022).
		
	\bibitem{MMSV}
	A. Mellit, A. Minets, O. Schiffmann, E. Vasserot, \emph{Coherent sheaves on surfaces, COHAs and deformed $W_{1+\infty}$-algebras}, arXiv:2311.13415 (2023).
	
	\bibitem{milnor1965structure}
	J. Milnor, J. Moore, \emph{On the structure of Hopf algebras}, Ann. of Math. (2) \textbf{81} (1965), 211--264.
	
	\bibitem{Minets}
	A. Minets, \emph{Cohomological Hall algebras for Higgs torsion sheaves, moduli of triples and sheaves on surfaces}, Selecta Math. (N.S.) \textbf{26} (2020), no. 2,
	Paper No. 30, 67 pp.
	
	\bibitem{MukaiInvent84}
	S. Mukai, \emph{Symplectic structure of the moduli space of sheaves on an Abelian or $K3$ surface}, Invent. Math. \textbf{77} (1984), no. 1, 101–116.
	
	\bibitem{MukaiBombay}
	S. Mukai, \emph{On the moduli space of bundles on K3 surfaces. I.}, in \textit{Vector bundles on algebraic varieties (Bombay, 1984)}, volume 11 of Tata Inst. Fund. Res. Stud. Math., pages 341–413. Bombay, 1987.
	
	\bibitem{N11}
	N. Nitsure, \emph{Schematic Harder-Narasimhan stratification}, Internat. J. Math.
	\textbf{22} (2011), 1365--1373.
	
	\bibitem{OGrady}
	K. O'Grady, \emph{Desingularized moduli spaces of sheaves on a $K3$}, J. Reine Angew. Math. \textbf{512} (1999), 49–117.
	
	\bibitem{Orlov}
	D. Orlov, \emph{Equivalences of derived categories and K3 surfaces}, J. Math. Sci. (N.Y.) \textbf{84} (1997), no. 5, 1361--1381.

	
	\bibitem{PTVV13} 
	T. Pantev, B. To\"en, M. Vaqui\'e, G. Vezzosi, \emph{Shifted symplectic structures}, Publ. Math. Inst. Hautes \'Etudes Sci.
	\textbf{117} (2013),  271--328.
	
	\bibitem{P19}
	M. Porta, \emph{GAGA theorems in derived complex geometry}, J. Algebraic Geom. \textbf{28}, no. 3, 519--565
	(2019).
	
	\bibitem{PortaSala}
	M. Porta, F. Sala, \emph{Two dimensional categorified Hall algebras}, J. Eur. Math. Soc. (JEMS) \textbf{25} (2023), no. 3, 1113--1205.
	
	\bibitem{Roz}
	N. Rozenblyum, \emph{Connections on moduli spaces and infinitesimal Hecke modifications}, arXiv:2108.07745 (2021).

	\bibitem{Rydh}
	D. Rydh, Families of cycles, \url{https://people.kth.se/~dary/thesis/thesis-paperIV.pdf}.
	
	\bibitem{SaintDonat}
	B. Saint-Donat, \emph{Projective models of $K3$ surfaces}, Amer. J. Math. \textbf{96} (1974), 602–639.
	
	\bibitem{SS}
	F. Sala, O. Schiffmann, \emph{Cohomological Hall algebra of Higgs sheaves on a curve}, Algebr. Geom. \textbf{7} (2020), 346--376.
	
	\bibitem{Schefers22} 
	K. Schefers, \emph{An equivalence between vanishing cycles and microlocalization}, arXiv:2205.12436 (2022).
	
	\bibitem{SchiffmannKac}
	O. Schiffmann, \emph{ Indecomposable vector bundles and stable Higgs bundles over smooth projective curves},
	Ann. of Math. (2) \textbf{183} (2016), no. 1, 297–362.
	
	\bibitem{SVIHES}
	O. Schiffmann, E. Vasserot, \emph{Cherednik algebras, W-algebras and the equivariant cohomology of the moduli space of instantons on $\mathbb{A}^2$}, Publ. Math. Inst. Hautes Études Sci. \textbf{118} (2013), 213–342.
	
	\bibitem{SVgenerators}
	O. Schiffmann, E. Vasserot, \emph{On cohomological Hall algebras of quivers: generators}, J. Reine Angew. Math. \textbf{760} (2020), 59–132.
	

	\bibitem{STV}
	T. Schürg, B. Toën, G. Vezzosi, \emph{Derived algebraic geometry, determinants of perfect complexes, and applications to obstruction theories for maps and complexes}, J. Reine Angew. Math. \textbf{702} (2015), 1--40.
		
	\bibitem{TanakaThomas}
	Y. Tanaka, R. Thomas, \emph{Vafa--Witten invariants for projective surfaces I: stable case}, 	J. Algebraic Geom. \textbf{29} (2020), 603–668.

	\bibitem{Th00}
	R.~Thomas, \emph{A holomorphic Casson invariant for Calabi--Yau 3-folds,
              and bundles on $K3$ fibrations}, J. Differential Geom. \textbf{54} (2000), 367--438.
	
	\bibitem{T23}
	Y. Toda, \emph{Gopakumar--Vafa invariants and wall-crossing}, J. Differential Geom. \textbf{123} (2023), 141--193.
	
	\bibitem{TV07} 
	B. To\"en, M. Vaqui\'e, \emph{Moduli of objects in dg-categories}, Ann. Sci. \'Ec. Norm. Sup\'er. (4) \textbf{40} (2007), no. 3, 387–444.
	
	\bibitem{Tubstack}
	S. Tubach, \emph{Mixed Hodge modules on stacks}, arXiv:2407.02256 (2024).
	
	\bibitem{Yang-Zhao}
	Y. Yang, G. Zhao, \emph{ The cohomological Hall algebra of a preprojective algebra}, Proc. Lond. Math. Soc. (3) \textbf{116} (2018), no. 5, 1029–1074.
	
	\bibitem{Yoshioka}
	K. Yoshioka, \emph{Moduli spaces of stable sheaves on Abelian surfaces}, Math. Ann. \textbf{321} (2001), 817--884.
	
	\bibitem{Yoshioka2}
	K. Yoshioka, \emph{Chamber structure of polarizations and the moduli of stable sheaves on a
	ruled surface}, Internat. J. Math. \textbf{7} (1996), no. 3, 411–431. 
	
	\bibitem{Yoshioka3}
	K. Yoshioka, \emph{Twisted stability and Fourier--Mukai transform II}, Manuscripta Math. \textbf{110} (2003), no. 4, 433--465.
	
\end{thebibliography}
\end{document}